\documentclass[11pt]{amsart}
\usepackage[a4paper,margin=28mm]{geometry}
\usepackage{amsmath,amssymb,amsthm,mathrsfs,mathtools,booktabs,enumitem,slashed,tikz,array}
\usepackage[colorlinks=true,linkcolor=blue!50!black,citecolor=green!40!black,urlcolor=blue!50!black]{hyperref}
\makeatletter
\def\paragraph{\@startsection{paragraph}{4}{\z@}{.5\linespacing\@plus.2\linespacing}{-.5em}{\normalfont\itshape}}
\makeatother
\usetikzlibrary{arrows.meta,calc,decorations.pathreplacing,positioning}
\numberwithin{equation}{section}
\newtheorem{theorem}{Theorem}[section]
\newtheorem{lemma}[theorem]{Lemma}
\newtheorem{proposition}[theorem]{Proposition}
\newtheorem{corollary}[theorem]{Corollary}
\newtheorem{maintheorem}{Theorem}

\theoremstyle{definition}
\newtheorem{definition}[theorem]{Definition}

\theoremstyle{remark}
\newtheorem{remark}[theorem]{Remark}

\newcommand{\gd}{\slashed g}
\newcommand{\gdc}{\slashed g_c}
\newcommand{\gdcH}{\slashed g_{c,H}}
\newcommand{\Nd}{\slashed\nabla}
\newcommand{\Div}{\operatorname{div}}
\newcommand{\Curl}{\operatorname{curl}}
\newcommand{\tr}{\operatorname{tr}}
\newcommand{\chih}{\hat\chi}
\newcommand{\chib}{\underline\chi}
\newcommand{\chibh}{\hat{\underline\chi}}
\newcommand{\etab}{\underline\eta}
\newcommand{\omb}{\underline\omega}
\newcommand{\ab}{\underline\alpha}
\newcommand{\bb}{\underline\beta}
\newcommand{\Lb}{\underline L}
\newcommand{\Db}{\underline D}
\newcommand{\ohat}{\mathbin{\widehat\otimes}}
\newcommand{\gac}{\mathring\gamma}
\newcommand{\av}[1]{\langle #1\rangle}
\newcommand{\Ecal}{\mathcal E}
\newcommand{\Bcal}{\mathcal B}
\newcommand{\Ccal}{\mathcal C}
\newcommand{\Pcal}{\mathcal P}
\newcommand{\Fcal}{\mathcal F}
\newcommand{\Kcal}{\mathcal K}
\newcommand{\Jcal}{\mathcal J}
\newcommand{\Hcal}{\mathcal H}
\newcommand{\CK}{\mathrm{CK}}
\newcommand{\R}{\mathbb R}
\newcommand{\C}{\mathbb C}
\newcommand{\Z}{\mathbb Z}
\newcommand{\bg}{\mathbf g}
\newcommand{\bR}{\mathbf R}
\newcommand{\sym}{\operatorname{sym}}
\newcommand{\skw}{\operatorname{skew}}

\newcommand{\Ric}{\operatorname{Ric}}
\newcommand{\BH}{\mathcal{BH}}
\newcommand{\scri}{\mathcal I}
\newcommand{\loc}{\mathrm{loc}}
\renewcommand{\Re}{\operatorname{Re}}
\renewcommand{\Im}{\operatorname{Im}}

\begin{document}

\title[Characteristic gluing to extremal Kerr and the third law]{Characteristic gluing to extremal Kerr horizons\\ and the third law of black hole thermodynamics \\ in vacuum}

\author{Stefan Czimek}
\address{Mathematisches Institut, Universit\"at Leipzig, Augustusplatz 10, 04109 Leipzig, Germany}
\email{stefan.czimek@uni-leipzig.de}

\begin{abstract} 
We prove a non-perturbative, non-axisymmetric characteristic gluing theorem for the Einstein vacuum equations: every Boyer–Lindquist sphere in the domain of outer communication of a Kerr black hole, in particular every symmetry sphere of a Schwarzschild exterior, can be glued along a single outgoing null hypersurface to a horizon sphere of an extremal Kerr black hole of sufficiently large mass. The gluing is exact at both ends and matches the full sphere data, including the transversal curvature. This proves a conjecture of Kehle and Unger in the extremal case. We first raise the angular momentum to the extremal value by rotating shear pulses of non-axisymmetric gravitational radiation, and then damp the incoming data exponentially along a long non-expanding neck. At the degenerate target horizon we invert the linearised problem at all angular frequencies, using a factorisation of a fourth-order angular operator into an angular Teukolsky–Starobinsky product.

As an application, we construct one-ended asymptotically flat vacuum initial data whose maximal development contains an exactly Schwarzschild apparent horizon and, at a later advanced time, an exactly extremal Kerr event horizon. This is the first rigorous counterexample to the third law of black hole thermodynamics in vacuum, and proves a second conjecture of Kehle and Unger.
\end{abstract}

\maketitle
\tableofcontents

\section{Introduction}\label{ch:intro}

\subsection{Overview}\label{sec:overview}
This paper is about characteristic gluing for the vacuum Einstein equations far from the perturbative regime, and its application to
the third law of black hole thermodynamics.

Our main result is a characteristic gluing theorem (Theorem~\ref{thm:A}): every Boyer--Lindquist sphere outside a Kerr black hole, in
particular every symmetry sphere of a Schwarzschild exterior, can be glued along a single outgoing null hypersurface to a horizon
sphere of an extremal Kerr black hole of sufficiently large mass, with the full sphere data matched exactly at both ends. Its main
features are the following.
\begin{itemize}
\item \emph{It is nonperturbative.} The data are large, and no step is a perturbation of Minkowski space or of another small
solution. The perturbative arguments are made around two explicit large structures: an exact partial bridge which reaches the
intrinsic horizon data with zero defect, and the extremal horizon itself.
\item \emph{It is non-axisymmetric.} In vacuum, angular momentum can only be carried by non-axisymmetric gravitational radiation,
and here it has to be raised all the way to the extremal value $J=M_f^2$. This is done by rotating pulses of shear whose
contribution to the angular momentum is computed exactly.
\item \emph{The target is degenerate.} At the extremal horizon the linearised matching problem for the transversal data is an
isomorphism at all angular frequencies, by a factorisation of the relevant fourth-order angular operator through an angular
Teukolsky--Starobinsky product. Whether the terminal sphere is an actual horizon section is decided by the sign of a boost, which
is computed explicitly from an exact trace identity.
\item \emph{The transversal data are matched.} The data agree at the target up to the transversal curvature $\ab$, so the realised
spacetime is $C^2$ across the gluing interfaces. This is what allows the construction to be continued by an exactly extremal Kerr
exterior in Theorem~\ref{thm:B}.
\end{itemize}

As an application we construct one-ended asymptotically flat vacuum initial data whose maximal development first contains an
exactly Schwarzschild apparent horizon and later an exactly extremal Kerr event horizon (Theorem~\ref{thm:B}). This disproves the
third law of black hole thermodynamics in vacuum, as conjectured by Kehle and Unger \cite[Conjecture~2]{KU}. The global part of this
deduction, namely the collapse, the trapped surfaces and the identification of the event horizon, follows Kehle and Unger
\cite{KU22,KU}; the new input is Theorem~\ref{thm:A}. The strategy of the proof is outlined in Section~\ref{sec:strategy}.

\subsection{Characteristic gluing}\label{sec:chargluing}
Let $(\mathcal M,\bg)$ be a solution of the vacuum Einstein equations $\Ric(\bg)=0$, and let $\Hcal$ be an outgoing null
hypersurface foliated by spacelike spheres $S_v$. The Einstein equations restricted to $\Hcal$ are constraint equations;
they are ordinary differential equations along the null generators of $\Hcal$, in which angular derivatives appear only as
coefficients. Conversely, constraint-satisfying data on two transversally intersecting null hypersurfaces determine a vacuum
solution in a neighbourhood \cite{Rendall,Luk}. The relevant notion of data on a sphere $S$ is that of \emph{sphere data}
\cite[Definition~2.4]{ACR},
\begin{equation}\label{eq:x}
 x=(\Omega,\gd,\Omega\tr\chi,\chih,\Omega\tr\chib,\chibh,\eta,\omega,D\omega,\omb,\Db\omb,\alpha,\ab),
\end{equation}
consisting of the metric, the two null second fundamental forms, the torsion, the lapse and its derivatives, and the two
extreme null curvature components (see Section~\ref{ch:null} for definitions). Sphere data encode the metric two-jet at $S$
(see also Section~\ref{sec:twojet}); two spacetimes inducing the same sphere data on $S$ can therefore be pasted together in a $C^2$ way.
\emph{Characteristic gluing} asks whether, for given sphere data $x_1$ on $S_1$ and $x_2$ on $S_2$, there is a solution of the null
constraint equations along a null hypersurface from $S_1$ to $S_2$ inducing $x_1$ and $x_2$.\\

\emph{Perturbative gluing.} Characteristic gluing for the Einstein equations was initiated by Aretakis, Rodnianski and the author
\cite{ACRann,ACR,ACRkerr}, who proved a perturbative gluing theorem for the vacuum Einstein equations near Minkowski space, modulo a
ten-dimensional space of conserved charges arising from the linearised null constraint equations at Minkowski. These linearised charges
at Minkowski correspond to energy, linear momentum, angular momentum and centre of mass, and by relating them to the ADM quantities,
asymptotically flat data can be characteristically glued to the data of a suitably chosen Kerr spacetime \cite{ACRkerr}. Subsequently,
Rodnianski and the author \cite{CR} developed an obstruction-free gluing method, which removes these obstructions by exploiting quadratic
terms of the null constraint equations through a high-frequency ansatz. Characteristic gluing has also been extended to nonzero
cosmological constant, higher dimensions and higher differentiability. Chru\'sciel and Cong \cite{CC} and Cong, Chru\'sciel and Gray
\cite{CCG2} treat the linearised problem near static vacuum backgrounds with cosmological constant, Chru\'sciel, Cong and Gray
\cite{CCG3} prove a $C^k$-gluing theorem for every finite $k$ near Birmingham--Kottler backgrounds, and Chru\'sciel and Cong
\cite{CCvar} study gluing, embedding and extension of characteristic data in all dimensions. Sansom \cite{Sansom} solves the
$C^3$-gluing problem near Minkowski space up to a twenty-dimensional space of obstructions, ten of which are new at third order.\\

\emph{Nonperturbative gluing and extremal black holes.} In spherical symmetry, a complementary, nonperturbative ansatz was introduced by
Kehle and Unger \cite{KU22} for characteristic gluing of the Einstein--Maxwell--charged scalar field system, where it leads to the
formation of extremal Reissner--Nordstr\"om black holes in finite time: they glue a Schwarzschild sphere to a horizon sphere of extremal
Reissner--Nordstr\"om, and obtain gravitational collapse spacetimes, arising from regular one-ended asymptotically flat initial data,
which first contain an exactly Schwarzschild apparent horizon and later, after a finite advanced time, an exactly extremal
Reissner--Nordstr\"om event horizon. Marin \cite{Marin} extends this construction to nonzero cosmological constant. Gadioux, Reall and
Santos \cite{GRS} implement the gluing numerically for several ans\"atze for the scalar field at the horizon, and find that it succeeds
only if the final black hole is sufficiently massive, with a threshold that depends on the ansatz. Schneider \cite{Schneider} shows that
such collapse to an extremal Reissner--Nordstr\"om black hole of radius $r_+$ requires the charge $\mathfrak e$ of the scalar field to
satisfy $|\mathfrak e|r_+>\frac13$, and that a Schwarzschild black hole can become extremal if $|\mathfrak e|r_+>\sqrt{(3+\sqrt{33})/3}$.

In vacuum, Kehle and Unger \cite{KU} combined the obstruction-free null gluing of \cite{CR} with a short pulse closely related to Christodoulou
\cite{Chr} to glue Minkowski space to any round symmetry sphere of Schwarzschild and further to the event horizon of very slowly rotating Kerr,
with applications to gravitational collapse. They conjectured that a Schwarzschild symmetry sphere can be characteristically glued to a
sphere in Kerr of large final mass $M_f$ and arbitrary $|a_f|\le M_f$ \cite[Conjecture~1]{KU}. Numerically, the formation of extremal rotating black holes in five-dimensional vacuum gravity was recently shown by Crump, Gadioux, Reall and Santos \cite{CGRS}.\\

\emph{The third law of black hole thermodynamics.} A Kerr black hole with mass $M$ and angular momentum $J=aM$ is \emph{extremal} if
$|a|=M$, that is, if its surface gravity vanishes. The third law of black hole mechanics \cite{BCH}, in the precise formulation of
Israel \cite{Israel}, asserts that a non-extremal black hole cannot become extremal in finite advanced time. The constructions of Kehle-Unger
\cite{KU22} show that it is false for the Einstein--Maxwell--charged scalar field system; Israel's argument does not apply because it
implicitly assumes that the outermost apparent horizon is connected \cite[Section~1.4.3]{KU22}. Kehle and Unger conjectured that the
third law is also false in vacuum (see Conjecture~2 in \cite{KU}), and explained that Conjecture~1 of \cite{KU}, for target spheres on the event
horizon of extremal Kerr, implies this by the arguments of \cite{KU22} and \cite[Corollary~1]{KU}. 

For other directions, we refer to \cite{KUcrit}, where extremal black holes arise at the threshold of gravitational collapse for the Einstein--Maxwell--Vlasov system; to
the third laws of Reall \cite{Reall}, McSharry and Reall \cite{McSharryReall}, and McSharry, Reall and Santos \cite{MRS} for
supersymmetric black holes and for the extremal BTZ black hole, under energy conditions on the matter.\\

\emph{The present problem.} In the present paper we prove Conjecture 1 of \cite{KU} for target spheres on the event horizon of extremal
Kerr, for arbitrary Boyer--Lindquist spheres outside a Kerr black hole as initial spheres, and deduce that the third law fails in vacuum.
Target spheres in sub-extremal Kerr, $|a_f|<M_f$, are not treated here: the matching at the target uses the degeneracy of the
extremal horizon, which leaves the incoming second fundamental form of a horizon section undetermined by its intrinsic geometry
(Remark~\ref{rem:nonextremal}).
This gives the first rigorous counterexample to the third law in vacuum gravity. The problem combines the difficulties of the settings
above: the data are large, the target is a degenerate horizon, and there is no symmetry that reduces the equations to ordinary
differential equations. In spherical symmetry the charge can be increased by charged matter. In vacuum the only way to reach extremal
Kerr is to spin the black hole up to $J=M_f^2$, and angular momentum can only be carried by \emph{non-axisymmetric} gravitational
radiation: in axisymmetry the Komar angular momentum is conserved (Corollary~\ref{cor:komar} is the null-hypersurface version of this).
The construction is therefore genuinely non-symmetric, a problem for partial differential equations on the sphere.

\subsection{Main results}
Throughout, $\mathcal R_\pi:\varphi\mapsto\varphi+\pi$ denotes the half-turn about the axis.

\begin{maintheorem}[Gluing a Boyer--Lindquist sphere to an extremal Kerr horizon sphere]\label{thm:A}
Let $|a_1|\le M_1$, and let $S_1=\{t=t_1,\ r=r_1\}$, $r_1>r_+(M_1,a_1)$, be a Boyer--Lindquist sphere in the exterior of the
Kerr solution with parameters $(M_1,a_1)$. Let $L_1$ be an axisymmetric future directed outgoing null normal of $S_1$, and let $x_1$
be the sphere data of $S_1$ generated by $L_1$ as in \cite[Lemma~A.1]{KU}, that is, in the local double null foliation near $S_1$ with
$\Omega^2=1$ on the bifurcate null hypersurface through $S_1$. Fix an integer $s\ge14$. There is $M_0=M_0(S_1,s)$ such that for every $M_f\ge M_0$ there is a
solution $v\mapsto x(v)$, $v\in[v_1,v_3]$, of the null structure and null Bianchi equations along an outgoing null hypersurface
with the following properties.
\begin{enumerate}[label=\textup{(\roman*)}]
\item \emph{Start.} $x(v_1)=x_1$; indeed the data agree with those of the outgoing Kerr cone through $S_1$ near $v_1$, in this
foliation.
\item \emph{End.} Let $\mathcal H_f$ be the future event horizon of the extremal Kerr solution with mass $M_f$ and $a_f=M_f$,
in ingoing Kerr coordinates $(v_K,r,\vartheta,\varphi)$, and let $S_3=\{v_K=v_3,\ r=M_f\}\subset\mathcal H_f$.
Near $v_3$ the data agree with the extremal Kerr horizon data, with generator field $L=\xi=\partial_{v_K}+\frac1{2M_f}\partial_\varphi$
and cuts $\{v_K=v\}$. In particular
\[
 x(v_3)=x_{S_3},
\]
the sphere data of $S_3$ in the local double null foliation near $S_3$ determined by the null normal $\xi$ and the transversal gauge
of Section~\ref{sec:gauge}.
\item \emph{Symmetry.} All data are invariant under $\mathcal R_\pi$.
\item \emph{Regularity.} The angular regularity is $H^s$ for $\gd$, $\Omega\tr\chi$, $\chih$, $H^{s-1}$ for the torsion,
$H^{s-2}$ for $\Omega\tr\chib$, $\chibh$, and $H^{s-4}$ for $\ab$, uniformly in $v$, and the data are smooth in $v$.
\end{enumerate}
Moreover, the data are realised by a vacuum spacetime in a neighbourhood of the null hypersurface, which is one-sided $C^4$ up
to the hypersurface on each side and $C^2$ across it.
\end{maintheorem}

In particular, every symmetry sphere in a Schwarzschild exterior can be glued along an outgoing null hypersurface to a horizon
sphere of an extremal Kerr black hole of sufficiently large mass. Most of the paper proves a version in which the target is a
section of the horizon determined by the construction.

\begin{theorem}[Gluing to a horizon section]\label{thm:cutversion}
Under the hypotheses of Theorem~\ref{thm:A}, for every $M_f\ge M_0$ there is a solution $v\mapsto x(v)$, $v\in[v_1,v_2]$,
satisfying \textup{(i)}, \textup{(iii)}, \textup{(iv)} of Theorem~\ref{thm:A}, and
\begin{enumerate}[label=\textup{(\roman*$'$)},start=2]
\item there are a spacelike section $S_2$ of $\mathcal H_f$, a constant $\lambda>0$ and a choice of transversal gauge near $S_2$
such that $x(v_2)=\mathfrak b_\lambda(x_{S_2})$, and near $v_2$ the data agree with the extremal Kerr horizon data.
\end{enumerate}
\end{theorem}

Here $\mathfrak b_\lambda$ is a boost, rescaling the null normals by
$L\mapsto\lambda^{-1}L$, $\Lb\mapsto\lambda\Lb$. The section $S_2$ and the boost $\lambda$ are part of the solution. Once the
data agree with extremal Kerr horizon data near a section $S_2$, the null hypersurface can be continued along the extremal
horizon itself, and the parametrisation of the generators and the foliation can be chosen to interpolate to the Killing
parametrisation with the coordinate spheres; this absorbs the shape of $S_2$ and the boost (Section~\ref{sec:horizonsphere}).
The positivity of $\lambda$ is essential: for $\lambda\le0$ the formulas at $S_2$ still make sense, but they are not the data of
a horizon section with future-directed normals, and there is no horizon along which to continue.

\begin{maintheorem}[The third law fails in vacuum]\label{thm:B}
There exist one-ended asymptotically flat vacuum initial data $(\Sigma,g_0,k_0)$, $\Sigma\cong\R^3$,
$(g_0,k_0)\in H^{7/2-}_{\loc}\times H^{5/2-}_{\loc}$, whose maximal globally hyperbolic development first contains an exactly
Schwarzschild apparent horizon and later an exactly extremal Kerr event horizon. The initial hypersurface lies in the causal past
of future null infinity, $\Sigma\subset J^-(\scri^+)$; in particular, $\Sigma$ is disjoint from the event horizon and contains no
trapped surfaces. More precisely, the black hole region contains
an exact Schwarzschild double null rectangle with an edge on $\{r=2M_1\}$, whose sections are marginally trapped and are
approached from the inside by trapped spheres, and, after a sufficiently late Cauchy level, the event horizon and the whole
escaping region coincide with those of an extremal Kerr solution. In particular, \cite[Conjecture~2]{KU} holds.
\end{maintheorem}

Theorem~\ref{thm:B} is the first rigorous counterexample to the third law in vacuum gravity. The precise statement is Theorem~\ref{thm:thirdlaw}. The metric is $C^2$ across the characteristic gluing interfaces, but need
not be $C^2$ at the central crease of the construction, as in \cite[Remark~4.1]{KU}. The deduction of Theorem~\ref{thm:B} from
Theorem~\ref{thm:A} follows the two-stage assembly of Kehle and Unger \cite[Section~5.3]{KU22}, \cite[Section~4]{KU}; we do not
apply their small-rotation result at $a=M$, but rerun the argument and verify every place where small rotation or spherical
symmetry was used (Section~\ref{ch:thirdlaw}).

\subsection{The difficulties}
\emph{Large data.} We start at a Kerr sphere of mass $M_1$ and end at an extremal Kerr horizon of mass $M_f\gg M_1$. Nothing is
close to Minkowski space, so there is no implicit function theorem around a trivial solution, and the ten-charge picture of
\cite{ACR} does not apply directly.

\emph{The target is a degenerate horizon.} An extremal horizon section has vanishing outgoing expansion and shear, and the
Killing generator is affinely parametrised. The transversal data $\Omega\tr\chib$, $\chibh$, $\ab$ at the target are
infinitely many conditions, which must all be matched, together with the angular momentum.

\emph{A sign.} The terminal section and the boost $\lambda$ are free parameters used to absorb two components of the endpoint
map. The natural parametrisation of the target extends smoothly through $\lambda=0$ to negative values, where it no longer
describes a horizon section, so the construction must decide the sign of $\lambda$.

\subsection{Strategy of the proof}\label{sec:strategy}

\begin{figure}[ht]
\centering
\begin{tikzpicture}[scale=0.9]
 \draw[thick] (0,0) -- (14,0);
 \foreach \x in {0,5.2,6.4,11.4,14}{\draw (\x,-0.15)--(\x,0.15);}
 \node[left] at (0,0) {$S_1$};
 \node[right] at (14,0) {$S_2$};
 \draw[decorate,decoration={brace,amplitude=6pt}] (0,0.3) -- (5.2,0.3) node[midway,above=6pt] {\small early bridge};
 \draw[decorate,decoration={brace,amplitude=6pt}] (5.2,0.3) -- (6.4,0.3) node[midway,above=6pt] {\small tail};
 \draw[decorate,decoration={brace,amplitude=6pt}] (6.4,0.3) -- (11.4,0.3) node[midway,above=6pt] {\small neck, length $T$};
 \draw[decorate,decoration={brace,amplitude=6pt}] (11.4,0.3) -- (14,0.3) node[midway,above=6pt] {\small collar};
 \node[below,align=center] at (2.6,-0.25) {\scriptsize Kerr germ; area raised;\\ \scriptsize spin pulses; controller};
 \node[below,align=center] at (5.8,-0.25) {\scriptsize focusing\\ \scriptsize to $\tr\chi=0$};
 \node[below,align=center] at (8.9,-0.25) {\scriptsize $\gdc=\gdcH$, $\omega$ constant,\\ \scriptsize incoming data decay};
 \node[below,align=center] at (12.7,-0.25) {\scriptsize small seed $k$\\ \scriptsize supported in $(0,1)$};
\end{tikzpicture}
\caption{Layout of the outgoing null hypersurface in the parameter $v$. The data are exactly Kerr near $S_1$. On the neck and
near $S_2$ the intrinsic data are exactly those of the extremal horizon.}
\label{fig:bridge}
\end{figure}
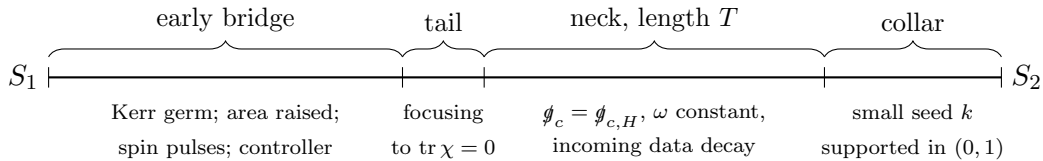

The proof of Theorem~\ref{thm:cutversion} has six steps (Figure~\ref{fig:bridge}).
\begin{enumerate}[label=\textup{Step~\arabic*.},leftmargin=4.2em]
\item \emph{A large partial bridge} (Sections~\ref{ch:tail} and~\ref{ch:spin}). Starting from $S_1$ we construct a solution of
the null constraint equations whose terminal sphere carries exactly the \emph{intrinsic} data of an extremal horizon section:
the metric, vanishing outgoing second fundamental form, and the horizon torsion. This includes the angular momentum $J=M_f^2$,
which is pumped in by non-axisymmetric rotating pulses of shear. The bridge ends with a \emph{focusing tail}: along an explicit
path of conformal metrics from the round one to that of the horizon, Raychaudhuri's equation is solved backwards from the exact
terminal data, so that the expansion vanishes at exactly the right area and at every angle simultaneously, while it stays positive
at the start of the tail. The transversal data at the end are not controlled.
\item \emph{Local controllability} (Section~\ref{ch:controller}). Additional pulses in the control window before the focusing tail make the map from seed perturbations to the
terminal area radius, expansion and torsion locally surjective with a smooth right inverse.
\item \emph{A long neck} (Section~\ref{ch:neck}). The hypersurface is continued along an exactly non-expanding segment of
length $T$ on which the generator has constant non-affinity $2\omega>0$. Along such a segment the incoming data decay like
$e^{-2\omega v}$ and $e^{-4\omega v}$, so the neck damps the uncontrolled transversal data of Steps~1 and~2 by
$O(e^{-2\omega T})$.
\item \emph{A late collar and a linear isomorphism} (Section~\ref{ch:linear}). On a final unit interval we insert a small
seed. To first order its effect on the incoming data is given by two radial moments, through angular operators $E$ and
$\Pcal=\Bcal E$. Together with the shape of the terminal section and the boost this gives an isomorphism at all angular
frequencies; the key is a factorisation of $\Pcal$ through an angular Teukolsky--Starobinsky product.
\item \emph{The nonlinear endpoint equation} (Section~\ref{ch:fixedpoint}). A contraction argument with constants independent
of $T$ solves the full matching problem, with error $O(e^{-2\omega T}+\varepsilon^2)$, where $\varepsilon$ is the amplitude of a
fixed seed direction $k_*$ invisible to the linearisation.
\item \emph{The sign of the boost} (Section~\ref{ch:sign}). An exact trace identity expresses $\lambda$ through a weighted
time integral of a scalar functional $\Fcal$ of the metric and the torsion. The functional is critical at the horizon, and its
second variation along $k_*$ can be computed explicitly and made negative. Hence
$\lambda=c_*\varepsilon^2+O(\varepsilon^3+e^{-2\omega T})$ with $c_*>0$.
\end{enumerate}
Finally (Section~\ref{ch:realisation}) the matched geometric quantities are completed to full sphere data by a choice of
transversal gauge, the hypersurface is continued along the extremal horizon from $S_2$ to a horizon sphere $S_3$, and the data
are realised in a vacuum spacetime by Luk's theorem \cite{Luk}.

\emph{Explicit structures at both ends.} Although the result concerns large data, no step is perturbative around Minkowski space or
around any other small solution. The perturbative arguments are carried out near two explicit structures, one at each end of the
hypersurface, and the neck connects them.

At the beginning, the partial bridge of Step~1 is a large but \emph{exact} solution: it reaches the intrinsic horizon data with zero
defect. It cannot be axisymmetric, since in axisymmetry the angular momentum is conserved (Corollary~\ref{cor:komar}), but it is
designed so that every quantity entering the equations nonlinearly is axisymmetric. The shear energy sees only the speed of the
distortion, not its phase, so pulses of a single helicity with disjoint supports have axisymmetric energy. Then the area radius and the
lapse depend only on $(v,z)$, and Raychaudhuri's equation reduces to ordinary differential equations. By the triangular structure of
the hierarchy, the non-axisymmetric information enters only the torsion, through a linear equation with explicit sources. There the
angular momentum is pumped in by closed rotating pulses, whose skew memory is computed exactly in the hyperbolic plane of unimodular
$2\times2$ matrices; the axial torsion profile is matched by an explicitly invertible ordinary differential equation; and all remaining
non-axisymmetric memories cancel exactly, because the area is, by design, exactly affine on the control window. The discrete symmetry $\mathcal R_\pi$
keeps only the modes $e^{\pm2i\varphi}$ and reduces the obstructions for the controller of Step~2 from six to two. Since the reference
bridge has zero defect, the controller is an inverse function theorem around a large exact solution, applied to the small changes of
the targets caused by the late seed.

At the other end, the extremal Kerr horizon is explicit because of its own algebraic structure: the near-horizon equation, exact
formulas for all sections and boosts (from the Killing structure and the Petrov type~D of Kerr), weight identities, and an exact
helicity calculus in which $\Pcal$ factorises into Teukolsky--Starobinsky operators. This makes the linearisation of Step~4 an
isomorphism at all angular frequencies and the sign in Step~6 computable. Finally, the neck makes the interaction between the two ends
small: the incoming data at the end of the early bridge are of size one and uncontrolled, but they reach the collar only through a
defect of size $O(e^{-2\omega T})$. This dictates the order of choices: first the large bridge, then $\varepsilon$, then $T$.

\subsection{Two mechanisms}\label{sec:toy}
Two ideas of the proof can be seen in finite dimensions.

\emph{(i) A damped equation forgets its initial data.} For the scalar equation $y'=-2\omega y+f$ on $[-T,0]$, $\omega>0$,
\[
 y(0)=e^{-2\omega T}y(-T)+\int_{-T}^0e^{2\omega v'}f(v')\,dv'.
\]
The initial value enters with the factor $e^{-2\omega T}$, and the kernel $e^{2\omega v'}$ has $L^1$-norm at most
$(2\omega)^{-1}$ uniformly in $T$, so the solution map $f\mapsto y(0)$ is bounded independently of $T$, without a growing Gronwall
factor. On the neck the incoming second fundamental form satisfies exactly such an equation,
$D(\Omega^{-1}\chib)=-2\omega\,\Omega^{-1}\chib+(\text{source})$, and the source vanishes on the horizon by the
near-horizon equation (Proposition~\ref{prop:NH}).

\emph{(ii) The sign of a parameter is decided by a second variation.} Let $\mathcal A:X\to Y$ be a bounded isomorphism onto a
closed subspace of codimension one, $\Lambda\in Y^*$ with $\ker\Lambda=\operatorname{ran}\mathcal A$, $y_0\in Y$ with
$\Lambda(y_0)\ne0$, and consider $\mathcal Ax+\lambda y_0+\mathcal N(x,\varepsilon)=0$ with a smooth nonlinearity satisfying
$\mathcal N(0,0)=0$, $D\mathcal N(0,0)=0$. Applying $\Lambda$ gives
$\lambda=-\Lambda(\mathcal N(x,\varepsilon))/\Lambda(y_0)$. If $x=O(\varepsilon^2)$ and
$\mathcal N(0,\varepsilon)=\varepsilon^2q+O(\varepsilon^3)$, then
$\lambda=-\varepsilon^2\Lambda(q)/\Lambda(y_0)+O(\varepsilon^3)$, so the sign of $\lambda$ is determined by the single number
$\Lambda(q)$. In the proof, $\mathcal A$ is the linearised endpoint map, $\Lambda$ is a weighted trace
(Corollary~\ref{cor:range}), $y_0$ is the incoming second fundamental form of the horizon, applying $\Lambda$ is the exact trace
identity of Proposition~\ref{prop:trace}, and $\Lambda(q)$ is the second variation of $\Fcal$.

\subsection{Outline}
Section~\ref{ch:null} derives the null constraint hierarchy along one null hypersurface, and Section~\ref{ch:optical} rewrites
its first two levels in terms of the area radius, the conformal metric and the torsion. Section~\ref{ch:kerr} computes the
data of the extremal Kerr horizon and of all its sections. Sections~\ref{ch:tail}--\ref{ch:sign} carry out Steps~1--6, and
Section~\ref{ch:realisation} completes the proofs of Theorems~\ref{thm:A} and~\ref{thm:cutversion}. Section~\ref{ch:thirdlaw}
proves Theorem~\ref{thm:B}. Appendix~\ref{app:sphere} collects the facts about tensors on the sphere and their smoothness at
the poles that are used in the construction of the seeds, and Appendix~\ref{app:cutoff} proves a cutoff lemma around a non-flat background used in Section~\ref{ch:thirdlaw}.

\subsection{Use of AI tools}\label{sec:aiusage}

The author used the large language models ChatGPT (OpenAI) and Claude (Anthropic) for exploratory calculations and numerical simulations. The models also contributed technical observations that informed the development of the proofs. They were also used to assist with literature searches and exposition. The author takes full responsibility for the mathematical results of this paper.

\section{Null hypersurfaces and the constraint hierarchy}\label{ch:null}

In this section we fix conventions and state the equations that hold along one outgoing null hypersurface of a
vacuum spacetime, in the renormalised form used throughout the paper. All statements in the proof of the gluing theorem
concern solutions of these equations. They are the null structure and null Bianchi equations of \cite{ACR,Chr}, to which we
refer for their derivation.

\subsection{Conventions}\label{sec:conv}
\emph{Signature and curvature.} The spacetime $(\mathcal M,\bg)$ is four-dimensional with
signature $(-,+,+,+)$ and Levi-Civita connection $\nabla$. The curvature endomorphism is
$R(X,Y)Z=\nabla_X\nabla_YZ-\nabla_Y\nabla_XZ-\nabla_{[X,Y]}Z$ and, following Christodoulou
\cite[Chapter~1]{Chr}, the curvature tensor with four arguments is
\begin{equation}\label{eq:Rconv}
 \bR(W,Z,X,Y):=\bg\big(W,R(X,Y)Z\big).
\end{equation}
It is antisymmetric in $(W,Z)$ and in $(X,Y)$, satisfies the pair symmetry
$\bR(W,Z,X,Y)=\bR(X,Y,W,Z)$ and the first Bianchi identity $R(X,Y)Z+R(Y,Z)X+R(Z,X)Y=0$. The Ricci
tensor is $\Ric(Y,Z)=\tr\big(X\mapsto R(X,Y)Z\big)$, so that the round unit sphere has $\Ric=\gac$ and
Gauss curvature $K=1$. In a frame $(e_\mu)$ with dual frame $(e^\mu)$ (that is,
$\bg(e^\mu,e_\nu)=\delta^\mu_\nu$, with $e^\mu$ a vector),
\begin{equation}\label{eq:Ricframe}
 \Ric(Y,Z)=\sum_\mu\bg\big(R(e_\mu,Y)Z,\,e^\mu\big).
\end{equation}
The vacuum Einstein equations are $\Ric(\bg)=0$.

\paragraph{Two-dimensional tensor operations.} Let $(S,\gd)$ be a Riemannian two-sphere with area
form $\slashed\epsilon$ and Levi-Civita connection $\Nd$. Indices $A,B,\ldots$ are raised and lowered
with $\gd$. For one-forms $\xi,\zeta$ and symmetric two-tensors $\theta,\theta'$ we write
\begin{align*}
 &(\xi,\zeta)=\gd^{AB}\xi_A\zeta_B,\quad |\xi|^2=(\xi,\xi),\quad (\theta,\theta')=\gd^{AC}\gd^{BD}\theta_{AB}\theta'_{CD},\quad
 (\theta\cdot\xi)_A=\theta_A{}^B\xi_B,\\
 &\Div\xi=\Nd^A\xi_A,\quad \Curl\xi=\slashed\epsilon^{AB}\Nd_A\xi_B,\quad (\Div\theta)_A=\Nd^B\theta_{AB},\\
 &{}^*\xi_A=\slashed\epsilon_A{}^B\xi_B,\qquad ({}^*\theta)_{AB}=\slashed\epsilon_A{}^C\theta_{CB}\ \ (\text{the left dual}),\\
 &(\xi\ohat\zeta)_{AB}=\xi_A\zeta_B+\xi_B\zeta_A-(\xi,\zeta)\gd_{AB},\qquad
 (\Nd\ohat\xi)_{AB}=\Nd_A\xi_B+\Nd_B\xi_A-(\Div\xi)\gd_{AB},\\
 &\theta\wedge\theta'=\slashed\epsilon^{AB}\theta_A{}^C\theta'_{CB}.
\end{align*}
The trace of $\theta$ is $\tr\theta=\gd^{AB}\theta_{AB}$ and its tracefree part is
$\hat\theta=\theta-\frac12(\tr\theta)\gd$.

\subsection{Null hypersurfaces and optical functions}\label{sec:nullhyp}

A hypersurface $\Hcal\subset\mathcal M$ is null if its normal vector $N$ is null; then $N$ is also
tangent to $\Hcal$, and its integral curves are the generators of $\Hcal$. A function $u$ is an
\emph{optical function} if $\bg^{-1}(du,du)=0$; its level sets are null hypersurfaces.

\begin{lemma}\label{lem:geodesic}
If $u$ is an optical function, the vector field $N:=-\nabla u$ (that is, $\bg(N,X)=-du(X)$) is
a null geodesic vector field: $\nabla_NN=0$.
\end{lemma}
\begin{proof}
In index notation $N_a=-\nabla_au$, and $\nabla_bN_a=-\nabla_b\nabla_au$ is symmetric in $a,b$. Hence
$N^b\nabla_bN_a=N^b\nabla_aN_b=\tfrac12\nabla_a(N^bN_b)=0$, since $N^bN_b=\bg^{-1}(du,du)=0$
everywhere.
\end{proof}
Thus null hypersurfaces are ruled by null geodesics, and the constraint equations along $\Hcal$
become ordinary differential equations along the generators.

\subsection{Double null gauge}\label{sec:dn}

We use the double null gauge of \cite[\S2.1]{ACR} (see also \cite[Appendix~A]{KU}). There are two
optical functions $u,v$, whose level sets $\{u=\text{const}\}$ (outgoing) and $\{v=\text{const}\}$
(incoming) intersect in spacelike two-spheres $S_{u,v}$. Angular coordinates $\theta^A$ are
transported along the incoming direction, and
\begin{equation}\label{eq:metric}
 \bg=-4\Omega^2du\,dv+\gd_{AB}(d\theta^A-b^Adv)(d\theta^B-b^Bdv),
\end{equation}
with lapse $\Omega>0$, shift $b=b^A\partial_A$ and induced metric $\gd=\gd_{AB}d\theta^Ad\theta^B$
on $S_{u,v}$. The frame dual to the coframe $(du,dv,d\theta^A-b^Adv)$ is
$(\partial_u,\ \partial_v+b^A\partial_A,\ \partial_A)$, in which \eqref{eq:metric} has
$\bg(\partial_u,\partial_v+b)=-2\Omega^2$, $\bg(\partial_A,\partial_B)=\gd_{AB}$, and all other products
vanish. Hence
\[
 \bg^{-1}=-\frac1{2\Omega^2}\big(\partial_u\otimes(\partial_v+b)+(\partial_v+b)\otimes\partial_u\big)+\gd^{AB}\partial_A\otimes\partial_B ,
\]
so $\bg^{-1}(du,du)=\bg^{-1}(dv,dv)=0$, and $\nabla u=-\frac1{2\Omega^2}(\partial_v+b)$,
$\nabla v=-\frac1{2\Omega^2}\partial_u$. Following \cite[(2.1)--(2.4)]{ACR}, we define
\begin{equation}\label{eq:frames}
 L:=\partial_v+b,\quad \Lb:=\partial_u,\qquad L':=\Omega^{-2}L=-2\nabla u,\quad \Lb':=\Omega^{-2}\Lb=-2\nabla v,
\end{equation}
\begin{equation}\label{eq:e34}
 e_4:=\Omega L'=\Omega^{-1}L,\qquad e_3:=\Omega\Lb'=\Omega^{-1}\Lb .
\end{equation}
By Lemma~\ref{lem:geodesic}, $L'$ and $\Lb'$ are geodesic, and
$\bg(L,\Lb)=-2\Omega^2$, $\bg(L',\Lb')=-2\Omega^{-2}$, $\bg(e_4,e_3)=-2$.

\begin{remark}[A sign in the shift]
In \cite[(2.3)]{ACR} the coframe is written as $d\theta^A+b^Adv$, together with $L=\partial_v+b$
\cite[(2.4)]{ACR}. With the coframe $d\theta^A+b^Adv$ the null normal tangent to $\{u=\text{const}\}$ is
$\partial_v-b$, so $L=\partial_v+b$ requires the coframe $d\theta^A-b^Adv$ of \eqref{eq:metric}. This plays
no role in the present paper, since we work on one outgoing null hypersurface with $b=0$.
\end{remark}

\subsection{One null hypersurface and the frame \texorpdfstring{$(L,\Lb')$}{(L,L')}}\label{sec:onehyp}

From now on we work on one outgoing null hypersurface $\Hcal=\{u=u_0\}$, foliated by the spheres
$S_v=S_{u_0,v}$. We choose the angular coordinates on $\Hcal$ constant along the generators (label
each generator by the point where it meets a fixed sphere); equivalently $b=0$ on $\Hcal$. Then on
$\Hcal$, $L=\partial_v$ is tangent to the generators, $[L,\partial_A]=0$, and the coordinate fields
$\partial_A$ are tangent to the spheres $S_v$. For an $S$-tensor $\theta$ (a tensor field tangent to the
spheres), the projected Lie derivative $D\theta:=\slashed{\mathcal L}_L\theta$ has components
$(D\theta)_{AB}=\partial_v(\theta_{AB})$, since $[\partial_v,\partial_A]=0$ and the flow of $\partial_v$
maps $S_v$ to $S_{v'}$. Thus $D$ is $\partial_v$ on components, and all equations below are ODEs along
the generators.

\paragraph{The generator-normalised frame.} Instead of the frame $(e_4,e_3)$ we use the pair
\begin{equation}\label{eq:Lpair}
 L=\partial_v=\Omega e_4,\qquad \Lb'=\Omega^{-1}e_3,\qquad \bg(L,\Lb')=-2 .
\end{equation}
Here $L$ is the generator field along which we integrate, and $\Lb'$ is the unique null normal to the
cuts $S_v$ normalised against it (the normal space of $S_v$ is a Lorentzian plane containing exactly
two null lines, one of them spanned by $L$). For $X,Y$ tangent to $S_v$ we set
\begin{equation}\label{eq:chiprime}
 \chi_L(X,Y):=\bg(\nabla_XL,Y),\quad \chib'(X,Y):=\bg(\nabla_X\Lb',Y),\quad
 \tau(X):=\tfrac12\bg(\nabla_XL,\Lb'),\quad \nabla_LL=:\kappa L .
\end{equation}
The null second fundamental forms $\chi_L$, $\chib'$ of $S_v$ are symmetric: for instance
$\chi_L(X,Y)-\chi_L(Y,X)=-\bg(L,[X,Y])=0$, as $[X,Y]$ is tangent to $S_v$. The one-form $\tau$ is the
torsion of the pair $(L,\Lb')$. Finally $\nabla_LL$ is proportional to
$L$ because $L=\Omega^2L'$ with $L'$ geodesic; the factor $\kappa$ is the non-affinity of $L$.

\begin{lemma}[Frame derivatives]\label{lem:framederiv}
On $\Hcal$, for $X$ tangent to $S_v$,
\begin{align}
 \nabla_XL&=\chi_L(X)^\sharp-\tau(X)\,L, &\nabla_X\Lb'&=\chib'(X)^\sharp+\tau(X)\,\Lb',\label{eq:fd1}\\
 \nabla_LL&=\kappa\,L, &\nabla_L\Lb'&=-\kappa\,\Lb'-2\tau^\sharp,\label{eq:fd2}
\end{align}
where $\chi_L(X)^\sharp$ is the vector tangent to $S_v$ with $\gd(\chi_L(X)^\sharp,Y)=\chi_L(X,Y)$, and
similarly for $\chib'(X)^\sharp$ and $\tau^\sharp$. Moreover $\nabla_L\partial_A=\nabla_{\partial_A}L$.
\end{lemma}
\begin{proof}
Every vector $V$ at a point of $\Hcal$ can be written as $V=V^A\partial_A+aL+c\Lb'$, and pairing with
$\Lb'$, $L$, $\partial_B$ gives
$a=-\tfrac12\bg(V,\Lb')$, $c=-\tfrac12\bg(V,L)$, $V_B=\bg(V,\partial_B)$.
For $V=\nabla_XL$: $\bg(V,\partial_B)=\chi_L(X,\partial_B)$, $\bg(V,\Lb')=2\tau(X)$, and
$\bg(V,L)=\frac12X\bg(L,L)=0$. For $V=\nabla_X\Lb'$: $\bg(V,\partial_B)=\chib'(X,\partial_B)$,
$\bg(V,\Lb')=0$, and $\bg(V,L)=X\bg(\Lb',L)-\bg(\Lb',\nabla_XL)=-2\tau(X)$. For $V=\nabla_L\Lb'$:
$\bg(V,\Lb')=0$; $\bg(V,L)=-\bg(\Lb',\nabla_LL)=2\kappa$; and
$\bg(V,\partial_B)=-\bg(\Lb',\nabla_L\partial_B)=-\bg(\Lb',\nabla_{\partial_B}L)=-2\tau_B$, using
$\nabla_L\partial_B-\nabla_{\partial_B}L=[L,\partial_B]=0$. Inserting these pairings into the formulas
for $a$, $c$, $V_B$ gives \eqref{eq:fd1}--\eqref{eq:fd2}.
\end{proof}

\subsection{The Ricci coefficients of \texorpdfstring{\cite{ACR}}{ACR} and a dictionary}\label{sec:dict}

The Ricci coefficients of \cite[(2.8)]{ACR}, defined with the frame $(e_4,e_3)$, are
\[
 \chi(X,Y)=\bg(\nabla_Xe_4,Y),\quad \chib(X,Y)=\bg(\nabla_Xe_3,Y),\quad
 \zeta(X)=\tfrac12\bg(\nabla_Xe_4,e_3),
\]
\[
 \eta=\zeta+\Nd\log\Omega,\qquad \etab=-\zeta+\Nd\log\Omega,\qquad
 \omega=D\log\Omega,\qquad \omb=\Db\log\Omega,
\]
and the null curvature components are \cite[(2.10)]{ACR}, \cite[(1.53), (1.58), (1.75), (1.77), (1.159)]{Chr}
\begin{align*}
 \alpha(X,Y)&=\bR(X,e_4,Y,e_4), & \ab(X,Y)&=\bR(X,e_3,Y,e_3),\\
 \beta(X)&=\tfrac12\bR(X,e_4,e_3,e_4), & \bb(X)&=\tfrac12\bR(X,e_3,e_3,e_4),\\
 \rho&=\tfrac14\bR(e_3,e_4,e_3,e_4), & \sigma\,\slashed\epsilon(X,Y)&=\tfrac12\bR(X,Y,e_3,e_4).
\end{align*}
We decompose $\chi=\frac12\tr\chi\,\gd+\chih$ and $\chib=\frac12\tr\chib\,\gd+\chibh$.

\begin{lemma}[Dictionary]\label{lem:dict}
On $\Hcal$, for $X,Y$ tangent to $S_v$,
\[
 \chi_L=\Omega\chi,\qquad \chib'=\Omega^{-1}\chib,\qquad \tau=-\etab,\qquad \kappa=2\omega,
\]
\[
 \bR(X,\Lb',Y,\Lb')=\Omega^{-2}\ab(X,Y),\qquad \tfrac12\bR(X,\Lb',\Lb',L)=\Omega^{-1}\bb(X),\qquad
 \tfrac14\bR(\Lb',L,\Lb',L)=\rho .
\]
\end{lemma}
\begin{proof}
$\chi_L(X,Y)=\bg(\nabla_X(\Omega e_4),Y)=\Omega\chi(X,Y)+X(\Omega)\bg(e_4,Y)=\Omega\chi(X,Y)$, and
likewise for $\chib'$. For the torsion,
\[
 \tau(X)=\tfrac12\bg\big(\nabla_X(\Omega e_4),\Omega^{-1}e_3\big)
 =\tfrac12\bg(\nabla_Xe_4,e_3)+\tfrac12X(\log\Omega)\bg(e_4,e_3)=-\etab(X),
\]
since $\zeta-\Nd\log\Omega=-\etab$.
Since $L=\Omega^2L'$ with $\nabla_{L'}L'=0$,
$\nabla_LL=\Omega^2L'(\Omega^2)L'=L(\log\Omega^2)\,L=2\omega\,L$. The curvature identities follow from
multilinearity, e.g.\ $\bR(X,\Lb',Y,\Lb')=\Omega^{-2}\bR(X,e_3,Y,e_3)$.
\end{proof}

Thus $2\omega$ is the non-affinity of the generator field $L$, and $-\etab$ is the torsion of the pair
$(L,\Lb')$. In the frame $(L,\Lb')$ the natural quantities are $\Omega\chi$, $\Omega^{-1}\chib$, $\etab$,
$\omega$, $\Omega^{-1}\bb$, $\Omega^{-2}\ab$, that is, the quantities of \cite{ACR} multiplied by $\Omega^k$,
where $k$ is the number of $e_4$'s minus the number of $e_3$'s in their definition. We call them
\emph{renormalised}. Their transport equations for the incoming second fundamental form and the
incoming curvature contain no explicit $\Nd\log\Omega$. The torsion equation contains $2\Nd\omega$,
$\omega=D\log\Omega$, so a suitable angular regularity of the lapse is needed; within that regularity
class the lapse remains freely prescribable, and on the neck and the late collar
(Section~\ref{ch:neck}) we choose $\omega$ independent of the angle, so that this term vanishes.

\subsection{Two identities on a null hypersurface}\label{sec:outgoing}

We record two consequences of Lemma~\ref{lem:framederiv}, which are used in Sections~\ref{ch:kerr} and~\ref{ch:spin} to
identify vanishing curvature components on horizons.

\begin{lemma}[Transport and Codazzi identities]\label{lem:frameid}
On $\Hcal$,
\begin{equation}\label{eq:chiLtransport}
 D(\chi_L)_{AB}=\kappa\,(\chi_L)_{AB}+(\chi_L)_A{}^C(\chi_L)_{CB}+\bg\big(R(L,\partial_A)L,\partial_B\big),
\end{equation}
and for $X,Y,Z$ tangent to $S_v$
\begin{equation}\label{eq:Codout}
 \bg\big(R(Z,X)L,Y\big)=(\Nd_Z\chi_L)(X,Y)-(\Nd_X\chi_L)(Z,Y)-\tau(X)\chi_L(Z,Y)+\tau(Z)\chi_L(X,Y).
\end{equation}
If moreover $\Ric(X,L)=0$ for $X$ tangent to $S_v$, then
\begin{equation}\label{eq:RicXL}
 \tfrac12\bg\big(R(L,X)L,\Lb'\big)=(\Div\chi_L)(X)-X(\tr\chi_L)-\tau(X)\tr\chi_L+\chi_L(X,\tau^\sharp).
\end{equation}
\end{lemma}
\begin{proof}
Let $X=\partial_A$, $Y=\partial_B$, so that $[L,X]=0$ and $\nabla_LX=\nabla_XL$. Then
$\nabla_L\nabla_XL=\nabla_X\nabla_LL+R(L,X)L=X(\kappa)L+\kappa\nabla_XL+R(L,X)L$, and since $\bg(L,Y)=0$,
$L\big(\chi_L(X,Y)\big)=\kappa\,\chi_L(X,Y)+\bg(R(L,X)L,Y)+\bg(\nabla_XL,\nabla_YL)$. By \eqref{eq:fd1},
$\bg(\nabla_XL,\nabla_YL)=(\chi_L\cdot\chi_L)(X,Y)$, which gives \eqref{eq:chiLtransport}. For \eqref{eq:Codout},
\eqref{eq:fd1} gives $\bg(\nabla_Z\nabla_XL,Y)=(\Nd_Z\chi_L)(X,Y)+\chi_L(\Nd_ZX,Y)-\tau(X)\chi_L(Z,Y)$; antisymmetrising in
$Z,X$ and subtracting $\nabla_{[Z,X]}L$ gives \eqref{eq:Codout}, since $\Nd$ is torsion free. Finally, \eqref{eq:Ricframe} in the
frame $(\partial_B,L,\Lb')$, with dual frame $(\partial^B,-\frac12\Lb',-\frac12L)$, gives
$\Ric(X,L)=\gd^{BC}\bg(R(\partial_B,X)L,\partial_C)-\frac12\bg(R(L,X)L,\Lb')$, the term $\bg(R(\Lb',X)L,L)$ vanishing by
antisymmetry; contracting \eqref{eq:Codout} in $Z$ and $Y$ yields \eqref{eq:RicXL}.
\end{proof}

\subsection{The renormalised hierarchy}\label{sec:hierarchy}

The following system is solved in the rest of the paper.

\begin{proposition}[Renormalised null constraint hierarchy]\label{prop:hier}
Along $\Hcal$ (with $b=0$, $L=\partial_v$) the vacuum equations imply
\begin{align}
 D\gd&=\Omega\tr\chi\,\gd+2\Omega\chih,\label{eq:H1}\\
 D(\Omega\tr\chi)&=2\omega\,\Omega\tr\chi-\tfrac12(\Omega\tr\chi)^2-|\Omega\chih|^2,\label{eq:H2}\\
 D\etab&=-\Div(\Omega\chih)+\tfrac12\Nd(\Omega\tr\chi)-\Omega\tr\chi\,\etab+2\Nd\omega,\label{eq:H3}\\
 D(\Omega^{-1}\tr\chib)&=-(2\omega+\Omega\tr\chi)\,\Omega^{-1}\tr\chib
   +2\big(\Div\etab+|\etab|^2-K\big),\label{eq:H4tr}\\
 D(\Omega^{-1}\chib)&=-2\omega\,\Omega^{-1}\chib+\tfrac12\Omega\tr\chi\,\Omega^{-1}\chibh
   +(\Omega\chih,\Omega^{-1}\chibh)\gd\notag\\
 &\qquad+\tfrac12\Omega^{-1}\tr\chib\,\Omega\chih+\Ecal(\gd,\etab),\label{eq:H4}\\
 \widehat D(\Omega^{-2}\ab)&=\big(\tfrac12\Omega\tr\chi-4\omega\big)\Omega^{-2}\ab
   -\Nd\ohat(\Omega^{-1}\bb)-5\,\etab\ohat(\Omega^{-1}\bb)\notag\\
 &\qquad-3\rho\,\Omega^{-1}\chibh+3\sigma\,{}^*(\Omega^{-1}\chibh),\label{eq:H5}
\end{align}
where
\begin{align}
 \Omega^{-1}\bb&=\Div(\Omega^{-1}\chibh)-\tfrac12\Nd(\Omega^{-1}\tr\chib)
   +(\Omega^{-1}\chibh)\cdot\etab-\tfrac12\Omega^{-1}\tr\chib\,\etab,\label{eq:Cod}\\
 \rho&=-K-\tfrac14\Omega\tr\chi\,\Omega^{-1}\tr\chib+\tfrac12(\Omega\chih,\Omega^{-1}\chibh),\label{eq:Gauss}\\
 \sigma&=-\Curl\etab+\tfrac12(\Omega\chih)\wedge(\Omega^{-1}\chibh).\label{eq:curl}
\end{align}
No angular derivative of $\Omega$ appears, except through the term $2\Nd\omega$ in \eqref{eq:H3}. The tensor
\begin{equation}\label{eq:Ecal}
 \Ecal(\gd,\etab)_{AB}:=\Nd_A\etab_B+\Nd_B\etab_A+2\etab_A\etab_B-K\gd_{AB}
\end{equation}
is called the \emph{near-horizon tensor}; its vanishing is the near-horizon equation of extremal
horizons (Section~\ref{ch:kerr}).
\end{proposition}

\begin{proof}
These are the null structure equations and the null Bianchi equation along an outgoing null hypersurface, see
\cite[Section~2]{ACR} and \cite[Chapter~1]{Chr}, written for the renormalised quantities of Lemma~\ref{lem:dict}; here $\sigma$ is
defined as in Section~\ref{sec:dict} and ${}^*$ is the left dual. The conversion consists in attaching the powers of $\Omega$. Equation
\eqref{eq:H1} is the first variation formula $D\gd=2\chi_L$. Raychaudhuri's equation
$D\tr\chi=-\frac\Omega2(\tr\chi)^2+\omega\tr\chi-\Omega|\chih|^2$ and $D(\Omega\tr\chi)=\omega\Omega\tr\chi+\Omega D\tr\chi$ give
\eqref{eq:H2}; the torsion equation
$D\etab=-\Omega(\chi\cdot\etab-\beta)+2\Nd\omega$ together with the Codazzi equation for $\Omega\beta$ and the identities
$\Omega\Div\chih=\Div(\Omega\chih)-\Omega\chih\cdot\Nd\log\Omega$, $\Omega\Nd\tr\chi=\Nd(\Omega\tr\chi)-\Omega\tr\chi\,\Nd\log\Omega$,
$\zeta-\Nd\log\Omega=-\etab$ gives \eqref{eq:H3}; and in the incoming equations the terms containing $\Nd\log\Omega$ cancel after
renormalisation. The equation for $\Omega^{-1}\tr\chib$ is the trace of \eqref{eq:H4}, and
$D\theta=\widehat D\theta+(\Omega\chih,\theta)\gd$ for tracefree $\theta$. The optical and torsion equations
\eqref{eq:H1}--\eqref{eq:H3} use the vacuum equations only through the constraint components $\Ric(L,L)=0$ and $\Ric(L,\partial_A)=0$.
\end{proof}

\begin{remark}[Triangular structure]\label{rem:triangular}
Proposition~\ref{prop:hier} has a triangular structure:
\begin{itemize}
\item the optical equations \eqref{eq:H1}--\eqref{eq:H2} involve only $\gd$, $\Omega\chi$ and $\omega$;
\item the torsion equation \eqref{eq:H3} is linear in $\etab$, with coefficients from the optical level;
\item the equation \eqref{eq:H4} for the incoming second fundamental form is linear in $\Omega^{-1}\chib$,
with a source $\Ecal(\gd,\etab)$ from the previous two levels;
\item the equation \eqref{eq:H5} is linear in $\Omega^{-2}\ab$, with sources from the previous three levels.
\end{itemize}
Each level costs angular derivatives: one for $\etab$ (through $\Div(\Omega\chih)$ and
$\Nd\Omega\tr\chi$), one more for $\Omega^{-1}\chib$ (through $\Nd\etab$ and $K$), and two more for
$\Omega^{-2}\ab$ (through $\Nd\ohat\Div\chibh$ in $\Nd\ohat(\Omega^{-1}\bb)$). This gives the regularity hierarchy
\[
 \gd,\ \Omega\chi\in H^s,\qquad \etab\in H^{s-1},\qquad \Omega^{-1}\chib\in H^{s-2},\qquad \Omega^{-2}\ab\in H^{s-4},
\]
as in the norm of \cite[Definition~2.5]{ACR}. In each equation the coefficient of the unknown is
algebraic. Hence each level is a linear ODE in $v$ with values in a fixed Sobolev space, which may be
integrated forwards or backwards, and no derivative loss occurs within a level.
\end{remark}

\subsection{Free data}\label{sec:freedata}

The hierarchy also shows what can be prescribed along $\Hcal$. Suppose we are given the sphere data
at the initial sphere $S_{v_1}$; the lapse $\Omega(v,x)$ along $\Hcal$ (equivalently $\omega=D\log\Omega$ and
the initial lapse); and the conformal class of $\gd(v)$, i.e.\ a family $\gdc(v)$ of metrics with fixed
area form (Section~\ref{ch:optical}). Then Raychaudhuri's equation \eqref{eq:H2} becomes a linear
second-order ODE for the area radius (Section~\ref{ch:optical}), \eqref{eq:H3} determines $\etab$,
\eqref{eq:H4} determines $\Omega^{-1}\chib$, and \eqref{eq:H5} determines $\Omega^{-2}\ab$. These are
Christodoulou's characteristic free data \cite[Chapter~2]{Chr}. In the proof we shall often reverse
the roles: we prescribe the area and the conformal class and solve for the lapse. Both are
legitimate, since \eqref{eq:H2} is a single scalar relation between the three.

\section{The optical level and the torsion}\label{ch:optical}

The first two levels of the hierarchy, the optical equations \eqref{eq:H1}--\eqref{eq:H2} and the
torsion equation \eqref{eq:H3}, carry the intrinsic geometry of the null hypersurface: area, shear,
expansion and angular momentum. In this section we rewrite them in the variables used throughout
the proof: the area radius $\phi$, the conformal metric $\gdc$, and $Z=\phi^2\eta$.

\subsection{The conformal decomposition}\label{sec:confdec}

We fix on the sphere of generators the unit round metric $\gac=d\vartheta^2+\sin^2\!\vartheta\,d\varphi^2$
in transported coordinates $(\vartheta,\varphi)$; it is a reference structure, not a geometric property of
the cuts. Following \cite[(2.6)]{ACR}, we define
\begin{equation}\label{eq:confdec}
 \phi^2:=\frac{\sqrt{\det\gd}}{\sqrt{\det\gac}},\qquad \gdc:=\phi^{-2}\gd,\qquad\text{so that}\qquad
 \gd=\phi^2\gdc,\quad \det\gdc=\det\gac .
\end{equation}
The ratio of determinants is a function, so $\phi$ is well defined. Thus $d\mu_{\gd}=\phi^2\,d\mu_{\gac}$, and
we call $\phi$ the \emph{area radius} ($\phi=r$ for a round sphere of radius $r$), while $\gdc$ has the
same area form as $\gac$ and carries the shape of $\gd$. The \emph{shear energy} is
\begin{equation}\label{eq:e}
 e:=\tfrac18\,\gdc^{AB}\gdc^{CD}\,D\slashed g_{c,AC}\,D\slashed g_{c,BD}=\tfrac18\tr\big((\gdc^{-1}D\gdc)^2\big).
\end{equation}

\begin{proposition}[Optics in conformal variables]\label{prop:opt}
\begin{equation}\label{eq:opt}
 \Omega\tr\chi=\frac{2D\phi}{\phi},\qquad \Omega\chih=\tfrac12\phi^2D\gdc,\qquad
 |\Omega\chih|^2_{\gd}=2e .
\end{equation}
\end{proposition}
\begin{proof}
Differentiating $\gd=\phi^2\gdc$ gives $D\gd=2\phi D\phi\,\gdc+\phi^2D\gdc$, whose $\gd$-trace is
\[
 \gd^{AB}D\gd_{AB}=\frac{4D\phi}\phi+\gdc^{AB}D\slashed g_{c,AB}=\frac{4D\phi}\phi+D\log\det\gdc=\frac{4D\phi}\phi,
\]
by Jacobi's formula and since $\det\gdc=\det\gac$ does not depend on $v$. By \eqref{eq:H1},
$\gd^{AB}D\gd_{AB}=2\Omega\tr\chi$, which gives the first formula. The tracefree part of $D\gd$ is then
$\phi^2D\gdc$, and by \eqref{eq:H1} it equals $2\Omega\chih$. Finally, since $\gd^{AB}=\phi^{-2}\gdc^{AB}$,
\[
 |\Omega\chih|^2_{\gd}=\tfrac14\phi^4\,\gd^{AC}\gd^{BD}D\slashed g_{c,AB}D\slashed g_{c,CD}
 =\tfrac14\,\gdc^{AC}\gdc^{BD}D\slashed g_{c,AB}D\slashed g_{c,CD}=2e. \qedhere
\]
\end{proof}

Thus the expansion is the logarithmic derivative of the area radius and the shear is the derivative
of the shape; $e$ depends at each point only on the curve $v\mapsto\gdc(v)$, not on $\phi$ or $\Omega$.

\subsection{Unimodular \texorpdfstring{$2\times2$}{2x2} matrices}\label{sec:2x2}

We describe the curve $\gdc(v)$ in the $\gac$-orthonormal frame
\begin{equation}\label{eq:roundframe}
 e_1=\partial_\vartheta,\quad e_2=\frac1{\sin\vartheta}\partial_\varphi;\qquad
 \theta^1=d\vartheta,\quad \theta^2=\sin\vartheta\,d\varphi ,
\end{equation}
and record $\gdc$ by its frame matrix $G_{ab}(v,x):=\gdc(e_a,e_b)$; since the frame is $\gac$-orthonormal, $G$ is also the matrix of
$\gac^{-1}\gdc$. Then $G$ is symmetric, positive definite, and $\det G=1$. The frame does not depend on $v$, so $(D\gdc)(e_a,e_b)=\partial_vG_{ab}$, and the
endomorphism $(\gdc^{-1}D\gdc)^A{}_B:=\gdc^{AC}(D\gdc)_{CB}$ has frame matrix $G^{-1}\partial_vG$; hence
$e=\frac18\tr\big((G^{-1}\partial_vG)^2\big)$. A rotation $\varphi\mapsto\varphi+\alpha$ maps $e_1,e_2$ to
themselves, so rotating a tensor field about the axis shifts the argument $\varphi$ of its frame
components. The optical level is thus $2\times2$ matrix algebra, pointwise on the sphere.

\begin{definition}\label{def:SR}
For $\xi,\alpha\in\R$ let
\[
 S(\xi):=\begin{pmatrix}\cos\xi&\sin\xi\\\sin\xi&-\cos\xi\end{pmatrix},\qquad
 S'(\xi):=\frac{d}{d\xi}S(\xi)=\begin{pmatrix}-\sin\xi&\cos\xi\\\cos\xi&\sin\xi\end{pmatrix},
\]
\[
 R_\alpha:=\begin{pmatrix}\cos\alpha&-\sin\alpha\\\sin\alpha&\cos\alpha\end{pmatrix},\qquad
 K_0:=\begin{pmatrix}0&1\\-1&0\end{pmatrix}=R_{-\pi/2}.
\]
\end{definition}

\begin{lemma}[Parametrisation of unimodular metrics]\leavevmode\label{lem:param}
\begin{enumerate}[label=(\alph*)]
\item $S(\xi)$ is symmetric and tracefree, $S(\xi)^2=I$ and $\det S(\xi)=-1$. So $S(\xi)$ is \emph{not} a
rotation matrix; it is the \emph{reflection} in the line through the origin at angle $\xi/2$:
\[
 S(\xi)=R_{\xi/2}\,\mathrm{diag}(1,-1)\,R_{\xi/2}^{-1}.
\]
Moreover
\[
 S(\xi)S(\xi')=R_{\xi-\xi'},\qquad S'(\xi)=S(\xi+\tfrac\pi2),\qquad S(\xi+\pi)=-S(\xi),
\]
\[
 S(\xi)S'(\xi)=K_0,\qquad S'(\xi)S(\xi)=-K_0 .
\]
The tracefree symmetric $2\times2$ matrices are exactly the matrices $\rho S(\xi)$, $\rho\ge0$.
\item $\exp(\rho S(\xi))=\cosh\rho\,I+\sinh\rho\,S(\xi)=R_{\xi/2}\,\mathrm{diag}(e^{\rho},e^{-\rho})\,R_{\xi/2}^{-1}$.
\item Every symmetric positive definite $G$ with $\det G=1$ is $G=\exp(\rho S(\xi))$ for a unique
$\rho\ge0$, and a unique $\xi\in\R/2\pi\mathbb Z$ if $\rho>0$.
\item If $G(t)=\exp(\rho(t)S(\xi(t)))$, then
\begin{equation}\label{eq:Apulse}
 G^{-1}\partial_tG=\rho_t\,S(\xi)+\xi_t\sinh\rho\cosh\rho\,S'(\xi)-\xi_t\sinh^2\!\rho\,K_0 ,
\end{equation}
\begin{equation}\label{eq:e2x2}
 e=\tfrac18\tr\big((G^{-1}\partial_tG)^2\big)=\tfrac14\big(\rho_t^2+\sinh^2\!\rho\;\xi_t^2\big).
\end{equation}
More generally, for two derivations $X,Y$ (e.g.\ $\partial_t$ and $\partial_\varphi$) acting on the
parameters,
\begin{equation}\label{eq:hypcurrent}
 \tr\big(G^{-1}XG\;G^{-1}YG\big)=2\,X\rho\,Y\rho+2\sinh^2\!\rho\;X\xi\,Y\xi .
\end{equation}
\end{enumerate}
\end{lemma}

\begin{proof}
(a) Identify $(x,y)\in\R^2$ with $w=x+iy$. Then $R_\alpha$ is multiplication by $e^{i\alpha}$ and $S(\xi)w=e^{i\xi}\overline w$. Hence
$S(\xi)S(\xi')w=e^{i(\xi-\xi')}w$, i.e.\ $S(\xi)S(\xi')=R_{\xi-\xi'}$ (in particular $S^2=I$), and
$R_{\xi/2}\mathrm{diag}(1,-1)R_{\xi/2}^{-1}w=e^{i\xi/2}\overline{e^{-i\xi/2}w}=e^{i\xi}\overline w$. The identities $S'(\xi)=S(\xi+\frac\pi2)$ and
$S(\xi+\pi)=-S(\xi)$ are read off from the entries, and give $SS'=R_{-\pi/2}=K_0$ and $S'S=R_{\pi/2}=-K_0$; a tracefree symmetric matrix
$\begin{pmatrix}a&b\\b&-a\end{pmatrix}$ equals $\rho S(\xi)$ with $a+ib=\rho e^{i\xi}$. Part (b) follows from $S^2=I$,
which splits the exponential series into $\cosh\rho\,I+\sinh\rho\,S$; the diagonal form follows from (a), since $\exp$ commutes with
conjugation. For (c), the spectral theorem gives
$G=R_\alpha\mathrm{diag}(\lambda,\lambda^{-1})R_\alpha^{-1}$ with $\lambda\ge1$, which is (b) with $\rho=\log\lambda$, $\xi=2\alpha$;
for $\rho>0$ the eigenline of $\lambda$ is unique, so $\xi$ is unique modulo $2\pi$.

(d) Write $c=\cosh\rho$, $s=\sinh\rho$. By (b), $G=cI+sS$ and
$G^{-1}=cI-sS$, and $\partial_tG=s\rho_tI+c\rho_tS+s\xi_tS'$. Using $S^2=I$ and $SS'=K_0$,
\[
 G^{-1}\partial_tG=cs\rho_tI+c^2\rho_tS+cs\xi_tS'-s^2\rho_tS-sc\rho_tI-s^2\xi_tK_0
 =(c^2-s^2)\rho_tS+cs\,\xi_tS'-s^2\xi_tK_0,
\]
which is \eqref{eq:Apulse}. Next, $\tr(S^2)=\tr(S'^2)=2$, $\tr(SS')=\tr K_0=0$, $\tr(SK_0)=\tr(SSS')=\tr S'=0$,
$\tr(S'K_0)=\tr(S'SS')=\tr S=0$ and $\tr(K_0^2)=-2$. Writing $A_X=X\rho\,S+cs\,X\xi\,S'-s^2X\xi\,K_0$ and
similarly $A_Y$,
\[
 \tr(A_XA_Y)=2X\rho\, Y\rho+2c^2s^2X\xi\, Y\xi-2s^4X\xi\, Y\xi=2X\rho\, Y\rho+2s^2X\xi\, Y\xi,
\]
which is \eqref{eq:hypcurrent}; with $X=Y=\partial_t$ it gives \eqref{eq:e2x2}.
\end{proof}

By (b), $\gdc$ stretches lengths in the direction at angle $\xi/2$ from $\partial_\vartheta$ by the factor
$e^{\rho/2}$ relative to $\gac$ and shrinks them in the orthogonal direction by $e^{-\rho/2}$: $\rho$ is the
strength of the distortion and $\xi/2$ the direction of its axis. Since an axis is only defined modulo
$\pi$, the doubled angle $\xi$ is the natural $2\pi$-periodic parameter, and rotating the tensor by an
angle $\beta$ maps $\xi$ to $\xi+2\beta$; this is the statement that tracefree symmetric tensors have spin
$2$, and it produces the phases in Lemma~\ref{lem:poles} and the helicities $\pm2$ in
Section~\ref{ch:linear}. The unimodular positive matrices form the hyperbolic plane
$\mathrm{SL}(2,\R)/\mathrm{SO}(2)$, with metric $\frac12\tr\big((G^{-1}dG)^2\big)=d\rho^2+\sinh^2\!\rho\,d\xi^2$ by
\eqref{eq:hypcurrent}, so that $(\rho,\xi)$ are geodesic polar coordinates centred at $G=I$. By
\eqref{eq:e2x2}, $e$ is a quarter of the squared hyperbolic speed of $t\mapsto G(t,x)$. A change of $\rho$
alone is a radial motion (the shape changes, the axes stay fixed); a change of $\xi$ alone is a circular
motion (the axes rotate). The skew part $-\xi_t\sinh^2\!\rho\,K_0$ of \eqref{eq:Apulse} is present only for
circular motion, and it carries the angular momentum (Section~\ref{ch:spin}).

In complex notation, the tracefree matrix $\log G$ has the complex entry
$(\log G)_{11}+i(\log G)_{12}=\rho e^{i\xi}$, the \emph{complex shear}. We write $\gdc=\gac\exp(A)$ if $A$ is a
$\gac$-symmetric tracefree endomorphism field and $\gdc(X,Y)=\gac(X,\exp(A)Y)$; the frame matrix of
$\gdc$ is then the exponential of the frame matrix of $A$.

The frame \eqref{eq:roundframe} is singular at the poles; identities between frame-independent continuous quantities, such as $e$
and $\rho$, extend there by continuity, and which frame matrices define smooth tensors is decided by Lemma~\ref{lem:poles}.

Lemma~\ref{lem:param}(c) shows that every unimodular $\gdc$ is of the form $\gac\exp(A)$, but only
pointwise and away from the poles. For the constructions (specifically, the connector of Section~\ref{ch:tail}) we
need the logarithm as a smooth field on the whole sphere, which the following frame-free formula
provides. Here $\mathcal R_\pi:\varphi\mapsto\varphi+\pi$ denotes the half-turn about the axis.

\begin{lemma}[The logarithm of a unimodular metric]\label{lem:log}
Let $\gdc$ be a metric on an open set of $S^2$ with $\det\gdc=\det\gac$, and let $H:=\gac^{-1}\gdc$ be the endomorphism obtained by raising the first
index with $\gac$ ($\gac$-self-adjoint, positive, $\det H=1$). Then there is exactly one $\gac$-symmetric tracefree endomorphism field $A$ with
$\gdc=\gac\exp(A)$, namely
\begin{equation}\label{eq:logformula}
 A=\log H:=\int_0^\infty\Big(\frac1{1+s}\,I-(s\,I+H)^{-1}\Big)\,ds .
\end{equation}
Moreover:
\begin{enumerate}[label=(\alph*)]
\item $A$ is as regular as $\gdc$, at every point including the poles, and depends smoothly on any parameter (such as $v$) on which $\gdc$ does;
\item if an isometry $\psi$ of $\gac$ preserves $\gdc$, it preserves $A$. In particular $A$ is axisymmetric if $\gdc$ is, and $\mathcal R_\pi$-invariant if
$\gdc$ is.
\end{enumerate}
\end{lemma}
\begin{proof}
At each point $H$ is $\gac$-self-adjoint and positive, so $sI+H$ is invertible for $s\ge0$, and the integrand is
$(1+s)^{-1}(sI+H)^{-1}(H-I)=O(s^{-2})$ as $s\to\infty$; the integral converges. In a $\gac$-orthonormal eigenbasis of $H$, with
eigenvalues $h_1,h_2>0$, the integral is diagonal with entries $\int_0^\infty\big(\frac1{1+s}-\frac1{s+h_i}\big)ds=\log h_i$. Hence $A$ is
$\gac$-self-adjoint, $\exp(A)=H$, and $\tr A=\log\det H=0$. If $A'$ is $\gac$-self-adjoint with $\exp(A')=H$, then $H$ is diagonal in an
eigenbasis of $A'$ with entries $e^{a'_i}$, so $A'$ has the eigenvalues $\log h_i$ on the corresponding eigenspaces of $H$, and $A'=A$.

For (a): on a compact set of positive self-adjoint $H$, all derivatives of the integrand with respect to $H$ are $O((1+s)^{-2})$, since
$\partial_H(sI+H)^{-1}=-(sI+H)^{-1}\,\cdot\,(sI+H)^{-1}$; so $H\mapsto\log H$ is smooth, and differentiation under the integral sign is
justified. The formula involves no frame, so $A$ is as regular as $H$, hence as $\gdc$, at every point including the poles, and in any
parameter. For (b): an isometry $\psi$ of $\gac$ acts on endomorphisms by conjugation with its differential, which commutes with
inversion and with the integral; if $\psi$ preserves $\gdc$, it preserves $H$, hence $A$.
\end{proof}

\subsection{Raychaudhuri's equation as a linear ODE}\label{sec:Ray}

\begin{proposition}\label{prop:Raylin}
In the variables \eqref{eq:confdec}, Raychaudhuri's equation \eqref{eq:H2} is equivalent to
\begin{equation}\label{eq:Ray}
 D^2\phi=2\omega\,D\phi-e\,\phi .
\end{equation}
\end{proposition}
\begin{proof}
By \eqref{eq:opt}, $\Omega\tr\chi=2D\phi/\phi$, so
$D(\Omega\tr\chi)=2D^2\phi/\phi-\frac12(\Omega\tr\chi)^2$. Inserting this and $|\Omega\chih|^2=2e$ into
\eqref{eq:H2}, the squares cancel and we obtain $2D^2\phi/\phi=4\omega D\phi/\phi-2e$.
\end{proof}

For $\Omega\equiv1$ this is \cite[(2.11)]{KU}, $\partial_v^2\phi+e\phi=0$. Equation \eqref{eq:Ray} is a
linear second-order ODE for $\phi$ at each angle, and it is one scalar relation between three
quantities: the area radius $\phi$, the lapse (through $\omega=D\log\Omega$) and the shape $\gdc$ (through
$e$). Any two of them can be prescribed and the third solved for. Prescribing $\gdc$ and $\Omega$, one
solves the linear ODE for $\phi$ (Christodoulou's method; used on the focusing tail and the neck,
Sections~\ref{ch:tail} and~\ref{ch:neck}). Prescribing $\gdc$ and $\phi$ with $D\phi>0$, one solves for the
lapse algebraically (used on the early bridge and the control window,
Sections~\ref{ch:tail}--\ref{ch:controller}):
\begin{equation}\label{eq:lapse}
 \omega=D\log\Omega=\frac{D^2\phi+e\phi}{2D\phi}.
\end{equation}
Prescribing $\phi$ and $\Omega$ determines $e$, which restricts the speed of $\gdc$.

\begin{lemma}[Monotonicity and focusing]\label{lem:focus}
Let $\omega\equiv0$ on an interval. Then $\phi$ is concave there wherever $\phi>0$: $D^2\phi=-e\phi\le0$.
In particular, if $D\phi(v_0)=0$ and $\phi>0$ on $[v,v_0]$, then $D\phi\ge0$ on $[v,v_0]$.
\end{lemma}
\begin{proof}
$D\phi(v')=D\phi(v_0)+\int_{v'}^{v_0}e\phi\,dv''\ge0$ for $v'\le v_0$.
\end{proof}
Thus shear always decreases the expansion, and a non-expanding cut ($D\phi=0$) at the end of an
interval with $\omega=0$ forces the expansion to be non-negative before it. This is how the focusing
tail of Section~\ref{ch:tail} reaches $\tr\chi=0$.

\subsection{The torsion in terms of \texorpdfstring{$Z=\phi^2\eta$}{Z}}\label{sec:Z}

In the early part of the construction we use $\eta=-\etab+2\Nd\log\Omega$, in the late part $\etab$. Both
satisfy simple equations after multiplication by $\phi^2$.

\begin{lemma}[Conformal covariance of the divergence]\label{lem:confdiv}
Let $\gd=\phi^2\gdc$ on a two-dimensional manifold, and let $\theta$ be a symmetric two-tensor that is
tracefree. Then $\phi^2\Div_{\gd}\theta=\Div_{\gdc}\theta$.
\end{lemma}
\begin{proof}
The Christoffel symbols differ by
\[
 C^D_{AB}:=\Gamma[\gd]^D_{AB}-\Gamma[\gdc]^D_{AB}=\delta^D_A\partial_B\log\phi+\delta^D_B\partial_A\log\phi
 -\slashed g_{c,AB}\,\gdc^{DE}\partial_E\log\phi ,
\]
so that
$(\Div_{\gd}\theta)_A=\phi^{-2}\gdc^{BC}\big(\nabla^{\gdc}_B\theta_{CA}-C^D_{BC}\theta_{DA}-C^D_{BA}\theta_{CD}\big)$.
Now $\gdc^{BC}C^D_{BC}=(1+1-2)\partial^D\log\phi=0$, using $\gdc^{BC}\slashed g_{c,BC}=2$ (here the dimension
enters), and
$\gdc^{BC}C^D_{BA}\theta_{CD}=(\tr_{\gdc}\theta)\partial_A\log\phi+\partial^C\log\phi\,\theta_{CA}-\partial^D\log\phi\,\theta_{AD}=0$
for tracefree symmetric $\theta$.
\end{proof}

\begin{lemma}\label{lem:Z}
$Z:=\phi^2\eta$ and $\underline Z:=\phi^2\etab$ satisfy
\begin{align}
 DZ&=\Div_{\gdc}\big(\tfrac12\phi^2D\gdc\big)+D\phi\,\Nd\phi-\phi\,\Nd D\phi+4\phi D\phi\,\Nd\log\Omega,\label{eq:Z}\\
 D\underline Z&=-\Div_{\gdc}\big(\tfrac12\phi^2D\gdc\big)+\phi\,\Nd D\phi-D\phi\,\Nd\phi+2\phi^2\Nd\omega ,\label{eq:Zb}
\end{align}
where $\Nd\phi$, $\Nd D\phi$, $\Nd\log\Omega$, $\Nd\omega$ are differentials of functions on $S^2$.
\end{lemma}
\begin{proof}
\eqref{eq:Zb}: $D(\phi^2\etab)=2\phi D\phi\,\etab+\phi^2D\etab$, and $2\phi D\phi=\phi^2\Omega\tr\chi$ by
\eqref{eq:opt}. By \eqref{eq:H3},
\[
 D(\phi^2\etab)=\phi^2\Omega\tr\chi\,\etab+\phi^2\big(-\Div_{\gd}(\Omega\chih)+\tfrac12\Nd(\Omega\tr\chi)-\Omega\tr\chi\,\etab+2\Nd\omega\big),
\]
and the $\etab$ terms cancel. By Lemma~\ref{lem:confdiv} and \eqref{eq:opt},
$\phi^2\Div_{\gd}(\Omega\chih)=\Div_{\gdc}(\frac12\phi^2D\gdc)$, and
$\frac12\phi^2\Nd(2D\phi/\phi)=\phi\Nd D\phi-D\phi\Nd\phi$.

\eqref{eq:Z}: Since $\eta=-\etab+2\Nd\log\Omega$, we have $Z=-\underline Z+2\phi^2\Nd\log\Omega$, and
$D(\phi^2\Nd\log\Omega)=2\phi D\phi\Nd\log\Omega+\phi^2\Nd\omega$ because $D$ commutes with $\Nd$ on functions.
Hence $DZ=-D\underline Z+4\phi D\phi\Nd\log\Omega+2\phi^2\Nd\omega$, and the $\Nd\omega$ terms cancel.
\end{proof}

The right sides of \eqref{eq:Z} and \eqref{eq:Zb} do not contain $Z$, $\underline Z$: once the optics
is known, the torsion is obtained by a single integration in $v$.

\subsection{The angular momentum flux and the Komar law}\label{sec:komar}

\begin{proposition}[Flux identity]\label{prop:flux}
For a fixed ($v$-independent) vector field $Y$ on $S^2$, the charge $C_Y:=\int_{S^2}Z(Y)\,d\mu_{\gac}$ obeys
\begin{equation}\label{eq:flux}
 D C_Y=\int_{S^2}\Big[D\phi\,Y\phi-\phi\,Y(D\phi)+4\phi D\phi\,Y\log\Omega
 -\frac{\phi^2}4\tr\big(\gdc^{-1}D\gdc\,\gdc^{-1}\mathcal L_Y\gdc\big)\Big]d\mu_{\gac}.
\end{equation}
\end{proposition}
\begin{proof}
Pair \eqref{eq:Z} with $Y$ and integrate. The gradient terms give the first three terms. For the
divergence term, let $\theta=\frac12\phi^2D\gdc$. Integrating by parts with respect to $\gdc$, and using
$(\mathcal L_Y\gdc)_{AB}=\nabla^{\gdc}_AY_B+\nabla^{\gdc}_BY_A$ and the symmetry of $\theta$,
\[
 \int(\Div_{\gdc}\theta)_AY^A\,d\mu_{\gdc}=-\int\theta_{AB}\,\nabla^{\gdc,B}Y^A\,d\mu_{\gdc}
 =-\tfrac12\int\theta^{AB}(\mathcal L_Y\gdc)_{AB}\,d\mu_{\gdc},
\]
with indices of $\theta$ raised by $\gdc$. Finally
$\theta^{AB}(\mathcal L_Y\gdc)_{AB}=\frac12\phi^2\tr(\gdc^{-1}D\gdc\,\gdc^{-1}\mathcal L_Y\gdc)$, and
$d\mu_{\gdc}=d\mu_{\gac}$.
\end{proof}

Since $Z(Y)\,d\mu_{\gac}=\eta(Y)\,d\mu_{\gd}$, for the axial Killing field $Y=\partial_\varphi$ on the extremal
Kerr horizon the charge is $C_{\partial_\varphi}=8\pi J$, where $J=\frac1{8\pi}\int\eta(\partial_\varphi)\,d\mu_{\gd}$ is
the Komar angular momentum (Section~\ref{ch:kerr}). For a general transported field $Y$, $C_Y$ is a
field-dependent charge, and its identification with a physical angular momentum requires these
symmetry hypotheses.

The gradient terms in \eqref{eq:flux} do not vanish in general; for instance, on a non-spherical cone $\phi=a+bv$ in Minkowski space
($e=\omega=0$, $\gdc=\gac$, $\Omega=1$) the charge $C_{\partial_\varphi}$ changes, although no angular momentum is radiated, since the
transported field need not be the ambient axial Killing field. When the area radius $\phi$ and the lapse are axisymmetric, the gradient
contributions vanish and the charge transfer comes from the non-axisymmetric shear term alone; its endpoint values can then be identified
with the Komar charges in the stationary regions.

\begin{corollary}[Komar conservation in axisymmetry]\label{cor:komar}
Suppose $\gdc$ and $\Omega$ are axisymmetric ($\partial_\varphi$ is a Killing field of $\gdc$ and $\partial_\varphi\Omega=0$),
and the initial values of $\phi$, $D\phi$ are axisymmetric. Then $\phi$ is axisymmetric for all $v$, and
$C_{\partial_\varphi}$ is constant.
\end{corollary}
\begin{proof}
$e$ and $\omega$ are axisymmetric. Applying $\partial_\varphi$ to \eqref{eq:Ray}, the function $y=\partial_\varphi\phi$ solves the
linear equation $D^2y=2\omega Dy-e\,y$ with zero initial data, so $y\equiv0$. Then in \eqref{eq:flux} with
$Y=\partial_\varphi$ the first three terms vanish, and the last vanishes because $\mathcal L_{\partial_\varphi}\gdc=0$.
\end{proof}

This is the null-hypersurface version of the Komar conservation law that forbids spin-up in
axisymmetry: angular momentum has to be supplied by
non-axisymmetric shear.

\begin{remark}[The profile versus the total]\label{rem:profile}
Corollary~\ref{cor:komar} conserves the total charge $C_{\partial_\varphi}$. The pointwise axial profile
$Z(\partial_\varphi)(z)$ can still change under axisymmetric evolution if $\gdc$ has a mixed component
$\gdc(\partial_\vartheta,\partial_\varphi)\ne0$: then $\Div_{\gdc}(\frac12\phi^2D\gdc)(\partial_\varphi)$ is a nonzero function of
$z$ with zero integral. If $\gdc$ is round, $D\gdc=0$ and the profile is frozen. This distinction is used
in Section~\ref{ch:tail}.
\end{remark}

\section{The extremal Kerr horizon}\label{ch:kerr}

The target of the gluing is a sphere on the event horizon of extremal Kerr. In this section we collect the
properties of that horizon used in the proof. The intrinsic horizon data and the near-horizon equation are classical
\cite{Hajicek,LP,KL}; we record them in the normalisation used here, together with the incoming second fundamental form. The new
ingredients are the exact formulas for all sections of the horizon and all constant boosts of the normals
(Section~\ref{sec:cuts}), and the identities for the weight function $w=1+z^2$ which determine the sign of the boost in
Section~\ref{ch:sign}.

\subsection{Kerr in ingoing coordinates}\label{sec:kerrcoords}

In ingoing Kerr coordinates $(v,r,\vartheta,\varphi)$ the Kerr metric with mass $M$ and rotation parameter
$a$ is \cite[(2.15)]{KU}
\begin{multline}\label{eq:kerr}
 \bg=-\Big(1-\frac{2Mr}\Sigma\Big)dv^2+2\,dv\,dr-\frac{4Mar\sin^2\!\vartheta}\Sigma\,dv\,d\varphi
 -2a\sin^2\!\vartheta\,dr\,d\varphi\\
 +\Sigma\,d\vartheta^2+\Big(r^2+a^2+\frac{2Ma^2r\sin^2\!\vartheta}\Sigma\Big)\sin^2\!\vartheta\,d\varphi^2,
\end{multline}
with $\Sigma=r^2+a^2\cos^2\!\vartheta$ and $\Delta=r^2-2Mr+a^2$; the coefficient of $d\varphi^2$ equals
$((r^2+a^2)^2-a^2\Delta\sin^2\vartheta)\sin^2\vartheta/\Sigma$. These coordinates are regular across the
future event horizon $\{r=r_+\}$, $r_+=M+\sqrt{M^2-a^2}$.

\paragraph{Units.} Unless $M$ appears explicitly, we take $M=a=1$ (extremal, $r_+=1$). The
extremal case of mass $M$ follows by the scaling $\bg\mapsto M^2\bg$, $(v,r)\mapsto(Mv,Mr)$, under which the
Christoffel symbols in scaled coordinates are unchanged. Under this scaling the Killing generator
introduced below becomes $\xi/M$ and the normal $n/M$, so $\tau(X)=-\etab(X)=\frac12\bg(\nabla_XL,\Lb')$ picks
up the factor $M^2\cdot M^{-1}\cdot M^{-1}=1$, the second fundamental form $\chib'$ picks up $M^2\cdot M^{-1}=M$,
and $d\mu$ picks up $M^2$. Consequently, for the extremal horizon of mass $M$, in the frame $L=\xi_M$ (the
Killing generator) and $\Lb'$ with $\bg(L,\Lb')=-2$,
\begin{equation}\label{eq:scaleM}
 \gd=M^2\gd_H,\qquad \etab=\etab_H,\qquad \Omega^{-1}\chib=M\chib_H,\qquad J=M^2,
\end{equation}
where $(\gd_H,\etab_H,\chib_H)$ are the unit-mass data of Proposition~\ref{prop:kerrdata} and $J$ is
the angular momentum defined in Section~\ref{sec:cutdata}.

\paragraph{Abbreviations.} With $z=\cos\vartheta$ we write
\[
 u_z:=1+z^2,\qquad s_z:=1-z^2=\sin^2\!\vartheta,\qquad w:=1+z^2 .
\]
The function $w$ coincides with $u_z$; it receives its own name because it plays the role of a weight.
Note $u_z+s_z=2$ and $d\vartheta^2=dz^2/s_z$.

\paragraph{Co-rotating coordinates.} The angular velocity of the horizon is
$\Omega_H=\frac a{r_+^2+a^2}=\frac12$. Put $\psi:=\varphi-\tfrac12v$, i.e.\ $d\varphi=d\psi+\tfrac12dv$. In the
coordinates $(v,r,z,\psi)$ the vector field
\[
 \xi:=\partial_v\big|_{(r,z,\psi)}=\partial_v\big|_{(r,z,\varphi)}+\tfrac12\partial_\varphi
\]
is a Killing field, being a constant-coefficient combination of the Killing fields $\partial_v$ and
$\partial_\varphi$ of \eqref{eq:kerr}. It is the \emph{horizon generator}.

Substituting $d\varphi=d\psi+\frac12dv$ in \eqref{eq:kerr} with $M=a=1$, $\Sigma=r^2+z^2$, a direct computation gives, on the
horizon $\{r=1\}$ (where $\Sigma=u_z$),
\begin{equation}\label{eq:kerrr1}
\begin{gathered}
 g_{vv}=0,\quad g_{vr}=\frac{u_z}2,\quad g_{vz}=g_{v\psi}=0,\quad g_{rr}=g_{rz}=0,\\
 g_{r\psi}=-s_z,\quad g_{zz}=\frac{u_z}{s_z},\quad g_{z\psi}=0,\quad g_{\psi\psi}=\frac{4s_z}{u_z},
\end{gathered}
\end{equation}
and, for the first derivatives at $r=1$ that we need,
\begin{equation}\label{eq:kerrdr}
\begin{gathered}
 \partial_rg_{vv}=\partial_zg_{vv}=0,\quad \partial_rg_{v\psi}=\frac{2s_z}{u_z},\quad \partial_zg_{vr}=z,\\
 \partial_rg_{zz}=\frac2{s_z},\quad \partial_rg_{\psi\psi}=\frac{8z^2s_z}{u_z^2},\quad
 \partial_zg_{\psi\psi}=-\frac{16z}{u_z^2}.
\end{gathered}
\end{equation}
Moreover $g_{vz}$, $g_{rz}$, $g_{z\psi}$ vanish identically, and $g_{v\psi}$ vanishes identically on $\{r=1\}$.

\subsection{The Killing generator: null, normal, affine}\label{sec:killinggen}

By \eqref{eq:kerrr1}, on $\{r=1\}$ we have $\bg(\xi,\xi)=g_{vv}=0$, $\bg(\xi,\partial_z)=g_{vz}=0$ and
$\bg(\xi,\partial_\psi)=g_{v\psi}=0$. Since $\xi,\partial_z,\partial_\psi$ are tangent to $\{r=1\}$, $\xi$ is a null
normal of the horizon, and its integral curves are the generators.

\begin{lemma}[Extremality]\label{lem:affine}
$\nabla_\xi\xi=0$ on the horizon.
\end{lemma}
\begin{proof}
For a Killing field $\nabla_\mu\xi_\nu=-\nabla_\nu\xi_\mu$, so
$(\nabla_\xi\xi)_\nu=\xi^\mu\nabla_\mu\xi_\nu=-\xi^\mu\nabla_\nu\xi_\mu=-\frac12\partial_\nu(\xi^\mu\xi_\mu)=-\frac12\partial_\nu g_{vv}$. At
$r=1$ all first derivatives of $g_{vv}$ vanish: $\partial_v,\partial_\psi$ trivially, $\partial_r$ and $\partial_z$ by
\eqref{eq:kerrdr}.
\end{proof}
For a general Killing horizon $\nabla_\xi\xi=\kappa\xi$ with $\kappa$ the surface gravity; $\kappa=0$ is
extremality. In the language of Section~\ref{ch:null}: with $L=\xi$, the horizon has $\omega=0$.

\subsection{The data on a horizon cut}\label{sec:cutdata}

Let $S_H=\{v=0,\ r=1\}$. Its future null normal $n$ with $\bg(\xi,n)=-2$ is determined by the linear conditions
$\bg(n,\partial_z)=\bg(n,\partial_\psi)=0$, $\bg(\xi,n)=-2$, $\bg(n,n)=0$; by \eqref{eq:kerrr1},
\begin{equation}\label{eq:nH}
 n=\Big(-\frac{s_z}{u_z},\,-\frac4{u_z},\,0,\,-1\Big)\quad\text{in }(v,r,z,\psi).
\end{equation}

We describe the cut with the frame $L=\xi$, $\Lb'=n$, i.e.\ Lemma~\ref{lem:dict} with $\Omega\equiv1$ and
$\omega=0$. Then $\eta=\zeta=-\etab$ on the cut, and we record the torsion as $\eta_H:=-\etab_H$, the form in
which the literature on near-horizon geometries writes it.

\begin{proposition}[Extremal Kerr horizon data, $M=1$]\label{prop:kerrdata}
\begin{align}
 \gd_H&=\frac{u_z}{s_z}\,dz^2+\frac{4s_z}{u_z}\,d\psi^2
 \qquad(\det\gd_H=4\ \text{in }(z,\psi),\ \text{area }8\pi),\label{eq:gH}\\
 \eta_H&=-\etab_H=-\frac{z}{u_z}\,dz+\frac{2s_z}{u_z^2}\,d\psi,\label{eq:etaH}\\
 K_H&=\frac{2(1-3z^2)}{u_z^3},\qquad
 \sigma_H=-\Curl\etab_H=-\frac{2z(3-z^2)}{u_z^3},\qquad K_H+i\sigma_H=\frac2{(1+iz)^3},\label{eq:KH}\\
 \chib_H&=-\frac{4}{s_zu_z}\,dz^2-\frac{8zs_z}{u_z^2}\,dz\,d\psi-\frac{16z^2s_z}{u_z^3}\,d\psi^2,
 \qquad \tr\chib_H=-\frac4{u_z},\label{eq:chibH}
\end{align}
and $\Omega\chi=0$, $\alpha=0$, $\beta=0$. Here $dz\,d\psi$ is the symmetric product, so the off-diagonal
components of $\chib_H$ are $-4zs_z/u_z^2$. The cut is oriented by $dz\wedge d\psi$,
$\slashed\epsilon_{z\psi}=+\sqrt{\det\gd_H}=2$; the sign of $\sigma_H$ refers to this orientation.
\end{proposition}

\begin{proof}
The induced metric is read off from \eqref{eq:kerrr1}; its determinant in $(z,\psi)$ is $4$, so $d\mu_H=2\,dz\,d\psi$ and the area is $8\pi$. For
the torsion, $\nabla_\mu\xi_\nu=\frac12(\partial_\mu\xi_\nu-\partial_\nu\xi_\mu)$ with $\xi_\nu=g_{v\nu}$ since $\xi$ is Killing, so
$\eta_H(\partial_A)=\frac12\bg(\nabla_A\xi,n)=\frac14n^\nu(\partial_Ag_{v\nu}-\partial_\nu g_{vA})$; by \eqref{eq:kerrdr} and \eqref{eq:nH} this gives
$\eta_z=\frac14n^r\partial_zg_{vr}=-\frac z{u_z}$ and $\eta_\psi=-\frac14n^r\partial_rg_{v\psi}=\frac{2s_z}{u_z^2}$. For a metric $P\,dz^2+Q\,d\psi^2$ with $PQ=4$
one has $K=-\frac18\partial_z^2Q$, which gives $K_H$ for $Q=\frac8{u_z}-4$. Since $\slashed\epsilon^{z\psi}=\frac12$,
$\Curl\eta_H=\frac12\partial_z\eta_\psi=-\frac{2z(3-z^2)}{u_z^3}$, and since $\chih=0$ on the horizon, \eqref{eq:curl} gives $\sigma_H=\Curl\eta_H$; with
$(1+iz)^3=(1-3z^2)+i(3z-z^3)$ this gives the complex form in \eqref{eq:KH}. Finally, since $\bg(n,\partial_B)=0$ along the cut,
$\chib_H(\partial_A,\partial_B)=-\bg(n,\nabla_A\partial_B)=-n^\mu\Gamma_{\mu AB}$. At $r=1$, $\Gamma_{vAB}=0$, and by \eqref{eq:kerrr1}, \eqref{eq:kerrdr}
the only nonzero Christoffel symbols of the first kind entering are $\Gamma_{rzz}=-\frac1{s_z}$, $\Gamma_{rz\psi}=z$,
$\Gamma_{r\psi\psi}=-\frac{4z^2s_z}{u_z^2}$, $\Gamma_{\psi z\psi}=-\frac{8z}{u_z^2}$. With $n^r=-\frac4{u_z}$, $n^\psi=-1$ this gives \eqref{eq:chibH}.

\emph{Outgoing quantities.} $\Omega\chi(X,Y)=\bg(\nabla_X\xi,Y)$ is symmetric (Section~\ref{sec:onehyp}) and
antisymmetric (Killing), hence zero. Then \eqref{eq:chiLtransport} with $\chi_L=0$, $\kappa=0$ gives
$\bg(R(\xi,X)\xi,Y)=0$, i.e.\ $\alpha=0$. And \eqref{eq:RicXL} with $\chi_L=0$ gives $\frac12\bg(R(\xi,X)\xi,n)=0$, which is $\beta=0$ by Lemma~\ref{lem:dict}.
\end{proof}

\paragraph{Area and angular momentum.} The area is $8\pi$ ($8\pi M^2$ in general). The angular momentum,
normalised as $J:=\frac1{8\pi}\int\eta(\partial_\psi)\,d\mu$, is
\[
 J=\frac1{8\pi}\int_0^{2\pi}\!\!\int_{-1}^1\frac{2s_z}{u_z^2}\cdot2\,dz\,d\psi=\int_{-1}^1\frac{1-z^2}{(1+z^2)^2}\,dz
 =\Big[\frac z{1+z^2}\Big]_{-1}^1=1 .
\]
So $J=M^2$ in general, by \eqref{eq:scaleM}; this is the Kerr angular momentum $aM$ with $a=M$. Since
$\partial_\psi$ is divergence free, exact one-forms have zero pairing with it, and $J$ only sees the part of the
torsion that is not a gradient.

\paragraph{The area coordinate.} Since $\det\gd_H=4$ in $(z,\psi)$, the area form is $2\,dz\wedge d\psi$, which is
$2$ times the round area form $dz\wedge d\varphi$ under the identification $\psi\leftrightarrow\varphi$. Thus Kerr's $z$ is
an area coordinate on the horizon; this makes the area radius of the horizon constant in the round
normalisation (Section~\ref{sec:hconf}).

\subsection{The near-horizon equation}\label{sec:NH}

\begin{proposition}[Near-horizon equation]\label{prop:NH}
The horizon data satisfy
\begin{equation}\label{eq:NH}
 \Ecal(\gd_H,\etab_H)=0,\qquad\text{i.e.}\qquad
 \Nd_{(A}\eta_{H,B)}=\eta_{H,A}\eta_{H,B}-\tfrac12K_H\gd_{H,AB}.
\end{equation}
\end{proposition}
\begin{proof}[Structural proof]
Foliate the horizon by the cuts $\{v=c\}$ and take $L=\xi$, $\Lb'$ the corresponding normal. Every quantity is
invariant under the flow of the Killing field $\xi$, so $D(\Omega^{-1}\chib)=0$. On the horizon $\Omega\chi=0$ and
$\omega=0$ (Lemma~\ref{lem:affine}), so \eqref{eq:H4} reduces to $0=\Ecal(\gd_H,\etab_H)$. Since $\etab_H=-\eta_H$,
$\Ecal=-2\Nd_{(A}\eta_{B)}+2\eta_A\eta_B-K\gd_{AB}$.
\end{proof}

The equation can also be verified directly from \eqref{eq:gH}, \eqref{eq:etaH} and \eqref{eq:KH}, using the non-vanishing Christoffel
symbols of $\gd_H=P\,dz^2+Q\,d\psi^2$, $P=\frac{u_z}{s_z}$, $Q=\frac{4s_z}{u_z}$,
\begin{equation}\label{eq:ChrH}
 \Gamma^z_{zz}=\frac{2z}{s_zu_z},\qquad
 \Gamma^z_{\psi\psi}=\frac{8zs_z}{u_z^3},\qquad
 \Gamma^\psi_{z\psi}=-\frac{2z}{s_zu_z},
\end{equation}
which we also use below.

Equation \eqref{eq:NH} is the extremal near-horizon geometry equation of H\'aj\'i\v{c}ek \cite{Hajicek},
Lewandowski--Pawlowski \cite{LP} and Kunduri--Lucietti \cite{KL}. It is the reason the extremal case is
special. On a non-expanding segment with $\Omega\chi=0$, horizon data $(\gd_H,\etab_H)$ and \emph{constant}
$\omega\ne0$, \eqref{eq:H4} reduces to $D(\Omega^{-1}\chib)=-2\omega\,\Omega^{-1}\chib$, a pure exponential decay.
This is the mechanism described in Section~\ref{sec:toy}, and the neck mechanism of Section~\ref{ch:neck}.

\begin{remark}[Non-extremal Killing horizons]\label{rem:nonextremal}
For a non-extremal Killing horizon, $\nabla_\xi\xi=\kappa\xi$ with $\kappa\ne0$, i.e.\ $\omega=\kappa/2$, and the same
argument gives $\Ecal(\gd,\etab)=\kappa\,\chib'$: the incoming second fundamental form is determined by the
intrinsic data. On an extremal horizon $\chib'$ is not determined by $(\gd,\etab)$; in fact the whole family
$\chib_H+\Ccal F$ of Section~\ref{sec:cuts} occurs.
\end{remark}

\subsection{The horizon in conformal variables}\label{sec:hconf}

We identify the sphere of generators with the horizon cut by $(z,\varphi)\leftrightarrow(z,\psi)$: the round coordinate
$z=\cos\vartheta$ is Kerr's $z$, and the transported azimuth is the co-rotating azimuth. This is a choice of
diffeomorphism, which we are free to make because the target is geometric data on a sphere.

\begin{proposition}\label{prop:hH}
For the horizon of mass $M$:
\begin{enumerate}[label=(\alph*)]
\item the area radius is constant, $\phi\equiv\sqrt2M$;
\item in the round orthonormal frame \eqref{eq:roundframe}, $\gdcH:=\gd_H/(2M^2)$ has frame matrix
\[
 h_H=\mathrm{diag}(e^{q},e^{-q}),\qquad q(z)=\log\frac{1+z^2}2\in[-\log2,0];
\]
\item $q(\pm1)=0$ and $q$ is $s_z$ times a smooth function; in the notation of Lemma~\ref{lem:param},
$h_H=\exp(|q|S(\pi))$, a fixed axis, with strength $|q|\le\log2$, maximal at the equator.
\end{enumerate}
\end{proposition}
\begin{proof}
(a) In $(z,\psi)$, $\det\gd_H=4M^4$, while $\det\gac=1$ in $(z,\varphi)$ ($\gac=dz^2/s_z+s_zd\varphi^2$). So
$\phi^4=4M^4$.
(b) With $e_1=\partial_\vartheta=-\sqrt{s_z}\,\partial_z$ and $e_2=s_z^{-1/2}\partial_\psi$, the frame matrix of $\gd_H$ (unit mass) is
$\mathrm{diag}(s_z\frac{u_z}{s_z},\frac1{s_z}\frac{4s_z}{u_z})=\mathrm{diag}(u_z,\frac4{u_z})$. Dividing by $2$ gives
$\mathrm{diag}(\frac{u_z}2,\frac2{u_z})=\mathrm{diag}(e^q,e^{-q})$.
(c) $q=\log(1-\frac12s_z)$, and $\log(1-x)/x$ is smooth near $x=0$. Since $q\le0$,
$\mathrm{diag}(e^q,e^{-q})=\exp(|q|\,\mathrm{diag}(-1,1))=\exp(|q|S(\pi))$.
\end{proof}
Since $q\le0$, $h_H$ is oblate: meridians are shortened and circles of latitude are lengthened, as
expected for a rapidly rotating horizon.

\subsection{The weight $w=1+z^2$ and the functional $\Fcal$}\label{subsec:weight}

The following identities drive the scalar part of the argument (Sections~\ref{ch:linear} and~\ref{ch:sign}).
\begin{proposition}[Weight identities]\label{prop:weight}
\begin{align}
 \big(2w\eta_H+dw\big)^\sharp&=\partial_\psi,\label{eq:W1}\\
 w\,\eta_H\otimes\eta_H+\tfrac12(dw\otimes\eta_H+\eta_H\otimes dw)+\tfrac12\Nd^2w
 &=\frac{1-3z^2}{u_z^2}\,\gd_H,\label{eq:W2}\\
 \Delta_{h_H}w-\frac{8s_z}{u_z^2}&=-4,\qquad h_H:=\tfrac12\gd_H .\label{eq:W3}
\end{align}
\end{proposition}
\begin{proof}
\eqref{eq:W1}: $2w\eta_H+dw=2u_z\big(-\frac z{u_z}dz+\frac{2s_z}{u_z^2}d\psi\big)+2z\,dz=\frac{4s_z}{u_z}d\psi$. Raising with
$\gd_H^{\psi\psi}=\frac{u_z}{4s_z}$ gives $\partial_\psi$.

\eqref{eq:W2}: By \eqref{eq:ChrH}, $(\Nd^2w)_{zz}=w''-\Gamma^z_{zz}w'=2-\frac{4z^2}{s_zu_z}$,
$(\Nd^2w)_{\psi\psi}=-\Gamma^z_{\psi\psi}w'=-\frac{16z^2s_z}{u_z^3}$ and $(\Nd^2w)_{z\psi}=0$. Hence the left side of
\eqref{eq:W2} has components
\begin{gather*}
 zz:\ \frac{z^2}{u_z}-\frac{2z^2}{u_z}+1-\frac{2z^2}{s_zu_z}=\frac{1-3z^2}{s_zu_z},\qquad
 \psi\psi:\ \frac{4s_z^2}{u_z^3}-\frac{8z^2s_z}{u_z^3}=\frac{4s_z(1-3z^2)}{u_z^3},\\
 z\psi:\ -\frac{2zs_z}{u_z^2}+\frac{2zs_z}{u_z^2}=0 .
\end{gather*}
For $zz$, the sum equals $\frac{-z^2s_z+s_zu_z-2z^2}{s_zu_z}$, and with $s_zu_z=1-z^4$ the numerator is
$-z^2+z^4+1-z^4-2z^2=1-3z^2$; for $\psi\psi$ use $s_z-2z^2=1-3z^2$. Since $\gd_{H,zz}=\frac{u_z}{s_z}$ and
$\gd_{H,\psi\psi}=\frac{4s_z}{u_z}$, these are the components of $\frac{1-3z^2}{u_z^2}\gd_H$.

\eqref{eq:W3}: $h_H$ has $\det h_H=1$ in $(z,\psi)$ and $h_H^{zz}=\frac{2s_z}{u_z}$, so
\[
 \Delta_{h_H}w=\partial_z\Big(\frac{2s_z}{u_z}\cdot2z\Big)=\frac{(4-12z^2)u_z-(4z-4z^3)2z}{u_z^2}=\frac{4(1-4z^2-z^4)}{u_z^2}.
\]
Subtracting $\frac{8(1-z^2)}{u_z^2}$ gives $-\frac{4(1+2z^2+z^4)}{u_z^2}=-4$.
\end{proof}

\begin{definition}[The weighted functional]\label{def:F}
For a metric $\gd$ on $S^2$ and a one-form $\etab$, set
\begin{equation}\label{eq:F}
 \Fcal(\gd,\etab):=\frac1{4\pi}\int_{S^2}w\big(|\etab|^2+\Div\etab-K\big)\,d\mu_{\gd}
 =\frac1{4\pi}\int_{S^2}\big(w|\etab|^2-dw\cdot\etab-wK\big)\,d\mu_{\gd}.
\end{equation}
\end{definition}
The second form follows by integrating $w\Div\etab$ by parts. By \eqref{eq:Ecal}, the integrand of the
first form is $\frac w2\tr\Ecal(\gd,\etab)$; in particular $\Fcal(\gd_H,\etab_H)=0$ by \eqref{eq:NH}.

\begin{lemma}[First variation of the Gauss curvature]\label{lem:dotK}
For a variation $\dot\gd=k$ of a metric on a surface, $\dot K=\frac12\big(\Div\Div k-\Delta\tr k-K\tr k\big)$. In
particular $\dot K=\frac12\Div\Div k$ if $k$ is tracefree.
\end{lemma}
\begin{proof}
The variation of the Christoffel symbols is the tensor
$\dot\Gamma^C_{AB}=\frac12(\Nd_Ak_B{}^C+\Nd_Bk_A{}^C-\Nd^Ck_{AB})$ (compute in normal coordinates). The Ricci tensor
varies by $\dot R_{AB}=\Nd_C\dot\Gamma^C_{AB}-\Nd_B\dot\Gamma^C_{CA}$, and the scalar curvature $R=\gd^{AB}R_{AB}$ by
$\dot R=-k^{AB}R_{AB}+\gd^{AB}\dot R_{AB}$. Now $\gd^{AB}\dot\Gamma^C_{AB}=(\Div k)^C-\frac12\Nd^C\tr k$ and
$\dot\Gamma^C_{CA}=\frac12\Nd_A\tr k$, so $\gd^{AB}\dot R_{AB}=\Div\Div k-\Delta\tr k$. In two dimensions $R_{AB}=K\gd_{AB}$
and $R=2K$.
\end{proof}

\begin{lemma}[Criticality]\label{lem:crit}
At $(\gd_H,\etab_H)$, $\Fcal$ is critical with respect to
\begin{enumerate}[label=(\alph*)]
\item tracefree variations of the metric;
\item variations of $\etab$ that are exact one-forms;
\item more generally, variations $\dot\etab$ with $\int\dot\etab(\partial_\psi)\,d\mu_{\gd_H}=0$.
\end{enumerate}
For a variation of $\etab$ alone, $\dot\Fcal=-\frac1{4\pi}\int_{S^2}\dot\etab(\partial_\psi)\,d\mu_{\gd_H}$.
\end{lemma}
\begin{proof}
For a variation $\dot\etab$, the second form of \eqref{eq:F} and \eqref{eq:W1} give
\[
 \dot\Fcal=\frac1{4\pi}\int(2w\etab_H-dw)\cdot\dot\etab\,d\mu
 =-\frac1{4\pi}\int(2w\eta_H+dw)\cdot\dot\etab\,d\mu=-\frac1{4\pi}\int\dot\etab(\partial_\psi)\,d\mu .
\]
An exact form $df$ gives $\int\partial_\psi f\,d\mu=0$, because $\partial_\psi$ is divergence free.

For a tracefree $\dot\gd=k$, the area form does not change to first order ($\dot{d\mu}=\frac12\tr k\,d\mu=0$), and
$\dot{(\gd^{AB})}=-k^{AB}$, so $\dot{(|\etab|^2)}=-k^{AB}\etab_A\etab_B$ and $\dot{(dw\cdot\etab)}=-k^{AB}w_A\etab_B$. By
Lemma~\ref{lem:dotK} and two integrations by parts, $\int w\dot K\,d\mu=\frac12\int k^{AB}\Nd_A\Nd_Bw\,d\mu$. Since
$-k^{AB}\etab_A\etab_B=-k^{AB}\eta_A\eta_B$ and $-\dot{(dw\cdot\etab)}=k^{AB}w_A\etab_B=-k^{AB}w_A\eta_B$, we obtain
\[
 \dot\Fcal=-\frac1{4\pi}\int k^{AB}\Big[w\eta_{H,A}\eta_{H,B}+w_{(A}\eta_{H,B)}+\tfrac12\Nd^2_{AB}w\Big]d\mu .
\]
By \eqref{eq:W2} the bracket is a multiple of $\gd_H$, and $k^{AB}\gd_{H,AB}=0$.
\end{proof}

\subsection{Cuts and boosts: the target chart}\label{sec:cuts}

The construction does not fix the target cut in advance. We now describe all cuts of the horizon and all
constant rescalings of the normals, with exact formulas (no smallness assumption).

Let $F$ be a smooth function on $S^2$ and $S_F=\{v=F,\ r=1\}$ the corresponding cut. Its tangent vectors are
$E_A=\partial_A+F_A\,\xi$, where $F_A=\partial_AF$. Write $\Nd F:=f^C\partial_C$, $f^C=\gd_H^{CD}F_D$, for the $\gd_H$-gradient on the
reference cut; because $\partial_C$ and $n$ are $\xi$-invariant, $\Nd F$ and $n$ are also defined at the points of
$S_F$ (by transport along $\xi$).

\begin{lemma}\label{lem:nF}
The vector field
\begin{equation}\label{eq:nF}
 n_F:=n+2\Nd F+|\Nd F|^2\,\xi
\end{equation}
satisfies $\bg(n_F,n_F)=0$, $\bg(n_F,E_A)=0$ and $\bg(\xi,n_F)=-2$. Hence it is the null normal of $S_F$ normalised
against $\xi$.
\end{lemma}
\begin{proof}
Use $\bg(n,\xi)=-2$, $\bg(\xi,\xi)=\bg(\xi,\partial_A)=\bg(n,\partial_A)=\bg(n,n)=0$ and $\bg(\Nd F,\partial_A)=F_A$. Then
$\bg(n_F,n_F)=4|\Nd F|^2+2|\Nd F|^2\bg(n,\xi)=0$, $\bg(n_F,E_A)=F_A\bg(n,\xi)+2F_A=0$ and $\bg(\xi,n_F)=\bg(\xi,n)=-2$.
\end{proof}

The incoming curvature of the cuts is computed from the Petrov type of Kerr. We use the following algebraic lemma.

\begin{lemma}[Null rotations about a principal null direction]\label{lem:nullrot}
Let $p$ be a point of a vacuum spacetime and $L,K\in T_p\mathcal M$ null with $\bg(L,K)=-2$. Let $P:=\{L,K\}^\perp$, with an oriented
orthonormal basis $(e_1,e_2)$, and assume that every component of $\bR$ in the frame $(L,K,e_1,e_2)$ of nonzero boost weight vanishes;
the boost weight of a component is the number of entries $L$ minus the number of entries $K$. Set $\rho:=\frac14\bR(K,L,K,L)$ and
$\sigma:=\frac12\bR(e_1,e_2,K,L)$, and for $Y\in P$ put
\[
 N_Y:=K+2Y+|Y|^2L .
\]
\begin{enumerate}[label=(\alph*)]
\item $N_Y$ is null with $\bg(L,N_Y)=-2$, and $P_Y:=\{L,N_Y\}^\perp=\{e+\bg(Y,e)L:\ e\in P\}$; the projection $\pi_Y:P_Y\to P$
along $L$ is an isometry.
\item For $X\in P_Y$, $N_Y+2X+|X|^2L=N_{Y+\pi_YX}$.
\item $\frac14\bR(N_Y,L,N_Y,L)=\rho$ and $\frac12\bR(X_1,X_2,N_Y,L)=\sigma$, where $X_a:=\pi_Y^{-1}e_a$.
\item Transported to $P$ by $\pi_Y$, the symmetric form $\ab_Y(X,X'):=\bR(X,N_Y,X',N_Y)$ on $P_Y$ is
\begin{equation}\label{eq:nullrot}
\begin{gathered}
 \ab_Y=6\big(\rho\,Y\ohat Y-\sigma\,Y\ohat{}^*Y\big),\qquad\text{equivalently}\\
 \tfrac12\big(\ab_Y(X_1,X_1)-i\,\ab_Y(X_1,X_2)\big)=3(\rho-i\sigma)(Y_1-iY_2)^2,
\end{gathered}
\end{equation}
where $Y_a=\bg(Y,e_a)$, $({}^*Y)_a=\slashed\epsilon_{ab}Y_b$, and $a\ohat b=a\otimes b+b\otimes a-\bg(a,b)\,\bg|_P$.
\end{enumerate}
\end{lemma}

\begin{proof}
(a) $\bg(N_Y,N_Y)=4|Y|^2+2|Y|^2\bg(K,L)=0$ and $\bg(L,N_Y)=\bg(L,K)=-2$. For $e\in P$, $\bg(e+cL,N_Y)=2\bg(e,Y)-2c$, and
$|e+cL|=|e|$ since $L$ is null and orthogonal to $P$.
(b) With $X=e+\bg(Y,e)L$: $N_Y+2X+|X|^2L=K+2(Y+e)+\big(|Y|^2+2\bg(Y,e)+|e|^2\big)L=N_{Y+e}$.

For (c) and (d) we expand by multilinearity, with $N_Y=K+2Y+|Y|^2L$ and $X_a=e_a+Y_aL$. By assumption only terms of boost weight
zero survive. In (c) every term other than $\bR(K,L,K,L)$, resp.\ $\bR(e_1,e_2,K,L)$, has more entries $L$ than $K$; this proves (c).
The components of boost weight zero are determined by $\rho$ and $\sigma$ through the algebraic symmetries of $\bR$ and $\mathrm{Ric}=0$
(see \cite[Chapter~1]{Chr}): besides $\bR(K,L,K,L)=4\rho$ and $\bR(e_a,e_b,K,L)=2\sigma\slashed\epsilon_{ab}$,
\begin{equation}\label{eq:wt0}
 \bR(e_a,L,e_b,K)=-\rho\,\delta_{ab}-\sigma\,\slashed\epsilon_{ab},\qquad \bR(e_a,e_b,e_c,e_d)=-\rho\,\slashed\epsilon_{ab}\slashed\epsilon_{cd}.
\end{equation}
In $\bR(X_a,N_Y,X_b,N_Y)$ the terms of boost weight zero are, by \eqref{eq:wt0} and the symmetries of $\bR$,
\begin{align*}
 |Y|^2\big(\bR(e_a,K,e_b,L)+\bR(e_a,L,e_b,K)\big)+4\bR(e_a,Y,e_b,Y)&=-2\rho|Y|^2\delta_{ab}-4\rho\,({}^*Y)_a({}^*Y)_b,\\
 2Y_a\big(\bR(L,K,e_b,Y)+\bR(L,Y,e_b,K)\big)&=2Y_a\big(\rho Y_b-3\sigma({}^*Y)_b\big),\\
 2Y_b\big(\bR(e_a,K,L,Y)+\bR(e_a,Y,L,K)\big)&=2Y_b\big(\rho Y_a-3\sigma({}^*Y)_a\big),\\
 Y_aY_b\,\bR(L,K,L,K)&=4\rho\,Y_aY_b .
\end{align*}
For instance, $\bR(L,K,e_b,Y)=-Y_d\bR(e_b,e_d,K,L)=-2\sigma({}^*Y)_b$ and $\bR(L,Y,e_b,K)=-Y_c\bR(e_c,L,e_b,K)=\rho Y_b-\sigma({}^*Y)_b$.
Since $({}^*Y)\otimes({}^*Y)=|Y|^2\delta-Y\otimes Y$ in two dimensions, the sum is
$12\rho\big(Y_aY_b-\frac12|Y|^2\delta_{ab}\big)-6\sigma\big(Y_a({}^*Y)_b+({}^*Y)_aY_b\big)$, which is the first form of \eqref{eq:nullrot}. For the
second, $(Y\ohat Y)_{11}-i(Y\ohat Y)_{12}=(Y_1-iY_2)^2$ and, as ${}^*Y=(Y_2,-Y_1)$, $(Y\ohat{}^*Y)_{11}-i(Y\ohat{}^*Y)_{12}=i(Y_1-iY_2)^2$.
\end{proof}

For a constant $\lambda>0$ consider the boosted frame
\[
 L=\lambda^{-1}\xi,\qquad \Lb'=\lambda\,n_F .
\]

\begin{proposition}[Exact cut formulas]\label{prop:cut}
For every $F$, without any smallness assumption:
\begin{enumerate}[label=(\alph*)]
\item the induced metric of $S_F$ is $\gd_H$, and $\Omega\chi=0$, $\alpha=\beta=0$;
\item the torsion of $(L,\Lb')$ is $\tau=\eta_H=-\etab_H$, independently of $F$ and $\lambda$; equivalently, its
coefficient $\etab$ is $\etab_H$;
\item the incoming second fundamental form is
\begin{equation}\label{eq:Ncut}
 \Omega^{-1}\chib=\lambda\,\chib_H+\Ccal G,\qquad
 \Ccal G:=2\Nd^2G+2(\etab_H\otimes dG+dG\otimes\etab_H),\qquad G:=\lambda F;
\end{equation}
\item the incoming curvature is a polynomial in $(\lambda,dG)$: with the vector field $Z_\lambda:=\Nd G-\frac\lambda2\partial_\psi$ on the
reference cut (vectors and one-forms are identified by $\gd_H$),
\begin{equation}\label{eq:Acutinv}
 \Omega^{-2}\ab=6\big(\rho_H\,Z_\lambda\ohat Z_\lambda-\sigma_H\,Z_\lambda\ohat{}^*Z_\lambda\big),\qquad \rho_H:=-K_H .
\end{equation}
Thus
\begin{equation}\label{eq:Acut}
 \Omega^{-2}\ab=\lambda^2\ab_H+\lambda\,\mathcal B_1(dG)+\mathcal B_2(dG,dG),
\end{equation}
with fixed tensor-valued linear and quadratic forms $\mathcal B_1,\mathcal B_2$, the parts of \eqref{eq:Acutinv} of degree one and two in
$dG$. Explicitly, in the helicity components of Definition~\ref{def:helicity} (with $g_\pm:=(dG)_\pm$),
\begin{equation}\label{eq:Acutexplicit}
 (\Omega^{-2}\ab)_+=\frac{6i\,s_z}{u_z(z-i)^3}\,\big(\lambda-2i\,g_+\big)^2,\qquad
 (\Omega^{-2}\ab)_-=-\frac{6i\,s_z}{u_z(z+i)^3}\,\big(\lambda+2i\,g_-\big)^2 .
\end{equation}
\end{enumerate}
\end{proposition}

\begin{proof}
We use four facts: $\xi$ is Killing and affine; $\bg(\nabla_X\xi,Y)=0$ for $X,Y$ tangent to the horizon (for
$X,Y$ tangent to a cut this is $\Omega\chi=0$, Proposition~\ref{prop:kerrdata}; if $X$ or $Y$ equals $\xi$ it holds
because $\bg(\nabla_X\xi,\xi)=\frac12X\bg(\xi,\xi)=0$ and $\nabla_\xi\xi=0$); $n$ and the coordinate fields $\partial_A$
are $\xi$-invariant, $[\xi,n]=[\xi,\partial_A]=0$; and $\frac12\bg(\nabla_A\xi,n)=\eta_H(\partial_A)$.

(a) Since $\xi$ is null and orthogonal to the horizon, $\bg(E_A,E_B)=\bg(\partial_A,\partial_B)$. The null second fundamental
form and $\alpha,\beta$ are properties of the horizon and its generator, not of the cut.

(b) $\frac12\bg(\nabla_{E_A}(\xi/\lambda),\lambda n_F)=\frac12\bg(\nabla_A\xi+F_A\nabla_\xi\xi,\,n+2\Nd F+|\Nd F|^2\xi)$. Here
$\nabla_\xi\xi=0$, $\bg(\nabla_A\xi,\Nd F)=0$ (both arguments tangent) and $\bg(\nabla_A\xi,\xi)=0$. What remains is
$\frac12\bg(\nabla_A\xi,n)=\eta_H(\partial_A)$.

(c) First take $\lambda=1$ and compute $\bg(\nabla_{E_A}n_F,E_B)$. There are three contributions.

\emph{From $n$.} By bilinearity,
$\bg(\nabla_{E_A}n,E_B)=\chib_H(\partial_A,\partial_B)+F_A\bg(\nabla_\xi n,\partial_B)+F_B\bg(\nabla_An,\xi)+F_AF_B\bg(\nabla_\xi n,\xi)$. Now
$\bg(\nabla_\xi n,\xi)=-\bg(n,\nabla_\xi\xi)=0$ and $\bg(\nabla_An,\xi)=-\bg(n,\nabla_A\xi)=-2\eta_{H,A}$. Moreover
$\nabla_\xi n=\nabla_n\xi$ since $[\xi,n]=0$, and the Killing equation gives
$\bg(\nabla_n\xi,\partial_B)=-\bg(\nabla_B\xi,n)=-2\eta_{H,B}$. Altogether: $\chib_{H,AB}-2F_A\eta_{H,B}-2F_B\eta_{H,A}$.

\emph{From $2\Nd F=2f^C\partial_C$.} We have
\[
 \bg(\nabla_{E_A}(f^C\partial_C),E_B)=\bg(\nabla_A(f^C\partial_C),\partial_B)+F_B\bg(\nabla_A(f^C\partial_C),\xi)
 +F_A\bg(\nabla_\xi(f^C\partial_C),E_B).
\]
The first term is the Hessian $\Nd^2_{AB}F$ (tangential projection). In the second,
$\bg(\nabla_A(f^C\partial_C),\xi)=-f^C\bg(\partial_C,\nabla_A\xi)=0$. In the third, $f^C$ is $\xi$-invariant and
$\nabla_\xi\partial_C=\nabla_C\xi$, whose pairing with tangent vectors vanishes. Contribution: $2\Nd^2_{AB}F$.

\emph{From $|\Nd F|^2\xi$.} $E_A(|\Nd F|^2)\bg(\xi,E_B)+|\Nd F|^2\bg(\nabla_{E_A}\xi,E_B)=0$.

Hence $\chib(F)=\chib_H+2\Nd^2F-2(dF\otimes\eta_H+\eta_H\otimes dF)$, with no quadratic terms. Since
$\etab_H=-\eta_H$, this is \eqref{eq:Ncut} with $\lambda=1$. In the boosted frame, $\Lb'=\lambda n_F$ multiplies $\chib$ by
$\lambda$, and $\lambda\chib_H+\lambda\Ccal F=\lambda\chib_H+\Ccal(\lambda F)$ since $\Ccal$ is linear.

(d) \emph{Step 1: the principal null directions.} Let $K:=-\frac4{u_z}\partial_r$. By \eqref{eq:kerrr1} it is null with
$\bg(\xi,K)=-2$, and comparison with \eqref{eq:nH} gives
\begin{equation}\label{eq:Kn}
 K=n+\partial_\psi+\frac{s_z}{u_z}\,\xi .
\end{equation}
Kerr is of Petrov type D, and on the horizon its two principal null directions are those of $\xi$ and of $\partial_r$: in ingoing
coordinates the ingoing principal null geodesics are the curves along which only $r$ varies (the Kinnersley vector $n$ is a multiple of
$\partial_r$), and the outgoing principal null direction is tangent to the horizon generators \cite{Chandra}. Hence at every point of the
horizon all components of $\bR$ of nonzero boost weight in a frame $(\xi,K,e_1,e_2)$ vanish, and Lemma~\ref{lem:nullrot} applies with $L=\xi$.

\emph{Step 2: $\rho$ and $\sigma$.} Let $\pi$ denote the projection along $\xi$ onto $P:=\{\xi,K\}^\perp$; then $\pi\partial_\psi=\partial_\psi+\frac{2s_z}{u_z}\xi$
and $|\pi\partial_\psi|^2=\gd_{H,\psi\psi}=\frac{4s_z}{u_z}$. With $Y_0:=-\frac12\pi\partial_\psi$, \eqref{eq:Kn} says exactly $n=N_{Y_0}$. By
Lemma~\ref{lem:nullrot}(c), $\rho$ and $\sigma$ may be computed in the frame $(\xi,n)$, where $\Omega\chi=0$; by the Gauss and curl equations
\eqref{eq:Gauss}, \eqref{eq:curl} they are $\rho=-K_H=\rho_H$ and $\sigma=-\Curl\etab_H=\sigma_H$.

\emph{Step 3: the cut.} By Lemma~\ref{lem:nullrot}(b) applied to $n=N_{Y_0}$ and $X=\Nd F\in P_{Y_0}$,
\[
 n_F=n+2\Nd F+|\Nd F|^2\xi=N_Y,\qquad Y=\pi\big(\Nd F-\tfrac12\partial_\psi\big).
\]
The tangent vectors $E_A=\partial_A+F_A\xi$ of $S_F$ lie in $P_Y$ (Lemma~\ref{lem:nF}) and project to $\pi\partial_A$. We identify $P$ with the
tangent space of the reference cut by $\pi$; this is an isometry for $\gd_H$, and it preserves orientation. Since $\rho,\sigma,\gd_H$ and
$\partial_\psi$ are invariant under the flow of the Killing field $\xi$, their values at the point $(F(x),x)$ of $S_F$ are those at the point
$(0,x)$ of the reference cut. Lemma~\ref{lem:nullrot}(d) therefore gives
\[
 \Omega^{-2}\ab(E_A,E_B)=\lambda^2\bR(E_A,n_F,E_B,n_F)=6\lambda^2\big(\rho_H\,Y\ohat Y-\sigma_H\,Y\ohat{}^*Y\big)(\partial_A,\partial_B),
\]
and $\lambda Y=\Nd G-\frac\lambda2\partial_\psi=Z_\lambda$, since $G=\lambda F$. This is \eqref{eq:Acutinv}, and \eqref{eq:Acut} follows by
sorting by degree in $dG$.

\emph{Step 4: helicity form.} In the oriented $\gd_H$-orthonormal frame $f_1=\sqrt{s_z/u_z}\,\partial_z$, $f_2=\frac12\sqrt{u_z/s_z}\,\partial_\psi$,
the second form of \eqref{eq:nullrot} reads $(\Omega^{-2}\ab)_+=3(\rho_H-i\sigma_H)\big(Z_\lambda(f_1)-iZ_\lambda(f_2)\big)^2$. By
Definition~\ref{def:helicity}, $\zeta(f_1)-i\zeta(f_2)=2\sqrt{s_z/u_z}\,\zeta_+$ for every one-form $\zeta$. The one-form
$\gd_H(\partial_\psi,\cdot)=\frac{4s_z}{u_z}d\psi$ has helicity components $\mp i$, so $(Z_\lambda)_+=g_++\frac{i\lambda}2=\frac i2(\lambda-2ig_+)$.
Finally $\rho_H-i\sigma_H=-(K_H+i\sigma_H)=-\frac2{(1+iz)^3}=-\frac{2i}{(z-i)^3}$ by \eqref{eq:KH}. Altogether
\[
 (\Omega^{-2}\ab)_+=3\cdot\Big(-\frac{2i}{(z-i)^3}\Big)\cdot\frac{4s_z}{u_z}\cdot\Big(-\frac14\Big)(\lambda-2ig_+)^2
 =\frac{6i\,s_z}{u_z(z-i)^3}(\lambda-2ig_+)^2 .
\]
The negative helicity component is the complex conjugate, since $\ab$ and $dG$ are real (for a real field the components $(\cdot)_\pm$
are complex conjugate). In particular $\ab_H$, $\mathcal B_1$, $\mathcal B_2$ are the terms of order $\lambda^2$, $\lambda$, $1$; the absence of
negative powers of $\lambda$ reflects the fact that $\xi$ is a principal null direction.
\end{proof}

\begin{definition}[Target chart]\label{def:target}
For $(\lambda,G)\in\R\times H^r(S^2)$ define the \emph{target incoming data}
\[
 \mathcal N(\lambda,G):=\lambda\chib_H+\Ccal G,\qquad
 \mathcal A(\lambda,G):=\lambda^2\ab_H+\lambda\mathcal B_1(dG)+\mathcal B_2(dG,dG),
\]
given explicitly by \eqref{eq:Acutexplicit}.
For $\lambda>0$ these are the actual data of the Kerr cut $F=G/\lambda$ in the frame $(\xi/\lambda,\lambda n_F)$. The formulas
make sense, and are smooth (polynomial), for all $\lambda\in\R$, including $\lambda\le0$, where they are not
Kerr data. At $(\lambda,G)=(0,0)$ both vanish, and the differential there is
$(\dot\lambda,\dot G)\mapsto(\dot\lambda\chib_H+\Ccal\dot G,\,0)$: the $\ab$-component has zero differential because
\eqref{eq:Acutexplicit} is a perfect square in $(\lambda,dG)$. The absence of negative powers of $\lambda$ in \eqref{eq:Acut}
(Proposition~\ref{prop:cut}(d)) is what makes $\mathcal A$ smooth through $\lambda=0$.
\end{definition}

We solve the matching problem near $(\lambda,G)=(0,0)$, where all formulas are smooth, and only at the end show
that the solution has $\lambda>0$, so that it corresponds to an actual cut of the Kerr horizon. The substitution
$G=\lambda F$ is what makes the chart smooth through $\lambda=0$; in terms of $F$, the cut would run off to infinity
as $\lambda\to0$.

\begin{lemma}[Weighted trace of the target]\label{lem:targettrace}
$\av{w\tr_{\gd_H}\chib_H}=-4$ and $\av{w\tr_{\gd_H}\Ccal G}=0$ for all $G$, where $\av f=\frac1{4\pi}\int f\,d\mu_{\gac}$.
Hence
\begin{equation}\label{eq:targetmean}
 \av{w\tr\mathcal N(\lambda,G)}=-4\lambda .
\end{equation}
\end{lemma}
\begin{proof}
$\av{w\tr\chib_H}=\frac1{4\pi}\int_0^{2\pi}\int_{-1}^1(1+z^2)\frac{-4}{1+z^2}\,dz\,d\varphi=-4$. Next
$\tr\Ccal G=2\Delta G+4\etab_H\cdot dG$, so, integrating by parts on $(S^2,\gd_H)$ and using \eqref{eq:W1},
\begin{align*}
 \int w\tr\Ccal G\,d\mu_H&=2\int G\big(\Delta w-2\Div(w\etab_H)\big)d\mu_H\\
 &=2\int G\,\Div(dw+2w\eta_H)\,d\mu_H=2\int G\,\Div\partial_\psi\,d\mu_H=0 .
\end{align*}
Since $d\mu_H=2\,d\mu_{\gac}$, the horizon and round averages differ only by a constant factor.
\end{proof}

\subsection{The incoming Bianchi operator on the horizon}\label{sec:Bop}

On the horizon background ($\Omega\chi=0$, horizon metric and torsion), \eqref{eq:H5} shows that the source of
$\Omega^{-2}\ab$ is a linear operator applied to $\Omega^{-1}\chib$.

\begin{definition}\label{def:B}
For a symmetric two-tensor $N$ on $S^2$ with trace $B$ and tracefree part $W$, set
\begin{align*}
 \beta_N&:=\Div W-\tfrac12dB+W\cdot\etab_H-\tfrac12B\,\etab_H,\\
 \Bcal N&:=-\Nd\ohat\beta_N-5\,\etab_H\ohat\beta_N+3K_HW+3\sigma_H\,{}^*W,
\end{align*}
all with respect to $\gd_H$ and with the left dual.
\end{definition}
On any segment where $(\gd,\etab)=(\gd_H,\etab_H)$ and
$\Omega\chi=0$, \eqref{eq:H5} reads
\begin{equation}\label{eq:H5horizon}
 \widehat D(\Omega^{-2}\ab)=-4\omega\,\Omega^{-2}\ab+\Bcal(\Omega^{-1}\chib).
\end{equation}
Indeed, by the Gauss equation \eqref{eq:Gauss} $\rho=-K_H$ (all products with $\Omega\chi$ vanish), by the curl
equation \eqref{eq:curl} $\sigma=-\Curl\etab_H=\sigma_H$, and by the Codazzi equation \eqref{eq:Cod}
$\Omega^{-1}\bb=\beta_{\Omega^{-1}\chib}$.

\begin{lemma}\label{lem:Bstat}
$\Bcal\chib_H=0$ and $\Bcal(\Ccal F)=0$ for every $F$.
\end{lemma}
\begin{proof}
Foliate the horizon by the cuts $\{v=F+c\}$, $c\in\R$, with $L=\xi$ (so $\omega=0$, $\Omega=1$). The generators are
the integral curves of $\xi$, and the angular coordinates $(z,\psi)$ are constant along them, so the hierarchy
applies. All data are $\xi$-invariant, so $\widehat D(\Omega^{-2}\ab)=0$, and \eqref{eq:H5horizon} together with
Proposition~\ref{prop:cut}(c) gives $\Bcal(\chib_H+\Ccal F)=0$. This holds for every $F$, including $F=0$; since $\Bcal$
is linear, both claims follow.
\end{proof}

\section{The partial bridge I: architecture and the focusing tail}\label{ch:tail}

This section and the next carry out Step~1 of the strategy (Section~\ref{sec:strategy}). We do not try to
match the incoming data. We only construct, starting from $S_1$, a solution of the characteristic hierarchy
whose terminal sphere carries the \emph{intrinsic} data of an extremal horizon cut: metric, vanishing outgoing
second fundamental form, and horizon torsion. In particular the angular momentum must come out exactly $M^2$;
since angular momentum can only be carried by non-axisymmetric shear (Corollary~\ref{cor:komar}), this requires
an explicit non-axisymmetric construction.

\subsection{The statement}

\begin{proposition}[Partial bridge]\label{prop:partial}
For every sufficiently large $M$ there is a smooth, $\mathcal R_\pi$-invariant solution of the hierarchy of
Proposition~\ref{prop:hier} on a compact interval $[v_1,v_*]$ with the following properties. It starts with the
actual outgoing Kerr germ through $S_1$ (including the transversal data at $S_1$). After a fixed
angle-dependent normalisation of the generator field at $v_*$ (Proposition~\ref{prop:reparam}), its terminal
data are
\[
 \gd=M^2\gd_H,\qquad \Omega\chi=0,\qquad \alpha=\beta=0,\qquad -\etab=\eta_H,
\]
the data of the extremal horizon of mass $M$ (Proposition~\ref{prop:kerrdata}, scaled). Its area is $8\pi M^2$
and its angular momentum $\frac1{8\pi}\int\eta_H(\partial_\psi)\,d\mu$ is $M^2$. The terminal incoming fields
$\Omega^{-1}\chib$, $\Omega^{-2}\ab$ are smooth but not prescribed.
\end{proposition}

Here $\mathcal R_\pi:\varphi\mapsto\varphi+\pi$ is the half-turn. ``Sufficiently large'' is relative to $S_1$: the
construction needs the target area to be much larger than the area of $S_1$; we do not claim that the
area inequality $8\pi M^2>|S_1|$ alone suffices. The incoming fields $\Omega^{-1}\chib$ and $\Omega^{-2}\ab$ are
determined by transport along the hypersurface (Remark~\ref{rem:triangular}) and cannot be prescribed at both
ends; matching them is the task of the neck and the collar (Sections~\ref{ch:neck}--\ref{ch:sign}).

The proof of Proposition~\ref{prop:partial} occupies this section and Section~\ref{ch:spin}.

\subsection{Architecture of the partial bridge}\label{sec:architecture}

The interval $[v_1,v_*]$ is divided into four parts (Figure~\ref{fig:bridgearea}).
\begin{enumerate}[label=(\alph*)]
\item \emph{The initial part} $[v_1,v_w]$: the Kerr germ, an axisymmetric \emph{connector} that makes the
conformal metric round, and an \emph{area raising} stretch with round conformal metric, which brings the
area to a fixed small multiple $cM^2$ of the final area scale (Section~\ref{sec:initialpart}).
\item \emph{The control window} $[v_w,v_{\rm tail}]$, with prescribed area which is affine in $v$ in its
middle part,
\begin{equation}\label{eq:affinearea}
 \phi^2=B(z)+\varpi(z)\,v,\qquad \varpi>0 .
\end{equation}
The pulses that pump in angular momentum live here (Section~\ref{ch:spin}), and later also the controller
pulses (Section~\ref{ch:controller}).
\item \emph{The focusing tail} $[v_{\rm tail},v_*]$, along which the expansion is brought to zero at exactly the
right area (Section~\ref{sec:tail}).
\item \emph{The terminal sphere}, where the torsion is normalised by an angle-dependent reparametrisation of
the generators (Section~\ref{sec:terminal}).
\end{enumerate}

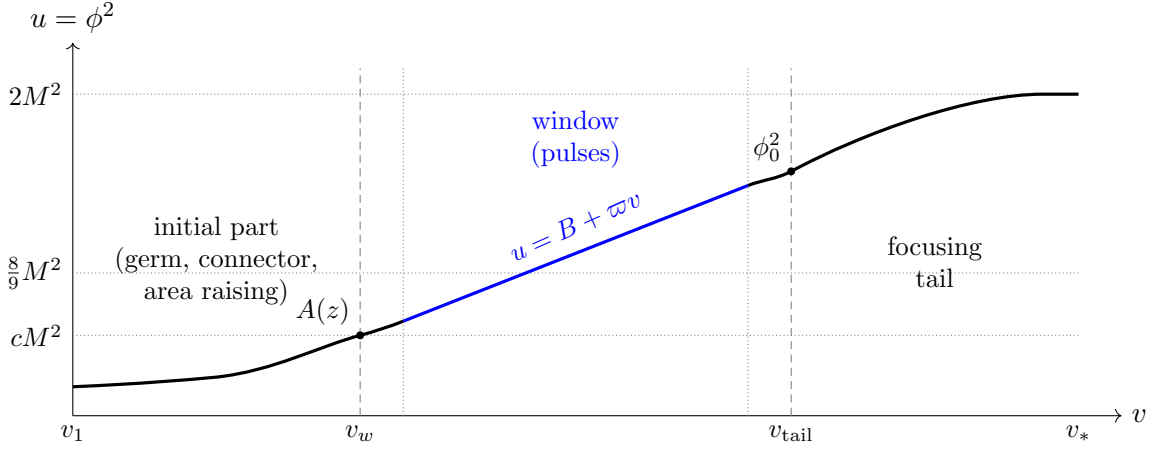
\begin{figure}[ht]
\centering
\begin{tikzpicture}[x=0.95cm,y=0.85cm]
 \draw[->] (0,0) -- (14.6,0) node[right] {$v$};
 \draw[->] (0,0) -- (0,5.8) node[above] {$u=\phi^2$};
 \foreach \y/\l in {1.25/cM^2,2.22/\frac89M^2,5/2M^2}
   {\draw[densely dotted,gray] (0,\y) -- (14,\y); \node[left,font=\small] at (0,\y) {$\l$};}
 \foreach \x in {4,10} \draw[densely dashed,gray] (\x,0) -- (\x,5.4);
 \foreach \x in {4.6,9.4} \draw[densely dotted,gray] (\x,0) -- (\x,5.4);
 \draw[very thick] (0,0.45) .. controls (1,0.5) and (1.5,0.55) .. (2,0.6)
   .. controls (2.8,0.7) and (3.4,1.07) .. (4,1.25)
   .. controls (4.2,1.31) and (4.4,1.382) .. (4.6,1.47);
 \draw[very thick,blue] (4.6,1.47) -- (9.4,3.58);
 \draw[very thick] (9.4,3.58) .. controls (9.6,3.668) and (9.8,3.68) .. (10,3.8)
   .. controls (11,4.4) and (12.5,5) .. (13.5,5) -- (14,5);
 \fill (4,1.25) circle (1.5pt) node[above left,font=\small] {$A(z)$};
 \fill (10,3.8) circle (1.5pt) node[above left,font=\small] {$\phi_0^2$};
 \node[blue,font=\small,rotate=24] at (7.0,2.95) {$u=B+\varpi v$};
 \foreach \x/\l in {0/v_1,4/v_w,10/v_{\rm tail},14/v_*} \node[below,font=\small] at (\x,0) {$\l$};
 \node[font=\small,align=center] at (2,2.4) {initial part\\(germ, connector,\\area raising)};
 \node[font=\small,align=center,blue] at (7,4.3) {window\\(pulses)};
 \node[font=\small,align=center] at (12,2.4) {focusing\\tail};
\end{tikzpicture}
\caption{The area $u=\phi^2$ along one generator of the partial bridge (qualitative). The area grows from
the area of $S_1$ to $A(z)=cM^2$ in the initial part. On the window it is interpolated in two short collars
(between the dashed and dotted lines) and affine in the middle, where all pulses are placed. The focusing
tail starts at $\phi_0^2\ge\frac89M^2$ and brings the expansion to zero exactly at $\phi^2=2M^2$. The only
quantitative input for the window is $A(z)<\phi_0(z)^2$.}
\label{fig:bridgearea}
\end{figure}

\paragraph{The lapse is recovered, not prescribed.} On (a) and (b) the area $\phi$ and the conformal metric
$\gdc$ are prescribed, and Raychaudhuri's equation \eqref{eq:Ray} determines the lapse by \eqref{eq:lapse},
\[
 \omega=D\log\Omega=\frac{D^2\phi+e\phi}{2D\phi},\qquad \log\Omega(v)=\log\Omega(v_1)+\int_{v_1}^v\omega .
\]
This is well defined because $D\phi$ will have a positive minimum on the compact region. The lapse need not
be small, and it is in general angle dependent; this is why $\Nd\log\Omega$ is kept in \eqref{eq:Z}. On the
tail (c), by contrast, the lapse is prescribed ($\omega=0$) and $\phi$ is solved for.

\subsection{The initial part}\label{sec:initialpart}

The initial part is built from one elementary tool, used three times.

\begin{lemma}[Positive slope interpolation]\label{lem:slope}
Let $[v_a,v_b]$ be an interval of length $\ell$, and let $u_a(v,z)$, $u_b(v,z)$ be smooth positive functions defined
for $v$ near $v_a$, resp.\ near $v_b$, with positive slopes $s_a:=\partial_vu_a>0$, $s_b:=\partial_vu_b>0$, all depending
smoothly on $z\in[-1,1]$ (or on $x\in S^2$). Suppose
\begin{equation}\label{eq:slopecond}
 u_a(v_a,z)<u_b(v_b,z)\qquad\text{for all }z .
\end{equation}
Then for every sufficiently small $\varepsilon>0$ there is a smooth $u$ on $[v_a,v_b]\times[-1,1]$ with
\begin{enumerate}[label=(\alph*)]
\item $u=u_a$ on $[v_a,v_a+\frac\varepsilon2]$ and $u=u_b$ on $[v_b-\frac\varepsilon2,v_b]$;
\item $\partial_vu>0$ everywhere, with a positive minimum;
\item $u=B(z)+\varpi(z)v$ with $\varpi>0$ on $[v_a+\varepsilon,v_b-\varepsilon]$.
\end{enumerate}
If $u_a,u_b$ are axisymmetric (resp.\ $\mathcal R_\pi$-invariant), so is $u$.
\end{lemma}
\begin{proof}
We construct the slope $s=\partial_vu$ rather than $u$, since positivity of a slope survives
convex combinations. Choose smooth cut-off functions of $v$ alone with $\chi_a+\chi_m+\chi_b=1$,
$0\le\chi_\bullet\le1$, $\chi_a=1$ on $[v_a,v_a+\frac\varepsilon2]$, $\chi_a=0$ beyond $v_a+\varepsilon$, $\chi_b=1$ on
$[v_b-\frac\varepsilon2,v_b]$, $\chi_b=0$ before $v_b-\varepsilon$. For a function $\varpi(z)$ to be fixed, set
\[
 s:=\chi_a\,s_a+\chi_m\,d+\chi_b\,s_b,\qquad u(v,z):=u_a(v_a,z)+\int_{v_a}^vs(v',z)\,dv' .
\]
\emph{Choice of $\varpi$.} The condition $u(v_b)=u_b(v_b)$ is one linear condition on $\varpi$:
\[
 \varpi(z)=\frac{u_b(v_b,z)-u_a(v_a,z)-I_a(z)-I_b(z)}{\int\chi_m\,dv},\qquad
 I_a:=\int\chi_as_a\,dv,\quad I_b:=\int\chi_bs_b\,dv .
\]
The collar integrals satisfy $0<I_a,I_b\le\varepsilon\max s_{a,b}$, with the maxima over compact sets, so
$I_a+I_b\to0$ uniformly in $z$ as $\varepsilon\to0$. By \eqref{eq:slopecond} and compactness,
$m:=\min_z(u_b(v_b)-u_a(v_a))>0$. Since $\int\chi_m\le\ell$, for $\varepsilon$ small we get $\varpi\ge m/(2\ell)>0$.

\emph{(b).} At each point $s$ is a convex combination of the positive numbers $s_a,\varpi,s_b$, so $s>0$, with a
positive minimum by compactness.

\emph{(a).} On $[v_a,v_a+\frac\varepsilon2]$, $s=s_a$ and $u(v_a)=u_a(v_a)$, so $u=u_a$. On $[v_b-\frac\varepsilon2,v_b]$,
$s=s_b$ and $u(v_b)=u_b(v_b)$ by the choice of $\varpi$, so $u=u_b$.

\emph{(c).} There $\chi_m=1$, so $s=\varpi(z)$. The symmetry statement is clear from the construction.
\end{proof}

\begin{lemma}[The initial part]\label{lem:initialpart}
Fix $c\in(0,\frac89)$. For every sufficiently large $M$ there is a smooth axisymmetric, $\mathcal R_\pi$-invariant
solution of the hierarchy on $[v_1,v_w]$ such that:
\begin{enumerate}[label=(\roman*)]
\item near $v_1$ it coincides with the outgoing Kerr cone through $S_1$, with all its fields;
\item near $v_w$: $\gdc=\gac$ and $\phi^2=M^2\hat u(v,z)$, where $\hat u$ is a fixed smooth function of $(v,z)$ for $v$ near
$v_w$, independent of $M$, with $\partial_v\hat u>0$ and $\hat u(v_w,z)=c$; in particular $A(z):=\phi^2(v_w,z)=cM^2$;
\item $D\phi>0$ throughout, and the lapse is the one recovered from \eqref{eq:lapse};
\item the axial torsion profile at $v_w$ differs from that at $v_1$ by a smooth function of $z$ bounded
independently of $M$.
\end{enumerate}
\end{lemma}

\begin{proof}
\emph{The Kerr germ.} Near $S_1$ we keep the actual Kerr data on a short interval $[v_1,v_1+\delta_0]$. All fields
are axisymmetric (and $\mathcal R_\pi$-invariant, since the double null foliation of Theorem~\ref{thm:A} is axisymmetric), and $D\phi>0$ there
because the outgoing expansion of $S_1$ is positive.

\emph{The connector.} Let $\gdc^K(v)$ be the conformal metric \eqref{eq:confdec} of the cuts of the outgoing Kerr
cone, in the transported coordinates, and let
\[
 A_K(v):=\log\big(\gac^{-1}\gdc^K(v)\big)
\]
be its logarithm in the sense of Lemma~\ref{lem:log}, so that $\gdc^K(v)=\gac\exp(A_K(v))$. By that lemma, $A_K$ is
a $\gac$-symmetric tracefree endomorphism field, smooth in $(v,x)$ (also at the poles, where the frame
\eqref{eq:roundframe} breaks down), and axisymmetric and $\mathcal R_\pi$-invariant because $\gdc^K$ is. Let $\chi_K(v)$
be a smooth cut-off equal to $1$ on $[v_1,v_1+\delta_0]$ and to $0$ after $v_1+2\delta_0$, and put
\[
 \gdc(v):=\gac\exp\big(\chi_K(v)A_K(v)\big)\qquad\text{on }[v_1,v_1+3\delta_0].
\]
It coincides with Kerr's near $v_1$, is round after $v_1+2\delta_0$, and is axisymmetric. For the area, apply
Lemma~\ref{lem:slope} on $[v_1+\frac12\delta_0,v_1+3\delta_0]$ with $u_a:=\phi_K^2$, the Kerr area, and $u_b:=u_1$, where
$u_1$ is any affine function of $v$ with positive slope and $u_1(v_1+3\delta_0)>\max\phi_K^2(v_1+\frac12\delta_0)$. This
gives a smooth axisymmetric $\phi$ with $D\phi>0$, equal to Kerr's near $v_1$. Define the lapse by \eqref{eq:lapse}
with $\Omega(v_1)$ the lapse of this foliation. Near $v_1$, where $\phi$ and $\gdc$ are Kerr's, \eqref{eq:lapse} returns Kerr's
$\omega$ (Raychaudhuri's equation holds for Kerr), so the lapse is Kerr's there too. Everything on this
stretch is independent of $M$.

\emph{The area raising.} On a further interval $[v_1+3\delta_0,v_w]$ we keep $\gdc=\gac$ and apply
Lemma~\ref{lem:slope} once more, now with $u_a:=u_1$, the affine function with which the connector ends, and
$u_b:=M^2\hat u$, for a fixed smooth $\hat u$ with $\partial_v\hat u>0$ and $\hat u(v_w,z)=c$. Condition
\eqref{eq:slopecond} reads $u_1(v_1+3\delta_0)<cM^2$, which holds for $M$ large because the left side is
independent of $M$. The lapse is again recovered by \eqref{eq:lapse}. This proves (i)--(iii), with $A(z)=cM^2$.

\emph{The axial profile.} By \eqref{eq:Z},
$DZ=\Div_{\gdc}(\frac12\phi^2D\gdc)+D\phi\Nd\phi-\phi\Nd D\phi+4\phi D\phi\Nd\log\Omega$. Since $\phi$ and $\Omega$ are
axisymmetric, the last three terms are multiples of $dz$, with zero $\partial_\varphi$-component. On the area raising
stretch $D\gdc=0$, so the profile $Z(\partial_\varphi)(z)$ does not change there, however much the area grows. On the
connector stretch the divergence term can have a nonzero $\partial_\varphi$-component (Remark~\ref{rem:profile}) if
$\gdc(\partial_\vartheta,\partial_\varphi)\ne0$; but this stretch is fixed and independent of $M$, so its contribution to
the profile is a fixed smooth function of $z$.
\end{proof}

\begin{remark}[Order of connector and area raising]\label{rem:connectorfirst}
If the conformal metric were made round \emph{after} raising the area, the connector's contribution to the
profile would be proportional to $\phi^2\sim M^2$, of the same size as the target profile $Z_H(\partial_\varphi)\sim M^2$,
and could spoil the positivity needed in Section~\ref{sec:pumping}. Doing the connector at bounded area keeps its
contribution bounded. Note that axisymmetric changes of the seed conserve the total angular momentum
(Corollary~\ref{cor:komar}), but in general not the axial profile (Remark~\ref{rem:profile}).
\end{remark}

\subsection{The focusing tail}\label{sec:tail}

The tail is the last segment $[v_{\rm tail},v_*]$ of the bridge. It brings the expansion to zero exactly, at
exactly the area and conformal geometry of the extremal horizon, at every angle simultaneously. It has to
satisfy three requirements.
\begin{enumerate}[label=(T\arabic*)]
\item \emph{Terminal data.} On a neighbourhood of $v_*$: $\gd=M^2\gd_H$ and $\Omega\chi=0$, i.e.\ $D\phi=0$ and
$D\gdc=0$.
\item \emph{Initial data that the window can reach.} At $v_{\rm tail}$ the conformal metric is round, the lapse
does not depend on $v$ ($\omega=0$), and $\phi>0$, $D\phi>0$ at every point of $S^2$, with $\phi$
axisymmetric. Positivity of $D\phi$ is essential: on the window the lapse is recovered from \eqref{eq:lapse},
which divides by $D\phi$.
\item \emph{No uncontrolled torsion.} The tail must not change the angular momentum, and the non-axisymmetric
torsion it produces must be explicit, so that it can be cancelled (Section~\ref{sec:cancel}).
\end{enumerate}
Forward integration is not suitable for (T1): Raychaudhuri's equation focuses each generator at its own rate,
so $D\phi$ would vanish at different parameter values at different angles, and the lapse formula degenerates
as $D\phi\to0$. The tail is therefore constructed \emph{backwards} from the exact terminal data: we choose an
explicit path of conformal metrics from $\gac$ to the horizon's conformal metric, solve Raychaudhuri's equation
for $\phi$ backwards, and verify (T2) at the initial cut.

In this subsection $\sigma=\pm1$ denotes the \emph{helicity} of a pulse, not the curvature component.

\subsubsection{Parameter, lapse and scaling}
We take the tail of fixed length $1$, independent of $M$, with parameter $t:=v-v_{\rm tail}\in[0,1]$, so that
$t=1$ is the terminal cut. On the tail we prescribe $\omega=0$, i.e.\ $\Omega(t,x)=\Omega_{\rm tail}(x)$. The function
$\Omega_{\rm tail}$ is not chosen here; it is the value at $v_{\rm tail}$ of the lapse recovered on the window, an
axisymmetric positive function (Section~\ref{sec:window}), and the tail only uses $\omega=0$. Then $t$ is an affine
parameter of the generators ($\nabla_LL=2\omega L=0$), and \eqref{eq:Ray} becomes
\begin{equation}\label{eq:tailRay}
 \partial_t^2\phi=-e\,\phi .
\end{equation}
This is the equation used on the window, read the other way round: on the window $\phi$ and $\gdc$ are
prescribed and $\Omega$ is recovered; on the tail $\omega$ and $\gdc$ are prescribed and $\phi$ is recovered.

\emph{Scaling.} The shear energy $e$ depends only on the path $\gdc(t)$, which will not depend on $M$, and
\eqref{eq:tailRay} is linear in $\phi$. With the terminal value $\phi(1)=\sqrt2M$ we get $\phi=M\hat\phi$, where $\hat\phi$
is the solution with $\hat\phi(1)=\sqrt2$. So the tail is independent of $M$ up to the factor $M$ on $\phi$, and
all estimates below are uniform in $M$.

\subsubsection{Regularity at the poles}
We describe $\gdc$ by its frame matrix $G$ in the round orthonormal frame \eqref{eq:roundframe}. This frame is
singular at the poles; which frame matrices give smooth tensors is decided by Lemma~\ref{lem:poles} in
Appendix~\ref{app:sphere}: for $a$ smooth on $[-1,1]$ and $\sigma\in\{-1,0,1\}$, the tracefree symmetric tensor field
with frame matrix $a(z)S(2\sigma\varphi)$ extends smoothly across the north pole $z=1$ if and only if
$a=(1-z)^{|\sigma+1|}b$, and across the south pole $z=-1$ if and only if $a=(1+z)^{|\sigma-1|}b$, with $b$ smooth.
Thus helicity $-1$ is unrestricted at the north pole and needs a factor $(1+z)^2$ at the south pole, helicity
$+1$ is the mirror image, and helicity $0$ needs a factor $1-z$, resp.\ $1+z$, at either pole. The same holds
for $\gdc=\gac\exp(T)$, since $\exp$ is an equivariant diffeomorphism from tracefree symmetric matrices onto
unimodular positive ones. Two consequences are used below.
\begin{enumerate}[label=(\roman*)]
\item \emph{An axisymmetric distortion vanishes at both poles.} If $\gdc$ is axisymmetric, at a pole it is invariant
under all rotations about the axis, hence a multiple of $\gac$ there, hence equal to $\gac$ (determinant one). So
for every axisymmetric path, $G(t,\pm1)=I$ and $e(t,\pm1)=0$ for all $t$.
\item \emph{At the north pole only helicity $-1$ survives, and at the south pole only helicity $+1$.} At the north pole
a helicity $-1$ field is a constant quadrupole: a tensor with Cartesian frame matrix $\mathrm{diag}(1,-1)$
has polar frame matrix $R_\varphi^{-1}\mathrm{diag}(1,-1)R_\varphi=S(-2\varphi)$.
\end{enumerate}

\subsubsection{The target and the path of conformal metrics}
By Proposition~\ref{prop:hH}, with the identification $(z,\varphi)\leftrightarrow(z,\psi)$, the target has constant area
radius $\phi=\sqrt2M$ and conformal frame matrix $h_H=\mathrm{diag}(e^{q},e^{-q})$, $q=\log\frac{1+z^2}2$. So the terminal
data of the backward problem, $\phi(1)=\sqrt2M$ and $\partial_t\phi(1)=0$, are the same at every angle.

The path is defined piecewise on $[0,1]$. In frame matrices:
\begin{equation}\label{eq:tailseed}
 G(t)=\begin{cases}
 \exp\big(a_-(z)R_-(t)\,S(-2\varphi)\big), & t\in\operatorname{supp}R_-,\\[2pt]
 \exp\big(a_+(z)R_+(t)\,S(+2\varphi)\big), & t\in\operatorname{supp}R_+,\\[2pt]
 I, & \text{other }t\in[0,\tfrac13],\\[2pt]
 \mathrm{diag}\big(e^{C(t)q(z)},e^{-C(t)q(z)}\big), & t\in[\tfrac13,1],
 \end{cases}
 \qquad a_-=\delta(1+z)^2,\quad a_+=\delta(1-z)^2 .
\end{equation}
Here:
\begin{itemize}
\item $R_+,R_-\in C^\infty_c\big((0,\frac13)\big)$ are two fixed, nonzero, nonnegative bump functions with
\emph{disjoint} supports, flat at the ends of their supports;
\item $\delta>0$ is a fixed small amplitude, independent of $M$, chosen so that the hypothesis of
Lemma~\ref{lem:tail} holds;
\item $\sigma=\pm1$ is the helicity of the pulse $\exp(a_\sigma R_\sigma S(2\sigma\varphi))$; its complex shear is
$\rho e^{i\xi}=a_\sigma R_\sigma e^{2i\sigma\varphi}$, a pure $e^{\pm2i\varphi}$ mode;
\item $C$ is a smooth step from $0$ to $1$, constant near $t=\frac13$ and near $t=1$, with $\int C'^2\,dt<2$. (The
linear ramp on an interval of length $\frac23$ has $C'=\frac32$ and $\int C'^2=\frac32$; smoothing its corners keeps
$\int C'^2<2$.)
\end{itemize}
There is no sum over $\sigma$: at a given time at most one pulse is active. Because the supports are
disjoint, one may equivalently write $\exp\{\sum_\sigma a_\sigma R_\sigma S(2\sigma\varphi)\}$ on $[0,\frac13]$. Near $t=0$ and
near $t=1$ the path is constant.

\emph{Smoothness.} Each piece is smooth on $S^2$ by Lemma~\ref{lem:poles}: $a_-=\delta(1+z)^2$ has helicity $-1$, which
is unrestricted at the north pole and needs $(1+z)^2$ at the south pole; $a_+$ is the mirror image; the connector
is $\exp(Cq\,\mathrm{diag}(1,-1))$, a helicity-$0$ field whose amplitude $Cq$ is divisible by $s_z$, hence by $1-z$
and by $1+z$. In $t$, the bumps are flat at the ends of their supports and $C$ is constant near $\frac13$ and $1$,
so the pieces join to all orders and the path is $C^\infty$ in $(t,x)$. The path is $\mathcal R_\pi$-invariant, since
$S(\pm2(\varphi+\pi))=S(\pm2\varphi)$.

\emph{The shear energy.} By Lemma~\ref{lem:param}(d) with $\xi_t=0$ on the pulses, and since
$G^{-1}\partial_tG=C'q\,\mathrm{diag}(1,-1)$ on the connector (diagonal matrices commute),
\begin{equation}\label{eq:taile}
 e(t,z)=\begin{cases}
 \frac14\,a_\sigma(z)^2R_\sigma'(t)^2, & t\in\operatorname{supp}R_\sigma,\\[2pt]
 0, & \text{other }t\in[0,\tfrac13],\\[2pt]
 \frac14\,C'(t)^2q(z)^2, & t\in[\tfrac13,1].
 \end{cases}
\end{equation}
Thus $e$ is axisymmetric, although the pulses are not: the energy sees only the speeds $\rho_t,\xi_t$, not the
phase $\xi=2\sigma\varphi$ itself.

\paragraph{Nonrotating pulses.} In a pulse of \eqref{eq:tailseed} the phase $\xi=2\sigma\varphi$ does not depend on $t$.
At each point of $S^2$ the axis of the distortion stays fixed; only its strength $\rho=a_\sigma(z)R_\sigma(t)$ grows from
$0$ to a maximum and returns to $0$. In the hyperbolic picture, $G(t,x)$ makes a radial excursion and comes back,
and $G^{-1}\partial_tG=\rho_tS(2\sigma\varphi)$ is symmetric. The \emph{rotating} pulses of Section~\ref{sec:closedpulses} have
phase $2\sigma(\varphi-\pi C(t))+\delta_j$: the frame matrix at time $t$ is the initial pattern rigidly rotated about the
polar axis by $\pi C(t)$, $G(t,x)$ moves on a circle in the hyperbolic plane, and $G^{-1}\partial_tG$ acquires the skew
part $-\xi_t\sinh^2\!\rho\,K_0$ of \eqref{eq:Apulse}, which feeds the angular momentum.

\paragraph{Why the pulses are needed.} Without them, $e(t,\pm1)=\frac14C'^2q(\pm1)^2=0$ for all $t$, so at
the poles $\phi(t,\pm1)\equiv\sqrt2M$ and $\partial_t\phi(0,\pm1)=0$: the expansion would vanish at the poles at the start
of the tail, violating (T2). By consequence (i) above this happens for every axisymmetric path. Focusing at the
poles therefore needs a non-axisymmetric seed, and by consequence (ii) the only distortions that are nonzero at
the north pole have helicity $-1$, and at the south pole $+1$. Hence one pulse of each helicity, with the
simplest amplitudes $a_\mp=\delta(1\pm z)^2$ that are nonzero at the respective pole ($a_-(1)=4\delta$) and smooth at
the other. Since $a_\mp$ vanish only at $z=\mp1$, at every latitude at least one pulse is nonzero.

\subsubsection{Solving backwards}
With the path \eqref{eq:tailseed} fixed, we solve \eqref{eq:tailRay} backwards from the terminal data, separately at
each point of $S^2$:
\begin{equation}\label{eq:tail}
 \partial_t^2\phi=-e\phi,\qquad \phi(1)=\sqrt2M,\qquad \partial_t\phi(1)=0 .
\end{equation}
Since $e=e(t,z)$, the solution $\phi=\phi(t,z)$ is axisymmetric.

\begin{lemma}[Tail]\label{lem:tail}
If $\sup_z\int_0^1t\,e(t,z)\,dt<\frac13$, then $\phi\ge\frac23\sqrt2M$, $\partial_t\phi\ge0$, and $\partial_t\phi(0,z)>0$ for
every $z$. Near $t=1$: $\phi\equiv\sqrt2M$, $\partial_t\phi\equiv0$, $\gdc\equiv h_H$. The hypothesis holds for $\delta$ small.
\end{lemma}

\begin{proof}
\emph{Integral form.} Integrating \eqref{eq:tail} once and twice from $t=1$, and using Fubini,
\[
 \partial_t\phi(t)=\int_t^1e\phi\,ds,\qquad
 \phi(t)=\sqrt2M-\int_t^1(s-t)\,e(s)\phi(s)\,ds .
\]

\emph{Positivity by continuity.} Let $t_0\in[0,1)$ be such that $\phi>0$ on $[t_0,1]$. Since $e\ge0$, the first formula
gives $\partial_t\phi\ge0$ on $[t_0,1]$, hence $\phi\le\sqrt2M$ there. Inserting this in the second formula,
\[
 \phi(t)\ge\sqrt2M\Big(1-\int_t^1(s-t)e\,ds\Big)\ge\sqrt2M\Big(1-\int_0^1se\,ds\Big)>\tfrac23\sqrt2M
 \qquad(t\in[t_0,1]).
\]
So the set of $t_0$ with $\phi>0$ on $[t_0,1]$ is open and closed in $[0,1]$ and contains a neighbourhood of $1$; hence
it is all of $[0,1]$, and the bounds hold everywhere.

\emph{Strict focusing at $t=0$.} $\partial_t\phi(0,z)=\int_0^1e\phi\,dt\ge\frac23\sqrt2M\int_0^1e\,dt>0$, because at every latitude
at least one of the two polar pulses is nonzero, so $e(\cdot,z)\not\equiv0$.

\emph{The hypothesis.} By \eqref{eq:taile}, on the connector $e\le\frac14(\log2)^2C'^2$ (since $|q|\le\log2$), so its
contribution to $\int te$ is at most $\frac14(\log2)^2\int C'^2<\frac12(\log2)^2=0.2402\ldots$. The pulses live in
$[0,\frac13]$ and $|a_\sigma|\le4\delta$, so they contribute at most $\frac13\cdot\frac14\cdot16\delta^2\int(R_-'^2+R_+'^2)\,dt$,
which is less than $\frac13-0.2403$ for $\delta$ small.

\emph{Near $t=1$} the seed is constant, so $e=0$, and uniqueness for \eqref{eq:tail} gives $\phi\equiv\sqrt2M$ there.
\end{proof}

\paragraph{What the tail hands to the window.} At $t=0$ we have $\phi(0,z)\in[\frac23\sqrt2M,\sqrt2M]$ and
$\partial_t\phi(0,z)>0$, both axisymmetric, together with $\omega=0$ and $\gdc=\gac$. Near $t=0$ the seed is constant ($R_\pm$
vanish there), so $e=0$ and $\phi$ is affine in $t$ near $t=0$:
\[
 \phi=\phi_0(z)+p_0(z)\,t,\qquad \phi_0\ge\tfrac23\sqrt2M,\quad p_0>0 .
\]
This is the germ to which the window is joined. In particular $\phi_0^2\ge\frac89M^2$, which explains the
constant $\frac89$ in Lemma~\ref{lem:initialpart}. Note also that the tail reaches $\tr\chi=0$ exactly without
dividing by $D\phi$, and that no lapse is recovered on it.

\subsubsection{What the tail leaves in the torsion}

We need two tools that will also be used in Section~\ref{ch:spin}.

\begin{lemma}[Divergence identity]\label{lem:divid}
Let $\gdc$ be a metric on $S^2$ with $\det\gdc=\det\gac$, and $\theta$ a symmetric two-tensor. Then
\begin{equation}\label{eq:divid}
 \big[\Div_{\gdc}\theta\big]_j=\nabla^{\gac}_i\big((\gdc^{-1}\theta)^i{}_j\big)
 -\tfrac12\tr\big(\gdc^{-1}\theta\;\gdc^{-1}\nabla^{\gac}_j\gdc\big).
\end{equation}
In particular, for $\theta=\frac12\phi^2D\gdc$,
\[
 \big[\Div_{\gdc}(\tfrac12\phi^2D\gdc)\big]_j=\tfrac12\nabla^{\gac}_i\big(\phi^2(\gdc^{-1}D\gdc)^i{}_j\big)
 -\tfrac14\phi^2\tr\big(\gdc^{-1}\nabla^{\gac}_j\gdc\;\gdc^{-1}D\gdc\big).
\]
\end{lemma}
\begin{proof}
Write $\nabla=\nabla^{\gac}$, $g=\gdc$, and let $C^k_{ij}=\frac12g^{kl}(\nabla_ig_{jl}+\nabla_jg_{il}-\nabla_lg_{ij})$ be the difference of
the Levi-Civita connections of $g$ and $\gac$. Then
$(\Div_g\theta)_j=g^{ik}\big(\nabla_i\theta_{kj}-C^l_{ik}\theta_{lj}-C^l_{ij}\theta_{kl}\big)$, while
$\nabla_i(g^{ik}\theta_{kj})=g^{ik}\nabla_i\theta_{kj}-g^{ia}g^{bk}(\nabla_ig_{ab})\theta_{kj}$. Hence
\[
 (\Div_g\theta)_j-\nabla_i(g^{ik}\theta_{kj})=g^{ia}g^{bk}(\nabla_ig_{ab})\theta_{kj}-g^{ik}C^l_{ik}\theta_{lj}-g^{ik}C^l_{ij}\theta_{kl}.
\]
For the second term,
\[
 g^{ik}C^l_{ik}=g^{ik}g^{lm}\nabla_ig_{km}-\tfrac12g^{lm}g^{ik}\nabla_mg_{ik},\qquad
 g^{ik}\nabla_mg_{ik}=\partial_m\log(\det g/\det\gac)=0;
\]
so the second term is $-g^{ik}g^{lm}(\nabla_ig_{km})\theta_{lj}$, which
cancels the first. With $\theta^{im}:=g^{ik}g^{ml}\theta_{kl}$ (symmetric), the third term is
$-\frac12\theta^{im}(\nabla_ig_{jm}+\nabla_jg_{im}-\nabla_mg_{ij})$; the first and last pieces cancel by the symmetry of
$\theta^{im}$, leaving $-\frac12\theta^{im}\nabla_jg_{im}=-\frac12\tr(g^{-1}\theta\,g^{-1}\nabla_jg)$.
\end{proof}

\begin{lemma}[The round frame connection]\label{lem:roundconn}
For a symmetric two-tensor $T$ with frame matrix (also denoted $T$) in the frame \eqref{eq:roundframe},
\[
 \nabla^{\gac}_{\partial_\vartheta}T=\partial_\vartheta T,\qquad \nabla^{\gac}_{\partial_\varphi}T=\partial_\varphi T+z\,[T,K_0]
\]
as frame matrices, where $z=\cos\vartheta$ and $K_0$ is as in Definition~\ref{def:SR}. Moreover
$[K_0,S(\xi)]=-2S'(\xi)$.
\end{lemma}
\begin{proof}
The round frame has $\nabla_{\partial_\vartheta}e_a=0$ and $\nabla_{\partial_\varphi}e_1=\cos\vartheta\,e_2$, $\nabla_{\partial_\varphi}e_2=-\cos\vartheta\,e_1$ (from the
Christoffel symbols $\Gamma^\varphi_{\vartheta\varphi}=\cot\vartheta$, $\Gamma^\vartheta_{\varphi\varphi}=-\sin\vartheta\cos\vartheta$). So
$\nabla_{\partial_\varphi}e_b=\sum_a(\Gamma_\varphi)_{ab}e_a$ with $\Gamma_\varphi=z\begin{pmatrix}0&-1\\1&0\end{pmatrix}=-zK_0$. Then
$(\nabla_{\partial_\varphi}T)(e_a,e_b)=\partial_\varphi T_{ab}-T(\nabla_{\partial_\varphi}e_a,e_b)-T(e_a,\nabla_{\partial_\varphi}e_b)$, i.e.\
$\nabla_{\partial_\varphi}T=\partial_\varphi T-\Gamma_\varphi^tT-T\Gamma_\varphi=\partial_\varphi T-zK_0T+zTK_0$, since $\Gamma_\varphi^t=zK_0$. For the
commutator, $K_0=SS'$ gives $K_0S=SS'S=-SK_0$, so $[K_0,S]=-2SK_0=-2SSS'=-2S'$.
\end{proof}

We now compute the torsion produced by the tail. On the tail $\phi$ and $\Omega=\Omega_{\rm tail}(z)$ are axisymmetric,
so in \eqref{eq:Z} the gradient terms $D\phi\Nd\phi-\phi\Nd D\phi+4\phi D\phi\Nd\log\Omega$ are multiples of $dz$ with
axisymmetric coefficients. Such one-forms are exact ($f(z)dz=d\int f$) and do not pair with $\partial_\varphi$. It
remains to integrate the divergence term over the tail.

\paragraph{The connector.} For $G=\mathrm{diag}(e^{Cq},e^{-Cq})$ and axisymmetric $\phi$, $G^{-1}\partial_tG=C'q\,E$ with
$E=\mathrm{diag}(1,-1)$. In \eqref{eq:divid}, the first term is $\frac12\nabla_i(fE)^i{}_j$ with $f=\phi^2C'q$ axisymmetric. As a
$(1,1)$-tensor $E=\partial_\vartheta\otimes d\vartheta-\partial_\varphi\otimes d\varphi$, and the $\varphi$-component of its divergence vanishes:
$(\partial_if)E^i{}_\varphi=-\partial_\varphi f=0$, and $\nabla_iE^i{}_\varphi=\frac1{\sin\vartheta}\partial_\varphi(\sin\vartheta E^\varphi{}_\varphi)-\Gamma^k_{i\varphi}E^i{}_k=0$
(only $\Gamma^\varphi_{\vartheta\varphi}$ and $\Gamma^\vartheta_{\varphi\varphi}$ are nonzero, and they meet the vanishing off-diagonal entries of
$E$). In the trace term, by Lemma~\ref{lem:roundconn}, $\nabla_{\partial_\varphi}G=z[G,K_0]$ (as $G$ does not depend on $\varphi$), and
\[
 \tr\big(G^{-1}z[G,K_0]\,E\big)=z\big(\tr(K_0E)-\tr(K_0GEG^{-1})\big)=0,
\]
because $G$ and $E$ commute. So the $\partial_\varphi$-component of $\Div_{\gdc}(\frac12\phi^2D\gdc)$ vanishes identically, and the
$dz$-component depends on $(t,z)$ only. The connector changes $Z$ by an exact axisymmetric one-form, which is
removed at the terminal sphere by the angular normalisation of the generator (Proposition~\ref{prop:reparam}).

\paragraph{The pulses.} On a pulse $\gdc^{-1}D\gdc=\rho_tS(2\sigma\varphi)$ with $\rho=a_\sigma(z)R_\sigma(t)$. For the trace term of
\eqref{eq:divid} we use
\begin{equation}\label{eq:fixedpol}
 \tr\big(\gdc^{-1}\nabla^{\gac}_j\gdc\;\gdc^{-1}D\gdc\big)=2\rho_t\,\partial_j\rho\qquad\text{for $G=\exp(\rho S(\xi))$ with $\xi_t=0$.}
\end{equation}
To prove \eqref{eq:fixedpol}, write $\nabla_jG=\partial_jG+z[G,K_0]$ for $j=\varphi$ and $\nabla_jG=\partial_jG$ for $j=\vartheta$
(Lemma~\ref{lem:roundconn}). The $\partial_jG$ part gives, by \eqref{eq:hypcurrent} with $X=\partial_j$, $Y=\partial_t$,
$2\partial_j\rho\,\rho_t+2\sinh^2\!\rho\,\partial_j\xi\,\xi_t=2\rho_t\partial_j\rho$. The connection part gives
$z\rho_t\tr\big(G^{-1}[G,K_0]S\big)=z\rho_t\big(\tr(K_0S)-\tr(K_0GSG^{-1})\big)=0$, since $G=\cosh\rho\, I+\sinh\rho\, S$ commutes
with $S$.

So the trace term is $-\frac14\phi^2\cdot2\rho_t\partial_j\rho$, again an axisymmetric multiple of $dz$, hence exact. The first
term of \eqref{eq:divid} is $\frac12\nabla^{\gac}_i\big(\phi^2\rho_tS^i{}_j\big)$, which is linear in $\phi^2\rho_tS$. Since
$\nabla^{\gac}$ and $S(2\sigma\varphi)$ do not depend on $t$, integration over the pulse gives
\[
 \Delta Z=\tfrac12\Div_{\gac}C_\sigma+(\text{exact}),\qquad
 C_\sigma:=\Big(\int\phi^2\rho_t\,dt\Big)\,S(2\sigma\varphi)=m_\sigma(z)\,S(2\sigma\varphi),
\]
\[
 m_\sigma(z)=a_\sigma(z)\int\phi^2R_\sigma'\,dt=-a_\sigma(z)\int\partial_t(\phi^2)\,R_\sigma\,dt ,
\]
the last by integration by parts ($R_\sigma$ has compact support). Since $R_\sigma\ge0$ and $\partial_t\phi\ge0$, we get
$m_\sigma\le0$, with strict inequality wherever $a_\sigma\ne0$: if $R_\sigma(t)>0$, then $R_\sigma$ decreases to $0$ after $t$,
so $e=\frac14a_\sigma^2R_\sigma'^2\not\equiv0$ on $[t,1]$ and $\partial_t\phi(t)=\int_t^1e\phi>0$. The memory would
vanish if $\phi$ were constant during the pulse; it does not, because the tail focuses.

Thus the only non-axisymmetric memory of the tail is the pair of tracefree tensors $C_\pm$, each of a single
helicity, i.e.\ a pure $e^{\pm2i\varphi}$ mode. They do not contribute to the $\varphi$-average of $Z(\partial_\varphi)$, but they
change the non-axisymmetric part of the torsion. They are cancelled on the window by two reserved nonrotating
pulses with the opposite memory (Section~\ref{sec:cancel}); this works because the first term of \eqref{eq:divid}
is linear in the memory for fixed-polarisation pulses.

\section{The partial bridge II: the window, the spin, the terminal sphere}\label{ch:spin}

The initial part (Lemma~\ref{lem:initialpart}) ends at $v_w$ with round conformal metric and area $cM^2$.
The focusing tail (Section~\ref{sec:tail}) starts at $v_{\rm tail}$ with round conformal metric and area
$\phi_0^2\ge\frac89M^2$. In this section we fill the gap with the control window, insert the pulses that pump in
the angular momentum, cancel all non-axisymmetric torsion, and normalise the torsion at the terminal sphere.
This completes the proof of Proposition~\ref{prop:partial}.

\subsection{The control window}\label{sec:window}

On the window $[v_w,v_{\rm tail}]$ we require:
\begin{enumerate}[label=(W\arabic*)]
\item \emph{Smooth junctions.} The data agree with the initial part near $v_w$ and with the tail near
$v_{\rm tail}$, to all orders.
\item \emph{Expansion.} $D\phi>0$ everywhere, with a positive minimum, so that the lapse formula
\eqref{eq:lapse} can be used.
\item \emph{Affine area in the middle.} On a middle interval $[v_w+\varepsilon,v_{\rm tail}-\varepsilon]$, which will contain all
pulses, $u:=\phi^2=B(z)+\varpi(z)\,v$ with $\varpi(z)>0$.
\item \emph{Axisymmetry} of $\phi$, and hence of $e$ and $\Omega$ as long as the energy is axisymmetric.
\end{enumerate}
Requirement (W3) yields the exact cancellations below. It is used three times:
\begin{itemize}
\item the weighted memory $\int\phi^2(\text{symmetric part})\,dt$ of a closed pulse depends only on its shape and
duration, not on its centre (Lemma~\ref{lem:closedmem}), so pulses with phases differing by $\pi$ cancel;
\item $\int\phi^2C'\,dt=\phi^2(t_j,z)$ for $C'$ even about $t_j$;
\item $\int\phi^2b'\,dt=-\varpi\int b$ for the reserved nonrotating pulses.
\end{itemize}
All three use that $\partial_v\phi^2=\varpi(z)$ does not depend on $v$. Requirement (W3) cannot hold on the whole window,
because neither end is affine: the tail germ has $u=(\phi_0+p_0t)^2$, quadratic in $t$, and the initial germ is
whatever the area raising produced. So the affine middle is joined to both ends by short transition
intervals, the \emph{collars}, which contain no pulses (Figure~\ref{fig:bridgearea}).

\paragraph{The lapse is recovered on the window.} Raychaudhuri's equation $D^2\phi=2\omega D\phi-e\phi$ is one scalar
equation at each point for the three unknowns $\phi,\omega,\gdc$. On the tail we prescribed $\omega=0$ and $\gdc$ and
solved for $\phi$. On the window this is not possible, because (W3) prescribes $\phi$: with $\omega=0$, an affine $u$
with $\partial_vu=\varpi$ would force
\[
 e=-\phi^{-1}D^2\phi=\tfrac14\varpi^2u^{-2},
\]
since $\phi=\sqrt u$ gives $D^2\phi=-\frac14\varpi^2u^{-3/2}$; this is not the energy of the pulses we want. So on the window
we prescribe $\phi$ and $\gdc$ and solve Raychaudhuri's equation for the lapse, by \eqref{eq:lapse}:
\[
 \omega=\frac{D^2\phi+e\phi}{2D\phi},\qquad \log\Omega(v)=\log\Omega(v_w)+\int_{v_w}^v\omega .
\]
As long as $D\phi>0$, every choice of $(\phi,\gdc)$ is admissible. The optical level on the window therefore reduces to
the choice of one positive function $u(v,z)$ with $\partial_vu>0$. The pulses only change $e$, and hence $\omega$; they do
not change $\phi$. This is why they can be inserted afterwards.

\begin{lemma}[Junction]\label{lem:join}
Let $u_a:=M^2\hat u$ be the area of the initial part near $v_w$ (Lemma~\ref{lem:initialpart}(ii)) and
$u_b=(\phi_0(z)+p_0(z)(v-v_{\rm tail}))^2$ the area of the tail near $v_{\rm tail}$, both extended a little into the window
with positive slopes. Since
\begin{equation}\label{eq:joincond}
 A(z)=u_a(v_w,z)=cM^2<\tfrac89M^2\le\phi_0(z)^2\qquad\text{for all }z,
\end{equation}
for every window length $\ell=v_{\rm tail}-v_w>0$ and every sufficiently small collar width $\varepsilon>0$ there is a smooth
axisymmetric $u$ on the window with: $u=u_a$ near $v_w$ and $u=u_b$ near $v_{\rm tail}$; $\partial_vu>0$
with a positive minimum; and $u=B(z)+\varpi(z)v$, $\varpi>0$, on $[v_w+\varepsilon,v_{\rm tail}-\varepsilon]$.
\end{lemma}
\begin{proof}
This is Lemma~\ref{lem:slope} with $[v_a,v_b]=[v_w,v_{\rm tail}]$; condition \eqref{eq:slopecond} is \eqref{eq:joincond}.
\end{proof}

\paragraph{Consequences.} Set $\phi:=\sqrt u$ on the window and $\gdc:=\gac$ (until pulses are inserted).
\begin{enumerate}[label=(\roman*)]
\item \emph{Expansion.} $D\phi=\partial_vu/(2\phi)>0$, with a positive minimum. This is (W2).
\item \emph{The lapse at the junctions.} In the collars there are no pulses, so $e=0$ and $\omega=D^2\phi/(2D\phi)$. Near $v_w$,
$\phi$ coincides with the area raising stretch, whose lapse was recovered from the same formula, so $\omega$
continues the earlier lapse to all orders. Near $v_{\rm tail}$, $\phi$ is affine, so $D^2\phi=0$ and $\omega=0$, which is
the tail's $\omega$. Integrating, $\Omega$ is smooth across both junctions, and on the tail it is the
$v$-independent function
\[
 \Omega_{\rm tail}(z)=\Omega(v_w,z)\exp\Big(\int_{v_w}^{v_{\rm tail}}\omega\,dv\Big),
\]
which is axisymmetric and positive. Inserting pulses in the middle changes $\omega$ only there, and hence changes
$\Omega_{\rm tail}$ by a positive factor; nothing on the tail depends on its value.
\item \emph{All other fields are smooth.} The seed $\gdc$, the area $\phi$ and the lapse $\Omega$ are $C^\infty$
on the whole bridge, because near each junction they coincide with the data on the other side. The remaining
fields ($\etab$, $\Omega^{-1}\chib$, $\Omega^{-2}\ab$) are not prescribed: they are obtained by integrating the linear
transport equations \eqref{eq:H3}--\eqref{eq:H5} forward from $S_1$, with smooth coefficients. So they are smooth
too, and nothing has to be matched by hand.
\end{enumerate}

\subsection{Closed rotating pulses}\label{sec:closedpulses}

All pulses of this section are placed in the affine middle of the window. The conformal metric is round on the
whole window except during the pulses, each of which starts and ends at $\gdc=\gac$ and is disjoint from the others.

We use four disjoint \emph{closed rotating} pulses, two of each helicity $\sigma=\pm1$:
\begin{equation}\label{eq:rotpulse}
 \gdc=\gac\exp\big\{a_\sigma(z)R(t)\,S\big(2\sigma\varphi-2\pi\sigma C(t)+\delta_j\big)\big\}.
\end{equation}
Here $t$ is the window parameter, $R$ is a plateau bump, the rotation $C$ takes place on the plateau, $C'$ is even
about the pulse centre $t_j$, and the two pulses of the same helicity have the same shape and phases $\delta_j$
differing by $\pi$. With $\rho=a_\sigma R$ and $\chi=2\sigma\varphi-2\pi\sigma C+\delta_j$, \eqref{eq:Apulse} gives
\begin{equation}\label{eq:Apulse2}
 \gdc^{-1}\partial_t\gdc=\rho_tS(\chi)+\chi_t\sinh\rho\cosh\rho\,S'(\chi)-\chi_t\sinh^2\!\rho\,K_0,\qquad
 e=\tfrac14\big(\rho_t^2+\sinh^2\!\rho\,\chi_t^2\big).
\end{equation}

\paragraph{Closed pulses.} A closed rotating pulse has the following time structure on an interval
$[t_0,t_3]$, $t_0<t_1<t_2<t_3$:
\begin{enumerate}[label=(C\arabic*)]
\item \emph{rise:} on $[t_0,t_1]$, $R$ increases from $0$ to $1$ and $C\equiv0$;
\item \emph{rotation:} on $[t_1,t_2]$, $R\equiv1$ and $C$ increases from $0$ to $1$;
\item \emph{fall:} on $[t_2,t_3]$, $R$ decreases from $1$ to $0$ and $C\equiv1$.
\end{enumerate}
At a fixed point of $S^2$ the distortion is switched on at a fixed axis, turned at fixed strength, and switched off.
The phase $\chi$ changes by $-2\pi\sigma$, one full period of $S$; equivalently, the spin-two pattern is rigidly rotated
about the polar axis by $\pi$, after which it coincides with itself. In the hyperbolic picture, $G(t,x)$
runs out radially to distance $a_\sigma$, once around the circle of radius $a_\sigma$, and back along the same ray.

Three properties of \eqref{eq:Apulse2} are used.
\begin{enumerate}[label=(\roman*)]
\item The energy $e$ is axisymmetric. The area $\phi^2=B(z)+\varpi(z)v$ is prescribed axisymmetric on the window, so the
recovered lapse is axisymmetric, and the terms $D\phi\Nd\phi-\phi\Nd D\phi+4\phi D\phi\Nd\log\Omega$ in \eqref{eq:Z} are gradients of
axisymmetric functions: they carry no angular momentum.
\item The symmetric part of $\gdc^{-1}\partial_t\gdc$ integrates to zero over a closed pulse, and its weighted memory
cancels within each pair (Lemma~\ref{lem:closedmem}).
\item The skew part $-\chi_t\sinh^2\rho\,K_0$ is phase-independent and gives the angular momentum (Lemma~\ref{lem:Lsigma}).
\end{enumerate}

\begin{lemma}[Memory of a closed rotating pulse]\label{lem:closedmem}
Let $G(t)=\exp\big(\rho(t)S(\chi(t))\big)$ be a closed rotating pulse at a fixed point of $S^2$, i.e.\ $\rho=aR$ and
$\chi=\chi_0-2\pi\sigma C$ with (C1)--(C3), where $\chi_0=2\sigma\varphi+\delta_j$. Write $\mathcal A=G^{-1}\partial_tG$. Then:
\begin{enumerate}[label=(\alph*)]
\item $\displaystyle\int_{t_0}^{t_3}\sym\mathcal A\,dt=0$;
\item $\displaystyle\int_{t_0}^{t_3}\skw\mathcal A\,dt=2\pi\sigma\sinh^2\!a\;K_0\neq0$;
\item if $u=u_0+\varpi\,t$ is affine, then
\[
 \int_{t_0}^{t_3}u\,\sym\mathcal A\,dt
 =-\varpi\Big[S(\chi_0)\int_{t_0}^{t_3}\rho\,dt+\sinh a\cosh a\int_{t_1}^{t_2}\big(S(\chi)-S(\chi_0)\big)dt\Big].
\]
\end{enumerate}
The right side of (c) does not contain $u_0$: translating the pulse in time changes only $u_0$, so the memory does not
depend on the position of the pulse. Replacing $\delta_j$ by $\delta_j+\pi$ replaces $S$ by $-S$ everywhere
($S(\xi+\pi)=-S(\xi)$) and so reverses the sign of the memory. Being an identity between functions of $(z,\varphi)$, this
holds with all angular derivatives.
\end{lemma}

\begin{proof}
By \eqref{eq:Apulse2}, $\sym\mathcal A=\rho_tS(\chi)+\chi_t\sinh\rho\cosh\rho\,S'(\chi)$ and $\skw\mathcal A=-\chi_t\sinh^2\!\rho\,K_0$.

\emph{(a)} On the rise and on the fall $\chi_t=0$ and $\chi$ is constant: $\chi_0$ on the rise, $\chi_0-2\pi\sigma$ on the
fall, and $S(\chi_0-2\pi\sigma)=S(\chi_0)$. So these phases contribute
$S(\chi_0)\big(\int_{t_0}^{t_1}\rho_t+\int_{t_2}^{t_3}\rho_t\big)=S(\chi_0)\big((a-0)+(0-a)\big)=0$. On the rotation $\rho\equiv a$,
so $\rho_t=0$, and by the chain rule $\chi_t\sinh\rho\cosh\rho\,S'(\chi)=\sinh a\cosh a\,\partial_t[S(\chi)]$. This phase contributes
$\sinh a\cosh a\,\big(S(\chi_0-2\pi\sigma)-S(\chi_0)\big)=0$.

\emph{(b)} Only the rotation contributes, where $\sinh^2\rho=\sinh^2a$ and $\int\chi_t\,dt=-2\pi\sigma$.

\emph{(c)} On the rise and the fall, $\sym\mathcal A=\rho_tS(\chi_0)$. Integrating by parts, using $u_t=\varpi$, $\rho(t_0)=\rho(t_3)=0$ and
$\rho(t_1)=\rho(t_2)=a$,
\[
 \Big(\int_{t_0}^{t_1}+\int_{t_2}^{t_3}\Big)u\rho_t\,dt=a\big(u(t_1)-u(t_2)\big)-\varpi\Big(\int_{t_0}^{t_1}+\int_{t_2}^{t_3}\Big)\rho\,dt .
\]
Since $u(t_1)-u(t_2)=-\varpi(t_2-t_1)$ and $\rho\equiv a$ on the plateau, the boundary term is $-\varpi\int_{t_1}^{t_2}\rho\,dt$. So rise and
fall together contribute $-\varpi\,S(\chi_0)\int_{t_0}^{t_3}\rho\,dt$. On the rotation, $\sym\mathcal A=\sinh a\cosh a\,\partial_t[S(\chi)-S(\chi_0)]$,
and the bracket vanishes at $t_1$ and $t_2$; integrating by parts gives $-\varpi\sinh a\cosh a\int_{t_1}^{t_2}(S(\chi)-S(\chi_0))dt$.
\end{proof}

Part (a) depends on the full turn: the rotation phase contributes $\sinh a\cosh a\,(S(\chi(t_2))-S(\chi(t_1)))$, which
vanishes only if $\chi$ changes by a multiple of $2\pi$.

\subsection{The axial transfer operator}\label{sec:axial}

\begin{lemma}[Axial transfer]\label{lem:Lsigma}
Let $\rho=\rho(t,z)$ and $u=\phi^2(t,z)$ be axisymmetric and $\chi=2\sigma\varphi+c(t)+\delta(z)$ with $\sigma=\pm1$, and let
$\gdc$ have frame matrix $G=\exp(\rho S(\chi))$. Then for each $t$, the $\varphi$-average of
$\big[\Div_{\gdc}(\frac12uD\gdc)\big](\partial_\varphi)$ equals $L_\sigma\big(u\chi_t\sinh^2\!\rho\big)$, where
\begin{equation}\label{eq:Lsigma}
 L_\sigma f:=\tfrac12(1-z^2)f'(z)-(\sigma+z)f(z).
\end{equation}
More precisely, $\big[\Div_{\gdc}(\frac12uD\gdc)\big](\partial_\varphi)=L_\sigma(u\chi_t\sinh^2\rho)+\frac12\big[\nabla^{\gac}_i\big(u\,(\sym\mathcal A)^i{}_\varphi\big)\big]$,
and the last term has zero $\varphi$-average.
\end{lemma}

\begin{proof}
By Lemma~\ref{lem:divid}, with $\nabla=\nabla^{\gac}$,
\[
 \big[\Div_{\gdc}(\tfrac12uD\gdc)\big]_\varphi=\tfrac12\nabla_i\big(u\mathcal A^i{}_\varphi\big)-\tfrac14u\tr\big(G^{-1}\nabla_\varphi G\,\mathcal A\big),
 \qquad \mathcal A=G^{-1}\partial_tG .
\]
Split $u\mathcal A=u\,\sym\mathcal A-f\,K_0$, $f:=u\chi_t\sinh^2\rho$; note $f=f(t,z)$.

\emph{The skew part of the first term.} As a $(1,1)$-tensor, $K_0$ has frame matrix $\begin{pmatrix}0&1\\-1&0\end{pmatrix}$, i.e.\
$K_0e_1=-e_2$, $K_0e_2=e_1$: it is the rotation by $-90^\circ$, which is parallel for $\gac$. Hence
$\nabla_i(-fK_0)^i{}_\varphi=-(\partial_if)(K_0)^i{}_\varphi$. Since $K_0\partial_\varphi=\sin\vartheta\,K_0e_2=\sin\vartheta\,\partial_\vartheta$, we have
$(K_0)^\vartheta{}_\varphi=\sin\vartheta$ and $(K_0)^\varphi{}_\varphi=0$. So this term is $-\frac12\sin\vartheta\,\partial_\vartheta f=\frac12(1-z^2)\partial_zf$, using
$\partial_\vartheta=-\sin\vartheta\,\partial_z$.

\emph{The trace term.} By Lemma~\ref{lem:roundconn}, $\nabla_\varphi G=\partial_\varphi G+z[G,K_0]$. Since $G=\cosh\rho\,I+\sinh\rho\,S(\chi)$ and
$[S,K_0]=2S'$, we have $[G,K_0]=2\sinh\rho\,S'(\chi)=2\partial_\chi G$. So
\[
 \nabla_\varphi G=(\partial_\varphi+2z\,\partial_\chi)G ,
\]
i.e.\ the frame connection acts as an additional rotation of the phase. By \eqref{eq:hypcurrent} with $X=\partial_\varphi+2z\partial_\chi$
(which acts on the parameters by $X\rho=0$, $X\chi=2\sigma+2z$) and $Y=\partial_t$,
\[
 \tr\big(G^{-1}\nabla_\varphi G\,\mathcal A\big)=2\sinh^2\!\rho\,(2\sigma+2z)\chi_t=4(\sigma+z)\chi_t\sinh^2\!\rho ,
\]
independently of $\varphi$. With the prefactor $-\frac14u$ this is $-(\sigma+z)f$.

\emph{The symmetric part of the first term.} $\sym\mathcal A=\rho_tS(\chi)+\chi_t\sinh\rho\cosh\rho\,S'(\chi)$ has a frame matrix whose
entries are $\cos\chi,\sin\chi$ times functions of $(t,z)$, with $\chi=2\sigma\varphi+\ldots$, $\sigma\ne0$. Rotations about the axis are
isometries of $\gac$ that preserve the frame and $\partial_\varphi$, so the $\varphi$-average of
$\frac12[\nabla_i(u\sym\mathcal A)^i{}_\varphi]$ is the same expression for the $\varphi$-averaged tensor field, whose frame matrix is
$0$.

Adding up gives $\frac12(1-z^2)f'-(\sigma+z)f=L_\sigma f$, plus a term of zero average.
\end{proof}

Thus the term $-zf$ in $L_\sigma f$ comes from the connection of the round frame (a spin-two tensor rotates by twice the
frame angle), and $-\sigma f$ from the explicit $\varphi$-dependence of the phase.

\paragraph{The transfer of one pair.} Integrate over the two pulses of helicity $\sigma$. The symmetric parts give
$\frac12\Div_{\gac}$ of the total weighted memory, which vanishes by Lemma~\ref{lem:closedmem}(c) (the two pulses have opposite
memories). The trace terms of the $\vartheta$-component are axisymmetric multiples of $dz$, hence exact. The axial component
changes by
\begin{equation}\label{eq:pairtransfer}
 \Delta Z(\partial_\varphi)=L_\sigma W,\qquad W:=\int u\chi_t\sinh^2\!\rho\,dt .
\end{equation}
For a closed pulse, $\rho=a_\sigma(z)$ and $\chi_t=-2\pi\sigma C'(t)$ on the plateau, and $\int C'=1$, so
$W=-2\pi\sigma\sinh^2a_\sigma\int uC'\,dt$. Since $u$ is affine and $C'$ is even about the centre $t_j$, $\int uC'=u(t_j,z)$. For the
two pulses of helicity $\sigma$:
\begin{equation}\label{eq:Wpair}
 W=-2\pi\sigma\sinh^2\!a_\sigma(z)\sum_j\phi^2(t_j,z).
\end{equation}

\subsection{Pumping in the angular momentum}\label{sec:pumping}

\paragraph{The target profile.} At the terminal sphere $\phi^2=2M^2$ and the torsion should be $\eta_H$, so the target axial
profile is
\[
 Z_H(\partial_\varphi)=2M^2\eta_H(\partial_\psi)=\frac{4M^2(1-z^2)}{(1+z^2)^2}.
\]
The profile at the start of the window is the Kerr profile at $S_1$ plus the
connector's contribution (Lemma~\ref{lem:initialpart}(iv)); we denote it by $Z_{\rm in}(\partial_\varphi)(z)$. It is a smooth function
of $z$ bounded independently of $M$. The area raising, the collars, the tail and all nonrotating pulses do not change the
axial profile (Lemma~\ref{lem:initialpart}, Section~\ref{sec:tail}, Section~\ref{sec:cancel}). The four rotating pulses must
therefore produce
\[
 H(z):=Z_H(\partial_\varphi)-Z_{\rm in}(\partial_\varphi)=(1-z^2)k(z).
\]
$H$ is divisible by $1-z^2$: the $\partial_\varphi$-component of a smooth axisymmetric one-form vanishes to second order at the
poles (at a pole the form is a rotation-invariant covector, hence zero, and $\partial_\varphi=x\partial_y-y\partial_x$ vanishes linearly),
i.e.\ it vanishes at $z=\pm1$ to first order in $z$. Moreover $k$ is smooth and, for $M$ large, positive:
$k=\frac{4M^2}{(1+z^2)^2}-\frac{Z_{\rm in}(\partial_\varphi)}{1-z^2}$, and the second term is bounded independently of $M$.

By \eqref{eq:pairtransfer} and \eqref{eq:Wpair}, the north pair ($\sigma=-1$) and the south pair ($\sigma=+1$) produce
$L_{-1}F_N-L_{+1}F_S$, where
\[
 F_N:=2\pi\sinh^2\!a_-\sum_{j\in N}\phi^2(t_j,z)\ge0,\qquad F_S:=2\pi\sinh^2\!a_+\sum_{j\in S}\phi^2(t_j,z)\ge0 .
\]
So we must solve $L_{-1}F_N-L_{+1}F_S=(1-z^2)k$ with nonnegative $F_N,F_S$ of the right pole behaviour.

\begin{lemma}[Positive inverse]\label{lem:Finv}
Let $k>0$ be smooth on $[-1,1]$, and
\begin{equation}\label{eq:Finverse}
 F(z)=\frac1{4(1+z^2)^{3/2}}\int_0^1t^{1/2}\sqrt{1+z^2t^2}\,k(zt)\,dt,\qquad
 F_N=(1+z)^4F,\quad F_S=(1-z)^4F .
\end{equation}
Then $F$ is smooth and positive on $[-1,1]$, $F(0)=k(0)/6$, and
\[
 4z(1+z^2)F'+6(1+3z^2)F=k,\qquad L_{-1}F_N-L_{+1}F_S=(1-z^2)k .
\]
\end{lemma}
\begin{proof}
\emph{Smoothness and positivity.} The integrand is smooth in $(t,z)\in[0,1]\times[-1,1]$ apart from the factor $t^{1/2}$,
which does not depend on $z$, so we may differentiate under the integral sign to any order. Positivity is clear, and
$F(0)=\frac14k(0)\int_0^1t^{1/2}dt=k(0)/6$.

\emph{The first identity.} Write $F=I/(4w^{3/2})$ with $w=1+z^2$ and $I(z)=\int_0^1t^{1/2}g(zt)\,dt$, $g(x):=\sqrt{1+x^2}\,k(x)$.
Then $F'=\frac{I'}{4w^{3/2}}-\frac{3zI}{4w^{5/2}}$, and
\[
 4z(1+z^2)F'+6(1+3z^2)F=\frac{zI'}{w^{1/2}}-\frac{3z^2I}{w^{3/2}}+\frac{3(1+3z^2)I}{2w^{3/2}}
 =w^{-1/2}\Big(zI'+\tfrac32I\Big),
\]
because $-3z^2+\frac32(1+3z^2)=\frac32(1+z^2)$. Now $z\,\partial_z[g(zt)]=t\,\partial_t[g(zt)]$, so integrating by parts in $t$,
\[
 zI'=\int_0^1t^{3/2}\partial_t[g(zt)]\,dt=g(z)-\tfrac32I ,
\]
the boundary term at $t=0$ vanishing because of $t^{3/2}$. Hence $zI'+\frac32I=\sqrt{1+z^2}\,k(z)$, which is the first identity, for
either sign of $z$.

\emph{The second identity.} With $L_\sigma W=\frac12(1-z^2)W'-(\sigma+z)W$,
\begin{align*}
 L_{-1}F_N&=(1+z)^4\big[3(1-z)F+\tfrac12(1-z^2)F'\big],\\
 L_{+1}F_S&=(1-z)^4\big[-3(1+z)F+\tfrac12(1-z^2)F'\big].
\end{align*}
(For the first: $\frac12(1-z^2)\cdot4(1+z)^3F=2(1-z)(1+z)^4F$ and $-(z-1)(1+z)^4F=(1-z)(1+z)^4F$; the second is analogous.)
Subtracting, and using $(1+z)^4(1-z)+(1-z)^4(1+z)=(1-z^2)\big((1+z)^3+(1-z)^3\big)=(1-z^2)(2+6z^2)$ and
$(1+z)^4-(1-z)^4=8z(1+z^2)$,
\[
 L_{-1}F_N-L_{+1}F_S=(1-z^2)\big[6(1+3z^2)F+4z(1+z^2)F'\big]=(1-z^2)k .\qedhere
\]
\end{proof}

\paragraph{The amplitudes.} We choose
\[
 a_-=\operatorname{arsinh}\sqrt{\frac{F_N}{2\pi\sum_{j\in N}\phi^2(t_j,z)}},\qquad
 a_+=\operatorname{arsinh}\sqrt{\frac{F_S}{2\pi\sum_{j\in S}\phi^2(t_j,z)}} .
\]
Since $\sqrt{F_N}=(1+z)^2\sqrt F$ with $\sqrt F$ smooth and positive, and $\operatorname{arsinh}y=y\,h(y^2)$ with $h$ smooth, $a_-$ is
$(1+z)^2$ times a smooth function, which is what Lemma~\ref{lem:poles} requires for helicity $-1$ at the south pole (there is no
condition at the north pole). Likewise $a_+=(1-z)^2\cdot(\text{smooth})$. So the rotating pulses are smooth and
$\mathcal R_\pi$-invariant (only the modes $e^{\pm2i\varphi}$ occur), and they match the whole axial profile $Z(\partial_\varphi)$, not
just the total angular momentum.

\subsection{Cancelling the remaining non-axisymmetric memory}\label{sec:cancel}

The polar pulses of the tail leave the tracefree memories $C_\sigma=m_\sigma(z)S(2\sigma\varphi)$ of single helicity
(Section~\ref{sec:tail}). We reserve two disjoint nonrotating pulses $\gdc=\gac\exp(b(t)T_\sigma)$ in the affine middle of the
window, with a fixed bump $b\ge0$ and $T_\sigma$ independent of $t$. Then $\gdc^{-1}\partial_t\gdc=b'T_\sigma$ exactly (commuting
matrices), and
\[
 \int\phi^2\,b'T_\sigma\,dt=-\varpi\Big(\int b\Big)T_\sigma
\]
by integration by parts, since $\partial_t\phi^2=\varpi$. Choosing $T_\sigma=C_\sigma/(\varpi\int b)$ makes the weighted memory equal to $-C_\sigma$,
which cancels the tail's memory in the first term of \eqref{eq:divid}. These pulses have no other effect:
\begin{itemize}
\item their energy $e=\frac18b'^2\tr(T_\sigma^2)=\frac14b'^2m_\sigma^2/(\varpi\int b)^2$ is axisymmetric, so $\phi$ and the lapse stay
axisymmetric;
\item their trace term in \eqref{eq:divid} is an axisymmetric gradient, by \eqref{eq:fixedpol};
\item their axial torque vanishes: the first term has zero $\varphi$-average (single helicity), and the trace term has no
$\partial_\varphi$-component.
\end{itemize}
As to pole regularity, $T_\sigma=m_\sigma S(2\sigma\varphi)/(\varpi\int b)$ with $m_\sigma=a_\sigma\cdot(\text{smooth})$ and
$a_\sigma=\delta(1\mp z)^2$; so $T_\sigma$ has the pole behaviour of the tail pulse it cancels, which is smooth by
Lemma~\ref{lem:poles}.

\subsection{The terminal sphere}\label{sec:terminal}

At $v_*$ we have $\gd=M^2\gd_H$ and $\Omega\chi=0$ on a neighbourhood (the seed is constant and $D\phi=0$). By
\eqref{eq:chiLtransport} and \eqref{eq:RicXL} (as in the proof of Proposition~\ref{prop:kerrdata}),
$\alpha=\beta=0$ there.

\paragraph{The torsion is correct up to an exact form.} At $v_*$, $Z=\phi^2\eta$ is axisymmetric: every non-axisymmetric
contribution (the memories of the rotating pulses and of the tail) has been cancelled, and everything else was
axisymmetric. Its axial component is $Z_H(\partial_\varphi)$. An axisymmetric one-form with zero $\partial_\varphi$-component is
$f(z)dz$, which is exact. Since $\phi^2=2M^2$ and $\eta=-\etab+2d\log\Omega$ with $\Omega=\Omega_{\rm tail}(z)$,
\[
 -\etab=\eta_H-db\qquad\text{for an axisymmetric function }b(z).
\]
By the next proposition, replacing $L$ by $aL$ with $a=e^{-b}$ at the terminal sphere changes $-\etab$ by $-d\log a=db$,
which gives exactly $\eta_H$.

\begin{proposition}[Angle-dependent reparametrisation of the generators]\label{prop:reparam}
Let $\Hcal$ be swept out by cuts $S_v$, $v\in[v_0,v_*]$, and let $a(x)>0$ be smooth on $S^2$. There is a smooth reparametrisation
$v=F(V,x)$ of the generators which is the identity near $v_0$, fixes both end cuts, and equals $v=v_*+a(x)(V-v_*)$ near $v_*$. At
the terminal cut its effect is: $\gd,\phi,\gdc$ are unchanged, the new generator field is $\tilde L=aL$, the new incoming normal is
$a^{-1}\Lb'$, and
\begin{gather*}
 \widetilde{D\phi}=a\,D\phi,\qquad \tilde\omega=a\,\omega,\qquad \widetilde{D\omega}=a^2D\omega,\\
 \tilde\etab=\etab+d\log a,\qquad
 \widetilde{\Omega^{-1}\chib}=a^{-1}\Omega^{-1}\chib,\qquad \widetilde{\Omega^{-2}\ab}=a^{-2}\Omega^{-2}\ab .
\end{gather*}
If the incoming optical coordinate $u$ is held fixed, then $\tilde\Omega^2=a\Omega^2$ and $\tilde\eta=\eta$ at the terminal cut.
\end{proposition}

\begin{proof}
\emph{Construction.} Choose nonnegative bumps $\varrho_e,\varrho_i\le1$ of $v$ with disjoint supports, $\varrho_e=1$ near $v_*$ and $\varrho_i$
supported in the interior, with integrals $I_e,I_i>0$. Define the inverse map $V=V(v,x)$ by
\[
 \partial_vV=1+(a^{-1}-1)\varrho_e(v)+(c-1)\varrho_i(v),\qquad V(v_0,x)=v_0,\qquad c(x):=1-(a^{-1}(x)-1)\frac{I_e}{I_i}.
\]
If the collar where $\varrho_e\ne0$ is short enough ($I_e/I_i$ small), then $c>0$, and on each support $\partial_vV$ is a convex combination of
positive numbers, so $V$ is strictly increasing along each generator. Its total increment is
$(v_*-v_0)+(a^{-1}-1)I_e+(c-1)I_i=v_*-v_0$, so the end cut is fixed, and near $v_*$, $\partial_vV=a^{-1}$, i.e.\ $v=v_*+a(V-v_*)$.

\emph{Endpoint formulas.} The new generator field is $\tilde L=\partial_V|_x=(\partial_VF)\,\partial_v=a\,L$ near $v_*$, with $a$ constant along
generators. The tangent fields of the new cuts are $\partial_A|_V=\partial_A+(\partial_AF)L$, and $\partial_AF=(V-v_*)\partial_Aa$ vanishes at the terminal
cut, so the terminal cut and its tangent spaces are unchanged; hence $\gd,\phi,\gdc$ are unchanged, and the normal normalised by
$\bg(\tilde L,\tilde\Lb')=-2$ is $\tilde\Lb'=a^{-1}\Lb'$. Then:
\begin{itemize}
\item $\widetilde{D\phi}=\tilde L\phi=aD\phi$;
\item $\nabla_{\tilde L}\tilde L=a\nabla_L(aL)=a\big(L(a)+2a\omega\big)L$, and this equals $2\tilde\omega\tilde L=2\tilde\omega aL$; with $L(a)=0$,
$\tilde\omega=a\omega$, and $\widetilde{D\omega}=\tilde L(a\omega)=a^2D\omega$;
\item $\tilde\tau(X)=\frac12\bg(\nabla_X(aL),a^{-1}\Lb')=\frac12X(\log a)\bg(L,\Lb')+\tau(X)=-X(\log a)+\tau(X)$, so with $\tau=-\etab$,
$\tilde\etab=\etab+d\log a$;
\item $\bg(\nabla_X(a^{-1}\Lb'),Y)=a^{-1}\bg(\nabla_X\Lb',Y)$ for $X,Y$ tangent, so $\widetilde{\Omega^{-1}\chib}=a^{-1}\Omega^{-1}\chib$; and
$\bR(X,a^{-1}\Lb',Y,a^{-1}\Lb')=a^{-2}\bR(X,\Lb',Y,\Lb')$;
\item if $u$ is fixed, $\Lb=\partial_u$ is unchanged and $\Lb'=\Omega^{-2}\Lb$, so $\tilde\Lb'=a^{-1}\Lb'$ means $\tilde\Omega^2=a\Omega^2$, and
$\tilde\eta=-\tilde\etab+2d\log\tilde\Omega=-\etab-d\log a+2d\log\Omega+d\log a=\eta$.\qedhere
\end{itemize}
\end{proof}

The reparametrisation $F$ is fixed once and for all, before any control parameter is varied, so its effect on the endpoint data is
multiplication by fixed smooth positive functions and addition of a fixed smooth one-form, a smooth diffeomorphism of the endpoint spaces
used later. (On intermediate cuts the transformation is not merely multiplicative, and it differs from the constant boost $\mathfrak b_\lambda$.)

\begin{remark}[The non-expanding continuation]\label{rem:isolated}
Continuing beyond $v_*$ with constant seed, $\phi\equiv\sqrt2M$ and $\omega=0$, gives a non-expanding segment. By
\eqref{eq:H4} and \eqref{eq:NH}, $D(\Omega^{-1}\chib)=\Ecal(\gd_H,\etab_H)=0$ on it, so the induced connection is invariant:
it is an extremal isolated horizon in the sense of Ashtekar--Beetle--Lewandowski \cite{ABL}. However, $\Omega^{-1}\chib$ on
it is not Kerr's; matching it is the task of Sections~\ref{ch:neck}--\ref{ch:sign}.
\end{remark}

This completes the proof of Proposition~\ref{prop:partial}.

\section{Making the endpoint controllable}\label{ch:controller}

The partial bridge reaches the intrinsic horizon data exactly. The backward integration from the final sphere in
Section~\ref{ch:neck} will, however, require slightly different data at the end of the early bridge, depending on the late
seed. We therefore show that every nearby terminal triple (area radius, expansion, torsion) is attained by a nearby seed, with a
smooth inverse. This is Step~2 of the strategy of Section~\ref{sec:strategy}.

The construction is local: it works around any smooth reference solution on an interval $[0,T_0]$ on which $\phi>0$ and
$D\phi>0$, and whose seed is round, $\gdc=\gac$, on a subinterval $U$. No smallness of the reference is assumed outside $U$.

\subsection{The statement}

We insert into $U$ two small but nonzero pulses on disjoint intervals $I_+$, $I_-$, followed by a round shear-free interval
$J\Subset U$:
\begin{equation}\label{eq:ctrlpulses}
 G=\exp\big(\varepsilon_0b_+(v)\,\mathbf B\big)\ \text{on }I_+,\qquad
 G=\exp\big(\varepsilon_0b_-(v)\,\mathbf B_-\big)\ \text{on }I_-,
\end{equation}
where $b_\pm\ge0$ are fixed bumps and, in the round orthonormal frame,
\[
 \mathbf B_-=(1-z)^2S(2\varphi),\qquad \mathbf B=(1+z)^2S(z-2\varphi).
\]
We call $\mathbf B$ the \emph{polar twist}. Writing $\mathsf E=S(0)=\mathrm{diag}(1,-1)$, $\mathsf F=S(\frac\pi2)$ and $\mathsf J$ for the
rotation of polarisation by $90^\circ$ ($\mathsf JS(\xi)=S(\xi+\frac\pi2)$, so $\mathsf J\mathsf E=\mathsf F$, $\mathsf J\mathsf F=-\mathsf E$), we have
$\mathbf B=\cos z\,\mathbf B_++\sin z\,\mathsf J\mathbf B_+$ with $\mathbf B_+=(1+z)^2S(-2\varphi)=(1+z)^2(\cos2\varphi\,\mathsf E-\sin2\varphi\,\mathsf F)$,
since $\cos z\,S(\xi)+\sin z\,S(\xi+\frac\pi2)=S(\xi+z)$. By Lemma~\ref{lem:poles}, $\mathbf B$ (helicity $-1$, with the smooth extra
phase $z$) and $\mathbf B_-$ (helicity $+1$) are smooth; $\mathbf B$ vanishes only at the south pole and $\mathbf B_-$ only at the north
pole. Both are $\mathcal R_\pi$-invariant.

\paragraph{Pairs of pulses.} In the application (Section~\ref{sec:insert}(i)) each of the two pulses is followed by a copy with the opposite
tensor sign and the same radial profile, so that their memories cancel and the reference bridge is not changed. Everything below
applies verbatim to this configuration. The energy $e$ does not depend on the sign, and the rank criterion of
Section~\ref{sec:lapsecontrol} only uses where $e$ is positive and where it vanishes. The phase controls of Section~\ref{sec:phasecontrol} act on the first pulse
of the twisted pair only, which we call $I_+$; its partner on the following interval is kept fixed. Thus $I_+$, $I_-$ denote the
intervals of the first pulse of each pair, and the lapse control interval $I$ contains both pairs and the gaps between them.

\begin{theorem}[Exact local optical--torsion control]\label{thm:ctrl}
Let $s\ge6$. There is a neighbourhood of the reference endpoint $(\phi_0,D\phi_0,Z_0)(T_0)$ in $H^s\times H^s\times H^{s-1}$ such that every
$\mathcal R_\pi$-invariant triple in it is attained exactly by a positive seed $\gdc$ and a positive lapse $\Omega$ with the same initial data,
modifications supported in $U$, $\phi,D\phi>0$, and $\mathcal R_\pi$-invariance. The correction depends smoothly on the target, with norm
bounded by the distance to the reference endpoint. Equivalently, one may prescribe $-\etab$ instead of $Z$.
\end{theorem}

The proof uses three families of controls, one for each component of the target (Table~\ref{tab:controls}), and the inverse
function theorem.
\begin{table}[ht]
\centering
\begin{tabular}{@{}lll@{}}
\toprule
target component & control & support\\\midrule
$(\phi,D\phi)$ (optics) & lapse, $\log\Omega\mapsto\log\Omega+a$ & interval $I$ containing both pairs\\
torsion, except two pairings & tensor, $\gdc\mapsto\gac\exp(\gac^{-1}H)$ & shear-free interval $J$\\
the two pairings $C_E,C_B$ & phase of the twisted pulse & $I_+$\\\bottomrule
\end{tabular}
\caption{The three families of controls.}\label{tab:controls}
\end{table}

\subsection{The optical components: lapse controls}\label{sec:lapsecontrol}

We vary the lapse, $\log\Omega\mapsto\log\Omega+a$, with $a$ supported in a control interval $I$ and the seed fixed. Then
$\omega\mapsto\omega+Da$, and $e$ does not change. Linearising \eqref{eq:Ray} at fixed angle, with $y=(\delta\phi,\delta D\phi)^t$, gives
\[
 D^2\delta\phi=2\,Da\,D\phi+2\omega D\delta\phi-e\,\delta\phi,
\]
i.e.
\[
 Dy=\mathsf Ay+\mathsf B\,Da,\qquad
 \mathsf A=\begin{pmatrix}0&1\\-e&2\omega\end{pmatrix},\qquad \mathsf B=\begin{pmatrix}0\\2D\phi\end{pmatrix}.
\]
To remove the derivative of the control we set $\tilde y=y-\mathsf Ba$. Since $a$ vanishes at the ends of $I$, $\tilde y$ and $y$ have the
same endpoint value, and
\[
 D\tilde y=\mathsf Ay+\mathsf BDa-(D\mathsf B)a-\mathsf BDa=\mathsf A(\tilde y+\mathsf Ba)-(D\mathsf B)a.
\]
As
$\mathsf A\mathsf B=(2D\phi,\,4\omega D\phi)^t$ and, by \eqref{eq:Ray}, $D\mathsf B=(0,\,2D^2\phi)^t=(0,\,4\omega D\phi-2e\phi)^t$, we obtain
\begin{equation}\label{eq:optlin}
 D\tilde y=\mathsf A\tilde y+\mathsf C a,\qquad \mathsf C=\mathsf A\mathsf B-D\mathsf B=\begin{pmatrix}2D\phi\\2e\phi\end{pmatrix}.
\end{equation}
Let $\mathsf U(T_0,v)$ be the fundamental matrix of $\mathsf A$ (from $v$ to $T_0$). Then
\[
 \tilde y(T_0)=\int_I\mathsf U(T_0,v)\mathsf C(v)a(v)\,dv,
\]
a linear map from controls to $\R^2$ at each point of $S^2$.

\begin{lemma}[Optical rank criterion]\label{lem:rank}
On a generator with $\phi,D\phi>0$, the endpoint map $a\mapsto(\delta\phi,\delta D\phi)(T_0)$ has rank two if and only if
$\nu:=e\phi^2/(D\phi)^2$ is not constant on $I$.
\end{lemma}
\begin{proof}
The map fails to be onto if and only if there is a covector $p_T\ne0$ with $p_T\cdot\mathsf U(T_0,v)\mathsf C(v)=0$ for all $v\in I$. Put
$p(v):=\mathsf U(T_0,v)^tp_T$; it solves the adjoint equation $Dp=-\mathsf A^tp$, i.e.\
\[
 Dp_1=ep_2,\qquad Dp_2=-p_1-2\omega p_2,
\]
and the condition is $\mathsf C^tp=0$ on $I$, i.e.\ $D\phi\,p_1+e\phi\,p_2=0$. Under this condition,
\[
 D(D\phi\,p_2)=(2\omega D\phi-e\phi)p_2+D\phi(-p_1-2\omega p_2)=-(e\phi p_2+D\phi p_1)=0,
\]
so $p_2=c/D\phi$ and $p_1=-ce\phi/(D\phi)^2$ for a constant $c$. The remaining adjoint equation $Dp_1=ep_2$ becomes
$c\,[D(e\phi/(D\phi)^2)+e/D\phi]=0$. By the product rule,
\[
 D\nu=D\Big(\phi\cdot\frac{e\phi}{(D\phi)^2}\Big)=D\phi\,\frac{e\phi}{(D\phi)^2}+\phi\,D\Big(\frac{e\phi}{(D\phi)^2}\Big)
 =\phi\,\Big[\frac e{D\phi}+D\Big(\frac{e\phi}{(D\phi)^2}\Big)\Big],
\]
and $\phi>0$, so the condition is $c\,D\nu=0$. If $\nu$ is not constant on $I$, then $c=0$ and $p\equiv0$, so the map is onto. If
$\nu$ is constant, any $c\ne0$ gives a nonzero covector annihilating the image.
\end{proof}

On a pulse $e>0$, and in the gaps $e=0$. We take $I$ to contain both pairs of pulses together with the gaps between them. Every generator,
including both poles (where $\mathbf B$, resp.\ $\mathbf B_-$, is nonzero), meets a pulse where $e>0$ and a gap where $e=0$, so $\nu$ is
not constant. Let $K_o(v,x):=\mathsf U(T_0,v,x)\mathsf C(v,x)$ and let $\chi_I$ be a bump function positive in $I$. The \emph{Gramian}
\[
 \mathsf G_o(x)=\int_I\chi_I^2K_oK_o^t\,dv
\]
is positive definite at every point by Lemma~\ref{lem:rank} (if $p^t\mathsf G_op=\int\chi_I^2|K_o^tp|^2=0$ then $K_o^tp=0$ on
$I$), hence uniformly positive by compactness of $S^2$. An explicit right inverse is
\begin{equation}\label{eq:lapseinv}
 a_r(v,x)=\chi_I(v)^2K_o(v,x)^t\mathsf G_o(x)^{-1}r(x),
\end{equation}
since $\int_I K_o a_r\,dv=\mathsf G_o\mathsf G_o^{-1}r=r$. Its coefficients are smooth, so it is bounded $H^s(S^2)^2\to C^k_vH^s_x$ for every $k$. The
controls are compactly supported in $I$, so the endpoint lapse is not changed.

\subsection{The torsion: tensor controls on the shear-free interval}\label{sec:tensorcontrol}

On $J$ we have $\gdc=\gac$ and $D\gdc=0$. We vary $\gdc$ by a tracefree $H$ supported in $J$, with the lapse fixed:
$\gdc=\gac\exp(\gac^{-1}H)$. Since $e$ is quadratic in $D\gdc$ and $D\gdc=0$ on the background, $\delta e=0$. Hence
$\delta\phi=\delta D\phi=0$ everywhere: these controls do not affect the optics to first order. In the linearisation of \eqref{eq:Z} only the
divergence term varies, and at the round background $\delta\big[\Div_{\gdc}(\frac12\phi^2D\gdc)\big]=\Div_{\gac}(\frac12\phi_0^2DH)$ (the variation
of the connection multiplies $D\gdc=0$). So $D\delta Z=\frac12\Div_{\gac}(\phi_0^2DH)$ on $J$ and $D\delta Z=0$ elsewhere. Integrating by parts
in $v$ ($H$ vanishes at the ends of $J$),
\begin{equation}\label{eq:torsiondiff}
 \delta Z(T_0)=\tfrac12\Div_{\gac}\int_J\phi_0^2DH\,dv=-\Div_{\gac}\int_J\phi_0\,D\phi_0\,H\,dv .
\end{equation}

In the notation of \cite[Appendix~D]{ACR} (after \cite{CK}), $\mathcal D_2^*\zeta:=-\frac12\Nd\ohat\zeta$ maps one-forms to
tracefree symmetric tensors, and $\mathcal D_2:=\Div$ maps tracefree symmetric tensors to one-forms; $\mathcal D_2^*$ is the $L^2$-adjoint of
$\mathcal D_2$. We choose $b_J\in C^\infty_c(J)$ with $\int b_J=1$ and set
\begin{equation}\label{eq:tensorcontrol}
 H_\zeta:=-\frac{b_J(v)}{\phi_0D\phi_0}\,\mathcal D_2^*\zeta .
\end{equation}
By \eqref{eq:torsiondiff}, $\delta Z(T_0)=\mathcal D_2\mathcal D_2^*\zeta$, and the optical response is zero.

\begin{lemma}[$\mathcal D_2\mathcal D_2^*$ on the round sphere]\label{lem:D2}
$\ker\mathcal D_2^*$ is the six-dimensional space $\CK$ of conformal Killing fields of $(S^2,\gac)$ (identified with one-forms), and
\begin{equation}\label{eq:D2D2}
 \mathcal D_2\mathcal D_2^*:\ H^{s+1}\cap\CK^\perp\to H^{s-1}\cap\CK^\perp\ \text{ is an isomorphism.}
\end{equation}
In the $\mathcal R_\pi$-invariant sector exactly two conformal Killing fields survive: $Y_E:=\Nd z=(1-z^2)\partial_z$ and
$Y_B:=\partial_\varphi$ (up to the identification of vectors and one-forms).
\end{lemma}
The first part is standard, see \cite{CK} and \cite[Appendix~D]{ACR}: $\mathcal D_2\mathcal D_2^*$ is elliptic and formally self-adjoint with
$\langle\zeta,\mathcal D_2\mathcal D_2^*\zeta\rangle=\|\mathcal D_2^*\zeta\|^2$, and $\CK$ is spanned by $\Nd x,\Nd y,\Nd z$ and ${}^*\Nd x,{}^*\Nd y,{}^*\Nd z$. Since
$\mathcal R_\pi$ acts by $(x,y,z)\mapsto(-x,-y,z)$, only $\Nd z$ and ${}^*\Nd z\propto\partial_\varphi$ are invariant.

Thus \eqref{eq:tensorcontrol} controls every invariant torsion component except the two pairings
\[
 C_i(Z)=\int_{S^2}Z(Y_i)\,d\mu_{\gac},\qquad i=E,B .
\]
$C_B$ is the angular momentum and $C_E$ a boost charge; they are the obstructions of \cite{ACR} that survive the symmetry.

\subsection{The two missing pairings: phase controls of the twisted pulse}\label{sec:phasecontrol}

On $I_+$ the pulse is $G=\exp(\rho S(\xi_0))$ with
\[
 \rho=\varepsilon_0b_+(v)(1+z)^2,\qquad \xi_0=z-2\varphi .
\]
By \eqref{eq:e2x2} and \eqref{eq:hypcurrent}, for any phase $\xi$,
\begin{equation}\label{eq:hyp}
\begin{gathered}
 e=\tfrac14\big((D\rho)^2+\sinh^2\!\rho\,(D\xi)^2\big),\\
 \tr\big(\gdc^{-1}D\gdc\,\gdc^{-1}\mathcal L_Y\gdc\big)=2D\rho\,Y\rho+2\sinh^2\!\rho\,D\xi\,Y\xi\quad(Y=Y_E,Y_B).
\end{gathered}
\end{equation}
For the second identity we compute the Lie derivatives. For $Y_B=\partial_\varphi$ the round frame is invariant, so the frame matrix of
$\mathcal L_Y\gdc$ is $\partial_\varphi G$, and \eqref{eq:hypcurrent} applies directly. For $Y_E=\Nd z=-\sin\vartheta\,\partial_\vartheta$ both
elements of the round coframe $(d\vartheta,\sin\vartheta\,d\varphi)$ are rescaled by the same factor:
\begin{gather*}
 \mathcal L_Yd\vartheta=d\big(Y(\vartheta)\big)=d(-\sin\vartheta)=-\cos\vartheta\,d\vartheta,\\
 \mathcal L_Y(\sin\vartheta\,d\varphi)=Y(\sin\vartheta)\,d\varphi=-\cos\vartheta\,\sin\vartheta\,d\varphi .
\end{gather*}
Since $G$ is the matrix of $\gdc$ in this coframe, the frame matrix of $\mathcal L_Y\gdc$ is $Y(G)-2\cos\vartheta\,G$. In
the trace in \eqref{eq:hyp} the extra term contributes $-2\cos\vartheta\,\tr(G^{-1}DG)=-2\cos\vartheta\,D\log\det G=0$, since $\det G=1$; the
remaining term is given by \eqref{eq:hypcurrent}.

We perturb the phase, $\xi=\xi_0+\tau\,k(v,x)$ with $\tau\in\R$ and $k$ supported in $I_+$. Since $D\xi_0=0$, the first identity in \eqref{eq:hyp} gives
$\delta e=0$ to first order, hence $\delta\phi=\delta D\phi=0$ on the whole interval: these controls, too, are tangent to fixed optics. With
\[
 \widetilde W:=\phi_0^2\sinh^2\!\rho,\qquad a_E:=Y_E\xi_0=1-z^2,\qquad a_B:=Y_B\xi_0=-2,
\]
the flux identity \eqref{eq:flux}, of which only the last term varies, gives
\begin{equation}\label{eq:chargevar}
 \delta C_i=-\tfrac14\int\!\!\int\phi_0^2\cdot2\sinh^2\!\rho\,(Dk)\,a_i
 =-\tfrac12\int_{I_+\times S^2}\widetilde W\,(Dk)\,a_i=\tfrac12\int_{I_+\times S^2}(D\widetilde W)\,k\,a_i ,
\end{equation}
by integration by parts in $v$ ($k$ vanishes at the ends of $I_+$). We take $k_j=(D\widetilde W)a_j$, $j=E,B$. The response matrix is the
\emph{charge Gramian}
\begin{equation}\label{eq:Gram}
 \mathsf G_{ij}=\tfrac12\int_{I_+\times S^2}(D\widetilde W)^2a_ia_j\,dv\,d\mu_{\gac}.
\end{equation}
\begin{lemma}\label{lem:Gram}
The matrix $\mathsf G$ is positive definite.
\end{lemma}
\begin{proof}
$\theta^t\mathsf G\theta=\frac12\int(D\widetilde W)^2\big[\theta_E(1-z^2)-2\theta_B\big]^2$. At every point except the south pole, $\widetilde W$ is positive somewhere in
the pulse and vanishes near its ends, so $\int_{I_+}(D\widetilde W)^2dv>0$. If $\theta^t\mathsf G\theta=0$, the polynomial in brackets vanishes for
$-1<z<1$, hence $\theta=0$.
\end{proof}
Without the polar twist we would have $a_E=Y_E(-2\varphi)=0$, and only the axial pairing could be controlled. The second pulse
$\mathbf B_-$ has a different role: it makes the optical rank criterion hold on the generator where $\mathbf B$ vanishes. The nonlinear
phase controls are realised by $\xi=\xi_0+\theta_Ek_E+\theta_Bk_B$, which only rotates polarisations.

\subsection{Assembly}\label{sec:ctrlassembly}

Let $\mathcal F(r,\zeta,\theta)=(\phi,D\phi,Z)(T_0)$ be the endpoint map with lapse control $a_r$, tensor control $H_\zeta$ and phase
controls $\theta=(\theta_E,\theta_B)$. Restricted to $\mathcal R_\pi$-invariant elements, it is a smooth map
\[
 \mathcal F:\ H^s\times H^s\times(H^{s+1}\cap\CK^\perp)\times\R^2\to H^s\times H^s\times H^{s-1}.
\]
Indeed, the matrix exponential and inversion are smooth on $H^s$ near positive metrics ($H^s$ is a Banach algebra for $s\ge2$),
\eqref{eq:Ray} is a linear ODE in $v$ with values in the algebra $H^s$, and \eqref{eq:Z} costs exactly one derivative. The differential at
$0$ has the triangular form
\[
 D\mathcal F(0)(r,\zeta,\theta)=\big(r,\ \mathcal Sr+\mathcal D_2\mathcal D_2^*\zeta+\theta_EZ_E+\theta_BZ_B\big),\qquad C_i(Z_j)=\mathsf G_{ij},
\]
where $\mathcal S:H^s\times H^s\to H^{s-1}$ is the (bounded) torsion response to lapse controls and $Z_j$ the torsion response to the phase
control $k_j$; the first component is $r$ because the tensor and phase controls have zero optical response. An explicit inverse is
given, for $(r,Z_*)$, by
\[
 \theta=\mathsf G^{-1}\big(C_E(Z_*-\mathcal Sr),\,C_B(Z_*-\mathcal Sr)\big)^t,\qquad
 \zeta=(\mathcal D_2\mathcal D_2^*)^{-1}(Z_*-\mathcal Sr-\theta_EZ_E-\theta_BZ_B).
\]
The argument of $(\mathcal D_2\mathcal D_2^*)^{-1}$ is orthogonal to the two invariant conformal Killing fields by the choice of $\theta$, and to
the other four by symmetry. Injectivity is also clear: $r=0$ forces $\theta=0$, and then $\zeta=0$. The inverse function theorem
gives Theorem~\ref{thm:ctrl}; note that the elliptic gain of two derivatives of $\mathcal D_2\mathcal D_2^*$
exactly matches the loss of one derivative at the torsion level, so no Nash--Moser scheme is needed.

Prescribing $-\etab$ instead of $Z$ is equivalent: the lapse controls are compactly supported, so the endpoint lapse
$f_*=\log\Omega(T_0)$ is fixed, and $\etab=2df_*-\phi^{-2}Z$ is a fixed smooth change of chart.

\subsection{Inserting the controller into the partial bridge}\label{sec:insert}

We insert the controller into the affine control window of the partial bridge. Four things must be checked.
\begin{enumerate}[label=(\roman*)]
\item \emph{The spin matching is not disturbed.} The controller pulses are inserted in pairs with identical radial profiles and opposite
tensor signs, $\exp(\pm\epsilon b(v)T)$, with $T$ independent of $v$. As in Section~\ref{sec:cancel}, each has weighted symmetric memory
$\mp\epsilon\,\varpi(\int b)T$, independent of its centre, so the pair cancels. The energy is axisymmetric, the trace term is
$\tr(\gdc^{-1}\partial_j\gdc\,\gdc^{-1}\partial_v\gdc)=2\rho_v\partial_j\rho$ (an axisymmetric gradient, by \eqref{eq:fixedpol}), and the axial
torque is zero. The lapse recovered by \eqref{eq:lapse} changes only axisymmetrically.
\item \emph{The hypotheses of Theorem~\ref{thm:ctrl} hold:} $D\phi>0$ on the window (affine area with $\varpi>0$), there is a round shear-free
interval $J$ after the pulses, the rank criterion holds (every generator meets a pulse and a zero-energy gap), and the Gramian is
positive.
\item \emph{Passing through the tail.} Theorem~\ref{thm:ctrl} is applied at the cut $T_0$ after the round interval $J$, before the
focusing tail. Beyond $T_0$ nothing depends on the controls except the data at $T_0$. The seed is the fixed reference seed (including the
polar pulses of the tail), and so is the lapse, because the lapse controls are supported in $I$; in particular $\Omega_{\rm tail}$ is fixed.
Hence, at each point of $S^2$:
\begin{itemize}
\item by \eqref{eq:Ray}, $(\phi,D\phi)$ solves a linear ODE with fixed coefficients $\omega$, $e$, so $(\phi,D\phi)(v)=\mathsf V(v)(\phi,D\phi)(T_0)$
with the invertible fundamental matrix $\mathsf V(v)$ of this ODE;
\item by \eqref{eq:Z}, $Z(v)=Z(T_0)+\mathcal Q_v(\phi,D\phi)$, where $\mathcal Q_v$ is a fixed smooth function of $(\phi,D\phi)(T_0)$ and its
first angular derivatives: the right side of \eqref{eq:Z} does not contain $Z$, and its terms are products of $\phi$, $D\phi$, their
first angular derivatives and fixed functions.
\end{itemize}
The same holds on the transition segment of Section~\ref{sec:necksetup}. Composing with the fixed angular normalisation of
Proposition~\ref{prop:reparam} (multiplication by fixed positive functions and addition of a fixed one-form) and with the fixed change of
chart $Z\mapsto\etab=2d\log\Omega-\phi^{-2}Z$ gives a map
\[
 \Pi:\ (\phi,D\phi,Z)(T_0)\longmapsto(\phi,D\phi,-\etab)\ \text{at the start of the neck},
\]
smooth from a neighbourhood of the reference in $H^s\times H^s\times H^{s-1}$ into the same space. Its inverse is obtained in the same way:
apply $\mathsf V^{-1}$ to the first two components, then subtract $\mathcal Q$. No division by $D\phi$ occurs, although $D\phi$ vanishes at
$v_*$. Thus Theorem~\ref{thm:ctrl} composed with $\Pi^{-1}$ is a local smooth controller for terminal triples $(\phi,D\phi,-\etab)$ near
$(\sqrt2,0,\eta_H)$ (unit mass) in $H^s\times H^s\times H^{s-1}$, with all modifications supported away from $S_1$.
\item \emph{There is room for everything.} One concrete layout, after rescaling the affine window to $(0,1)$, uses the ten disjoint closed
intervals $[\frac{2j-1}{24},\frac{2j}{24}]$, $j=1,\dots,10$, followed by the open round interval $J=(\frac{21}{24},\frac{22}{24})$, as in
Table~\ref{tab:layout}. All radial profiles vanish near the ends of their intervals, and there is room at both ends of the window for
the fixed joining germs. On the window the lapse is recovered from \eqref{eq:lapse} with $D\phi=\varpi/(2\sqrt{B+\varpi v})$, which has a positive
lower bound; every quantity in \eqref{eq:lapse} is smooth, also across pulse boundaries, so integrating $D\log\Omega=\omega$ from a
positive initial lapse gives a smooth positive lapse.
\end{enumerate}
\begin{table}[ht]
\centering
\begin{tabular}{@{}ll@{}}
\toprule
intervals & content\\\midrule
1, 2 & rotating north pair ($\sigma=-1$)\\
3, 4 & rotating south pair ($\sigma=+1$)\\
5, 6 & the two cancellers of the tail memories\\
7, 8 & the twisted controller pair (opposite signs; phase controls on 7)\\
9, 10 & the complementary-pole pair ($\mathbf B_-$, opposite signs)\\
$J$ & divergence (tensor) controls\\\bottomrule
\end{tabular}
\caption{Layout of the affine control window.}\label{tab:layout}
\end{table}

From now on we rescale to unit final mass and write $\Phi:=\sqrt2$. All constants below may depend on this fixed early bridge and
its controller.

\section{The neck and the late collar}\label{ch:neck}

After the partial bridge and its controller we have exact control over the intrinsic data at the start of the neck, but
no control over the incoming fields $\Omega^{-1}\chib$, $\Omega^{-2}\ab$. In this section we introduce the last two segments of the null
hypersurface: a long \emph{neck}, which forgets the incoming fields by the mechanism of Section~\ref{sec:toy}, and a short
\emph{collar}, which carries the seed that matches them (Section~\ref{ch:linear}).

\subsection{Set-up}\label{sec:necksetup}

We work at unit final mass and shift $v$ so that the remaining part of $\Hcal$ consists of
\[
 \text{the neck }[-T,0]\qquad\text{and the collar }[0,1].
\]
On both we choose the same angularly constant non-affinity $\omega\equiv\omega_0>0$, and we write simply $\omega$ for this constant.

\paragraph{The transition $\omega:0\to\omega_0$.} The terminal sphere of the partial bridge has $\omega=0$ (the tail is affine). A fixed
smooth transition from $\omega=0$ to $\omega_0$, through angularly constant values, is inserted after it and absorbed into the fixed early
flow; it prevents a jump in the lapse jets. Its background has $D\phi=0$ and constant seed, so Raychaudhuri's equation
$D^2\phi=2\omega D\phi$ keeps $D\phi=0$ and $\phi$ constant, the torsion equation $D\etab=2\Nd\omega=0$ keeps $\etab$, and the intrinsic
horizon data are unchanged.

\paragraph{The seeds.} On the neck the seed is the horizon seed, $\gdc=\gdcH:=\frac12\gd_H$. On the collar the seed is
\begin{equation}\label{eq:collarseed}
 \gdc(k)=\gdcH\exp\big(\gd_H^{-1}k\big),
\end{equation}
where $k$ is a $\gd_H$-tracefree symmetric tensor field supported in $(0,1)\times S^2$.
It is positive, has $\det\gdc=\det\gdcH$ (the exponent is tracefree), and at $\phi=\Phi$ the physical metric is
$\gd=\Phi^2\gdcH\exp(\gd_H^{-1}k)=\gd_H\exp(\gd_H^{-1}k)$, whose first variation is exactly $k$.

\paragraph{Terminal data.} We impose exactly
\begin{equation}\label{eq:terminal}
 \phi(1)=\Phi=\sqrt2,\qquad D\phi(1)=0,\qquad \etab(1)=\etab_H ,
\end{equation}
and integrate \eqref{eq:Ray} and \eqref{eq:H3} backwards from $v=1$. In this way the terminal sphere has exactly the horizon's metric,
expansion and torsion, whatever the seed does in between. The price is that the values at $v=-T$ become targets, which the early
controller of Section~\ref{ch:controller} must hit. Since $\gdc(v)$ is constant near $v=1$ and on the neck, the seed itself is continuous.
The incoming fields, on the other hand, are integrated forwards from $S_1$ through the neck and backwards from the target over the
collar; matching them at $v=0$ is the endpoint equation of Section~\ref{ch:linear}.

\subsection{The role of the collar seed}\label{sec:seedoverview}

The seed $k$ is a small burst of gravitational radiation which changes the shape, but not the area element, of the cross-sections; its
conformal seed is the horizon seed near both ends of the collar. It does not spoil the intrinsic levels, since these are integrated
backwards from \eqref{eq:terminal}; it only moves the targets at $v=-T$ (for instance, by \eqref{eq:Vphi}, $\phi(0)<\Phi$ wherever the shear
energy is nonzero). Its main role is to steer the incoming fields: to first order only the two moments $H_1=\int_0^1e^{2\omega v}k\,dv$ and
$H_2=\int_0^1e^{4\omega v}k\,dv$ enter (Theorem~\ref{thm:moments}), and together with the cut and the boost they produce every incoming datum
(Theorem~\ref{thm:iso}, whose core is the inversion of $\Pcal=\Bcal E$ in Theorem~\ref{thm:Pinv}). Finally, in Section~\ref{ch:fixedpoint} we take
$k=\varepsilon k_*+\mathcal R(H_1,H_2)$ with a moment-free direction $k_*$; it is invisible to first order, but its shear energy focuses the
generators at order $\varepsilon^2$ and, by the trace identity (Proposition~\ref{prop:trace}), gives
$\lambda=c_*\varepsilon^2+O(\varepsilon^3+e^{-2\omega T})$ for the boost, with $c_*>0$ for the choice of $k_*$ made in
Section~\ref{ch:sign}.

\subsection{Backward optics on the collar}\label{sec:Volterra}

\begin{lemma}\label{lem:Volterra}
With $e$ the shear energy of \eqref{eq:collarseed} (extended by zero to the neck),
\begin{align}
 D\phi(v)&=\int_v^1e^{-2\omega(v'-v)}e(v')\phi(v')\,dv',\label{eq:Vp}\\
 \phi(v)&=\Phi-\frac1{2\omega}\int_v^1\big(1-e^{-2\omega(v'-v)}\big)e(v')\phi(v')\,dv'.\label{eq:Vphi}
\end{align}
If $\sup_x\int_0^1e(v,x)\,dv<\omega$, then $\frac12\Phi\le\phi\le\Phi$ and $D\phi\ge0$ on $[-T,1]$, uniformly in $T$.
\end{lemma}
\begin{proof}
With $\omega$ constant, \eqref{eq:Ray} reads $D(e^{-2\omega v}D\phi)=e^{-2\omega v}(D^2\phi-2\omega D\phi)=-e^{-2\omega v}e\phi$. Integrating from $v$
to $1$ with $D\phi(1)=0$ gives $-e^{-2\omega v}D\phi(v)=-\int_v^1e^{-2\omega v'}e\phi\,dv'$, which is \eqref{eq:Vp}. Then
$\phi(v)=\Phi-\int_v^1D\phi(v'')\,dv''$, and by Fubini
\begin{align*}
 \int_v^1\!\!\int_{v''}^1e^{-2\omega(v'-v'')}e\phi(v')\,dv'\,dv''&=\int_v^1e\phi(v')\Big(\int_v^{v'}e^{-2\omega(v'-v'')}dv''\Big)dv'\\
 &=\int_v^1e\phi(v')\,\frac{1-e^{-2\omega(v'-v)}}{2\omega}\,dv',
\end{align*}
which is \eqref{eq:Vphi}. For the bounds we argue as in Lemma~\ref{lem:tail}: before a first zero of $\phi$ (going backwards), $\phi>0$
and $e\ge0$ give $D\phi\ge0$, hence $\phi\le\Phi$; then \eqref{eq:Vphi} gives $\phi\ge\Phi\big(1-\frac1{2\omega}\int_0^1e\big)\ge\frac12\Phi$, which
excludes the zero. On the neck $e=0$, and the explicit formulas of Section~\ref{sec:neckformulas} apply.
\end{proof}
The hypothesis holds on a fixed small ball of seeds $k$ (in $C^1_vH^s_x$, say), which we fix once and for all and call the
\emph{seed ball}.

\subsection{Explicit neck formulas}\label{sec:neckformulas}

On the neck $e=0$ and $\Omega\chih=\frac12\phi^2D\gdcH=0$. Let $(\phi_0,D\phi_0,\etab_0)$ be the values at $v=0$ and put
\[
 a_1=\frac{D\phi_0}{2\omega},\qquad a_0=\phi_0-a_1 .
\]
Then $D^2\phi=2\omega D\phi$ gives
\begin{equation}\label{eq:neckopt}
 \phi(v)=a_0+a_1e^{2\omega v},\qquad D\phi(v)=2\omega a_1e^{2\omega v}.
\end{equation}
For the torsion we use \eqref{eq:Zb} with $D\gdc=0$ and $\Nd\omega=0$: $\underline Z:=\phi^2\etab$ satisfies $D\underline Z=\phi\,\Nd D\phi-D\phi\,\Nd\phi$.
By \eqref{eq:neckopt},
\[
 \phi\Nd D\phi-D\phi\Nd\phi=(a_0+a_1e^{2\omega v})2\omega e^{2\omega v}da_1-2\omega a_1e^{2\omega v}(da_0+e^{2\omega v}da_1)
 =2\omega e^{2\omega v}(a_0da_1-a_1da_0),
\]
and integrating from $0$,
\begin{equation}\label{eq:neckZ}
 \underline Z(v)=\underline Z_0+\big(e^{2\omega v}-1\big)\big(a_0\,da_1-a_1\,da_0\big).
\end{equation}
As $v\to-T\to-\infty$ these converge exponentially: $\phi\to a_0$, $D\phi\to0$, $\underline Z\to\underline Z_0-(a_0da_1-a_1da_0)$. Consequently:
\begin{itemize}
\item the maps (seed $k$) $\mapsto$ (data at $v=-T$) have derivatives up to order three bounded uniformly in $T\ge0$ on the seed
ball, because with $x=e^{2\omega v}\in[e^{-2\omega T},1]$ the expressions \eqref{eq:neckopt} and \eqref{eq:neckZ} are polynomials in $x\in[0,1]$
with coefficients depending smoothly on $(\phi_0,D\phi_0,\etab_0)$;
\item the values at $v=-T$ lie in a fixed small neighbourhood of $(\Phi,0,\etab_H)$, on which the early controller works;
\item $\int_{-T}^0\|\Omega\tr\chi\|_{H^{s-1}}\,dv\le C$, since $\Omega\tr\chi=2D\phi/\phi$ and $\int_{-T}^0e^{2\omega v}dv\le\frac1{2\omega}$, and $\phi$
stays bounded away from zero.
\end{itemize}

\subsection{The incoming fields on the neck}\label{sec:neckincoming}

We use the early controller to attain the data at $v=-T$ required by the backward integration, keeping $S_1$ exactly, and propagate
$\Omega^{-1}\chib$ and $\Omega^{-2}\ab$ forwards through the neck. Let
\[
 \Kcal_T(k):=\big(\Omega^{-1}\chib,\ \Omega^{-2}\ab\big)\big|_{v=0}\in H^{s-2}(\mathrm{Sym}^2)\times H^{s-4}(\mathrm{Sym}^2_0)
\]
be the incoming data that arrive at the start of the collar from the left.

\paragraph{The $\chib$ propagator.} On the neck $\gd=\phi^2\gdcH$, so the trace and tracefree projections $\Pi_{\rm tr},\Pi_{\rm tf}$ with
respect to $\gdcH$ (equivalently $\gd$) do not depend on $v$. With $\Omega\chih=0$, \eqref{eq:H4} becomes
\[
 D(\Omega^{-1}\chib)=-2\omega\,\Omega^{-1}\chib+\tfrac12\Omega\tr\chi\,\Pi_{\rm tf}(\Omega^{-1}\chib)+\Ecal(\gd,\etab).
\]
Applying $\Pi_{\rm tr}$ and $\Pi_{\rm tf}$, which commute with $D$ here, we see that the trace part decays like $e^{-2\omega v}$ and the
tracefree part like $e^{-2\omega v}\exp\int\frac12\Omega\tr\chi=e^{-2\omega v}\phi$, since $\frac12\Omega\tr\chi=D\log\phi$. So the propagator from
$v'$ to $v\ge v'$ is
\begin{equation}\label{eq:UN}
 U_N(v,v')=e^{-2\omega(v-v')}\Big(\Pi_{\rm tr}+\frac{\phi(v)}{\phi(v')}\Pi_{\rm tf}\Big).
\end{equation}

\paragraph{The $\ab$ propagator.} With $\Omega\chih=0$ we have $D=\widehat D$ on tracefree tensors, and by the Gauss equation
$\rho=-K-\frac14\Omega\tr\chi\,\Omega^{-1}\tr\chib$, by the curl equation $\sigma=-\Curl\etab$. So \eqref{eq:H5} becomes
\[
 D(\Omega^{-2}\ab)=\big(\tfrac12\Omega\tr\chi-4\omega\big)\Omega^{-2}\ab+\Bcal_{\gd,\etab}(\Omega^{-1}\chib)
 +\tfrac34\Omega\tr\chi\,\Omega^{-1}\tr\chib\,\Omega^{-1}\chibh,
\]
where $\Bcal_{\gd,\etab}$ is the operator of Definition~\ref{def:B} with $(\gd_H,\etab_H)$ replaced by the current $(\gd,\etab)$; the extra term
is $-3(-\frac14\Omega\tr\chi\,\Omega^{-1}\tr\chib)\Omega^{-1}\chibh$. The homogeneous propagator is $U_A(v,v')=e^{-4\omega(v-v')}\phi(v)/\phi(v')$.

\begin{lemma}[The neck forgets]\label{lem:neck}
On the seed ball:
\begin{enumerate}[label=(\alph*)]
\item $\|\Kcal_T(0)\|\le Ce^{-2\omega T}$;
\item $\Kcal_T$ is $C^3$ with bounds independent of $T$;
\item $D\Kcal_T(0)=0$ for every $T$.
\end{enumerate}
The norm is $H^{s-2}$ for the first and $H^{s-4}$ for the second component.
\end{lemma}

\begin{proof}
(a) At $k=0$ the backward solution is the horizon itself: $\phi\equiv\Phi$, $D\phi=0$, $\etab\equiv\etab_H$ on $[-T,1]$. On the neck $\Ecal=0$ by
\eqref{eq:NH}, $\Omega\tr\chi=0$, and $\Bcal_{\gd,\etab}=\Bcal$. So $\Omega^{-1}\chib(v)=e^{-2\omega(v+T)}\Omega^{-1}\chib(-T)$, and by variation of
constants
\begin{align*}
 \Omega^{-2}\ab(0)&=e^{-4\omega T}\Omega^{-2}\ab(-T)+\int_{-T}^0e^{4\omega v}e^{-2\omega(v+T)}dv\;\Bcal(\Omega^{-1}\chib(-T))\\
 &=e^{-4\omega T}\Omega^{-2}\ab(-T)+\frac{e^{-2\omega T}-e^{-4\omega T}}{2\omega}\,\Bcal(\Omega^{-1}\chib(-T)).
\end{align*}
The data at $-T$ are the fixed outputs of the early bridge at its fixed endpoint, so they are bounded.

(b) There are three ingredients, none of which produces a factor growing with $T$.

\emph{The early output is $C^3$ on one fixed ball.} Let $U$ be a small fixed ball of targets $(\phi,D\phi,\etab)$ around $(\Phi,0,\etab_H)$ in
$H^s\times H^s\times H^{s-1}$, and let $I:U\to H^{s-2}\times H^{s-4}$ map a target to the incoming data at the end of the early bridge. By
Section~\ref{sec:insert}(iii),
\[
 I=\mathcal T\circ\mathcal C\circ\Pi^{-1}.
\]
Here $\Pi^{-1}$ pulls the target back through the fixed transition, normalisation and tail segments to the cut $T_0$ of the control window;
it is the explicit smooth map of Section~\ref{sec:insert}(iii), and it involves no division by $D\phi$. Next, $\mathcal C$ is the controller of
Theorem~\ref{thm:ctrl}, given by the inverse function theorem; it is smooth, with values in the controls, which have fixed smooth radial
profiles and $H^s$ angular coefficients. Finally, $\mathcal T$ solves the triangular linear transport equations \eqref{eq:H4}, \eqref{eq:H5}
for the incoming fields along the early bridge; their coefficients are products of the controlled fields and at most two of their angular
derivatives, so $\mathcal T$ is smooth in the controls. Hence $I$ is smooth, and continuity of $D^jI$ at the base point gives
$\sup_U\sum_{j\le3}\|D^jI\|<\infty$ after shrinking $U$ once. All these maps are fixed before $T$ is chosen.

\emph{The neck data are uniformly controlled.} With $x=e^{2\omega v}\in[e^{-2\omega T},1]$, \eqref{eq:neckopt} and \eqref{eq:neckZ} are
polynomials in $x\in[0,1]$ with coefficients depending smoothly on $(\phi_0,D\phi_0,\etab_0)$. So they and all their parameter derivatives
up to order three are bounded independently of $T$, and the values at $v=-T$ lie in $U$ once the seed ball is small (shrunk once,
independently of $T$).

\emph{Variation of constants without Gronwall.} The propagators \eqref{eq:UN} and $U_A$ are bounded by $Ce^{-2\omega(v-v')}$ and
$Ce^{-4\omega(v-v')}$, because $\frac12\le\phi/\Phi\le1$; their kernels have $L^1$ norms at most $\frac C{2\omega}$, $\frac C{4\omega}$, uniformly in $T$.
The sources are $\Ecal(\gd,\etab)\in H^{s-2}$ and $\Bcal_{\gd,\etab}(\Omega^{-1}\chib)+\frac34\Omega\tr\chi\,\Omega^{-1}\tr\chib\,\Omega^{-1}\chibh\in H^{s-4}$,
bounded uniformly in $v$ and $T$ (for instance $\Nd^2(\Omega^{-1}\chib)\in H^{s-4}$, and multiplication by the coefficients is bounded for
$s\ge10$). We first bound $\Omega^{-1}\chib$ by
\[
 \|\Omega^{-1}\chib(0)\|\le Ce^{-2\omega T}\|\Omega^{-1}\chib(-T)\|+C\int_{-T}^0e^{2\omega v}dv\,\sup\|\Ecal\|\le C',
\]
and then $\Omega^{-2}\ab$ in the same way. For a parameter derivative of order $j\le3$ we differentiate the integral formulas: every term
is a product of derivatives of propagators, coefficients and previously bounded quantities, of total order $\le j$, and the term
containing $D^j(\Omega^{-1}\chib)$ is bounded before $D^j(\Omega^{-2}\ab)$ is treated. Induction gives uniform bounds for all derivatives
up to order three, including mixed ones. Since the equations are linear in the incoming fields with triangular structure, no Gronwall
factor $e^{CT}$ appears.

(c) Let a dot denote a first variation in $k$ at $k=0$. Since $e$ is quadratic in $D\gdc$, which vanishes at $k=0$, $\dot e=0$, hence
$\dot\phi=\dot{D\phi}=0$ everywhere by \eqref{eq:Vp}, \eqref{eq:Vphi}. For the torsion we linearise \eqref{eq:H3} at the horizon background
($\Omega\tr\chi=0$, $\Nd\omega=0$, $\dot{(\Omega\tr\chi)}=0$): $D\dot\etab=-\Div\dot{(\Omega\chih)}$, and $\dot{(\Omega\chih)}=\frac12D\dot\gd=\frac12Dk$.
Since the divergence is taken with respect to the fixed metric $\gd_H$, it commutes with $D$, and with $\dot\etab(1)=0$, $k(1)=0$ we obtain
\begin{equation}\label{eq:etadot}
 \dot\etab(v)=-\tfrac12\Div k(v),\qquad\text{in particular }\dot\etab(0)=0,
\end{equation}
because $k(0)=0$. So all data at $v=0$, hence at $v=-T$, have zero first variation. The output of the controller, and with it $\Kcal_T$,
then also has zero first variation.
\end{proof}

Thus whatever the early bridge does to the incoming fields is damped by $e^{-2\omega T}$, and to first order the late seed does not
interact with the early bridge at all. In particular, no smallness of the incoming fields at the end of the early bridge is needed: they
may be large, for instance growing with the final mass, but they do not feed back into the optical and torsion levels
(Remark~\ref{rem:triangular}), and their size enters only the constant $C$ in Lemma~\ref{lem:neck}, which is fixed before $T$ is chosen.

\section{The linearised endpoint problem and its inverse}\label{ch:linear}

In this section we write the matching of the incoming data as a single equation, linearise it at a ``limiting corner'', and show that
the linearisation is an isomorphism at all angular frequencies. The key is an exact computation: in helicity components the relevant
fourth order operator $\Pcal$ becomes a $2\times2$ block whose diagonal is a purely imaginary constant and whose off-diagonal entries are
the angular Teukolsky--Starobinsky operators.

\subsection{The endpoint equation}\label{sec:endpointeq}

Let $\Jcal(k,\lambda,G)$ denote the incoming data at $v=0$ obtained by integrating \eqref{eq:H4}--\eqref{eq:H5} backwards over the
collar, with seed $k$, from the target values $(\mathcal N(\lambda,G),\mathcal A(\lambda,G))$ of Definition~\ref{def:target} at $v=1$.
The optical and torsion levels of the collar are already fixed by \eqref{eq:terminal}. The incoming fields arriving from the left are
$\Kcal_T(k)$ (Section~\ref{sec:neckincoming}). The matching problem is therefore the single equation
\begin{equation}\label{eq:endpoint}
 \Jcal(k,\lambda,G)-\Kcal_T(k)=0\qquad\text{in }H^{s-2}(\mathrm{Sym}^2)\times H^{s-4}(\mathrm{Sym}^2_0).
\end{equation}
Once \eqref{eq:endpoint} holds, all of $\gd,\Omega\chi,\etab,\Omega^{-1}\chib,\Omega^{-2}\ab$ are continuous at $v=0$. The concatenated
data then solve the hierarchy on the whole of $[v_1,1]$ (the equations are first order in $v$, so continuity at $v=0$ suffices, and
smoothness in $v$ follows from the equations), and at $v=1$ they agree with those of the cut $F=G/\lambda$ in the frame
$(\xi/\lambda,\lambda n_F)$, provided $\lambda>0$.

\begin{remark}[Bundle identification]
The incoming curvature tensors in \eqref{eq:endpoint} are tracefree for the metric $\gd(0;k)$, which depends on $k$. We identify this
bundle with the fixed $\gd_H$-tracefree bundle by the positive square-root conjugation $\theta\mapsto h^{1/2}\theta h^{1/2}$ with
$h=\gd_H^{-1}\gd(0;k)$, which is smooth near $k=0$. At the background this has no effect on the linearisation. Indeed, $k$ vanishes near
$v=0$, so $\gd(0;k)=\phi(0;k)^2\gdcH$ and $h=(\phi(0;k)^2/\Phi^2)I$. Thus the tracefree bundles coincide and the map is multiplication
by a positive scalar, whose first variation vanishes because $\phi(0;k)-\Phi=O(k^2)$.
\end{remark}

\subsection{Linearisation at the limiting corner}\label{sec:moments}

We linearise $\Jcal$ at the limiting corner $(k,\lambda,G)=(0,0,0)$. Here the collar background is the intrinsic horizon with
$\Omega^{-1}\chib\equiv0$ and $\Omega^{-2}\ab\equiv0$, because $\mathcal N(0,0)=\mathcal A(0,0)=0$. This background is not a horizon cut
of Kerr: it is the formal limit $\lambda\to0$, and working there is the purpose of the target chart.

\begin{definition}\label{def:EP}
For a $\gd_H$-tracefree symmetric $k$ set
\begin{equation}\label{eq:EP}
 Ek:=\frac{d}{d\epsilon}\Big|_{0}\Ecal\big(\gd_H+\epsilon k,\ \etab_H-\tfrac{\epsilon}2\Div k\big),\qquad \Pcal:=\Bcal\circ E .
\end{equation}
\end{definition}
The torsion variation $-\frac12\Div k$ is exactly \eqref{eq:etadot}.

\begin{lemma}[Covariant formula for $E$]\label{lem:Ecov}
With $\gamma:=\frac12\Div k$ and $\eta_H=-\etab_H$,
\[
 (Ek)_{AB}=-\Nd_A\gamma_B-\Nd_B\gamma_A+2\dot\Gamma^C_{AB}\eta_{H,C}+2(\eta_{H,A}\gamma_B+\gamma_A\eta_{H,B})-(\Div\gamma)\gd_{H,AB}-K_Hk_{AB},
\]
where $\dot\Gamma^C_{AB}=\frac12(\Nd_Ak_B{}^C+\Nd_Bk_A{}^C-\Nd^Ck_{AB})$.
\end{lemma}
\begin{proof}
Vary each term of $\Ecal=\Nd_A\etab_B+\Nd_B\etab_A+2\etab_A\etab_B-K\gd_{AB}$ with $\dot\gd=k$, $\dot\etab=-\gamma$:
$\dot{(\Nd_A\etab_B)}=\Nd_A\dot\etab_B-\dot\Gamma^C_{AB}\etab_C=-\Nd_A\gamma_B+\dot\Gamma^C_{AB}\eta_{H,C}$;
$\dot{(2\etab_A\etab_B)}=2(\gamma_A\eta_{H,B}+\eta_{H,A}\gamma_B)$ (two sign changes); and $\dot{(K\gd)}=\dot K\gd_H+K_Hk$ with
$\dot K=\frac12\Div\Div k=\Div\gamma$ (Lemma~\ref{lem:dotK}).
\end{proof}

\begin{theorem}[Two moments]\label{thm:moments}
For $k$ supported in $(0,1)$ and $(\lambda,G)=(0,0)$, set
\[
 H_j:=\int_0^1e^{2j\omega v}k(v)\,dv,\qquad j=1,2 .
\]
Then the first variation of $\Jcal$ in $k$ is
\[
 \delta(\Omega^{-1}\chib)(0)=-E H_1,\qquad \delta(\Omega^{-2}\ab)(0)=\frac1{2\omega}\,\Pcal(H_2-H_1).
\]
The first variation in $(\lambda,G)$ is $\big(e^{2\omega}(\dot\lambda\chib_H+\Ccal\dot G),\,0\big)$.
\end{theorem}
\begin{proof}
At the background $\Omega\chi=0$, $\Omega^{-1}\chib=0$ and $\Omega^{-2}\ab=0$. In the linearisations of \eqref{eq:H4}, \eqref{eq:H5}
every term that is a product of a vanishing background quantity with a first variation drops out, and what remains is
\[
 D\dot N=-2\omega\dot N+Ek,\qquad D\dot A=-4\omega\dot A+\Bcal\dot N,\qquad N=\Omega^{-1}\chib,\ A=\Omega^{-2}\ab .
\]
(The first variation of $\Ecal(\gd,\etab)$ along the collar is $E k(v)$, by \eqref{eq:etadot} and $\dot\phi=0$; in \eqref{eq:H5} at the
background the coefficients are those of $\Bcal$, Section~\ref{sec:Bop}.) With $\dot N(1)=0$ we get
$\dot N(v)=-e^{-2\omega v}\int_v^1e^{2\omega v'}Ek(v')\,dv'$, so $\dot N(0)=-EH_1$, since $E$ is angular and commutes with the
$v$-integral. Then, with $\dot A(1)=0$, inserting $\dot N$ and exchanging the order of integration over $0\le v\le v'\le1$,
\begin{align*}
 \dot A(0)&=-\int_0^1e^{4\omega v}\Bcal\dot N(v)\,dv
 =\int_0^1\Pcal k(v')e^{2\omega v'}\Big(\int_0^{v'}e^{2\omega v}dv\Big)dv'\\
 &=\frac1{2\omega}\int_0^1(e^{4\omega v'}-e^{2\omega v'})\Pcal k(v')\,dv'.
\end{align*}
For $(\lambda,G)$, the target variation $(\dot\lambda\chib_H+\Ccal\dot G,0)$ is transported back by $D\dot N=-2\omega\dot N$, which gives
the factor $e^{2\omega}$, and it produces no $\ab$, because $\Bcal\chib_H=\Bcal\Ccal\dot G=0$ (Lemma~\ref{lem:Bstat}).
\end{proof}

In particular, if $H_1=H_2=0$, the seed is invisible to first order at every angular frequency; this gives the free direction $k_*$ used
in Section~\ref{ch:sign}. Conversely, the two moments, the cut and the boost together must be able to produce every incoming datum. This
is proved in the remainder of this section.

\subsection{Ellipticity and index}\label{sec:symbols}

\paragraph{Principal symbols.} We compute principal symbols in the usual way: freeze the coefficients at a point, use a
$\gd_H$-orthonormal frame there, and replace each derivative $\partial_A$ by $i\zeta_A$ for a unit covector $\zeta$. All operators below
are built covariantly from $\gd_H$, so it suffices to take $\zeta=(1,0)$. A one-form is represented by its components $b=(b_1,b_2)$, a
symmetric tensor by its matrix $N$, and a tracefree symmetric $k$ by $(p,q)$, where
\[
 k=\begin{pmatrix}p&q\\q&-p\end{pmatrix}.
\]
For the first order operators that occur we then have, with $(\Nd\ohat b)_{AB}=\Nd_Ab_B+\Nd_Bb_A-(\Div b)\gd_{H,AB}$,
\begin{equation}\label{eq:symbols}
 \sigma(d)f=(if,0),\qquad \sigma(\Div)b=ib_1,\qquad \sigma(\Div)N=i(N_{11},N_{12}),\qquad
 \sigma(\Nd\ohat)b=i\begin{pmatrix}b_1&b_2\\b_2&-b_1\end{pmatrix},
\end{equation}
and $\sigma(\Nd^2)f=-\zeta\otimes\zeta\,f=\begin{pmatrix}-f&0\\0&0\end{pmatrix}$. For instance,
$\sigma(\Nd\ohat)b=i(\zeta\otimes b+b\otimes\zeta-(\zeta\cdot b)I)=i\begin{pmatrix}2b_1-b_1&b_2\\b_2&-b_1\end{pmatrix}$.

\emph{The operator $E$.} By Lemma~\ref{lem:Ecov}, the second order part of $Ek$ is $-(\Nd_A\gamma_B+\Nd_B\gamma_A)-(\Div\gamma)\gd_H$
with $\gamma=\frac12\Div k$; the terms $2\dot\Gamma^C_{AB}\eta_{H,C}$, $2(\eta_H\otimes\gamma+\gamma\otimes\eta_H)$ and $-K_Hk$ are of lower order.
By \eqref{eq:symbols}, $\gamma$ has symbol $\frac i2(p,q)$. The symbol of $\Nd_A\gamma_B+\Nd_B\gamma_A$ is
$i(\zeta\otimes\gamma+\gamma\otimes\zeta)=i\cdot\frac i2\begin{pmatrix}2p&q\\q&0\end{pmatrix}$, and that of $\Div\gamma$ is $i\cdot\frac i2p$. Hence
\[
 \sigma_2(E)k=\frac12\begin{pmatrix}2p&q\\q&0\end{pmatrix}+\frac12p\,I=\begin{pmatrix}\frac32p&\frac12q\\[2pt]\frac12q&\frac12p\end{pmatrix}.
\]

\emph{The operator $\Ccal$.} Since $\Ccal F=2\Nd^2F$ plus first order terms, $\sigma_2(\Ccal)F=\begin{pmatrix}-2F&0\\0&0\end{pmatrix}$.

\emph{The operator $\Bcal$.} By Definition~\ref{def:B}, the second order part of $\Bcal N$ is $-\Nd\ohat(\Div W-\frac12dB)$, where $B$ is the
trace and $W$ the tracefree part of $N$; the remaining terms of $\beta_N$ and of $\Bcal N$ are of lower order. By \eqref{eq:symbols},
$\Div W-\frac12dB$ has symbol $i\big(W_{11}-\frac12B,\,W_{12}\big)$, and therefore
\[
 \sigma_2(\Bcal)N=-i\cdot i\begin{pmatrix}W_{11}-\frac12B&W_{12}\\[2pt]W_{12}&-W_{11}+\frac12B\end{pmatrix}
 =\begin{pmatrix}W_{11}-\frac12B&W_{12}\\[2pt]W_{12}&\frac12B-W_{11}\end{pmatrix}.
\]

\emph{The operator $\Pcal=\Bcal E$.} For $N=\sigma_2(E)k$ we have $B=2p$ and $W=N-\frac B2I=\frac12\begin{pmatrix}p&q\\q&-p\end{pmatrix}$, so
$W_{11}-\frac12B=-\frac p2$ and
\[
 \sigma_4(\Pcal)k=\sigma_2(\Bcal)\sigma_2(E)k=\frac12\begin{pmatrix}-p&q\\q&p\end{pmatrix}.
\]

\emph{Ellipticity.} The operator $(k,F)\mapsto Ek+\Ccal F$ maps sections of a bundle of rank three (two tracefree components and one
scalar) to sections of a bundle of rank three (symmetric tensors). Its symbol is
\[
 (p,q,F)\mapsto\big(\tfrac32p-2F,\ \tfrac12q,\ \tfrac12p\big)\quad(\text{entries }11,12,22),\qquad
 \det\begin{pmatrix}\frac32&0&-2\\0&\frac12&0\\\frac12&0&0\end{pmatrix}=\frac12\ne0 .
\]
The operator $\Pcal$ maps tracefree tensors to tracefree tensors, with symbol $(p,q)\mapsto(-\frac p2,\frac q2)$, which is invertible. Both
operators are therefore elliptic between bundles of equal rank, hence Fredholm $H^r\to H^{r-2}$, resp.\ $H^r\to H^{r-4}$. To compute their
indices we compare with the round sphere.

\paragraph{Round sphere identities.} On the unit round sphere $\Nd\ohat$ is defined as above, with $\gac$ in place of $\gd_H$, and we
let $\Delta f=\nu_\ell f$, $\nu_\ell=-\ell(\ell+1)$.
\begin{lemma}\label{lem:E1}
For every one-form $\gamma$, $\Div(\Nd\ohat\gamma)=\Delta\gamma+\gamma$ (rough Laplacian). Consequently
\begin{equation}\label{eq:E1}
 \Div(\Nd\ohat df)=(\nu_\ell+2)\,df,\qquad \Div(\Nd\ohat{}^*df)=(\nu_\ell+2)\,{}^*df .
\end{equation}
\end{lemma}
\begin{proof}
$\Nd^A(\Nd_A\gamma_B+\Nd_B\gamma_A-\Div\gamma\,\gac_{AB})=\Delta\gamma_B+\Nd^A\Nd_B\gamma_A-\Nd_B\Div\gamma$, and the Ricci identity gives
$\Nd^A\Nd_B\gamma_A=\Nd_B\Nd^A\gamma_A+\mathrm{Ric}_B{}^A\gamma_A=\Nd_B\Div\gamma+\gamma_B$ ($\mathrm{Ric}=\gac$). For $\gamma=df$: $\Delta df=d\Delta f+df$ by
the same Ricci identity, so $\Div(\Nd\ohat df)=(\nu_\ell+1)df+df$. For $\gamma={}^*df$: $\Nd\slashed\epsilon=0$, so
$\Delta{}^*df={}^*\Delta df=(\nu_\ell+1){}^*df$.
\end{proof}
For $\ell\ge2$ the tensors $\Nd\ohat df$ and $\Nd\ohat{}^*df$ are nonzero (their divergence is); these are the electric and magnetic
tensor harmonics. They span the tracefree symmetric tensors: a tracefree $k$ orthogonal to all of them has $\Div k$ orthogonal to all
$df$, ${}^*df$, hence $\Div k=0$ (Hodge decomposition of one-forms on $S^2$), and then the integrated Weitzenb\"ock identity of the
following lemma forces $k=0$.

\begin{lemma}[Weitzenb\"ock]\label{lem:weitz}
For a tracefree symmetric $k$ on a surface of constant curvature $K$, $\int|\Nd k|^2=2\int|\Div k|^2-2K\int|k|^2$. On the unit sphere,
$2\int|\Div k|^2=\int|\Nd k|^2+2\int|k|^2$.
\end{lemma}
\begin{proof}
Set $T_{ABC}=\Nd_Ak_{BC}-\Nd_Bk_{AC}$. In an orthonormal frame with $k=\begin{pmatrix}a&b\\b&-a\end{pmatrix}$, the only components are
$T_{121}=-T_{211}=\Nd_1b-\Nd_2a$ and $T_{122}=-T_{212}=-\Nd_1a-\Nd_2b$, while $\Div k=(\Nd_1a+\Nd_2b,\ \Nd_1b-\Nd_2a)$. So
$|T|^2=2|\Div k|^2$ pointwise. Expanding $|T|^2=2|\Nd k|^2-2\Nd_Ak_{BC}\Nd^Bk^{AC}$ gives $|\Nd k|^2=|\Div k|^2+\Nd_Ak_{BC}\Nd^Bk^{AC}$.
Integrating by parts and using the Ricci identity $\Nd_A\Nd_Bk^{AC}-\Nd_B\Nd_Ak^{AC}=2Kk_B{}^C$ (for tracefree symmetric $k$ on a
surface the two curvature terms each give $Kk_B{}^C$, and the trace term vanishes), we obtain
$\int\Nd_Ak_{BC}\Nd^Bk^{AC}=-\int k_{BC}\Nd_A\Nd^Bk^{AC}=\int|\Div k|^2-2K\int|k|^2$.
\end{proof}

\paragraph{Round comparison operators.} The round operator with the same principal part as $(E,\Ccal)$ is
\[
 E_0(k,F)=-\sym\Nd\Div k-\tfrac12(\Div\Div k)\gac+2\Nd^2F
\]
($\sym$ including the factor $\frac12$); its symbol is $\sigma_2(E)k+\sigma_2(\Ccal)F$ by the computation above. We compute $E_0$ on
tensor harmonics, writing $\nu=\nu_\ell$.

\emph{Electric modes} $k=\Nd\ohat df$, $F=\kappa f$. By \eqref{eq:E1}, $\Div k=(\nu+2)df$, $\Div\Div k=(\nu+2)\nu f$, and
$\Nd^2f=\frac12k+\frac12\nu f\gac$. So
\begin{align*}
 E_0(k,\kappa f)&=-(\nu+2)\big(\tfrac12k+\tfrac12\nu f\gac\big)-\tfrac12(\nu+2)\nu f\gac+2\kappa\big(\tfrac12k+\tfrac12\nu f\gac\big)\\
 &=\big(\kappa-\tfrac{\nu+2}2\big)k+\big(\kappa\nu-\nu(\nu+2)\big)f\gac .
\end{align*}
Recording the output by (trace, tracefree coefficient) $=(2(\kappa\nu-\nu(\nu+2)),\ \kappa-\frac{\nu+2}2)$, the matrix acting on
$(1,\kappa)$ is
\[
 \begin{pmatrix}-2\nu(\nu+2)&2\nu\\-\frac12(\nu+2)&1\end{pmatrix},\qquad
 \det=-2\nu(\nu+2)+\nu(\nu+2)=-\nu(\nu+2)\neq0\quad(\ell\ge2).
\]
\emph{Magnetic modes} $k=\Nd\ohat{}^*df$: $\Div k=(\nu+2){}^*df$, $\Div\Div k=0$, and $\sym\Nd{}^*df=\frac12k$ (${}^*df$ is
divergence free); so $E_0k=-\frac12(\nu+2)k$.

\emph{$\ell=1$}: there are no tensor harmonics, and $2\Nd^2F=-2F\gac$ (first harmonics satisfy $\Nd^2F=-F\gac$) produces the pure-trace
output. \emph{$\ell=0$}: constants $F$ lie in the kernel, and the constant trace is the only output not attained.

Hence $\ker E_0=\{(0,\mathrm{const})\}$, $\operatorname{coker}E_0=\R\gac$, and $\operatorname{ind}E_0=0$.

The comparison operator for $\Pcal$ is $\Pcal_0=\frac12\Nd\ohat\Div\Nd\ohat\Div-\Nd\ohat d\Div\Div$. By \eqref{eq:symbols}, $\Nd\ohat\Div$ has symbol
$k\mapsto i\cdot i\,k=-k$, so the first term has symbol $\frac12k$; $\Div\Div$ has symbol $k\mapsto-p$, and $\Nd\ohat d$ has symbol
$f\mapsto-\mathrm{diag}(f,-f)$, so the second term has symbol $-\mathrm{diag}(p,-p)$. Altogether the symbol is
$\frac12k-\mathrm{diag}(p,-p)=\sigma_4(\Pcal)k$. On tensor harmonics, by \eqref{eq:E1}, $\Pcal_0$ acts by the multiplier
$\frac12(\nu+2)^2-\nu(\nu+2)=(\nu+2)(1-\frac\nu2)$ on electric modes and $\frac12(\nu+2)^2$ on magnetic modes;
both are nonzero for $\ell\ge2$. So $\Pcal_0$ is invertible $H^r\to H^{r-4}$ and has index zero.

\begin{lemma}[Index zero]\label{lem:index}
At the extremal horizon, $(k,F)\mapsto Ek+\Ccal F$ and $\Pcal$ have Fredholm index zero, on the full spaces and on the
$\mathcal R_\pi$-invariant subspaces.
\end{lemma}
\begin{proof}
Join $\gac$ to $\gd_H$ by $g_t=(1-t)\gac+t\gd_H$, identifying the tracefree bundles by the positive square root of $g_t^{-1}\gd_H$. The
operators $E_0,\Pcal_0$ built with $g_t$ have, in a $g_t$-orthonormal frame at a unit covector, the same principal symbols as above,
which are invertible for every $t$. At $t=1$ deform the lower order coefficients linearly to the actual horizon ones; this does not
change principal symbols. The whole path consists of elliptic operators between bundles of equal rank, depending continuously on the
parameter in operator norm, hence of Fredholm operators, and the index is constant along it. The
invariant subspaces are handled in the same way, since $\mathcal R_\pi$ is an isometry of every $g_t$.
\end{proof}
Thus $\Pcal$ is invertible as soon as it is injective. We prove injectivity at all frequencies by an exact computation.

\subsection{The helicity calculus on the horizon}\label{sec:helcalc}

Fix an azimuthal mode $e^{im\psi}$ and complexify; the Fourier factor is suppressed throughout, so $\partial_\psi$ acts as $im$. We
abbreviate
\[
 M_m:=\frac{mu_z}{2s_z},\qquad \mu:=M_m+\frac{4z}{s_zu_z}=\frac{mu_z^2+8z}{2s_zu_z},\qquad c:=\frac1{z+i},\qquad\bar c=\frac1{z-i},
\]
and note $(z+i)(z-i)=u_z$, $\frac{2z}{u_z}=c+\bar c$, $\frac1{s_zu_z}=\frac12\big(\frac1{s_z}+\frac1{u_z}\big)$ (since $u_z+s_z=2$),
and $\frac{s_z}{u_z}\bar c=\frac{s_z}{(z-i)^2(z+i)}$.

\begin{definition}[Helicity components]\label{def:helicity}
A one-form $\gamma$ and a tracefree symmetric tensor $W$ (with respect to $\gd_H$) of mode $m$ are written as
\begin{gather*}
 \gamma_z=\gamma_++\gamma_-,\quad \gamma_\psi=\frac{2is_z}{u_z}(\gamma_+-\gamma_-);\\
 W_{zz}=\frac{u_z}{s_z}(W_++W_-),\quad W_{z\psi}=2i(W_+-W_-),\quad W_{\psi\psi}=-\frac{4s_z}{u_z}(W_++W_-).
\end{gather*}
A symmetric tensor $N$ is described by its trace $B=\tr_{\gd_H}N$ and the helicity components of its tracefree part.
\end{definition}
In the $\gd_H$-orthonormal frame $f_1=\sqrt{s_z/u_z}\,\partial_z$, $f_2=\frac12\sqrt{u_z/s_z}\,\partial_\psi$ one has
$W(f_1,f_1)=W_++W_-$ and $W(f_1,f_2)=i(W_+-W_-)$, so $W_\pm=\frac12\big(W(f_1,f_1)\mp iW(f_1,f_2)\big)$ are the components of spin
weight $\pm2$. Similarly, $\gamma_\pm$ are $\frac12\sqrt{u_z/s_z}$ times the spin $\pm1$ components. Note that $f_1$ points along
$-\partial_\vartheta$; so for the frame $(e_1,e_2)$ of \eqref{eq:roundframe} the roles of $\pm$ are interchanged. With the present
convention, positive helicity and mode $m$ correspond to the round-frame pattern $S(m\varphi+\ldots)$, i.e.\ to helicity $\sigma=m/2$
in the sense of Lemma~\ref{lem:poles}.

\begin{lemma}[Building blocks]\label{lem:blocks}
For a function $f$, a one-form $\gamma=(\gamma_+,\gamma_-)$, a tracefree tensor $W=(W_+,W_-)$ of mode $m$, and $\etab_H=-\eta_H$:
\begin{align*}
 (df)_\pm&=\tfrac12\big(f'\pm M_mf\big), \qquad\qquad
 (\Div W)_\pm=\Big(\partial_z\mp M_m-\frac{4z}{s_zu_z}\Big)W_\pm,\\
 \Div\gamma&=\frac{s_z}{u_z}\Big[\Big(\partial_z-M_m-\frac{4z}{s_zu_z}\Big)\gamma_++\Big(\partial_z+M_m-\frac{4z}{s_zu_z}\Big)\gamma_-\Big],\\
 (\Nd\ohat\gamma)_\pm&=\frac{s_z}{u_z}\big(\partial_z\pm M_m\big)\gamma_\pm,\qquad\qquad
 (\etab_H)_+=\tfrac12\bar c,\quad(\etab_H)_-=\tfrac12c,\\
 (W\cdot\etab_H)_+&=cW_+,\quad(W\cdot\etab_H)_-=\bar cW_-,\qquad
 (\etab_H\ohat\gamma)_+=\frac{s_z}{u_z}\,\bar c\,\gamma_+,\quad(\etab_H\ohat\gamma)_-=\frac{s_z}{u_z}\,c\,\gamma_-,\\
 ({}^*W)_\pm&=\pm i\,W_\pm,\qquad\qquad \gamma\cdot\gamma'=\frac{2s_z}{u_z}(\gamma_+\gamma'_-+\gamma_-\gamma'_+).
\end{align*}
Here $'=\partial_z$, all operations are with respect to $\gd_H$, and ${}^*$ is the left dual. The divergence lowers the spin weight by
one, and $\Nd\ohat$ raises it by one; in particular $(\Div W)_+$ depends only on $W_+$, and $\Nd\ohat\gamma$ is tracefree.
\end{lemma}

\begin{proof}
We use the complex frame $\mathfrak f:=f_1-if_2$, $\bar{\mathfrak f}=f_1+if_2$. On functions of mode $m$,
\[
 \mathfrak f=\sqrt{s_z/u_z}\,(\partial_z+M_m),\qquad \bar{\mathfrak f}=\sqrt{s_z/u_z}\,(\partial_z-M_m),
\]
since $\partial_\psi$ acts as $im$ and $\frac m2\sqrt{u_z/s_z}=\sqrt{s_z/u_z}\,M_m$. By the remarks after Definition~\ref{def:helicity},
\begin{gather*}
 \gamma(\mathfrak f)=2\sqrt{s_z/u_z}\,\gamma_+,\quad \gamma(\bar{\mathfrak f})=2\sqrt{s_z/u_z}\,\gamma_-,\\
 W(\mathfrak f,\mathfrak f)=4W_+,\quad W(\bar{\mathfrak f},\bar{\mathfrak f})=4W_-,\quad W(\mathfrak f,\bar{\mathfrak f})=0,
\end{gather*}
and $\sum_af_a\otimes f_a=\frac12(\mathfrak f\otimes\bar{\mathfrak f}+\bar{\mathfrak f}\otimes\mathfrak f)$. By \eqref{eq:ChrH}, $\nabla_{\partial_z}f_1$ is a multiple of
$\partial_z$ and $\nabla_{\partial_\psi}f_1=\sqrt{s_z/u_z}\,\Gamma^\psi_{z\psi}\partial_\psi=-\frac{4z}{u_z^2}f_2$. Hence
\[
 \nabla_X\mathfrak f=i\Gamma(X)\,\mathfrak f,\qquad \Gamma=-\frac{4z}{u_z^2}\,d\psi,\qquad
 i\Gamma(\mathfrak f)=-i\Gamma(\bar{\mathfrak f})=-\sqrt{s_z/u_z}\,\frac{2z}{s_zu_z},
\]
so that $(\nabla_X\gamma)(\mathfrak f)=X(\gamma(\mathfrak f))-i\Gamma(X)\gamma(\mathfrak f)$, $(\nabla_XW)(\mathfrak f,\mathfrak f)=X(W(\mathfrak f,\mathfrak f))-2i\Gamma(X)W(\mathfrak f,\mathfrak f)$
and $(\nabla_XW)(\mathfrak f,\bar{\mathfrak f})=0$. We also use $\partial_z\sqrt{s_z/u_z}=-\frac{2z}{s_zu_z}\sqrt{s_z/u_z}$. We derive the $(+)$-formulas; the
$(-)$-formulas follow in the same way with $\mathfrak f$ and $\bar{\mathfrak f}$ exchanged.

\emph{$df$ and $\etab_H$.} $df(\mathfrak f)=\mathfrak f(f)=\sqrt{s_z/u_z}\,(f'+M_mf)$. With $\etab_{H,z}=\frac z{u_z}$ and $\etab_{H,\psi}=-\frac{2s_z}{u_z^2}$,
$\etab_H(\mathfrak f)=\sqrt{s_z/u_z}\,\frac z{u_z}-\frac i2\sqrt{u_z/s_z}\,\big(-\frac{2s_z}{u_z^2}\big)=\sqrt{s_z/u_z}\,\frac{z+i}{u_z}$, and $\frac{z+i}{u_z}=\bar c$.

\emph{$\Div W$.} In $(\Div W)(\mathfrak f)=\frac12\big[(\nabla_{\bar{\mathfrak f}}W)(\mathfrak f,\mathfrak f)+(\nabla_{\mathfrak f}W)(\bar{\mathfrak f},\mathfrak f)\big]$ the second term vanishes, and
\[
 \tfrac12(\nabla_{\bar{\mathfrak f}}W)(\mathfrak f,\mathfrak f)=2\big(\bar{\mathfrak f}-2i\Gamma(\bar{\mathfrak f})\big)W_+
 =2\sqrt{s_z/u_z}\,\Big(W_+'-M_mW_+-\frac{4z}{s_zu_z}W_+\Big).
\]

\emph{$\Div\gamma$ and $\Nd\ohat\gamma$.} We have $\Div\gamma=\frac12\big[(\nabla_{\bar{\mathfrak f}}\gamma)(\mathfrak f)+(\nabla_{\mathfrak f}\gamma)(\bar{\mathfrak f})\big]$ and
$(\Nd\ohat\gamma)(\mathfrak f,\mathfrak f)=2(\nabla_{\mathfrak f}\gamma)(\mathfrak f)$, since $\gd_H(\mathfrak f,\mathfrak f)=0$. Now
\[
 (\nabla_{\bar{\mathfrak f}}\gamma)(\mathfrak f)=\frac{2s_z}{u_z}\Big(\gamma_+'-M_m\gamma_+-\frac{4z}{s_zu_z}\gamma_+\Big),\qquad
 (\nabla_{\mathfrak f}\gamma)(\mathfrak f)=\frac{2s_z}{u_z}\big(\gamma_+'+M_m\gamma_+\big):
\]
the derivative of the factor $\sqrt{s_z/u_z}$ in $\gamma(\mathfrak f)$ contributes $-\frac{2z}{s_zu_z}\gamma_+$ to both, and the connection term contributes
$-\frac{2z}{s_zu_z}\gamma_+$ to the first and $+\frac{2z}{s_zu_z}\gamma_+$ to the second.

\emph{Algebraic terms.} $(W\cdot\etab_H)(\mathfrak f)=\frac12W(\mathfrak f,\mathfrak f)\,\etab_H(\bar{\mathfrak f})$,
$(\etab_H\ohat\gamma)(\mathfrak f,\mathfrak f)=2\etab_H(\mathfrak f)\gamma(\mathfrak f)$ and
$\gamma\cdot\gamma'=\frac12\big[\gamma(\mathfrak f)\gamma'(\bar{\mathfrak f})+\gamma(\bar{\mathfrak f})\gamma'(\mathfrak f)\big]$. For the dual, $({}^*W)(f_1,\cdot)=W(f_2,\cdot)$ and
$({}^*W)(f_2,\cdot)=-W(f_1,\cdot)$ in the positively oriented frame ($\slashed\epsilon_{12}=1$), so $({}^*W)(\mathfrak f,\cdot)=iW(\mathfrak f,\cdot)$.

Inserting the component relations above gives all formulas.
\end{proof}

\subsection{$E$ in helicity form}\label{sec:Ehel}

Take a positive helicity input $k=(a,0)$, i.e.\ $k_{zz}=\frac{u_z}{s_z}a$, $k_{z\psi}=2ia$, $k_{\psi\psi}=-\frac{4s_z}{u_z}a$, and put
\[
 D_a:=\Big(\partial_z-M_m-\frac{4z}{s_zu_z}\Big)a=a'-\mu a .
\]
\begin{proposition}\label{prop:Ehel}
\begin{align*}
 \tr(Ek)&=-\frac{2s_z}{u_z}\Big(\partial_z-M_m-\frac{4z}{s_zu_z}+2c\Big)D_a,\qquad (Ek)_-=0,\\
 (Ek)_+&=-\frac{s_z}{2u_z}\Big(\partial_z+M_m+\frac{2z}{u_z}\Big)D_a+\Big(\frac{(i-z)m}{2u_z}+\frac{2(z+i)^2}{u_z^3}\Big)a .
\end{align*}
\end{proposition}

\begin{proof}
We evaluate the five terms of Lemma~\ref{lem:Ecov}.

\emph{$\gamma=\frac12\Div k$.} By Lemma~\ref{lem:blocks}, $\gamma_+=\frac12D_a$, $\gamma_-=0$.

\emph{The terms $-\Nd_A\gamma_B-\Nd_B\gamma_A-(\Div\gamma)\gd$.} Since $\Nd_A\gamma_B+\Nd_B\gamma_A=(\Nd\ohat\gamma)_{AB}+\Div\gamma\,\gd_{AB}$, these equal
$-\Nd\ohat\gamma-2\Div\gamma\,\gd$. By Lemma~\ref{lem:blocks}, $(\Nd\ohat\gamma)_+=\frac{s_z}{2u_z}(\partial_z+M_m)D_a$, $(\Nd\ohat\gamma)_-=0$, and
$\Div\gamma=\frac{s_z}{2u_z}(\partial_z-M_m-\frac{4z}{s_zu_z})D_a$, contributing $-4\Div\gamma$ to the trace.

\emph{The term $2(\eta_H\otimes\gamma+\gamma\otimes\eta_H)=-2(\etab_H\ohat\gamma)-2(\etab_H\cdot\gamma)\gd$.} By Lemma~\ref{lem:blocks}, its trace is
$-4\etab_H\cdot\gamma=-4\cdot\frac{2s_z}{u_z}\gamma_+(\etab_H)_-=-\frac{2s_z}{u_z}c\,D_a$, and its $(+)$-component is
$-2\frac{s_z}{u_z}\bar c\,\gamma_+=-\frac{s_z}{u_z}\bar c\,D_a$.

\emph{The term $Q_{AB}:=2\dot\Gamma^C_{AB}\eta_{H,C}=\eta_H^C(\Nd_Ak_{BC}+\Nd_Bk_{AC}-\Nd_Ck_{AB})$.} Its trace is
$\eta_H^C\big(2(\Div k)_C-\Nd_C\tr k\big)=4\eta_H\cdot\gamma=-\frac{2s_z}{u_z}c\,D_a$, the same as for the previous term. For the helicity
components we use the complex frame of the proof of Lemma~\ref{lem:blocks}. Since
$\eta_H^\sharp=\frac12\big(\eta_H(\bar{\mathfrak f})\mathfrak f+\eta_H(\mathfrak f)\bar{\mathfrak f}\big)$ and $(\nabla_Xk)(\mathfrak f,\bar{\mathfrak f})=0$,
\[
 Q(\mathfrak f,\mathfrak f)=2(\nabla_{\mathfrak f}k)(\mathfrak f,\eta_H^\sharp)-(\nabla_{\eta_H^\sharp}k)(\mathfrak f,\mathfrak f)
 =\tfrac12\big[\eta_H(\bar{\mathfrak f})\,(\nabla_{\mathfrak f}k)(\mathfrak f,\mathfrak f)-\eta_H(\mathfrak f)\,(\nabla_{\bar{\mathfrak f}}k)(\mathfrak f,\mathfrak f)\big].
\]
Here $\eta_H(\mathfrak f)=-\sqrt{s_z/u_z}\,\bar c$ and $\eta_H(\bar{\mathfrak f})=-\sqrt{s_z/u_z}\,c$, and, with $k(\mathfrak f,\mathfrak f)=4a$ and the connection
of the proof of Lemma~\ref{lem:blocks},
\begin{gather*}
 (\nabla_{\bar{\mathfrak f}}k)(\mathfrak f,\mathfrak f)=4\sqrt{s_z/u_z}\,D_a,\\
 (\nabla_{\mathfrak f}k)(\mathfrak f,\mathfrak f)=4\sqrt{s_z/u_z}\Big(a'+M_ma+\frac{4z}{s_zu_z}a\Big)=4\sqrt{s_z/u_z}\,\big(D_a+2\mu a\big).
\end{gather*}
Hence, using $\bar c-c=\frac{2i}{u_z}$,
\[
 Q_+=\tfrac14Q(\mathfrak f,\mathfrak f)=\frac{s_z}{2u_z}\big[(\bar c-c)D_a-2\mu c\,a\big]=\frac{is_z}{u_z^2}D_a-\frac{s_z\mu}{u_z}c\,a ,
\]
and $Q_-=0$ by the same computation, since $k(\bar{\mathfrak f},\bar{\mathfrak f})=0$.

\emph{The term $-K_Hk$}: trace $0$, $(+)$-component $-K_Ha$.

\emph{Sum.} The trace is $-4\Div\gamma-\frac{4s_z}{u_z}c\,D_a=-\frac{2s_z}{u_z}(\partial_z-M_m-\frac{4z}{s_zu_z}+2c)D_a$. The $(-)$-component
vanishes. The $(+)$-component is
\[
 -\frac{s_z}{2u_z}(\partial_z+M_m)D_a-\frac{s_z}{u_z}\bar c\,D_a+\frac{is_z}{u_z^2}D_a-\frac{s_z\mu}{u_z}c\,a-K_Ha .
\]
Now $-\bar c+\frac i{u_z}=-\frac{z+i}{u_z}+\frac i{u_z}=-\frac z{u_z}$, so the $D_a$-terms are $-\frac{s_z}{2u_z}(\partial_z+M_m+\frac{2z}{u_z})D_a$.
For the $a$-terms use $\frac{s_z\mu}{u_z}=\frac{mu_z^2+8z}{2u_z^2}$, $c=\frac{z-i}{u_z}$ and $K_H=\frac{2(1-3z^2)}{u_z^3}$
(see \eqref{eq:KH}):
\begin{align*}
 -\frac{(mu_z^2+8z)(z-i)}{2u_z^3}-\frac{2(1-3z^2)}{u_z^3}
 &=\frac{(i-z)m}{2u_z}-\frac{8z(z-i)+4(1-3z^2)}{2u_z^3}\\
 &=\frac{(i-z)m}{2u_z}-\frac{2(1-z^2-2iz)}{u_z^3}
 =\frac{(i-z)m}{2u_z}+\frac{2(z+i)^2}{u_z^3},
\end{align*}
since $(z+i)^2=-(1-z^2-2iz)$.
\end{proof}

\subsection{$\Bcal$ in helicity form}\label{sec:Bhel}

\begin{proposition}\label{prop:Bhel}
For a symmetric $N$ with trace $B$ and tracefree helicity components $(n_+,n_-)$:
\begin{align*}
 (\beta_N)_+&=\Big(\partial_z-M_m-\frac{4z}{s_zu_z}+c\Big)n_+-\tfrac14\big(\partial_z+M_m+\bar c\big)B,\\
 (\beta_N)_-&=\Big(\partial_z+M_m-\frac{4z}{s_zu_z}+\bar c\Big)n_--\tfrac14\big(\partial_z-M_m+c\big)B,\\
 (\Bcal N)_+&=-\frac{s_z}{u_z}\big(\partial_z+M_m+5\bar c\big)(\beta_N)_++\frac{6}{(1+iz)^3}\,n_+,\\
 (\Bcal N)_-&=-\frac{s_z}{u_z}\big(\partial_z-M_m+5c\big)(\beta_N)_-+\frac{6}{(1-iz)^3}\,n_-,
\end{align*}
and $\tr\Bcal N=0$.
\end{proposition}
\begin{proof}
This follows directly from Lemma~\ref{lem:blocks}. For $\beta_N=\Div W-\frac12dB+W\cdot\etab_H-\frac12B\etab_H$ we get
$(\beta_N)_+=(\partial_z-M_m-\frac{4z}{s_zu_z})n_+-\frac14(B'+M_mB)+cn_+-\frac14\bar cB$, and similarly for $(-)$. For
$\Bcal N=-\Nd\ohat\beta_N-5\etab_H\ohat\beta_N+3K_HW+3\sigma_H{}^*W$, the first two terms give
$-\frac{s_z}{u_z}(\partial_z+M_m)(\beta_N)_+-5\frac{s_z}{u_z}\bar c(\beta_N)_+$, and the last two give $3(K_H+i\sigma_H)n_+=\frac6{(1+iz)^3}n_+$
by \eqref{eq:KH}. All terms are tracefree.
\end{proof}
The near-horizon curvature thus enters only through the combination $K_H\pm i\sigma_H=2/(1\pm iz)^3$.

\subsection{$\Pcal=\Bcal E$: the opposite helicity is a product}\label{sec:Phel}

\begin{proposition}\label{prop:Phel}
For the positive helicity input $k=(a,0)$:
\begin{align}
 (\Pcal k)_+&=-\frac{3im}{(z-i)^4}\,a,\label{eq:same}\\
 (\Pcal k)_-&=-\frac{s_z}{2u_z}\big(\partial_z-M_m+5c\big)\big(\partial_z-M_m+c\big)\label{eq:opp}\\
 &\qquad\cdot\frac{s_z}{u_z}\Big(\partial_z-M_m-\frac{4z}{s_zu_z}+2c\Big)\Big(\partial_z-M_m-\frac{4z}{s_zu_z}\Big)a .\notag
\end{align}
\end{proposition}

\begin{proof}
\emph{Opposite helicity.} $N=Ek$ has $n_-=(Ek)_-=0$ and $B=\tr Ek$. By Proposition~\ref{prop:Bhel}, $(\beta_N)_-=-\frac14(\partial_z-M_m+c)B$
and $(\Bcal N)_-=-\frac{s_z}{u_z}(\partial_z-M_m+5c)(\beta_N)_-$. Inserting
$B=-\frac{2s_z}{u_z}(\partial_z-M_m-\frac{4z}{s_zu_z}+2c)(\partial_z-M_m-\frac{4z}{s_zu_z})a$ from Proposition~\ref{prop:Ehel}, the constants
combine to $-\frac{s_z}{u_z}\cdot(-\frac14)\cdot(-2)=-\frac{s_z}{2u_z}$, which is \eqref{eq:opp}. No cancellation is involved: the
opposite-helicity output is the composition of four first order operators.

\emph{Same helicity.} Here there is a cancellation, which is best seen in terms of $D_a$. By Proposition~\ref{prop:Ehel},
\[
 (Ek)_+=-\frac{s_z}{2u_z}(\partial_z+t)D_a+\kappa\,a,\qquad \tr Ek=-\frac{2s_z}{u_z}(\partial_z+q)D_a,
\]
with
\begin{gather*}
 t:=M_m+\frac{2z}{u_z},\qquad q:=-M_m-\frac{4z}{s_zu_z}+2c,\\
 \kappa:=\frac{(i-z)m}{2u_z}+\frac{2(z+i)^2}{u_z^3}=-\frac m{2(z+i)}+\frac2{(z-i)^3(z+i)} .
\end{gather*}

\emph{Step 1: $(\beta_{Ek})_+$.} Insert $n_+=(Ek)_+$ and $B=\tr Ek$ in Proposition~\ref{prop:Bhel}. Since
$\partial_z\frac{s_z}{u_z}=-\frac{4z}{s_zu_z}\frac{s_z}{u_z}$, moving the factors $\frac{s_z}{2u_z}$ to the left gives
\[
 (\beta_{Ek})_+=\frac{s_z}{2u_z}\big[(\partial_z+p)(\partial_z+q)-(\partial_z+r)(\partial_z+t)\big]D_a
 +\Big(\partial_z-M_m-\frac{4z}{s_zu_z}+c\Big)(\kappa a),
\]
with $p:=M_m-\frac{4z}{s_zu_z}+\bar c$ and $r:=-M_m-\frac{8z}{s_zu_z}+c$. The operator in brackets is
$(p+q-r-t)\partial_z+(q'+pq-t'-rt)$, and $p+q-r-t=c+\bar c-\frac{2z}{u_z}=0$; so the terms with $D_a''$ and $D_a'$ cancel. A direct
computation of the zeroth order coefficient gives
\[
 \frac{s_z}{2u_z}\big(q'+pq-t'-rt\big)=\frac m{2(z+i)}-\frac{i(3z+i)}{(z-i)^3(z+i)} .
\]
In the last term we write $a'=D_a+\mu a$; since $\mu-M_m-\frac{4z}{s_zu_z}=0$, it equals $\kappa D_a+(\kappa'+c\kappa)a$, and
$\kappa'+c\kappa=-\frac6{(z-i)^4(z+i)}$ (the terms with $m$ cancel). The terms with $m$ also cancel in the coefficient of $D_a$, and since
$2-i(3z+i)=3(1-iz)=-3i(z+i)$,
\[
 (\beta_{Ek})_+=-\frac{3i}{(z-i)^3}\,D_a-\frac6{(z-i)^4(z+i)}\,a=-\frac{3i}{(z-i)^3}\Big(D_a-\frac{2i}{u_z}\,a\Big).
\]

\emph{Step 2: $(\Pcal k)_+$.} We have $(\partial_z+M_m+5\bar c)\big((z-i)^{-3}h\big)=(z-i)^{-3}(\partial_z+M_m+2\bar c)h$, because
$\partial_z(z-i)^{-3}=-3\bar c(z-i)^{-3}$, and $\frac6{(1+iz)^3}=\frac{6i}{(z-i)^3}$. By Proposition~\ref{prop:Bhel},
\[
 (\Pcal k)_+=\frac{3is_z}{u_z(z-i)^3}\Big[(\partial_z+M_m+2\bar c)\Big(D_a-\frac{2i}{u_z}a\Big)-(\partial_z+t)D_a\Big]+\frac{6i}{(z-i)^3}\,\kappa\,a .
\]
In the bracket the terms with $D_a'$ cancel, the remaining multiple of $D_a$ is $(M_m+2\bar c-t)D_a-\frac{2i}{u_z}D_a=0$ (as
$2\bar c-\frac{2z}{u_z}=\frac{2i}{u_z}$ and $a'=D_a+\mu a$), and the multiple of $a$ is
$\frac{2i}{u_z}\big(-\mu-M_m-2\bar c+\frac{2z}{u_z}\big)=-\frac{2i}{u_z}\big(2M_m+\frac{4z}{s_zu_z}+\frac{2i}{u_z}\big)$. Hence
\[
 (\Pcal k)_+=\frac6{(z-i)^3}\Big[\frac{s_z}{u_z^2}\Big(2M_m+\frac{4z}{s_zu_z}+\frac{2i}{u_z}\Big)+i\kappa\Big]a .
\]
In the bracket the terms with $m$ are $\frac m{u_z}-\frac{(1+iz)m}{2u_z}=\frac{(1-iz)m}{2u_z}=-\frac{im}{2(z-i)}$, and the terms without $m$ are
$\frac{4z+2is_z+2i(z+i)^2}{u_z^3}=0$. Therefore $(\Pcal k)_+=-\frac{3im}{(z-i)^4}\,a$, which is \eqref{eq:same}.
\end{proof}

\subsection{Conjugation and the Teukolsky--Starobinsky factorisation}\label{sec:TS}

Substitute $a=U\psi$ with $U:=\frac{z-i}{z+i}$, and multiply the outputs by $V:=u_z(z-i)^2$, resp.\ $\overline V=u_z(z+i)^2$. These are
smooth, nowhere vanishing factors on $[-1,1]$, i.e.\ bundle isomorphisms. Put
\[
 A_m:=\partial_z-M_m-\frac{2z}{s_z},\qquad T_m:=s_z^2A_m^4 .
\]

\begin{proposition}\label{prop:TS}
$V(\Pcal k)_+=-3im\,\psi$ and $\overline V(\Pcal k)_-=-\frac12T_m\psi$.
\end{proposition}
\begin{proof}
\emph{Same helicity.} $V\cdot\big(-\frac{3im}{(z-i)^4}\big)\frac{z-i}{z+i}\psi=-3im\frac{u_z(z-i)^3}{(z-i)^4(z+i)}\psi=-3im\psi$.

\emph{Opposite helicity.} We push $U$ and the scalar factors through \eqref{eq:opp} from the right.
\begin{enumerate}[label=(\arabic*)]
\item $U'/U=\bar c-c=\frac{2i}{u_z}$ and $\frac{4z}{s_zu_z}=\frac{2z}{s_z}+\frac{2z}{u_z}=\frac{2z}{s_z}+c+\bar c$. Hence
$(\partial_z-M_m-\frac{4z}{s_zu_z})(U\psi)=U\big(\psi'+(\bar c-c)\psi-M_m\psi-\frac{2z}{s_z}\psi-(c+\bar c)\psi\big)=U(A_m-2c)\psi$.
\item Similarly $(\partial_z-M_m-\frac{4z}{s_zu_z}+2c)(Ug)=U A_mg$. So the two inner factors give $U\,A_m(A_m-2c)\psi=:U h$.
\item $\frac{s_z}{u_z}U=\frac{s_z}{(z+i)^2}=s_zc^2$. Since $\partial_z(s_zc^2)=s_zc^2(-\frac{2z}{s_z}-2c)$, we have
$(\partial_z-M_m+c)(s_zc^2h)=s_zc^2(A_m-c)h$, and then $(\partial_z-M_m+5c)(s_zc^2h_2)=s_zc^2(A_m+3c)h_2$.
\item Altogether $(\Pcal k)_-=-\frac{s_z}{2u_z}s_zc^2(A_m+3c)(A_m-c)A_m(A_m-2c)\psi$, and multiplying by $\overline V=u_z(z+i)^2=u_zc^{-2}$,
\[
 \overline V(\Pcal k)_-=-\tfrac12s_z^2\,(A_m+3c)(A_m-c)A_m(A_m-2c)\,\psi .
\]
\item Finally, $(A_m+3c)(A_m-c)A_m(A_m-2c)=A_m^4$. Indeed $[A_m,c]=c'=-c^2$, so $A_mc^j=c^jA_m-jc^{j+1}$, and
\begin{align*}
 A_m(A_m-2c)&=A_m^2-2cA_m+2c^2,\\
 (A_m-c)(A_m^2-2cA_m+2c^2)&=A_m^3-3cA_m^2+6c^2A_m-6c^3,\\
 (A_m+3c)(A_m^3-3cA_m^2+6c^2A_m-6c^3)&=A_m^4+(3-3)cA_m^3+(3+6-9)c^2A_m^2\\
 &\quad+(-12-6+18)c^3A_m+(18-18)c^4 .
\end{align*}
In the last line we used $-3A_mcA_m^2=-3cA_m^3+3c^2A_m^2$, $6A_mc^2A_m=6c^2A_m^2-12c^3A_m$, $-6A_mc^3=-6c^3A_m+18c^4$, and the
terms $3cA_m^3-9c^2A_m^2+18c^3A_m-18c^4$ from $3c\times(\ldots)$.\qedhere
\end{enumerate}
\end{proof}

\paragraph{Teukolsky--Starobinsky.} In the polar angle, $T_m$ is the angular Teukolsky--Starobinsky product at spheroidal parameter
$m/2$ \cite{TP,Chandra}:
\[
 T_m=\mathcal L^{(m)}_{-1}\mathcal L^{(m)}_0\mathcal L^{(m)}_1\mathcal L^{(m)}_2,\qquad
 \mathcal L^{(m)}_n=\partial_\vartheta+\frac m{\sin\vartheta}-\frac m2\sin\vartheta+n\cot\vartheta=-\sqrt{s_z}\Big(\partial_z-M_m-\frac{nz}{s_z}\Big).
\]
Indeed, with $z=\cos\vartheta$ we have $\partial_\vartheta=-\sqrt{s_z}\partial_z$, $\cot\vartheta=z/\sqrt{s_z}$, and
$\frac m{\sin\vartheta}-\frac m2\sin\vartheta=\frac{m(2-s_z)}{2\sqrt{s_z}}=\sqrt{s_z}\,M_m$. Applying $\mathcal L_n$ to $(-1)^ks_z^{k/2}g$ gives
\[
 (-1)^{k+1}s_z^{(k+1)/2}\Big(\partial_z-M_m-\frac{(k+n)z}{s_z}\Big)g,
\]
because $\partial_zs_z^{k/2}=-kzs_z^{k/2-1}$. In the product
$\mathcal L_{-1}\mathcal L_0\mathcal L_1\mathcal L_2$ the pairs $(k,n)$ are $(0,2)$, $(1,1)$, $(2,0)$, $(3,-1)$, so always $k+n=2$, and the
successive outputs are $-\sqrt{s_z}A_mg$, $s_zA_m^2g$, $-s_z^{3/2}A_m^3g$, $s_z^2A_m^4g$.

\paragraph{The negative helicity input.} The operator $\Pcal$ is real. Complex conjugation maps a positive helicity tensor of mode $m$
to a negative helicity tensor of mode $-m$, and conjugation of the above (with $\overline U$, $\overline V$, and $T_m$ real) gives, for
the input $k=(0,b)$ with $b=\overline U\psi$, $\overline V(\Pcal k)_-=-3im\psi$ and $V(\Pcal k)_+=-\frac12T_{-m}\psi$. In the $m$-sector,
$\Pcal$ therefore becomes the $2\times2$ block operator
\begin{equation}\label{eq:block}
 \mathcal D_m=\begin{pmatrix}-3im&-\frac12T_{-m}\\[2pt]-\frac12T_m&-3im\end{pmatrix}\quad\text{acting on }(\psi_+,\psi_-),
\end{equation}
with outputs $(V(\Pcal k)_+,\overline V(\Pcal k)_-)$.

\subsection{Pole domains and the adjoint identity}\label{sec:adjoint}

For the adjoint identity we need to know exactly which amplitudes correspond to smooth tensors. By Appendix~\ref{app:sphere}, a field
of spin weight $s$ with polar complex component $\tau$ is smooth at the north pole if and only if $\tau e^{is\varphi}$ is smooth there,
and at the south pole if and only if $\tau e^{-is\varphi}$ is (Theorem~\ref{A:thm:spins}); combined with the description of smooth
Fourier modes $f(r)e^{ik\varphi}$ in polar coordinates, which are smooth if and only if $f=r^{|k|}g(r^2)$ with $g$ smooth
(Theorem~\ref{A:thm:fourier}), this shows that $\tau=a(\vartheta)e^{im\varphi}$ is smooth if and only if
$a=\vartheta^{|m+s|}f_N(\vartheta^2)$ near $\vartheta=0$ and $a=(\pi-\vartheta)^{|m-s|}f_S((\pi-\vartheta)^2)$ near $\vartheta=\pi$, with
$f_N,f_S$ smooth (Corollary~\ref{A:cor:modes}). The horizon frame $(f_1,f_2)$ differs from the round frame by a smooth nonsingular
axisymmetric bundle map, and $f_1$ points along $-\partial_\vartheta$, so that the roles of $\pm$ are interchanged
(Section~\ref{sec:helcalc}). Consequently, a tensor of mode $m$ and positive helicity is smooth if and only if its amplitude satisfies
\begin{equation}\label{eq:poles}
 a(\vartheta)=\vartheta^{|m+2|}f_N(\vartheta^2)\ \text{near }\vartheta=0,\qquad a(\pi-t)=t^{|m-2|}f_S(t^2)\ \text{near }\vartheta=\pi,
\end{equation}
with $f_N,f_S$ smooth: the spin-two frame contributes azimuthal weight $\pm2$. For negative helicity the two exponents are interchanged.
More generally, a section of spin weight $n$ has exponents $|m+n|$ and $|m-n|$. Since $U,V$ are smooth and nowhere zero, the same
conditions describe the domains after the conjugation.

On these domains $\mathcal L^{(m)}_n$ maps spin weight $n$ smoothly to spin weight $n-1$. Near the north pole its singular part is
$\partial_\vartheta+(m+n)/\vartheta$. If $m+n\ge0$ it lowers the leading power by one, from $|m+n|$ to $|m+n-1|$ (for $m+n=0$: the
derivative of an even function vanishes to first order). If $m+n<0$ the leading terms cancel and the next power,
$|m+n|+1=|m+n-1|$, survives. The south pole is the same with $\partial_t=-\partial_\vartheta$, and the smooth term
$-\frac m2\sin\vartheta$ has the right parity and order.

\begin{lemma}[Adjoint identity]\label{lem:adjoint}
For smooth sections in the domains \eqref{eq:poles}, $\langle T_mf,g\rangle_{L^2(dz)}=\langle f,T_{-m}g\rangle_{L^2(dz)}$.
\end{lemma}
\begin{proof}
On $[\delta,\pi-\delta]$ with the measure $\sin\vartheta\,d\vartheta=-dz$, integration by parts gives
$\int(\partial_\vartheta f)\bar g\sin\vartheta\,d\vartheta=[f\bar g\sin\vartheta]-\int f\,\overline{(\partial_\vartheta+\cot\vartheta)g}\sin\vartheta\,d\vartheta$,
and the multiplication terms of $\mathcal L_n$ are real. So
$(\mathcal L^{(m)}_n)^*=-(\partial_\vartheta+\cot\vartheta)+\frac m{\sin\vartheta}-\frac m2\sin\vartheta+n\cot\vartheta=-\mathcal L^{(-m)}_{1-n}$.
Reversing the order of the four factors, the four minus signs cancel and the indices $(-1,0,1,2)$ become $(2,1,0,-1)$, i.e.\ we obtain
$T_{-m}$. Each boundary term is $\sin\vartheta$ times a product of two intermediate sections, which are smooth by the preceding paragraph
and so have bounded amplitudes. Hence the boundary terms vanish as $\delta\to0$.
\end{proof}

\subsection{The all-frequency inverse}\label{sec:Pinv}

\begin{theorem}\label{thm:Pinv}
$\Pcal:H^r(\mathrm{Sym}^2_0)\to H^{r-4}(\mathrm{Sym}^2_0)$ is an isomorphism for every $r\ge4$, also on $\mathcal R_\pi$-invariant tensors.
\end{theorem}
\begin{proof}
\emph{Kernel.} An element of the kernel is smooth by elliptic regularity, so we may work mode by mode. For $m\ne0$ the off-diagonal
part of \eqref{eq:block} is self-adjoint by the adjoint identity (Lemma~\ref{lem:adjoint}):
$\langle T_m\psi_+,\psi_-\rangle=\langle\psi_+,T_{-m}\psi_-\rangle=\overline{\langle T_{-m}\psi_-,\psi_+\rangle}$. Hence
\[
 \langle\mathcal D_m\psi,\psi\rangle=-3im\big(\|\psi_+\|^2+\|\psi_-\|^2\big)-\tfrac12\big(X+\overline X\big),\qquad X=\langle T_{-m}\psi_-,\psi_+\rangle,
\]
so $\Im\langle\mathcal D_m\psi,\psi\rangle=-3m\big(\|\psi_+\|^2+\|\psi_-\|^2\big)$, and $\mathcal D_m\psi=0$ forces $\psi=0$. The point is
that the zeroth order diagonal term $-3im$, which comes from the same-helicity formula \eqref{eq:same}, is purely imaginary, while the
fourth order part is self-adjoint; no estimate uniform in $m$ is needed.

For $m=0$, $A_0=\partial_z-\frac{2z}{s_z}=s_z^{-1}\partial_zs_z$, so $T_0\psi=s_z\partial_z^4(s_z\psi)$. By \eqref{eq:poles} with $m=0$
(both exponents equal to $2$), a smooth axisymmetric spin-two section has $\psi=s_zf$ with $f$ smooth on $[-1,1]$. So $T_0\psi=0$ makes
$s_z^2f$ a polynomial of degree at most three with double zeros at $z=\pm1$; such a polynomial vanishes. Hence $f=0$.

\emph{Cokernel.} The adjoint block is $\mathcal D_m^*=\begin{pmatrix}3im&-\frac12T_{-m}\\-\frac12T_m&3im\end{pmatrix}$, and the same
argument applies. Since $\mathcal D_m$ is obtained from $\Pcal$ by invertible smooth bundle maps, $\Pcal$ is injective with trivial
cokernel. As it is Fredholm of index zero (Lemma~\ref{lem:index}), it is bijective with bounded inverse. By uniqueness and equivariance
the inverse preserves the $\mathcal R_\pi$-sector.
\end{proof}

\subsection{The cut, the boost and the single scalar obstruction}\label{sec:obstruction}

\begin{corollary}[Range of $E+\Ccal$]\label{cor:range}
Let $\Lambda(N):=\int_{S^2}w\,\tr_{\gd_H}N\,d\mu_H$. The operator $(k,F)\mapsto Ek+\Ccal F$,
$H^r(\mathrm{Sym}^2_0)\times H^r\to H^{r-2}(\mathrm{Sym}^2)$, has kernel $\{(0,\text{const})\}$ and range exactly $\ker\Lambda$. On
$\{\int F\,d\mu_H=0\}$ it is an isomorphism onto $\ker\Lambda$, with a gain of two derivatives.
\end{corollary}
\begin{proof}
\emph{Kernel.} If $Ek+\Ccal F=0$, apply $\Bcal$: since $\Bcal\Ccal F=0$, we get $\Pcal k=0$, so $k=0$ by Theorem~\ref{thm:Pinv}. Then
$\Ccal F=0$, and taking the trace, $\Delta F+2\etab_H\cdot dF=0$. By the maximum principle (an elliptic operator without zeroth order
term on a closed manifold) $F$ is constant.

\emph{$\Lambda$ annihilates the range.} $\Lambda(\Ccal F)=0$ by Lemma~\ref{lem:targettrace}. For $Ek$: since $\Ecal(\gd_H,\etab_H)=0$,
$\tr(Ek)$ is the first variation of $\tr\Ecal=2(\Div\etab+|\etab|^2-K)$ along $(k,-\frac12\Div k)$. As $d\mu$ does not vary
($k$ is tracefree), $\Lambda(Ek)=2\cdot4\pi\,\dot\Fcal$. By Lemma~\ref{lem:crit}, $\dot\Fcal=0$: the metric variation is tracefree, and the
torsion variation has axial pairing $-\frac12\int(\Div k)(\partial_\psi)\,d\mu=\frac12\int k^{AB}\Nd_A(\partial_\psi)_B\,d\mu=0$, since
$\partial_\psi$ is Killing for $\gd_H$.

\emph{Range.} The operator is elliptic of index zero (Lemma~\ref{lem:index}), with one-dimensional kernel. So its cokernel is
one-dimensional, and the nonzero functional $\Lambda$ spans the annihilator of the range.
\end{proof}

The obstruction $\Lambda$ is precisely what the boost absorbs, because $\Lambda(\chib_H)\ne0$ (Lemma~\ref{lem:targettrace}): the one
scalar condition that the seed cannot achieve is exactly the one that measures $\lambda$. This is why the sign of $\lambda$ is ultimately
controlled by the scalar functional $\Fcal$ (Section~\ref{ch:sign}), in accordance with the heuristic of Section~\ref{sec:toy}.

\begin{theorem}[Linear isomorphism at the limiting corner]\label{thm:iso}
Let $x=(H_1,H_2,G,\lambda)\in H^r(\mathrm{Sym}^2_0)^2\times H^r_0(S^2)\times\R$ ($G$ of mean zero). The map
\[
 \mathcal Ax:=\Big(-EH_1+e^{2\omega}(\lambda\chib_H+\Ccal G),\ \frac1{2\omega}\Pcal(H_2-H_1)\Big)\in H^{r-2}(\mathrm{Sym}^2)\times H^{r-4}(\mathrm{Sym}^2_0)
\]
is a Banach space isomorphism, also on $\mathcal R_\pi$-invariant elements.
\end{theorem}
\begin{proof}
Given $(N_0,A_0)$, apply $\Lambda$ to the first component; $\Lambda(EH_1)=\Lambda(\Ccal G)=0$ determines
$\lambda=\Lambda(N_0)/(e^{2\omega}\Lambda(\chib_H))$. Then $N_0-e^{2\omega}\lambda\chib_H\in\ker\Lambda$, and Corollary~\ref{cor:range}
determines $(-H_1,e^{2\omega}G)$, hence $(H_1,G)$, uniquely. Finally Theorem~\ref{thm:Pinv} gives $H_2=H_1+2\omega\Pcal^{-1}A_0$. The
estimates follow from the gains of two and four derivatives.
\end{proof}

\paragraph{Moments are realised by seeds.} Fix two smooth radial profiles $\varrho_1,\varrho_2$ supported in $(0,1)$ whose moment
matrix $\big(\int_0^1e^{2j\omega v}\varrho_i\,dv\big)_{ij}$ is invertible (e.g.\ two bumps with disjoint supports; the matrix is then a
Vandermonde-type matrix of distinct exponentials, invertible for generic positions). For $H_1,H_2$ let $\mathcal R(H_1,H_2)(v)$ be the
linear combination of $\varrho_1(v),\varrho_2(v)$ with tensor coefficients that has moments $(H_1,H_2)$. It is bounded
$H^r\to C^\infty_vH^r_x$.

\section{Solving the nonlinear endpoint equation}\label{ch:fixedpoint}

We solve the endpoint equation \eqref{eq:endpoint}, $\Jcal(k,\lambda,G)-\Kcal_T(k)=0$, by a contraction argument whose constants do not
depend on the neck length $T$. The unknowns are the two seed moments, the cut function and the boost. In addition, the seed contains a small
multiple $\varepsilon$ of a fixed direction $k_*$ which the linearisation does not see; this direction makes $\lambda$ positive in
Section~\ref{ch:sign}.

\subsection{The ansatz}

Fix a seed direction $k_*$, smooth, supported in $(0,1)$, $\mathcal R_\pi$-invariant, with vanishing moments
\[
 H_1(k_*)=\int_0^1e^{2\omega v}k_*(v)\,dv=0,\qquad H_2(k_*)=\int_0^1e^{4\omega v}k_*(v)\,dv=0 ;
\]
it is chosen in Section~\ref{ch:sign}. We look for solutions of \eqref{eq:endpoint} of the form
\[
 k=\varepsilon k_*+\mathcal R(H_1,H_2),\qquad x=(H_1,H_2,G,\lambda),
\]
where $\mathcal R$ is the moment realisation of Section~\ref{sec:obstruction}, so that the moments of $k$ are exactly $(H_1,H_2)$. Then
\eqref{eq:endpoint} reads
\[
 \Theta_T(x,\varepsilon):=\Jcal\big(\varepsilon k_*+\mathcal R(H_1,H_2),\lambda,G\big)-\Kcal_T\big(\varepsilon k_*+\mathcal R(H_1,H_2)\big)=0 .
\]
Let $X:=H^r(\mathrm{Sym}^2_0)^2\times H^r_0(S^2)\times\R$ and $Y:=H^{r-2}(\mathrm{Sym}^2)\times H^{r-4}(\mathrm{Sym}^2_0)$, restricted to
$\mathcal R_\pi$-invariant elements.

\subsection{Four properties of $\Theta_T$}

\begin{lemma}\label{lem:fourfacts}
There are $\rho_0>0$, $\varepsilon_1>0$, $T_1$ and $C$ such that for all $T\ge T_1$, on $B_X(\rho_0)\times(-\varepsilon_1,\varepsilon_1)$:
\begin{enumerate}[label=(F\arabic*)]
\item $\|\Theta_T(0,0)\|_Y\le Ce^{-2\omega T}$;
\item $D_x\Theta_T(0,0)=\mathcal A$, the isomorphism of Theorem~\ref{thm:iso}; in particular it is independent of $T$;
\item $D_\varepsilon\Theta_T(0,0)=0$;
\item $\Theta_T$ is $C^3$, with $\|D^j\Theta_T\|\le C$ for $j\le3$.
\end{enumerate}
\end{lemma}

\begin{proof}
(F1) We have $\Jcal(0,0,0)=0$: at the limiting corner the target is $(\mathcal N,\mathcal A)(0,0)=(0,0)$, and backward transport of zero
incoming data over an unperturbed collar gives zero. Hence $\Theta_T(0,0)=-\Kcal_T(0)$, which is $O(e^{-2\omega T})$ by
Lemma~\ref{lem:neck}(a).

(F2) By the chain rule, $D_x\Jcal$ at the corner in the direction $(\dot H_1,\dot H_2,\dot G,\dot\lambda)$ equals
$D_k\Jcal[\mathcal R(\dot H_1,\dot H_2)]+D_{(\lambda,G)}\Jcal[\dot\lambda,\dot G]$. Since the moments of $\mathcal R(\dot H_1,\dot H_2)$ are
$(\dot H_1,\dot H_2)$, Theorem~\ref{thm:moments} shows that the first term is $(-E\dot H_1,\frac1{2\omega}\Pcal(\dot H_2-\dot H_1))$; the
second is $(e^{2\omega}(\dot\lambda\chib_H+\Ccal\dot G),0)$. Moreover $D\Kcal_T(0)=0$ by Lemma~\ref{lem:neck}(c). Thus
$D_x\Theta_T(0,0)=\mathcal A$.

(F3) $D_\varepsilon\Theta_T(0,0)=D_k\Jcal(0)[k_*]-D\Kcal_T(0)[k_*]$. The first term vanishes by Theorem~\ref{thm:moments}, because $k_*$ has
zero moments, and the second by Lemma~\ref{lem:neck}(c).

(F4) For $\Kcal_T$ this is Lemma~\ref{lem:neck}(b). For $\Jcal$ it holds because the collar is a fixed unit interval and the transport is
triangular and lossless (Remark~\ref{rem:triangular}): the backward solution of \eqref{eq:H4}--\eqref{eq:H5} depends smoothly on the
coefficients (the seed and its optics and torsion, which are smooth in $k$ by Lemma~\ref{lem:Volterra} and \eqref{eq:H3}) and on the final
values, which are polynomial in $(\lambda,dG)$ by Definition~\ref{def:target}. Finally, $(H_1,H_2)\mapsto\mathcal R(H_1,H_2)$ is linear and
bounded.
\end{proof}

\subsection{The contraction}

\begin{proposition}[Uniform contraction]\label{prop:contraction}
There are $C_*$, $\varepsilon_0$ and $T_0$ such that for $|\varepsilon|<\varepsilon_0$ and $T>T_0$ the equation $\Theta_T(x,\varepsilon)=0$
has a unique solution $x_T(\varepsilon)$ in a fixed small ball of $X$. It depends smoothly on $\varepsilon$ and satisfies
\begin{equation}\label{eq:xbound}
 \|x_T(\varepsilon)\|_X\le C_*\big(e^{-2\omega T}+\varepsilon^2\big).
\end{equation}
\end{proposition}

At this stage $\lambda$ may have either sign.

\begin{proof}
Write $\Theta_T(x,\varepsilon)=-\Kcal_T(0)+\mathcal Ax+\mathcal N_T(x,\varepsilon)$; this defines the nonlinearity $\mathcal N_T$. By
(F1)--(F3),
\[
 \mathcal N_T(0,0)=0,\qquad D_x\mathcal N_T(0,0)=0,\qquad D_\varepsilon\mathcal N_T(0,0)=0,
\]
and by (F4) the second derivatives of $\mathcal N_T$ are bounded by $C_2:=C$ on $B_X(\rho_0)\times(-\varepsilon_1,\varepsilon_1)$, uniformly
in $T$. The mean value inequality, applied to $D\mathcal N_T$, gives for $\|x\|,\|y\|\le\rho\le\rho_0$ and $|\varepsilon|\le\varepsilon_1$
\begin{align}
 \|D_x\mathcal N_T(x,\varepsilon)\|&\le C_2\big(\|x\|+|\varepsilon|\big),\label{eq:DN}\\
 \|\mathcal N_T(x,\varepsilon)-\mathcal N_T(y,\varepsilon)\|&\le C_2\big(\rho+|\varepsilon|\big)\|x-y\|,\label{eq:Nlip}\\
 \|\mathcal N_T(0,\varepsilon)\|&\le\tfrac12C_2\,\varepsilon^2 ;\label{eq:N0}
\end{align}
for \eqref{eq:N0}, note that $\varepsilon\mapsto\mathcal N_T(0,\varepsilon)$ vanishes together with its first derivative at $0$ and has second
derivative bounded by $C_2$.

The equation $\Theta_T=0$ is equivalent to the fixed point equation
\[
 x=\Phi(x):=\mathcal A^{-1}\big(\Kcal_T(0)-\mathcal N_T(x,\varepsilon)\big).
\]
Let $a:=\|\mathcal A^{-1}\|$ and
\[
 \rho:=2a\big(Ce^{-2\omega T}+\tfrac12C_2\varepsilon^2\big).
\]
Choose $T_0$ large and $\varepsilon_0$ small such that $\rho\le\rho_0$ and $aC_2(\rho+\varepsilon_0)\le\frac12$; this is possible since
$\rho\to0$ as $T\to\infty$, $\varepsilon\to0$. Then for $\|x\|\le\rho$, by (F1), \eqref{eq:N0} and \eqref{eq:Nlip},
\[
 \|\Phi(x)\|\le a\big(\|\Kcal_T(0)\|+\|\mathcal N_T(0,\varepsilon)\|+\|\mathcal N_T(x,\varepsilon)-\mathcal N_T(0,\varepsilon)\|\big)
 \le a\big(Ce^{-2\omega T}+\tfrac12C_2\varepsilon^2\big)+\tfrac12\rho=\rho,
\]
so $\Phi$ maps the closed ball $\bar B(\rho)$ into itself, and by \eqref{eq:Nlip} $\|\Phi(x)-\Phi(y)\|\le\frac12\|x-y\|$. The Banach fixed
point theorem gives a unique fixed point $x_T(\varepsilon)\in\bar B(\rho)$, and \eqref{eq:xbound} holds with $C_*=2a\max(C,\frac12C_2)$.
Uniqueness in the fixed ball $B(\rho_0)$, where $aC_2(\rho_0+\varepsilon_0)\le\frac12$, follows from the same contraction estimate.

Smooth dependence on $\varepsilon$: at the solution, $D_x\Theta_T=\mathcal A+D_x\mathcal N_T$ with $\|\mathcal A^{-1}D_x\mathcal N_T\|\le\frac12$
by \eqref{eq:DN}, so $D_x\Theta_T$ is invertible by a Neumann series. The implicit function theorem then shows that $\varepsilon\mapsto
x_T(\varepsilon)$ is $C^3$, i.e.\ as smooth as $\Theta_T$.
\end{proof}

All constants are independent of $T$: the linearisation $\mathcal A$ does not involve the neck, the error $\Kcal_T(0)$ decays like
$e^{-2\omega T}$, and the nonlinearity is bounded uniformly in $T$. This allows us to choose first $\varepsilon$ and then $T$ in
Section~\ref{ch:sign}.

\subsection{Regularity}

With the seed in $C^\infty_vH^r_x$ the hierarchy produces $\gd\in H^r$, $\etab\in H^{r-1}$, $\Omega^{-1}\chib\in H^{r-2}$ and
$\Omega^{-2}\ab\in H^{r-4}$ (Remark~\ref{rem:triangular}). The maps $\Jcal$ and $\Kcal_T$ take values exactly in the target space $Y$ of
$\mathcal A$, which gains back two and four derivatives (Theorem~\ref{thm:iso}). No loss of derivatives occurs, so the ordinary contraction
principle suffices. We fix $r=s$ large enough that all Sobolev products used are bounded; $s\ge10$ suffices here, and we take $s\ge14$,
which leaves room for the arguments of Section~\ref{ch:realisation}.

\section{The sign of the boost}\label{ch:sign}

By Section~\ref{ch:fixedpoint} the endpoint equation has a solution $x_T(\varepsilon)=(H_1,H_2,G,\lambda)$ for every small $\varepsilon$ and
large $T$, but $\lambda$ may have either sign. We show that, for a suitable seed direction $k_*$,
\[
 \lambda_T(\varepsilon)=c_*\varepsilon^2+O\big(\varepsilon^3+e^{-2\omega T}\big),\qquad c_*>0 .
\]
The mechanism is that of the heuristic of Section~\ref{sec:toy}: applying the weighted trace that detects the boost to the equations gives an
exact formula for $\lambda$ as a time integral of the scalar functional $\Fcal$. Since $\Fcal$ is critical at the horizon, $\lambda$ is
governed by a second variation, which we compute explicitly and make negative.

\subsection{An exact trace identity}

\begin{proposition}\label{prop:trace}
For any solution of the full problem, i.e.\ once \eqref{eq:endpoint} holds,
\begin{equation}\label{eq:lambdaexact}
 \lambda=-\frac1{4\Phi^2}\Big[\,2\int_{-T}^1e^{2\omega(v-1)}\Fcal\big(\gd(v),\etab(v)\big)\,dv
 +e^{-2\omega(T+1)}\av{w\,\phi^2\,\Omega^{-1}\tr\chib}\big|_{v=-T}\Big].
\end{equation}
\end{proposition}

\begin{proof}
On $[-T,1]$ the non-affinity $\omega$ is constant. Multiply \eqref{eq:H4tr} by $e^{2\omega(v-1)}\phi^2$ and use $D\phi^2=\Omega\tr\chi\,\phi^2$:
\begin{align*}
 D\big(e^{2\omega(v-1)}\phi^2\,\Omega^{-1}\tr\chib\big)
 &=e^{2\omega(v-1)}\phi^2\big(2\omega+\Omega\tr\chi\big)\Omega^{-1}\tr\chib+e^{2\omega(v-1)}\phi^2D(\Omega^{-1}\tr\chib)\\
 &=2e^{2\omega(v-1)}\phi^2\big(\Div\etab+|\etab|^2-K\big).
\end{align*}
Apply $\av{w\,\cdot\,}=\frac1{4\pi}\int w\,\cdot\,d\mu_{\gac}$ and use $\phi^2d\mu_{\gac}=d\mu_{\gd}$; by \eqref{eq:F} the right-hand side
becomes $2e^{2\omega(v-1)}\Fcal(\gd,\etab)$. Integrate from $-T$ to $1$. At $v=1$ we have $\phi=\Phi$ and $\Omega^{-1}\chib=\mathcal N(\lambda,G)$,
so $\av{w\Phi^2\tr\mathcal N}=-4\Phi^2\lambda$ by \eqref{eq:targetmean}. Hence
$-4\Phi^2\lambda-e^{-2\omega(T+1)}\av{w\phi^2\Omega^{-1}\tr\chib}(-T)=2\int_{-T}^1e^{2\omega(v-1)}\Fcal\,dv$.
\end{proof}

Two comments are in order. First, the identity uses the transport equation on all of $[-T,1]$; it therefore needs $\Omega^{-1}\tr\chib$ to be
continuous at $v=0$, which is exactly what \eqref{eq:endpoint} provides. Second, the last term in \eqref{eq:lambdaexact} is
$O(e^{-2\omega T})$, with uniformly bounded seed derivatives, since the incoming data at $v=-T$ are bounded (Lemma~\ref{lem:neck}). Moreover,
$\Fcal$ on $[-T,1]$ depends only on the late seed: the optics and torsion on $[-T,1]$ are obtained by backward integration from the fixed
terminal values, and the early controller only has to match them.

\subsection{First variation}

At $k=0$ we have $\Fcal\equiv\Fcal(\gd_H,\etab_H)=0$, since the integrand is $\frac w2\tr\Ecal(\gd_H,\etab_H)=0$ by \eqref{eq:NH}. For a
first variation along a supported tracefree $k$ we have $\dot\phi=0$, the metric variation is tracefree, and the torsion variation
$-\frac12\Div k$ has zero axial pairing, since $\partial_\psi$ is Killing. Hence $\dot\Fcal(v)=0$ for every $v$, by Lemma~\ref{lem:crit}.
Consequently, along $k=\varepsilon k_*+\mathcal R(x_T(\varepsilon))$,
\[
 \int_{-T}^1e^{2\omega(v-1)}\Fcal\,dv=\varepsilon^2\,\mathcal E_T[k_*]+O\big(\varepsilon^3+\varepsilon\,\|x_T\|+\|x_T\|^2\big),
\]
where $\mathcal E_T[k_*]$, the $\varepsilon^2$-coefficient, is the second variation along $k_*$. By \eqref{eq:xbound} the error is
$O(\varepsilon^3+e^{-2\omega T})$. Thus the correction $\mathcal R(H_1,H_2)$ cannot change the leading coefficient: it is quadratically small and
enters $\Fcal$ only at second order. Inserting this into \eqref{eq:lambdaexact} gives
\begin{equation}\label{eq:lambdaE}
 \lambda_T(\varepsilon)=-\frac{\varepsilon^2}{2\Phi^2}\mathcal E_T[k_*]+O\big(\varepsilon^3+e^{-2\omega T}\big).
\end{equation}

\subsection{The second variation}\label{sec:secondvar}

In the $\gd_H$-orthonormal frame $f_1=\sqrt{s_z/u_z}\,\partial_z$, $f_2=\frac12\sqrt{u_z/s_z}\,\partial_\psi$ we write the seed as
\[
 \gdc=\gdcH^{1/2}\exp(\varepsilon S_c)\gdcH^{1/2},\qquad
 S_c=a(z)\begin{pmatrix}\Re(c(v)e^{im\psi})&-\Im(c(v)e^{im\psi})\\-\Im(c(v)e^{im\psi})&-\Re(c(v)e^{im\psi})\end{pmatrix},
\]
with $c\in C^\infty_c((0,1);\C)$ and a fixed smooth real profile $a$. If $P$ is the coframe matrix of $\gd_H$, i.e.\ $\gd_H=P^tP$, then at
$\phi=\Phi$ the physical metric is $P^te^{\varepsilon S_c}P$, so the first order metric variation is
\[
 k=P^tS_cP ;
\]
the components of $k$ in the frame $(f_1,f_2)$ are exactly $S_c$. In helicity components (Definition~\ref{def:helicity}), $k$ is the sum of
the positive helicity mode $\frac12a\,c\,e^{im\psi}$ and its complex conjugate. Note that $|k|^2=2a^2|c|^2$ is independent of $\psi$. Let
$N:=\av{a^2}=\frac12\int_{-1}^1a^2\,dz$.

The $\varepsilon^2$-coefficient of $\Fcal(v)$ has three contributions: (i) the static one, from the first order metric $k(v)$, its
exponential second order part, and the first order torsion $-\frac12\Div k(v)$; (ii) the axial memory, from the second order torsion; (iii)
the conformal factor, from the second order area radius. There are no others: the first order variations of $\phi$ and of the lapse vanish,
and by Lemma~\ref{lem:crit} the second order torsion enters only through its axial pairing.

\subsubsection{(i) The static contribution}

\begin{proposition}[Static second variation]\label{prop:static}
Let $\gd(\varepsilon)=\gd_H\exp(\varepsilon\gd_H^{-1}k)$ and $\etab(\varepsilon)=\etab_H-\varepsilon\gamma$, $\gamma:=\frac12\Div k$, with $k$
tracefree. The $\varepsilon^2$-coefficient of $\Fcal(\gd(\varepsilon),\etab(\varepsilon))$ is
\begin{equation}\label{eq:Qcov}
 Q=\frac1{4\pi}\int_{S^2}\Big[w|\gamma|^2-2w\,k(\eta_H,\gamma)-\tfrac32\,k(dw,\gamma)+|k|^2\Big(\frac{1-3z^2}{2u_z^2}-\frac{\Delta w}{16}\Big)\Big]d\mu_H ,
\end{equation}
where all quantities are taken with respect to $\gd_H$.
\end{proposition}

\begin{proof}
Write $H=\gd_H^{-1}k$, a tracefree endomorphism. Since $\exp(\varepsilon H)$ is unimodular, $d\mu_{\gd(\varepsilon)}=d\mu_H$ exactly. Also
$\gd(\varepsilon)^{-1}=\exp(-\varepsilon H)\gd_H^{-1}$, and in two dimensions $H^2=\frac12\tr(H^2)I=\frac12|k|^2I$; hence, to second order,
\[
 \gd^{-1}=\gd_H^{-1}-\varepsilon\,k^\sharp+\tfrac{\varepsilon^2}4|k|^2\gd_H^{-1},\qquad k^\sharp:=\gd_H^{-1}k\gd_H^{-1}.
\]
We use the second form of \eqref{eq:F}.

\emph{Torsion terms.} With $\etab=-(\eta_H+\varepsilon\gamma)$, the $\varepsilon^2$-coefficients are
\[
 |\etab|^2:\ \tfrac14|k|^2|\eta_H|^2-2k(\eta_H,\gamma)+|\gamma|^2,\qquad -dw\cdot\etab:\ \tfrac14|k|^2\,dw\cdot\eta_H-k(dw,\gamma).
\]

\emph{Curvature term.} Along $\gd(\varepsilon)$ the derivative $\gd'=\gd(\varepsilon)H$ is tracefree for every $\varepsilon$, so by
Lemma~\ref{lem:dotK} $K'(\varepsilon)=\frac12\Div_\varepsilon\Div_\varepsilon(\gd(\varepsilon)H)$. Differentiating once more at $\varepsilon=0$
gives two pieces: the variation of the operator $\Div\Div$ (through the metric) applied to $k$, and $\Div\Div$ applied to
$\frac{d}{d\varepsilon}(\gd H)=\gd_HH^2=\frac12|k|^2\gd_H$. So $K''(0)=\frac12\delta_k(\Div\Div)(k)+\frac14\Delta|k|^2$, and the
$\varepsilon^2$-coefficient of $\int wK\,d\mu$ is $\frac12\int wK''(0)\,d\mu_H$.

For the first piece, since $d\mu$ does not vary, we have for fixed $T$
\[
 \int w\,\delta_k(\Div\Div)(T)\,d\mu=\delta_k\Big[\int T_{AB}\gd^{AC}\gd^{BD}\Nd^2_{CD}w\,d\mu\Big],
\]
and $\delta(\Nd^2_{CD}w)=-\dot\Gamma^E_{CD}w_E$. With $T=k$,
\[
 \int w\,\delta_k(\Div\Div)(k)=-2\int(k\cdot k)^{CB}\Nd^2_{CB}w-\int k^{CD}\dot\Gamma^E_{CD}w_E
 =-\int|k|^2\Delta w-\int k^{CD}\dot\Gamma^E_{CD}w_E .
\]
Now $k^{CD}\dot\Gamma^E_{CD}=k^{CD}\Nd_Ck_D{}^E-\frac14\Nd^E|k|^2$, and for tracefree symmetric $k$ in two dimensions
$k^{CD}\Nd_Ck_{DE}=\frac12\Nd_E|k|^2-k_E{}^D(\Div k)_D$, as follows by differentiating $k\cdot k=\frac12|k|^2\gd$. Integrating by parts,
$\int k^{CD}\dot\Gamma^E_{CD}w_E=-\frac14\int|k|^2\Delta w-2\int k(dw,\gamma)$. Hence
\begin{align*}
 \int wK''(0)\,d\mu_H&=\tfrac12\Big(-\tfrac34\int|k|^2\Delta w+2\int k(dw,\gamma)\Big)+\tfrac14\int|k|^2\Delta w\\
 &=-\tfrac18\int|k|^2\Delta w+\int k(dw,\gamma),
\end{align*}
and the $\varepsilon^2$-coefficient of $-\int wK\,d\mu$ is $\frac1{16}\int|k|^2\Delta w-\frac12\int k(dw,\gamma)$.

\emph{Sum.} Collecting terms, the coefficient of $|k|^2$ is $\frac14w|\eta_H|^2+\frac14dw\cdot\eta_H+\frac1{16}\Delta w$. By the trace of
\eqref{eq:W2}, $w|\eta_H|^2+dw\cdot\eta_H+\frac12\Delta w=\frac{2(1-3z^2)}{u_z^2}$, so this equals
$\frac{1-3z^2}{2u_z^2}-\frac18\Delta w+\frac1{16}\Delta w$. The coefficient of $k(dw,\gamma)$ is $-1-\frac12=-\frac32$.
\end{proof}

\begin{proposition}[One-dimensional form]\label{prop:Q1d}
For the seed above, the static contribution to the $\varepsilon^2$-coefficient of $\Fcal(v)$ is $Q_m[a]\,|c(v)|^2$, where
\begin{equation}\label{eq:Q1d}
 Q_m[a]:=\Big\langle\frac{s_z}2D_a^2-\frac{s_zz}{u_z}aD_a+\frac{\frac32-4z^2+\frac12z^4}{u_z^2}a^2\Big\rangle,
 \qquad D_a=a'-\Big(\frac{mu_z}{2s_z}+\frac{4z}{s_zu_z}\Big)a .
\end{equation}
\end{proposition}

\begin{proof}
We evaluate each term of \eqref{eq:Qcov} with the helicity calculus and average over $\psi$. Note that
$\frac1{4\pi}\int f\,d\mu_H=\frac1{4\pi}\int_0^{2\pi}\!\int_{-1}^1f\,2\,dz\,d\psi=2\av{\bar f}$, where $\bar f$ is the $\psi$-average. By
Lemma~\ref{lem:blocks}, $\gamma=\frac12\Div k$ has $\gamma_+=\frac14c\,D_a\,e^{im\psi}$ and $\gamma_-=\overline{\gamma_+}$ (the conjugate mode has
$M_m\mapsto-M_m$, which gives the same operator), and $D_a$ is the operator of the same name in Section~\ref{sec:Ehel}.
\begin{itemize}
\item $|\gamma|^2=\frac{2s_z}{u_z}(\gamma_+\gamma_-+\gamma_-\gamma_+)=\frac{s_z}{4u_z}|c|^2D_a^2$, so $2\av{w|\gamma|^2}=\av{\frac{s_z}2D_a^2}|c|^2$.
\item In the complex frame $\mathfrak f=f_1-if_2$ of the proof of Lemma~\ref{lem:blocks}, $\sum_af_a\otimes f_a=\frac12(\mathfrak f\otimes\bar{\mathfrak f}+\bar{\mathfrak f}\otimes\mathfrak f)$
and $k(\mathfrak f,\bar{\mathfrak f})=0$, so for one-forms $X,Y$
\[
 k(X^\sharp,Y^\sharp)=\tfrac14\big[k(\mathfrak f,\mathfrak f)X(\bar{\mathfrak f})Y(\bar{\mathfrak f})+k(\bar{\mathfrak f},\bar{\mathfrak f})X(\mathfrak f)Y(\mathfrak f)\big]
 =\frac{4s_z}{u_z}\big(k_+X_-Y_-+k_-X_+Y_+\big).
\]
Here $k_+=\frac12ac\,e^{im\psi}$, $k_-=\overline{k_+}$, $\gamma_-=\overline{\gamma_+}$, and by Lemma~\ref{lem:blocks} $(\eta_H)_\pm=-\frac1{2(z\mp i)}$ and
$(dw)_\pm=\frac12w'=z$. The products $k_+\gamma_-$ and $k_-\gamma_+$ do not depend on $\psi$, and
\begin{gather*}
 k(\eta_H,\gamma)=\frac{4s_z}{u_z}\cdot2\Re\Big[\frac12ac\cdot\Big(-\frac1{2(z+i)}\Big)\cdot\frac14\bar cD_a\Big]=-\frac{s_zz}{2u_z^2}|c|^2aD_a,\\
 k(dw,\gamma)=\frac{4s_z}{u_z}\cdot2\Re\Big[\frac12ac\cdot z\cdot\frac14\bar cD_a\Big]=\frac{zs_z}{u_z}|c|^2aD_a .
\end{gather*}
Hence
$2\av{-2w\,k(\eta_H,\gamma)-\frac32k(dw,\gamma)}=2\av{\frac{s_zz}{u_z}aD_a-\frac32\frac{s_zz}{u_z}aD_a}|c|^2=-\av{\frac{s_zz}{u_z}aD_a}|c|^2$.
\item $|k|^2=2a^2|c|^2$ and $\Delta_{\gd_H}w=\frac12\Delta_{h_H}w=\frac{2(1-4z^2-z^4)}{u_z^2}$ by \eqref{eq:W3}, so
$\frac{1-3z^2}{2u_z^2}-\frac{\Delta w}{16}=\frac{4(1-3z^2)-(1-4z^2-z^4)}{8u_z^2}=\frac{3-8z^2+z^4}{8u_z^2}$, and
$2\av{|k|^2(\ldots)}=\av{\frac{3-8z^2+z^4}{2u_z^2}a^2}|c|^2$.\qedhere
\end{itemize}
\end{proof}

\paragraph{The value for the chosen profile.} Take $m=-2$ and $a=(1+z)^2$. Then $M_m=\frac{mu_z}{2s_z}=-\frac{u_z}{s_z}$ and
\begin{align*}
 D_a&=2(1+z)+(1+z)^2\frac{u_z^2-4z}{s_zu_z}=\frac{(1+z)\big(2(1-z)u_z+u_z^2-4z\big)}{(1-z)u_z}\\
 &=\frac{(1+z)(1-z)^2(z^2+3)}{(1-z)u_z}=\frac{s_z(z^2+3)}{u_z},
\end{align*}
using $2(1-z)u_z+u_z^2-4z=z^4-2z^3+4z^2-6z+3=(z-1)^2(z^2+3)$. Inserting $a$ and $D_a$, the integrand of \eqref{eq:Q1d} is
$P(z)/(2u_z^2)$ with
\begin{align*}
 P&=s_z^3(z^2+3)^2-2s_z^2z(1+z)^2(z^2+3)+(3-8z^2+z^4)(1+z)^4\\
 &=-z^{10}-2z^9-6z^8-20z^5-14z^4-16z^3-23z^2+6z+12 .
\end{align*}
Dividing by $2u_z^2$ gives the partial fraction decomposition
\[
 -\tfrac12z^6-z^5-2z^4+2z^3+\tfrac92z^2-13z-14+\frac{16z+12}{u_z}+\frac8{u_z^2}.
\]
Odd terms integrate to zero, and $\int_{-1}^1\frac{dz}{u_z}=\frac\pi2$, $\int_{-1}^1\frac{dz}{u_z^2}=\frac12+\frac\pi4$; therefore
\[
 \int_{-1}^1(\ldots)\,dz=-\tfrac17-\tfrac45+3-28+6\pi+4+2\pi=8\pi-\tfrac{768}{35},\qquad Q_{-2}[a]=4\pi-\tfrac{384}{35}.
\]
Also $N=\frac12\int_{-1}^1(1+z)^4dz=\frac{16}5$. Hence
\[
 q_a:=\frac{Q_{-2}[a]}N=\frac{5\pi}4-\frac{24}7=0.49842\ldots<\frac12 .
\]

\subsubsection{(ii) The axial memory}

Write $\ell=\varepsilon^2\ell_2+O(\varepsilon^3)$ (see (iii)) and $\etab=\etab_H+\varepsilon\etab_1+\varepsilon^2\etab_2+O(\varepsilon^3)$. The
coefficient of $\varepsilon^2$ in $\phi^2\etab$ is $\Phi^2(\etab_2+2\ell_2\etab_H)$. Define
\[
 \etab_{2,\mathrm{mem}}:=\etab_2+2\ell_2\etab_H-d\ell_2 ;
\]
the subtracted exact form has zero axial integral. We compute this memory contribution here and the conformal contribution in (iii). By
Lemma~\ref{lem:crit}, $\etab_{2,\mathrm{mem}}$ enters $\Fcal$ only through $-\frac1{4\pi}\int\etab_{2,\mathrm{mem}}(\partial_\psi)\,d\mu_H$.
Pairing \eqref{eq:Zb} with $\partial_\psi$ and integrating over $S^2$ (as in Proposition~\ref{prop:flux}, with the opposite sign because
$\underline Z$ enters with $-\Div$) gives
\[
 D\int\phi^2\etab(\partial_\psi)\,d\mu_{\gac}=\frac14\int\phi^2\tr\big(\gdc^{-1}D\gdc\,\gdc^{-1}\mathcal L_{\partial_\psi}\gdc\big)d\mu_{\gac}
 +\int\big(\phi\,\partial_\psi D\phi-D\phi\,\partial_\psi\phi\big)d\mu_{\gac}.
\]
The last integral vanishes at second order because the second order area radius is axisymmetric (see (iii)). The frame $(f_1,f_2)$ is
$\psi$-invariant, so at second order the trace is $\varepsilon^2\tr(S_c'\,\partial_\psi S_c)$, with $'=D$. For matrices of the form
$S_X=a\begin{pmatrix}\Re X&-\Im X\\-\Im X&-\Re X\end{pmatrix}$ one has $\tr(S_XS_Y)=2a^2\Re(X\bar Y)$; with $X=c'e^{im\psi}$ and
$Y=imce^{im\psi}$,
\[
 \tr(S_c'\,\partial_\psi S_c)=2a^2\Re\big(c'\,\overline{imc}\big)=2a^2m\Im(\bar cc').
\]
Thus $D\int\Phi^2\etab_{2,\mathrm{mem}}(\partial_\psi)\,d\mu_{\gac}=\frac{\Phi^2}4\cdot2m\Im(\bar cc')\int a^2d\mu_{\gac}
=\frac{\Phi^2}2\,m\Im(\bar cc')\cdot4\pi N$. Integrating back from $v=1$, where $\etab=\etab_H$ exactly,
\[
 \int\Phi^2\etab_{2,\mathrm{mem}}(\partial_\psi)\,d\mu_{\gac}(v)=-\frac{\Phi^2}2\,4\pi N\,m\int_v^1\Im(\bar cc')\,dv' .
\]
For this coefficient we use the fixed factor $\Phi^2$. Since $d\mu_H=2\,d\mu_{\gac}$, we obtain
\[
 \int\etab_{2,\mathrm{mem}}(\partial_\psi)\,d\mu_H=\frac2{\Phi^2}\int\Phi^2\etab_{2,\mathrm{mem}}(\partial_\psi)\,d\mu_{\gac}
 =-4\pi Nm\int_v^1\Im(\bar cc')\,dv',
\]
and the contribution to the $\varepsilon^2$-coefficient of $\Fcal(v)$ is
\[
 -\frac1{4\pi}\int\etab_{2,\mathrm{mem}}(\partial_\psi)\,d\mu_H=mN\int_v^1\Im(\bar cc')(v')\,dv' .
\]

\subsubsection{(iii) The conformal factor and the torsion dilution}

Write $\phi=\Phi e^\ell$, where $\ell$ is of second order. The Raychaudhuri equation gives $D^2\ell+(D\ell)^2=2\omega D\ell-e$, and at second
order $e=\frac18\tr\big((\varepsilon S_c')^2\big)=\frac{\varepsilon^2}4a^2|c'|^2$; so, per $\varepsilon^2$,
$D^2\ell=2\omega D\ell-\frac14a^2|c'|^2$. With $\ell(1)=D\ell(1)=0$, as in Lemma~\ref{lem:Volterra},
\[
 \ell(v)=-\frac{a^2}{8\omega}\int_v^1\big(1-e^{-2\omega(v'-v)}\big)|c'(v')|^2dv' .
\]
In particular $\ell$ is axisymmetric. It enters $\Fcal$ in two ways.
\begin{itemize}
\item Through the metric. Under $\gd\mapsto e^{2\ell}\gd$ in two dimensions, $|\etab|^2d\mu$ and $dw\cdot\etab\,d\mu$ are invariant and
$K\,d\mu\mapsto(K-\Delta\ell)\,d\mu$. So the contribution is $\frac1{4\pi}\int w\Delta\ell\,d\mu_H=\frac1{4\pi}\int\ell\,\Delta_{\gd_H}w\,d\mu_H$.
\item Through the torsion (dilution). By \eqref{eq:Zb}, at second order $D\underline Z\supset\Phi^2\Nd D\ell$, so
$\underline Z(v)=\Phi^2\etab_H+\Phi^2d\ell(v)+\ldots$ and $\etab=\underline Z/\phi^2=\etab_H(1-2\ell)+d\ell+\ldots$. The exact part $d\ell$
does not contribute (Lemma~\ref{lem:crit}); the part $-2\ell\etab_H$ contributes
$-\frac1{4\pi}\int(-2\ell\etab_H)(\partial_\psi)\,d\mu_H=-\frac1{4\pi}\int2\ell\,\eta_H(\partial_\psi)\,d\mu_H=-\frac1{4\pi}\int\ell\frac{4s_z}{u_z^2}\,d\mu_H$.
\end{itemize}
Together, with $\Delta_{\gd_H}=\frac12\Delta_{h_H}$ and \eqref{eq:W3},
\[
 \frac1{4\pi}\int\ell\Big(\tfrac12\Delta_{h_H}w-\frac{4s_z}{u_z^2}\Big)d\mu_H=\frac1{4\pi}\int\ell\cdot\tfrac12\Big(\Delta_{h_H}w-\frac{8s_z}{u_z^2}\Big)d\mu_H
 =-\frac2{4\pi}\int\ell\,d\mu_H=-4\av{\ell},
\]
i.e.\ the contribution is $\frac N{2\omega}\int_v^1\big(1-e^{-2\omega(v'-v)}\big)|c'(v')|^2dv'$. Thus the weight identity \eqref{eq:W3} combines
the two optical effects into a single multiple of $\av\ell$, with the favourable sign.

\subsection{The radial quadratic form}

\begin{lemma}\label{lem:radial}
Let $\mathcal E_T[c]$ be the $\varepsilon^2$-coefficient of $\int_{-T}^1e^{2\omega(v-1)}\Fcal\,dv$, with $c$ extended by zero to the neck.
Then, exactly for every $T\ge0$,
\begin{equation}\label{eq:ET}
 \frac{\mathcal E_T[c]}N=\int_0^1e^{2\omega(t-1)}\Big(q_a|c|^2+\frac m{2\omega}\big(1-e^{-2\omega(t+T)}\big)\Im(\bar cc')
 +\frac1{8\omega^2}\big(1-e^{-2\omega(t+T)}\big)^2|c'|^2\Big)dt .
\end{equation}
In particular $\mathcal E_T\to\mathcal E_\infty$ with error $O(e^{-2\omega T})$, uniformly on the fixed seed set, where
\begin{equation}\label{eq:Einf}
 \frac{\mathcal E_\infty[c]}N=\int_0^1e^{2\omega(v-1)}\Big(q_a|c|^2+\frac m{2\omega}\Im(\bar cc')+\frac{|c'|^2}{8\omega^2}\Big)dv .
\end{equation}
\end{lemma}

\begin{proof}
By (i)--(iii), the $\varepsilon^2$-coefficient of $\Fcal(v)$ is
$N\big[q_a|c(v)|^2+m\int_v^1\Im(\bar cc')+\frac1{2\omega}\int_v^1(1-e^{-2\omega(t-v)})|c'(t)|^2dt\big]$. On the neck ($v<0$) the static term
vanishes and the other two are given by the same formulas (with $c=0$ there), so the neck is included exactly. Integrate against
$e^{2\omega(v-1)}$ over $[-T,1]$ and exchange the order of integration ($v\le t$), using
\begin{align*}
 \int_{-T}^te^{2\omega(v-1)}dv&=\frac{e^{2\omega(t-1)}}{2\omega}\big(1-e^{-2\omega(t+T)}\big),\\
 \int_{-T}^te^{2\omega(v-1)}\big(1-e^{-2\omega(t-v)}\big)dv&=\frac{e^{2\omega(t-1)}}{4\omega}\big(1-e^{-2\omega(t+T)}\big)^2 .
\end{align*}
For the second identity, the two pieces are $\frac{e^{2\omega(t-1)}}{2\omega}(1-X)$ and $\frac{e^{2\omega(t-1)}}{4\omega}(1-X^2)$ with
$X=e^{-2\omega(t+T)}$, and $2(1-X)-(1-X^2)=(1-X)^2$. Letting $T\to\infty$ replaces $1-e^{-2\omega(t+T)}$ by $1$, with error
$O(e^{-2\omega T})$ since $c$ is supported in $(0,1)$.
\end{proof}

\subsection{Completing the square and choosing the seed}

Substitute $y=2\omega v$ and $c(v)=e^{(-\frac12-im)y}b_0(y)$. Then
\begin{equation}\label{eq:energy}
 \frac{\mathcal E_\infty[c]}N=\frac{e^{-2\omega}}{2\omega}\int_0^{2\omega}\Big(\tfrac12|b_0'|^2+\big(q_a-\tfrac{m^2}2+\tfrac18\big)|b_0|^2\Big)dy .
\end{equation}
Indeed, with $'=d/dy$ we have $\frac{dc}{dv}=2\omega\frac{dc}{dy}=2\omega e^{(-\frac12-im)y}\big(b_0'-(\tfrac12+im)b_0\big)$ and
$|c|^2=e^{-y}|b_0|^2$, hence
\begin{align*}
 \Im\big(\bar c\tfrac{dc}{dv}\big)&=2\omega e^{-y}\big(\Im(\bar b_0b_0')-m|b_0|^2\big),\\
 \big|\tfrac{dc}{dv}\big|^2&=4\omega^2e^{-y}\big(|b_0'|^2-\Re(\bar b_0'b_0)-2m\Im(\bar b_0b_0')+(\tfrac14+m^2)|b_0|^2\big).
\end{align*}
Since $e^{2\omega(v-1)}=e^{y-2\omega}$, the integrand of \eqref{eq:Einf} becomes
\[
 e^{-2\omega}\Big[q_a|b_0|^2+m\Im(\bar b_0b_0')-m^2|b_0|^2+\tfrac12|b_0'|^2-\tfrac12\Re(\bar b_0'b_0)-m\Im(\bar b_0b_0')+(\tfrac18+\tfrac{m^2}2)|b_0|^2\Big];
\]
the $\Im$ terms cancel, $\Re(\bar b_0'b_0)=\frac12\frac d{dy}|b_0|^2$ integrates to zero, and $dv=dy/(2\omega)$.

Now take $m=-2$ and $a=(1+z)^2$. Since $\pi<\frac{22}7$,
\[
 q_a-\tfrac{m^2}2+\tfrac18=\tfrac{5\pi}4-\tfrac{24}7-\tfrac{15}8<\tfrac{55}{14}-\tfrac{48}{14}-\tfrac{15}8=-\tfrac{11}8 .
\]
Take $2\omega=16$ and three smooth nonnegative bumps $b^{(1)},b^{(2)},b^{(3)}$ with disjoint supports of length $4$ in $(0,16)$, each with
$\int|b'|^2<2\int|b|^2$. Such bumps exist: $\cos^4(\pi y/4)$ on an interval of length $4$ has quotient
\[
 \frac{\int|b'|^2}{\int|b|^2}=\frac{\pi^2\cdot\frac4\pi\int_{-\pi/2}^{\pi/2}\cos^6x\sin^2x\,dx}{\frac4\pi\int_{-\pi/2}^{\pi/2}\cos^8x\,dx}
 =\pi^2\,\frac{\frac{5\pi}{16}-\frac{35\pi}{128}}{\frac{35\pi}{128}}=\frac{\pi^2}7\approx1.41,
\]
and smoothing its $C^3$ endpoints preserves the strict inequality. On the complex span of the bumps, by disjointness of the supports,
$\frac12\int|b_0'|^2<\int|b_0|^2$ for $b_0\ne0$, so \eqref{eq:energy} is $<\frac{e^{-2\omega}}{2\omega}(1-\frac{11}8)\int|b_0|^2<0$: the form is
negative definite on a three-dimensional complex space.

The moment conditions $H_1=H_2=0$ for $k=P^tS_cP$ are two complex linear conditions on $c$,
\[
 \int_0^1e^{2\omega v}c\,dv=\int_0^1e^{4\omega v}c\,dv=0,
\]
since the angular factor ($a(z)$ times the helicity tensor) is common. Hence some nonzero $c_*$ in the span satisfies them. Let $k_*$ be the
corresponding seed. It is smooth (for $m=-2$ the factor $(1+z)^2$ gives smoothness at the poles by \eqref{eq:poles}, with exponents
$|m+2|=0$ and $|m-2|=4$), $\mathcal R_\pi$-invariant ($m$ is even), moment-null in every angular component, and
\[
 \mathcal E_\infty[k_*]<0 .
\]

\subsection{Conclusion}

Inserting this into \eqref{eq:lambdaE}, with $\mathcal E_T[k_*]=\mathcal E_\infty[k_*]+O(e^{-2\omega T})$, we obtain
\begin{equation}\label{eq:lambdasign}
 \lambda_T(\varepsilon)=c_*\varepsilon^2+O\big(\varepsilon^3+e^{-2\omega T}\big),\qquad c_*=-\frac{\mathcal E_\infty[k_*]}{2\Phi^2}>0 .
\end{equation}
Choose $\varepsilon\ne0$ so small that the $O(\varepsilon^3)$ term is below $\frac14c_*\varepsilon^2$, and then $T$ so large that the
$O(e^{-2\omega T})$ term is below $\frac14c_*\varepsilon^2$. Then $\lambda_T(\varepsilon)>0$.

\begin{remark}[Order of choices; the constant term]
The order is: first $\varepsilon$, then $T$; the exponentially small error from the early bridge must be beaten by the quadratic gain. As for
the constant term, $\Fcal(\gd_H,\etab_H)=0$, because its integrand is $w\cdot\frac12\tr\Ecal(\gd_H,\etab_H)$, which vanishes pointwise by
\eqref{eq:NH}. So the expansion of $\int e^{2\omega(v-1)}\Fcal$ starts at order $\varepsilon^2$, and $\lambda_T(0)=O(e^{-2\omega T})$.
\end{remark}

\section{From matched data to a spacetime}\label{ch:realisation}

By Section~\ref{ch:sign} we have, for $\lambda>0$, a solution of the null constraint hierarchy on $[v_1,v_2]$ (with $v_2=1$ in the
normalisation of Sections~\ref{ch:neck}--\ref{ch:sign}) whose data at $v_2$ are those of an actual extremal Kerr horizon cut
$S_2=\{v_K=F\}$, $F=G/\lambda$, in the frame $L=\xi/\lambda$, $\Lb'=\lambda n_F$. No smallness of $F$ is needed, since the cut formulas of
Proposition~\ref{prop:cut} are exact; in fact $F$ is large, because $\lambda$ is small. In this section we list what has been matched and
complete the remaining gauge entries of the sphere data (Section~\ref{sec:gauge}), continue the solution along the extremal Kerr horizon from
the cut $S_2$ to a horizon sphere $S_3$ (Section~\ref{sec:horizonsphere}), show that matched sphere data give $C^2$ pasting
(Section~\ref{sec:twojet}), and realise the data in a vacuum spacetime (Section~\ref{sec:luk}). This completes the proofs of
Theorem~\ref{thm:A} and Theorem~\ref{thm:cutversion} (Section~\ref{sec:proofA}).

\subsection{What has been matched, and the gauge completion}\label{sec:gauge}

We write $v_K$ for the ingoing Kerr coordinate and $v$ for the parameter along the bridge.

\paragraph{What has been matched at $S_2$.}
\begin{itemize}
\item $\gd$, $\Omega\chi$ ($=0$), $-\etab$ ($=\eta_H$) and $\omega$ ($=\omega_0$) match by construction of the collar. More is true: on a whole
neighbourhood of $S_2$ the seed is constant and $\phi\equiv\Phi$, so all outgoing quantities and all their $D$-derivatives coincide with those
of the Kerr horizon, parametrised with the same constant non-affinity (see \eqref{eq:collarparam} below). In particular $\alpha=0$, and the
higher-order outgoing data $x^{\rm high}$ of \cite[Definition~2.2]{KU} match. This does not say that all higher transversal derivatives
match, and nothing below needs that.
\item $\Omega^{-1}\chib$ and $\Omega^{-2}\ab$ match by the endpoint equation \eqref{eq:endpoint}. Indeed, on the collar the incoming fields of
the bridge are, by definition of $\Jcal$, the backward transport of the target data $(\mathcal N(\lambda,G),\mathcal A(\lambda,G))$ of
Definition~\ref{def:target} from $v=1$, and near $v=1$ the seed vanishes. Hence on a neighbourhood of $S_2$ the incoming fields of the bridge
and of Kerr solve the same linear transport equations with the same data at $S_2$, and agree.
\item $\rho,\sigma,\beta,\bb$ match by the Gauss, curl and Codazzi equations.
\end{itemize}

\paragraph{The parametrisation of the collar.} On the collar near $v=1$ the generator field $L=\partial_v$ of the bridge is a positive
multiple of the Killing field, $L=h\xi$. Its non-affinity is $2\omega_0$, and $\nabla_LL=h\xi(h)\xi+h^2\nabla_\xi\xi=L(\log h)\,L$, so
$\partial_v\log h=2\omega_0$. With $h(1)=\lambda^{-1}$ (the frame at $S_2$ is $\xi/\lambda$), and since the cuts $\{v=\text{const}\}$ are the
translates $\{v_K=F+\text{const}\}$ (they all have torsion coefficient $\etab_H$, and the foliation of the bridge is determined by $L$ and
$S_2$), we obtain
\begin{equation}\label{eq:collarparam}
 L=h(v)\,\xi,\quad h(v)=\lambda^{-1}e^{2\omega_0(v-1)},\qquad v_K=F(x)+g(v),\quad g(v)=\frac{e^{2\omega_0(v-1)}-1}{2\omega_0\lambda}.
\end{equation}
Since $h$ does not depend on the angles, the coefficient $\etab$ of every cut of the collar is $\etab_H$ (Proposition~\ref{prop:cut}(b)), and
$\Omega^{-1}\chib=h^{-1}(\chib_H+\Ccal F)$, which indeed solves $D(\Omega^{-1}\chib)=-2\omega_0\Omega^{-1}\chib$.

\paragraph{The remaining entries of the sphere data.} In \eqref{eq:x} the geometric quantities above determine everything except the
lapse $\Omega$ itself at the sphere and the transversal lapse jets $\omb=\Db\log\Omega$ and $\Db\omb$; since $\eta=-\etab+2\Nd\log\Omega$, also
$\eta$ is determined once $\Omega$ is matched. On the bridge the lapse is part of the construction along $\Hcal$, and $\omb$, $\Db\omb$ are determined
along $\Hcal$ by transport from the Kerr germ at $S_1$, by the transport equations for the lapse jets in double null gauge
\cite[Appendix~A]{KU}; we write $\Omega_b$, $\omb_b$, $(\Db\omb)_b$ for their
values at $S_2$. On the Kerr side they are gauge. The ingoing null hypersurface $\underline{\mathcal C}$ of Kerr through $S_2$, generated by the
null geodesics with initial velocity $\Lb'=\lambda n_F$, is fixed, and so is the geodesic extension of $\Lb'$ along it; but the parameter $u$
along its generators is free, subject to $u=0$ on $S_2$ and $\partial_u$ future directed. Given such a parameter, the lapse is defined by
$\Lb=\partial_u=\Omega^2\Lb'$.

Let $u$ be one such parameter, with lapse $\Omega$, and let $\tilde u$ be another, $u=F(\tilde u,x)$ with $F(0,x)=0$ and $\partial_{\tilde u}F>0$.
Along the generators $\partial_{\tilde u}=\partial_{\tilde u}F\,\partial_u$, while $\Lb'$ is unchanged, so $\tilde\Omega^2=\partial_{\tilde u}F\cdot\Omega^2\circ F$,
i.e.\ $\log\tilde\Omega=\frac12\log\partial_{\tilde u}F+\log\Omega\circ F$. Differentiating once and twice in $\tilde u$ and evaluating at $\tilde u=0$ gives,
with $F_j(x):=\partial_{\tilde u}^jF(0,x)$,
\begin{equation}\label{eq:lapsejets}
 \tilde\Omega^2=F_1\Omega^2,\qquad \tilde\omb=\frac{F_2}{2F_1}+F_1\omb,\qquad
 \Db\tilde\omb=\frac12\Big(\frac{F_3}{F_1}-\frac{F_2^2}{F_1^2}\Big)+F_2\,\omb+F_1^2\,\Db\omb .
\end{equation}
These equations can be solved successively for the jets of $F$ that give the values of the bridge:
\[
 F_1=\frac{\Omega_b^2}{\Omega^2}>0,\qquad F_2=2F_1\big(\omb_b-F_1\omb\big),\qquad
 F_3=\frac{F_2^2}{F_1}+2F_1\big((\Db\omb)_b-F_2\,\omb-F_1^2\,\Db\omb\big).
\]
We take $F(\tilde u,x):=F_1\tilde u+\frac12F_2\tilde u^2+\frac16F_3\tilde u^3$ on a short segment $0\le\tilde u\le\delta$ of $\underline{\mathcal C}$, with
$\delta>0$ so small that $\partial_{\tilde u}F>0$ there.

\paragraph{Which Kerr sphere data are used.}
The normalised construction of \cite[Lemma~A.1]{KU} sets the lapse derivatives to zero; this is not the collar gauge, whose outgoing
coefficient is $\omega_0>0$. We therefore define $x_{S_2}$ by applying the inverse constant boost to the Kerr data in the collar gauge
\eqref{eq:collarparam}, with the parameter $\tilde u$ just chosen along $\underline{\mathcal C}$. At the sphere its outgoing normal is $\xi$, and its extension has
$\omega=\lambda\omega_0$. Consequently
\[
 x(v_2)=\mathfrak b_\lambda(x_{S_2})
\]
with exactly the boost convention of \cite[Definition~2.3]{KU}. (A constant boost of the normalised data with $\omega=0$ could not have this
property.)

\paragraph{Agreement on the collar.} Near $S_2$ we use, for Kerr, the double null gauge determined by the parameter $v$ of
\eqref{eq:collarparam} along $\Hcal$ and the parameter $\tilde u$ just constructed along $\underline{\mathcal C}$. In this gauge the Kerr data on the
collar satisfy the same transport equations along $\Hcal$ as the data of the bridge, namely the hierarchy of Proposition~\ref{prop:hier}
together with the transport equations for the lapse jets \cite[Appendix~A]{KU}. The outgoing quantities agree on the collar (first item
above), and all remaining entries agree at $S_2$ by construction. Uniqueness for these linear transport equations gives agreement of all
entries of \eqref{eq:x} on a neighbourhood of $S_2$ in $\Hcal$. This proves Theorem~\ref{thm:cutversion} at the level of null data.

\subsection{From the cut to a horizon sphere}\label{sec:horizonsphere}

\begin{proposition}[Continuation along the horizon]\label{prop:horizonsphere}
Let the bridge on $[v_1,1]$ be as above, with $\lambda>0$. For every sufficiently large constant $\kappa$ there are $V_3>1$ and an extension
of the solution of the null constraint hierarchy to $[v_1,V_3]$ such that:
\begin{enumerate}[label=(\alph*)]
\item on $[1,V_3]$ the null hypersurface is the extremal Kerr horizon $\mathcal H_f$, and all data are induced by the Kerr spacetime;
\item the extension is smooth across $v=1$ (it coincides with the bridge on a neighbourhood of $v=1$);
\item near $V_3$ the cuts are the horizon spheres $\{v_K=V+\kappa\}$ and the generator field is $L=\xi$; in particular the final sphere is
$S_3=\{v_K=V_3+\kappa,\ r=M_f\}$.
\end{enumerate}
\end{proposition}

\begin{proof}
\emph{The parametrisation.} Extend $g$ by the same formula to $V\ge1$, choose $\delta_1>0$ and $V_3>1+2\delta_1$, and a smooth step $\chi(V)$
with $\chi=0$ for $V\le1+\delta_1$, $\chi=1$ for $V\ge V_3-\delta_1$, and $\chi'\ge0$. Put
\[
 \Psi(V,x):=(1-\chi(V))\big(F(x)+g(V)\big)+\chi(V)\,(V+\kappa),\qquad V\in[1,V_3].
\]
Then
\[
 \partial_V\Psi=(1-\chi)g'(V)+\chi+\chi'(V)\big(V+\kappa-F(x)-g(V)\big).
\]
The first two terms form a convex combination of the positive numbers $g'(V)=\lambda^{-1}e^{2\omega_0(V-1)}$ and $1$, and the last term is
nonnegative as soon as $\kappa\ge\max_xF+\max_{[1,V_3]}g$. For such $\kappa$ we have $\partial_V\Psi>0$ everywhere.

\emph{The data.} For each $V$ let $S'_V=\{v_K=\Psi(V,x),\ r=M_f\}$. These are cuts of $\mathcal H_f$ (graphs over the sphere of generators),
hence spacelike: their tangent vectors $E_A=\partial_A+\partial_A\Psi\,\xi$ satisfy $\bg(E_A,E_B)=\bg(\partial_A,\partial_B)$, which is
positive definite. Since $\partial_V\Psi>0$, they foliate the segment of $\mathcal H_f$ between $S'_1=S_2$ and $S'_{V_3}$. The generator field of
this foliation, with the generators labelled by $x$, is $L=\partial_V\Psi\,\xi$, which is future directed. Define the data on $[1,V_3]$ as the
data induced by the Kerr metric on the cuts $S'_V$ in the frame $(L,\Lb')$, $\bg(L,\Lb')=-2$. Since Kerr is a vacuum solution and the
angular coordinates are constant along the generators, these data solve the hierarchy of Proposition~\ref{prop:hier}, whose derivation used
nothing else. This proves (a).

\emph{(b).} For $V\le1+\delta_1$ we have $\Psi=F+g$, so the cuts and the generator field are exactly those of \eqref{eq:collarparam}: the
extension describes the same piece of the Kerr horizon with the same foliation and parametrisation as the collar of the bridge near $v=1$.
On a neighbourhood of $v=1$ the data of the bridge coincide with those of Kerr (Section~\ref{sec:gauge}). Hence the two coincide on an open
interval, and the concatenation is smooth.

\emph{(c).} For $V\ge V_3-\delta_1$ we have $\Psi=V+\kappa$: the cuts are $\{v_K=\text{const}\}$ and $L=\xi$.
\end{proof}

Along the continuation the lapse is determined, up to a fixed function of the angles, by the non-affinity, $2\omega=\partial_V\log\partial_V\Psi$.
At $S_3$ the remaining gauge entries $\Omega$, $\omb$, $\Db\omb$ are completed as in Section~\ref{sec:gauge}, by choosing the parametrisation
of the ingoing Kerr null hypersurface through $S_3$. The boost $\lambda$ and the shape $F$ of the cut no longer appear in the final sphere:
they have been absorbed into the parametrisation $\Psi$. This is \cite[Remark~1.1]{KU}: any horizon section can be
connected to a Kerr coordinate sphere along the horizon.

The positivity $\lambda>0$ is used here in an essential way: it makes $h=\lambda^{-1}e^{2\omega_0(v-1)}$ positive, so that the generator field
of the bridge at $S_2$ is a future directed multiple of the Killing field and $F=G/\lambda$ is a genuine function. For $\lambda\le0$ there is
no Kerr horizon along which to continue.

\subsection{Matched sphere data and \texorpdfstring{$C^2$}{C2} pasting}\label{sec:twojet}

In double null gauge, sphere data determine all derivatives of the metric up to order two on the sphere: the corresponding tuple of
Ricci coefficients, curvature components and their $D$- and $\Db$-derivatives reduces to \eqref{eq:x} by the null structure equations
\cite[Section~1.5 and Remark~(2) after Definition~2.4]{ACR}; see also \cite[Section~2.1]{KU}. For the convenience of the reader we indicate
the argument along a segment of $\Hcal=\{u=u_0\}$, in the double null gauge \eqref{eq:metric} with $b=0$ on $\Hcal$.

The metric is determined by $(\Omega,\gd,b)$: $\bg_{uu}=\bg_{uA}=0$, $\bg_{uv}=-2\Omega^2$, $\bg_{AB}=\gd_{AB}$, $\bg_{vA}=-\gd_{AB}b^B$,
$\bg_{vv}=\gd_{AB}b^Ab^B$. Since the $\theta^A$ are transported along $\Lb=\partial_u$, we have $[\partial_u,\partial_A]=0$, and $\Db$ acts as
$\partial_u$ on components. On $\Hcal$,
\[
 \partial_u\Omega=\Omega\,\omb,\qquad \partial_u\gd_{AB}=2\Omega\chib_{AB},\qquad \partial_ub^A=4\Omega^2\zeta^A,\quad \zeta=\eta-\Nd\log\Omega .
\]
The first two identities are the definitions of $\omb$ and $\chib$. For the third, $[\partial_u,L]=(\partial_ub^A)\partial_A$, and
$\bg(\nabla_{\Lb}L,\partial_B)=-\bg(L,\nabla_{\partial_B}\Lb)=2\Omega^2\eta_B$, $\bg(\nabla_L\Lb,\partial_B)=-\bg(\Lb,\nabla_{\partial_B}L)=2\Omega^2\etab_B$,
with $\eta-\etab=2\zeta$. Differentiating once more, $\partial_u^2\Omega=\Omega(\omb^2+\Db\omb)$, while $\partial_u^2\gd_{AB}$ and $\partial_u^2b^A$
involve $\Db\chib$ and $\Db\zeta=\Db\eta-\Nd\omb$. By the incoming null structure equations \cite[Section~2.2]{ACR}, $\Db\chib$ is an
algebraic expression in $\Omega$, $\omb$, $\gd$, $\chib$ and $\ab$, and $\Db\eta$ one in $\Omega$, $\chib$, $\eta$, $\etab$ and $\bb$, which is
given by the Codazzi equation \eqref{eq:Cod}. Tangential and mixed derivatives are $\partial_v$- and $\partial_A$-derivatives of quantities
given on the whole segment.

Hence two solutions in double null gauge with $\Hcal=\{u=u_0\}$, the same parameter $v$ and generator labels $\theta^A$ on $\Hcal$, and the
same sphere data on an open segment of $\Hcal$, have the same metric two-jet there in the coordinates $(u,v,\theta^A)$; agreement of
higher transversal jets is not needed. The gauge entries $\Omega$, $\omb$, $\Db\omb$ of \eqref{eq:x} enter through the lapse; this is why
they are matched in Section~\ref{sec:gauge}.

\subsection{Realisation in a spacetime}\label{sec:luk}

The constructed null data have positive definite $\gd$ on a compact cylinder and agree near $S_1$ with the actual Kerr data. We take a short
Kerr germ as the transversal null hypersurface at $S_1$. Luk's local existence theorem \cite{Luk} gives a vacuum solution in a neighbourhood of
the whole cylinder (of possibly very small width). It induces the constructed data on the cylinder, in particular $x_{S_3}$ at $S_3$ and
$\mathfrak b_\lambda(x_{S_2})$ at $S_2$. Nothing here identifies a global event horizon; this is done in Section~\ref{ch:thirdlaw}.

\subsubsection{The precise regularity input}\label{sec:regularityinput}
We use characteristic local existence in the following form. Two compact intersecting null segments with positive definite section
metrics, constraint-satisfying data and compatible corner values admit a unique double null vacuum development on a sufficiently thin
neighbourhood of both segments. For data with a fixed large finite number of derivatives, the development retains the corresponding
finite-order regularity on each side; in particular, sufficiently high finite Sobolev regularity gives a one-sided $C^4$ metric up to either
initial segment. The width is positive for each fixed data set; no bound uniform in the final mass, the neck length or the inverse boost is
required.

For smooth data this is \cite[Theorem~1]{Luk}. As Luk remarks after the statement of that theorem, the width of the neighbourhood depends
only on an $H^5$ norm of the data, and $H^5$ data give an $H^4$ solution by approximation. The higher one-sided regularity for data with
finitely many derivatives follows from the a priori estimates of \cite[Section~5]{Luk} by the standard persistence of regularity argument,
which involves only finitely many commuted derivatives. Characteristic local existence is used in the same way in the vacuum construction
of Kehle and Unger \cite[proof of Corollary~1]{KU}, whose characteristic data likewise have only finite regularity \cite[Remark~1.4]{KU}.

\subsubsection{Allocation and corner compatibility}\label{sec:allocation}
For a requested index $s\ge14$ we run the null construction at $S=\max(s,30)$, retaining thirty radial derivatives of the prescribed seed
profiles. All controls have smooth fixed radial profiles, so this is a finite choice of Banach norms, not a demand for uniform bounds at all
orders. The hierarchy gives
\[
 \gdc,\ \phi,\ D\phi\in H^S,\qquad \etab\in H^{S-1},\qquad
 \Omega^{-1}\chib\in H^{S-2},\qquad \Omega^{-2}\ab\in H^{S-4}.
\]
The mixed derivatives needed in the characteristic estimates follow by differentiating the same transport equations; the regularity reserve is
much larger than what is needed for a $C^4$ metric. 

\emph{Corner compatibility.} In the characteristic initial value problem the data on the two intersecting null segments must be compatible
at the corner sphere: they must induce the same metric, connection coefficients and curvature components there, including the transversal
derivatives that can be computed from either side. At $S_1$ this holds automatically. The controls vanish on a neighbourhood of $S_1$, so
there the outgoing data are those of the outgoing Kerr cone through $S_1$, and the transversal segment is a short segment of the ingoing
Kerr cone through $S_1$. Both data sets are therefore restrictions of one and the same vacuum spacetime, in one and the same double null gauge,
and every compatibility condition, of any order, is satisfied. The same argument applies at the other end of the bridge, where the data near
$S_2$ (and near $S_3$) are restrictions of the Kerr spacetime in the gauge of Section~\ref{sec:gauge}.

 At the other end $F=G/\lambda\in H^S$ for the fixed $\lambda>0$, so the target Kerr branch and its chosen
gauge also have the required finite regularity. Large norms may shrink the strip, but do not prevent its existence. Taking the mass threshold
at this $S$ proves the claimed result at index $s$ as well.

Each individual side is thus $C^4$ up to the gluing segment. Across it we claim only $C^2$, from the matching two-jet of
Section~\ref{sec:twojet}. In particular, no matching of third transversal derivatives, and no globally $C^4$ glued spacetime, is assumed in
Section~\ref{ch:thirdlaw}.

\subsection{Proof of Theorem~\ref{thm:A} and Theorem~\ref{thm:cutversion}}\label{sec:proofA}

\begin{proof}[Proof of Theorem~\ref{thm:A} and Theorem~\ref{thm:cutversion}]
We collect the steps.
\begin{enumerate}[label=(\arabic*)]
\item The hierarchy (Proposition~\ref{prop:hier}) gives triangular, lossless transport along the null hypersurface.
\item The partial bridge (Proposition~\ref{prop:partial}) and the controller (Theorem~\ref{thm:ctrl}) attain every nearby triple
$(\phi,D\phi,\etab)$ at the start of the neck, keeping $S_1$.
\item Along the neck (Lemma~\ref{lem:neck}) the incoming errors are $O(e^{-2\omega T})$, $D\Kcal_T(0)=0$, and the $C^3$ bounds are uniform.
\item The collar (Theorem~\ref{thm:moments}), the Teukolsky--Starobinsky factorisation (Theorem~\ref{thm:Pinv}) and the cut and boost
(Corollary~\ref{cor:range}) give an isomorphism $\mathcal A$ independent of $T$ (Theorem~\ref{thm:iso}).
\item The contraction (Proposition~\ref{prop:contraction}) gives exact matching in the extended chart $(\lambda,G)$.
\item The trace identity (Proposition~\ref{prop:trace}), criticality and the negative second variation (Section~\ref{ch:sign}) give $\lambda>0$,
so $S_2$ is an actual horizon cut.
\item The gauge completion (Section~\ref{sec:gauge}) gives $x(v_2)=\mathfrak b_\lambda(x_{S_2})$, which is Theorem~\ref{thm:cutversion} at the
level of null data; the continuation along the horizon (Proposition~\ref{prop:horizonsphere}) leads to the horizon sphere $S_3$ with Killing
generator. After the translation $v_K\mapsto v_K-\kappa$, an isometry of Kerr preserving $\mathcal H_f$ and $\xi$,
Proposition~\ref{prop:horizonsphere}(c) gives $v_K=v$ near $v_3=V_3$, so $S_3=\{v_K=v_3,\ r=M_f\}$.
\item The two-jet argument of Section~\ref{sec:twojet} and Luk's theorem with persistence of regularity (Section~\ref{sec:luk}) give the $C^2$ characteristic
gluing of the Boyer--Lindquist sphere $S_1$ to the extremal Kerr horizon sphere $S_3$ (to the cut $S_2$, respectively), with one-sided $C^4$
regularity.
\end{enumerate}
The regularity statements, including the choice of the index $S=\max(s,30)$, are those of Section~\ref{sec:allocation}.
\end{proof}

\begin{remark}[The output used in Section~\ref{ch:thirdlaw}]\label{rem:outputforKU}
For $a_1=0$, the solution of Theorem~\ref{thm:A} has the following three properties, which are those used in Section~\ref{ch:thirdlaw}.
\begin{enumerate}[label=(\alph*)]
\item \emph{Exact ends.} Near $S_1$ the data agree with those of the outgoing Schwarzschild cone through $S_1$ (Theorem~\ref{thm:A}(i)), and
the transversal branch at $S_1$ is a short branch of the Schwarzschild spacetime itself (Section~\ref{sec:luk} and
Section~\ref{sec:allocation}), so the transversal data agree as well. Near the terminal section $S_3$ the data are exact extremal Kerr horizon
data with $L=\xi$ and cuts $\{v_K=\text{const}\}$ (Proposition~\ref{prop:horizonsphere}(a),(c)); in particular they extend beyond $S_3$ by the
exact extremal Kerr horizon data.
\item \emph{Full $C^2$ sphere-data matching.} At the interfaces $S_1$ and $S_3$ (and $S_2$) all entries of \eqref{eq:x} match, in the sense of
\cite[Section~2.1]{KU}, including the gauge entries $\Omega$, $\omb$, $\Db\omb$ (Section~\ref{sec:gauge} and the paragraph following
Proposition~\ref{prop:horizonsphere}). By Section~\ref{sec:twojet} the metric two-jets of the two spacetimes agree there.
\item \emph{Regularity reserve.} On every compact segment the development is a one-sided $C^4$ spacetime metric up to the gluing hypersurface,
on each side (Section~\ref{sec:regularityinput} and Section~\ref{sec:allocation}); across the interfaces only $C^2$ is claimed.
\end{enumerate}
\end{remark}

\section{The third law of black hole thermodynamics in vacuum}\label{ch:thirdlaw}

In this section we deduce Theorem~\ref{thm:B} from Theorem~\ref{thm:A}. The logical structure is a single implication: Theorem~\ref{thm:A},
in the form of Proposition~\ref{tl:prop:input}, implies Theorem~\ref{thm:thirdlaw}, the precise form of Theorem~\ref{thm:B}. The argument is the
two-stage assembly of Kehle and Unger, which they proposed for exactly this purpose \cite[Section~1.2]{KU}. We use their construction as a
template and indicate precisely what has to be adjusted because the endpoint is extremal and because the second Cauchy extension takes place
around a non-flat background. General results (characteristic local existence, Cauchy evolution, maximal developments, causal smoothing,
low-regularity focusing) are used by reference.

The argument is an adaptation of published proofs, not a citation of a published theorem. The vacuum collapse result \cite[Corollary~1]{KU}
assumes small rotation $|a|/M\ll1$, and we do not apply it at $a=M$. Instead we rerun its proof, and that of \cite[Theorem~1]{KU22}, with the
second gluing step replaced by Proposition~\ref{tl:prop:input}, and we check every place where small rotation or spherical symmetry was used
(Section~\ref{tl:sec:adjust}). The two-stage arrangement is \cite[Section~5.3 and Figure~12]{KU22}; its vacuum realisation is
\cite[Section~4 and Figure~5]{KU}. Table~\ref{tl:tab:map} shows, for each step of the proof in Section~\ref{tl:sec:proof}, what is taken from
\cite{KU,KU22} and what is added here.

\begin{table}[ht]
\centering
\small
\begin{tabular}{@{}>{\raggedright\arraybackslash}p{0.27\textwidth}>{\raggedright\arraybackslash}p{0.31\textwidth}>{\raggedright\arraybackslash}p{0.33\textwidth}@{}}
\toprule
step & taken from \cite{KU,KU22} & added or adjusted here\\\midrule
1. first collapse, rectangle & \cite[Corollary~1]{KU} with $a=0$; \cite[Section~5.3]{KU22} & keep its open central vacuum margin\\
2. spin-up and Kerr exterior & arrangement of \cite[Section~5.3]{KU22} & Proposition~\ref{tl:prop:input}; a genuinely null Kerr attachment
 (Section~\ref{tl:sec:kerr})\\
3. vacuum strips, $C^2$ matching & \cite[Section~4]{KU}, \cite{Luk} & full sphere-data matching (Section~\ref{tl:sec:regularity})\\
4. closing the bottom & \cite[Lemma~5.1]{KU22}, \cite[Lemma~4.1]{KU} & non-flat background with a rate (Appendix~\ref{app:cutoff})\\
5. initial slice $\Sigma$ & \cite[Section~4]{KU} & explicit graph and Kerr end\\
6. global properties & \cite[Section~5.3]{KU22} & the differences listed in Section~\ref{tl:sec:global}\\
\bottomrule
\end{tabular}
\caption{Map of the proof of Theorem~\ref{thm:thirdlaw}.}
\label{tl:tab:map}
\end{table}

\subsection{Conventions}\label{tl:sec:conv}

The signature is $(-,+,+,+)$. In a double null chart the metric has the form \eqref{eq:metric}, with $\Omega>0$ and with $u$ and $v$ increasing
to the future. For an outgoing hypersurface $\{u=u_B\}$, the \emph{exterior} is $\{u<u_B\}$ and the \emph{interior} is $\{u>u_B\}$. We write
$H^{s-}:=\bigcap_{\delta>0}H^{s-\delta}$. The black hole region is $\BH=\mathcal M\setminus J^-(\scri^+)$, relative to the single asymptotically
flat end, and the event horizon is $\mathcal H^+=\partial J^-(\scri^+)$. The following elementary fact is used repeatedly.

\begin{lemma}[Monotonicity]\label{tl:lem:mono}
Along every future causal curve in the domain of a double null chart, $u$ and $v$ are non-decreasing. Hence a future causal curve crosses
$\{u=u_B\}$ only from the exterior to the interior, and a past causal curve only from the interior to the exterior.
\end{lemma}

\begin{proof}
By \eqref{eq:metric}, $\bg^{uu}=\bg^{vv}=0$, so $\nabla u$ and $\nabla v$ are past-directed null vectors. For future causal $X$ therefore
$du(X)=\bg(\nabla u,X)\ge0$ and $dv(X)\ge0$.
\end{proof}

\subsection{The input and the result}\label{tl:sec:input}

\begin{proposition}[Input from Theorem~\ref{thm:A}]\label{tl:prop:input}
Fix a Schwarzschild exterior symmetry sphere of mass $M_1>0$ and radius $R_1>2M_1$. For every sufficiently large $M_f$ there are vacuum
characteristic data along an outgoing null segment, joining that sphere to a section of the future event horizon of the extremal Kerr
solution with $a_f=M_f$, such that:
\begin{enumerate}[label=\textup{(\alph*)}]
\item \emph{Exact ends.} Near the initial sphere the data are the prescribed Schwarzschild germ; near the terminal section they extend by
the exact Kerr horizon data.
\item \emph{Full matching.} The data match in the full $C^2$ sphere-data sense of \cite[Section~2.1]{KU}, after the permitted boosts and
sphere diffeomorphisms, with compatible gauge data. Equivalently, local spacetime realisations have matching metric two-jets at the
interfaces.
\item \emph{Regularity reserve.} The data have enough finite differentiability for characteristic local existence, leaving one-sided
$C^4$ spacetime metrics in the matching charts on every compact segment.
\end{enumerate}
\end{proposition}

\begin{proof}
A Schwarzschild symmetry sphere $\{t=t_1,\ r=R_1\}$ is a Boyer--Lindquist sphere of the Kerr solution with parameters $(M_1,0)$, and
$R_1>2M_1=r_+(M_1,0)$. Hence Theorem~\ref{thm:A} applies with $a_1=0$
(and any fixed $s\ge14$), and gives, for every $M_f\ge M_0$, data along an outgoing null hypersurface from this sphere to a horizon sphere
$S_3\subset\mathcal H_f$. By Section~\ref{ch:realisation}, in the form summarised in Remark~\ref{rem:outputforKU}, these data have exact ends
(Theorem~\ref{thm:A}(i), (ii)), full $C^2$ sphere-data matching, equivalently matching metric two-jets, and a one-sided $C^4$ regularity
reserve. These are (a), (b) and (c).
\end{proof}

Each clause is needed. The proposition concerns the \emph{realised} characteristic theorem: an endpoint equation alone does not give
(a)--(c), which are supplied by Section~\ref{ch:realisation}. Clause~(b) is strong on purpose: matching only the area, mass, angular momentum or
intrinsic two-metric is not enough, since the transversal curvature and gauge entries determine the second transversal derivatives of the
metric, and without them the pasted metric is not $C^2$. Clause~(a) is obtained by choosing the free data to be exactly Schwarzschild,
resp.\ Kerr, near the endpoints and using uniqueness for the null constraint hierarchy; it is part of the verified output. The regularity
in~(c) is one-sided: across the gluing hypersurface only $C^2$ is needed, and no $C^4$ matching is required.

\begin{theorem}[Precise form of Theorem~\ref{thm:B}]\label{thm:thirdlaw}
For every sufficiently large $M_f$ there are one-ended asymptotically flat vacuum initial data
\[
 \Sigma\cong\R^3,\qquad (g_0,k_0)\in H^{7/2-}_{\loc}\times H^{5/2-}_{\loc},
\]
whose maximal globally hyperbolic development $(\mathcal M,\bg)$ has the following properties.
\begin{enumerate}[label=\textup{(\roman*)}]
\item \emph{Collapse.} $\scri^+$ is complete and $\BH\neq\varnothing$. Every point of $\Sigma$ can reach $\scri^+$; hence $\Sigma$ is disjoint from
the event horizon and contains no closed future-trapped surface.
\item \emph{Early exactly Schwarzschild apparent horizon.} $\BH$ contains an exact Schwarzschild exterior double null rectangle $\mathfrak R$ of
mass $M_1$ with an edge $\mathcal A$ on $\{r=2M_1\}$. The sections of $\mathcal A$ are marginally trapped, and there are trapped spheres
arbitrarily close to $\mathcal A$ on its inner side.
\item \emph{Late exactly extremal Kerr event horizon.} $\mathcal M$ contains a full extremal Kerr exterior and its future horizon beyond a
finite ingoing optical level. After a sufficiently late Cauchy level, the whole escaping region and the whole event horizon are the
corresponding Kerr regions.
\item \emph{Regularity.} The metric is $C^2$ across the characteristic interfaces, and has the stated Cauchy regularity near the central cap.
\item \emph{No symmetry.} The spacetime is not axisymmetric.
\end{enumerate}
\end{theorem}

Two clarifications. ``Apparent horizon'' refers to the marginally trapped tube $\mathcal A$ together with the trapped spheres just inside it; we
do not assert that it is outermost with respect to some unspecified global foliation. ``After a sufficiently late Cauchy level'' replaces
``for late advanced time'', because no global advanced time is available on the unknown part of a non-symmetric spacetime.

\subsection{The construction}\label{tl:sec:picture}

\begin{figure}[ht]
\centering
\begin{tikzpicture}[scale=0.56,>=Latex,font=\small]
 \def\P(#1,#2){({(#2)-(#1)},{(#2)+(#1)})}
 \fill[blue!7] \P(-0.8,8.5) -- \P(-5,8.5) -- \P(-5,11) -- \P(-0.8,11) -- cycle;
 \node[blue!60!black] at \P(-2.6,10.4) {exact extremal Kerr};
 \fill[blue!20] \P(-0.8,-0.5) -- \P(-1.3,-0.5) -- \P(-1.3,8.5) -- \P(-0.8,8.5) -- cycle;
 \fill[gray!40] \P(-0.8,3) -- \P(0,3) -- \P(0,4.5) -- \P(-0.8,4.5) -- cycle;
 \node at \P(-0.4,3.75) {\scriptsize$\mathfrak R$};
 \draw[thick] \P(0,-0.3) -- \P(0,9.2) node[above] {$C_0$};
 \draw[line width=2.2pt] \P(0,3) -- \P(0,4.5);
 \node[left] at \P(0.05,3.9) {$\mathcal A$};
 \fill \P(0,-0.3) circle (2pt) node[left] {\scriptsize$S_0$};
 \draw[line width=4pt,blue!35] \P(-0.8,5.2) -- \P(-0.8,7.5);
 \draw[very thick,blue] \P(-0.8,-0.5) -- \P(-0.8,11) node[above] {$C_B$};
 \fill[blue] \P(-0.8,-0.5) circle (2pt) node[left=1pt] {\scriptsize$S_B$};
 \fill[blue] \P(-0.8,5.2) circle (2pt) node[left=1pt] {\scriptsize$S_1$};
 \fill[blue] \P(-0.8,8.5) circle (2pt) node[left=1pt] {\scriptsize$S_4$};
 \draw[dashed,blue!70!black] \P(-0.8,8.5) -- \P(-5,8.5) node[right] {$N_2$};
 \draw[dashed,blue] \P(-0.8,-0.5) -- \P(-0.8,-3.0);
 \draw[red,thick] \P(-0.95,-3.2) .. controls \P(-0.98,-1.6) .. \P(-1.0,-0.1)
   .. controls \P(-1.05,3) .. \P(-1.1,8.4) .. controls \P(-1.6,8.9) and \P(-3,9.3) .. \P(-5,9.3) node[right] {$\Sigma$};
 \node[red!70!black,right] at \P(-1.0,-2.4) {cap};
 \node[blue!60!black,right] (ta) at \P(-3.2,4.4) {bridge (Proposition~\ref{tl:prop:input})};
 \draw[->,blue!60!black] (ta.west) -- \P(-0.9,6.3);
 \node[blue!60!black,right] (e2) at \P(-3.0,1.2) {exterior strip};
 \draw[->,blue!60!black] (e2.west) -- \P(-1.15,2.4);
 \node[left] at \P(2.4,6.2) {retained old interior};
\end{tikzpicture}
\caption{Schematic double null picture (not a Penrose diagram): $u$ increases to the upper left, $v$ to the upper right. Black: the
first collapse, with event horizon $C_0$ (exactly Schwarzschild after the ingoing sheet $N_1$). Blue: the new event horizon $C_B$, slightly
outside $C_0$; it runs through the bridge from $S_1$ and then along the exact extremal Kerr horizon. Grey: the exact Schwarzschild rectangle
$\mathfrak R$, whose edge $\mathcal A$ on $C_0$ is the apparent-horizon segment. Dashed below $S_B$: the past boundary of the bottom domain;
the cap lies just below it. Red: the initial slice $\Sigma$.}
\label{tl:fig:assembly}
\end{figure}
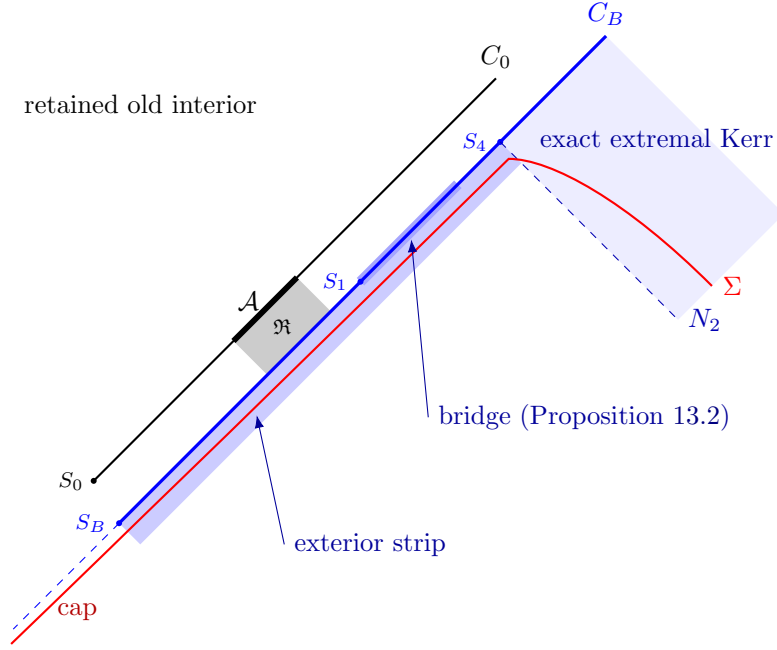

The construction is shown in Figure~\ref{tl:fig:assembly}. We build a black hole by a first collapse; discard its exterior beyond a slightly
outer cone $C_B$; along $C_B$ spin the horizon up to extremality by Proposition~\ref{tl:prop:input}; attach the exact extremal Kerr exterior;
solve for thin vacuum strips on both sides of $C_B$; close the bottom by a Cauchy stability argument; and finally verify the global causal
properties.

The order of choices is essential: first the first collapse and its central vacuum margin; then the outer cone $C_B$, the rectangle and the
sphere $S_1$; then the final mass $M_f$ and the bridge; then the strip widths; then the second cutoff width; and last the cap, the strip graph
and the Kerr end. Every small quantity is chosen after the objects relative to which it must be small, so no estimate has to be uniform in
$M_f$, in the bridge length, or in the strip widths.

\subsection{Proof of Theorem~\ref{thm:thirdlaw}}\label{tl:sec:proof}

\paragraph{Step 1: the first collapse and the retained rectangle.}
We first produce a black hole with an exactly Schwarzschild horizon segment and room for a new horizon just outside it. Apply the vacuum
construction of \cite[Corollary~1 and Section~4]{KU} with final spin $a=0$. Call its outgoing gluing hypersurface $C_0$ (it becomes the
Schwarzschild horizon $\{r=2M_1\}$) and its stationary ingoing attachment hypersurface $N_1$; beyond $N_1$, on the exterior side, the spacetime
is exactly Schwarzschild. We use the actual construction, not only the statement of the corollary; in particular we retain the open central
vacuum margin produced by its first Cauchy cutoff.

In the Schwarzschild part choose an outgoing hypersurface $C_B=\{u=u_B\}$, $u_B<0$ small, slightly outside $C_0=\{u=0\}$, and continue it
backwards through the fixed compact characteristic strip. Continuous dependence of the (regular) null geodesic flow keeps it embedded and close
to $C_0$. Its bottom cut $S_B$ bounds a ball $B$ whose closed past domain lies in the central vacuum margin (Appendix~\ref{tl:app:place}). We keep
the Schwarzschild rectangle and a small actual interior collar of its old horizon edge:
\[
 \mathfrak R=\{u_B\le u\le0,\ v_{N_1}\le v\le v_R\},\qquad \mathcal A=\mathfrak R\cap\{u=0\}.
\]

\paragraph{Step 2: the spin-up and the Kerr exterior.}
Along $C_B$ we now prescribe data that are first Schwarzschild and eventually exactly extremal Kerr. Choose a later Schwarzschild symmetry
sphere $S_1\subset C_B$ (beyond $\mathfrak R$) and apply Proposition~\ref{tl:prop:input}. Continue the terminal data along the exact extremal
horizon for all future times. Attach the exact Kerr exterior beyond a later ingoing \emph{null} hypersurface $N_2$ through a horizon cut $S_4$
(Section~\ref{tl:sec:kerr} explains why this needs care). The final mass is chosen after the first spacetime and $S_1$; everything chosen later
may depend on $M_f$ and on the bridge.

\paragraph{Step 3: thin vacuum strips and $C^2$ matching.}
We construct a vacuum spacetime on both sides of the whole future part of $C_B$.

\emph{Exterior.} As in the proof of \cite[Corollary~1]{KU}, solve the characteristic problem backwards with data on $N_2$ and on the entire finite
segment of $C_B$ from its bottom cut up to $S_4$. Luk's theorem \cite{Luk} gives a strip of positive transverse width on this fixed compact
segment (possibly very small, depending on everything chosen before).

\emph{Interior.} Solve forwards from an ingoing segment in the retained old region (at $v=v_R$, inside $\mathfrak R$) and the new data on $C_B$.
Do this on increasing finite segments and patch by uniqueness; this gives an interior neighbourhood of every future horizon point. The widths
may shrink to zero at late times (we choose them non-increasing in $v$). The incoming data of this interior problem are the actual old data;
they are not replaced by Kerr interior data.

\emph{Matching.} By Proposition~\ref{tl:prop:input}(b), propagated along the null constraint hierarchy, the two sides induce the same full sphere
data on $C_B$, which gives the $C^2$ pasting used in \cite{KU}. Concretely, the transversal data of the exterior solution, transported backwards
from the final Kerr cut, coincide with the forward bridge data and, before $S_1$, with the data of the first spacetime. This is uniqueness for
one and the same linear transport hierarchy.

\paragraph{Step 4: closing the bottom.}
We need a vacuum region containing the whole closed past domain of the bottom ball $B$. On the bottom slice take the data of the first
spacetime inside $B$ and the newly constructed exterior-strip data outside. They match to second order at $S_B$. Apply the non-flat version of
the Kehle--Unger cutoff argument, Lemma~\ref{tl:lem:cap} in Appendix~\ref{app:cutoff}, with the first spacetime as background. This supplies the
central vacuum region needed in the second application of their assembly.

\paragraph{Step 5: the initial slice $\Sigma$.}
We construct a one-ended asymptotically flat spacelike slice through all the pieces.

\emph{Cap.} Choose a smooth spacelike cap slightly below the past boundary of the closed bottom domain, from which every point can escape to the
exterior strip (Section~\ref{tl:sec:global}(1)).

\emph{Strip graph.} Continue the cap through the thin exterior strip by a graph $u=h(v)$ with $h'<0$ and arbitrarily small total variation, so
that it stays inside the strip, however thin. It is spacelike: with $du=h'\,dv$, \eqref{eq:metric} induces in the coordinates $(v,\theta^A)$
the metric
\[
 \begin{pmatrix}-4\Omega^2h'+|b|^2_{\gd} & -b_B\\ -b_A & \gd_{AB}\end{pmatrix},\qquad b_A=\gd_{AB}b^B.
\]
The angular block $\gd$ is positive definite and its Schur complement is
\[
 \big(-4\Omega^2h'+|b|^2_{\gd}\big)-b_A\gd^{AB}b_B=-4\Omega^2h'>0,
\]
so the induced metric is positive definite, whatever the shift $b$.

\emph{Kerr end.} Continue the graph into the exact Kerr exterior to an asymptotically flat slice which coincides with a constant
Boyer--Lindquist time slice far out (Section~\ref{tl:sec:kerr}).

The cap (a ball), the strip annulus and the end (the complement of a ball) form $\Sigma\cong\R^3$. The constraint equations hold because every
piece lies in a vacuum region. Regularity and the domain argument are verified in Sections~\ref{tl:sec:regularity} and~\ref{tl:sec:global}.

\paragraph{Step 6: global properties.}
We verify (i)--(v) of Theorem~\ref{thm:thirdlaw}.

\emph{Maximal development.} The domain of dependence of $\Sigma$ in the assembled spacetime embeds into the maximal development by geometric
uniqueness \cite{MGHD}. As in \cite[proof of Corollary~1]{KU}, with the modifications of Section~\ref{tl:sec:global}, this domain contains the whole
future cylinder $C_B$, the old rectangle and a complete Kerr asymptotic end, and the late escaping region is exactly that end. The escape curves from the
cap and the strip give $\Sigma\subset J^-(\scri^+)$.

\emph{No trapped surface in $\Sigma$.} This is the usual exclusion in a predictable exterior with complete $\scri^+$. Since the metric is only
known to be $C^1$ near the cap, we use the maximising-geodesic approximation and focusing estimates of \cite[Corollary~2.16,
Lemmas~5.5--5.6]{Graf}.

\emph{Trapped spheres.} We use the old Schwarzschild horizon, as in the rectangle step of \cite[Section~5.3]{KU22}, not the final Kerr horizon.
On $\mathcal A$ we have $\tr\chi=0$, $\chih=\eta=0$, $\tr\chib<0$ and $\rho=-2M_1/r^3$. In the transport equation \cite[(A.16)]{KU}
\[
 \Db(\Omega\tr\chi)=2\Omega^2\Div\eta+2\Omega^2|\eta|^2-\Omega^2(\chih,\chibh)-\tfrac12\Omega^2\tr\chi\,\tr\chib+2\Omega^2\rho
\]
all terms but the last vanish on $\mathcal A$, so
\[
 \Db(\Omega\tr\chi)=2\Omega^2\rho=-\frac{4M_1\Omega^2}{r^3}<0\qquad\text{on }\mathcal A.
\]
The actual inner metric has the same two-jet on $\mathcal A$ (it is $C^2$ there), so by Taylor's formula spheres slightly inside $\mathcal A$ have
$\tr\chi<0$, and still $\tr\chib<0$: they are trapped. The rectangle $\mathfrak R$ lies in the closed interior side $\{u\ge u_B\}$ of $C_B$; its edge on $C_B$ lies on the new
event horizon (Section~\ref{tl:sec:global}(2)), which is disjoint from the open set $J^-(\scri^+)$, and its remaining points lie on
the inner side of the new event horizon. Hence $\mathfrak R\subset\BH$: the old apparent horizon lies inside the new black hole.

\emph{No axial symmetry.} For a global axial Killing field $Z$, the vacuum Komar integral $(16\pi)^{-1}\int_S\star dZ^\flat$ would take the same
value on the homologous Schwarzschild and final Kerr spheres. It vanishes on the Schwarzschild spheres and is proportional to
$M_fa_f=M_f^2\neq0$ on the Kerr spheres. This is the obstruction noted in \cite[Remark~1.8]{KU}.

This completes the proof of Theorem~\ref{thm:thirdlaw}, modulo the verifications in Section~\ref{tl:sec:adjust}. \qed

\subsection{The adjustments}\label{tl:sec:adjust}

\subsubsection{Where small rotation was used}\label{tl:sec:smallrot}
Table~\ref{tl:tab:smallrot} lists the uses of $|a|/M\ll1$ in \cite{KU} and their replacements. Thus the claim is an adaptation of the proof,
not a direct invocation of the small-spin corollary at $a=M$.

\begin{table}[ht]
\centering
\small
\begin{tabular}{@{}>{\raggedright\arraybackslash}p{0.44\textwidth}>{\raggedright\arraybackslash}p{0.48\textwidth}@{}}
\toprule
use of $|a|/M\ll1$ in \cite{KU} & replacement here\\\midrule
the gluing theorem in \cite[Corollary~1]{KU} & first collapse: applied with $a=0$; second stage: Proposition~\ref{tl:prop:input}\\
final trapped-surface check (perturbatively from Schwarzschild) & the Schwarzschild calculation on $\mathcal A$ (Step~6)\\
characteristic existence, Cauchy stability & need only finite regularity and non-degenerate Lorentzian charts, not small $a/M$\\
exact Kerr exterior attachment & checked directly at $a=M$ (Section~\ref{tl:sec:kerr})\\
\bottomrule
\end{tabular}
\caption{Uses of small rotation in \cite{KU} and their replacements.}
\label{tl:tab:smallrot}
\end{table}

\subsubsection{Regularity and matching}\label{tl:sec:regularity}
The regularity of the initial data comes from the fractional regularity and cutoff argument of \cite[Lemma~4.1]{KU}, with the one-sided
derivative reserve of Proposition~\ref{tl:prop:input}(c). The mechanism is as follows. A $C^2$ induced metric with one-sided $C^4$ bounds has
third derivatives that are bounded and Lipschitz on each side of the interface, with no delta term across it (the second derivatives are
continuous). Likewise a $C^1$ second fundamental form with one-sided $C^3$ bounds has piecewise Lipschitz second derivatives. Such a piecewise
Lipschitz function $w$ satisfies $\|w(\cdot+z)-w\|_{L^2}^2\le C|z|$: the translates differ by $O(1)$ only on a slab of volume $O(|z|)$ around the
interface, and by $O(|z|)$ elsewhere. By the fractional difference-quotient characterisation,
\[
 \int\|w(\cdot+z)-w\|^2_{L^2}\,|z|^{-3-2\sigma}\,dz<\infty\qquad\text{for }\sigma<\tfrac12,
\]
that is, $w\in H^\sigma$. Hence $(g_0,k_0)\in H^{7/2-}\times H^{5/2-}$. Hyperbolic persistence gives the same indices on the cap. The maximal
development theorem is applied at the integer index $H^3\times H^2$ \cite{MGHD}, and the fractional regularity is recovered by persistence.

Two remarks. Only finitely many derivatives are needed on each compact characteristic segment, and they need not agree to that order across
the gluing hypersurface. And the full $C^2$ matching includes the transversal curvature and gauge information; solving only the tangential null
equations is not a substitute for it. As in \cite[Remark~4.1]{KU}, pointwise $C^2$ regularity may fail near the central crease.

\subsubsection{The exact extremal Kerr attachment}\label{tl:sec:kerr}
The attachment hypersurface $N_2$ must be a genuine ingoing null hypersurface from the horizon to infinity. The surfaces $\{v_K=\text{const}\}$
of ordinary ingoing Kerr coordinates are not null when $a\neq0$, so they are not admissible.

We use instead the ingoing quasi-spherical light cones of Pretorius and Israel \cite{PI}, which have no caustics for $r>0$, at arbitrary
rotation. For the final parameters ($a_f=M_f$) write
\[
 \Delta=(r-M_f)^2,\qquad B_K=r^2+a_f^2,\qquad \frac{dr_*}{dr}=\frac{B_K}{\Delta},\qquad \mathsf v=t+r_*^{\rm opt}=v_K+q(r,\vartheta).
\]
The optical function $\mathsf v$ has radial coefficient $Q/\Delta$ with $Q^2=B_K^2-a_f^2\nu\Delta$, $0\le\nu\le1$ \cite[(7), (10)]{PI}. Hence
\[
 \partial_rq=\frac{Q-B_K}{\Delta}=\frac{Q^2-B_K^2}{\Delta\,(Q+B_K)}=-\frac{a_f^2\nu}{Q+B_K}.
\]
The last expression is regular where $\Delta=(r-M_f)^2$ vanishes (there $Q=B_K>0$), and the angular construction is regular in smooth polar
charts. So $\mathsf v$ is a regular optical function up to and across the extremal horizon, and a level set $N_2=\{\mathsf v=v_4\}$ is a
legitimate horizon-to-infinity ingoing attachment. It meets the horizon in a shifted cut $S_4$ rather than in the coordinate sphere at which the
gluing data end; one simply continues the exact horizon up to that cut.

\emph{The spacelike end, without small spin.} Put $\Sigma_K=r^2+a_f^2\cos^2\vartheta$. In Boyer--Lindquist coordinates
$\bg^{-1}(d\mathsf v,d\mathsf v)=0$, $\bg^{-1}(d\mathsf v,dr)=Q/\Sigma_K$ and $\bg^{-1}(dr,dr)=\Delta/\Sigma_K$. Hence
\[
 \bg^{-1}\big(d(\mathsf v-r_*),d(\mathsf v-r_*)\big)=-2\frac{B_K}\Delta\frac Q{\Sigma_K}+\frac{B_K^2}{\Delta^2}\frac\Delta{\Sigma_K}
 =\frac{B_K^2-2B_KQ}{\Sigma_K\Delta}<0,
\]
because for $r\ge M_f=a_f$ we have $\Delta\le r^2$ and $4a_f^2r^2\le B_K^2$, so $a_f^2\nu\Delta\le B_K^2/4$ and $Q\ge\frac{\sqrt3}2B_K>\frac12B_K$.
So $T=\mathsf v-r_*=t+q$ is a time function on the exterior. We cut off the bounded function $q$ at large radius: for a cutoff $\chi$ at a
large radius $R$, $\widetilde T=t+(1-\chi(r))q$ is still temporal and equals Boyer--Lindquist time near infinity. Its spacelike level set through
the joining sphere is the Kerr end of $\Sigma$; it is joined to the strip graph inside the exact Kerr neighbourhood of the corner, where both are
spacelike graphs and can be smoothed.

\subsubsection{Differences from the proof of \texorpdfstring{\cite[Corollary~1]{KU}}{[KU, Corollary 1]}}\label{tl:sec:global}
The Cauchy extension, the initial slice and the causal structure of the assembled spacetime are verified as in the proof of
\cite[Corollary~1]{KU} (see \cite[Figure~5]{KU}), which does not use spherical symmetry. Above $\Sigma$ the spacetime again consists of thin
vacuum strips along the gluing hypersurface, an exact Kerr exterior beyond an ingoing null hypersurface, and a region near the bottom obtained
by Cauchy stability. The domain of dependence of $\Sigma$, the escape of $\Sigma$ to $\scri^+$ and the identification of the late exterior
with Kerr follow from this arrangement as there. The two-stage arrangement is that of \cite[Section~5.3]{KU22}. We list the points where the
present construction differs, and how each of them is handled.

\paragraph{(1) The interior is not Minkowski.} In \cite{KU} the region inside the bottom sphere of the gluing hypersurface is a piece of
Minkowski space; here it is the central vacuum margin of the first collapse spacetime.
\begin{itemize}
\item Closing the bottom requires the cutoff argument with a non-flat background, Lemma~\ref{tl:lem:cap} in Appendix~\ref{app:cutoff}.
\item The cap still escapes. The new metric agrees with the background on the closed past domain of the bottom ball $B$, so the null
generators of its past boundary lead to $\partial B$. A point slightly to the past of this boundary therefore has a timelike curve to a point
just outside $B$, in the exterior strip, and from there it escapes through the strip into the Kerr exterior. A smooth spacelike cap with a
prescribed collar in the strip is provided by causal smoothing for locally Lipschitz cones \cite[Theorems~2.17 and~2.21]{Minguzzi}.
\item The retained part of the first spacetime, i.e.\ the rectangle $\mathfrak R$ and the interior strip, lies in the domain of dependence of
$\Sigma$. By Lemma~\ref{tl:lem:mono}, a past causal curve starting there cannot leave through an inner outgoing or an upper ingoing boundary,
so it reaches the exterior strip or the bottom domain, where the data agree on open collars and Cauchy uniqueness applies.
\end{itemize}

\paragraph{(2) The exterior is extremal Kerr.}
\begin{itemize}
\item The attachment hypersurface $N_2$, the time function on the exterior and the Kerr end of $\Sigma$ are constructed at $a=M$ in
Section~\ref{tl:sec:kerr}; no smallness of the rotation is used. The strip graph is joined to the Kerr end with a positive buffer inside the
strip width, so that above $\Sigma$ the hypersurface $N_2$ is met only inside the strip.
\item $C_B$ is a portion of the event horizon. Points just outside $C_B$ escape through the strip into Kerr, and by Lemma~\ref{tl:lem:mono} a
future causal curve cannot cross $C_B$ from the interior to the exterior side. The extremal case needs one remark: beyond $S_4$ the generators
of $C_B$ are integral curves of the Killing field $\xi$, which is affinely parametrised on the extremal horizon (Lemma~\ref{lem:affine}), with
parameter running to $+\infty$. So the generators are future complete, and $C_B$ has no future endpoint in the maximal development.
\item The late exterior is Kerr, as in \cite{KU}; the Kerr-specific input is the time function $\widetilde T$ and the optical function
$\mathsf v$ of Section~\ref{tl:sec:kerr}. Far out $\Sigma=\{t=t_0\}$, and a point with $t\le t_0$ and $\mathsf v>t_0+\max_\vartheta r_*^{\rm opt}(R,\vartheta)$
cannot reach $\{r=R\}$ before $\{t=t_0\}$ along a future causal curve; this gives the full late Kerr ingoing region in $D(\Sigma)$ and the
complete Kerr $\scri^+$. At bounded $\mathsf v$, a lower bound on $\widetilde T$ bounds the radius, so on a Cauchy level later than the compact
part of $N_2$ above $\Sigma$ and than the non-stationary part of the horizon, the escaping region is the retained Kerr exterior. We state this
at late Cauchy levels rather than at late advanced times, because no global advanced time is available on the unknown part of the
spacetime.
\end{itemize}

\paragraph{(3) The third-law statements.} The early apparent horizon, and its position inside the new black hole, are not part of
\cite[Corollary~1]{KU}. They come from the two-stage arrangement of \cite[Section~5.3]{KU22}, and Step~6 verifies them without spherical
symmetry: the marginally trapped tube $\mathcal A$ and the nearby trapped spheres are obtained in the exact Schwarzschild rectangle, and
$\mathfrak R$ lies in the black hole region because it lies on the interior side of $C_B$, which is a portion of the event horizon by (2).

\subsubsection{What is imported from \cite{KU22}}\label{tl:sec:imported}
Only the two-stage assembly and the preservation of the early Schwarzschild rectangle are imported from \cite[Section~5.3]{KU22}. Its further
conclusions about the outermost apparent horizon, its disconnectedness, or trapped surfaces at arbitrarily late times rely on additional
arguments in spherical symmetry and are not part of Theorem~\ref{thm:thirdlaw}. No exact Kerr interior is prescribed anywhere in the
construction.

\appendix
\section{Tensors on the sphere and smoothness at the poles}\label{app:sphere}

This appendix decides which frame matrices in the round polar frame describe tensor fields that are smooth at the poles; this is
used in Sections~\ref{ch:optical}, \ref{ch:tail}, \ref{ch:spin}, \ref{ch:controller} and~\ref{ch:linear}. Throughout,
$\gac=d\vartheta^2+\sin^2\!\vartheta\,d\varphi^2$ is the unit round metric, $z=\cos\vartheta$, and the polar frame and coframe are
\begin{equation}\label{A:eq:polarframe}
 e_1:=\partial_\vartheta,\quad e_2:=\frac1{\sin\vartheta}\,\partial_\varphi;\qquad \theta^1:=d\vartheta,\quad \theta^2:=\sin\vartheta\,d\varphi ,
\end{equation}
as in \eqref{eq:roundframe}. For one-forms $\alpha,\beta$ we write $\alpha\beta:=\frac12(\alpha\otimes\beta+\beta\otimes\alpha)$ and $\alpha^2:=\alpha\alpha$.
We use the matrices $S(\xi)$, $R_\alpha$, $K_0$ of Definition~\ref{def:SR} and Lemma~\ref{lem:param}.

\subsection{Tensor fields on the sphere}\label{A:sec:sphere}

Let $S^2\subset\R^3$ be the unit sphere, $N=(0,0,1)$, $S=(0,0,-1)$, with spherical coordinates
$(x,y,z)=(\sin\vartheta\cos\varphi,\,\sin\vartheta\sin\varphi,\,\cos\vartheta)$ on $S^2\setminus\{N,S\}$. Near the poles we use the graph charts $(x,y)$
on $U_N=\{z>0\}$ and $U_S=\{z<0\}$, in which the poles are the origin, and $w:=x+iy=re^{i\varphi}$, $r=\sin\vartheta$. Smoothness near a pole
means smoothness in $(x,y)$ at the origin; for instance $a(z)$ is smooth on $S^2$ for every $a$ smooth on an open interval containing
$[-1,1]$, since $z=\pm\sqrt{1-x^2-y^2}$. From $dw=e^{i\varphi}(dr+i\,r\,d\varphi)$,
\begin{equation}\label{A:eq:drdphi}
 dr=\tfrac12\big(e^{-i\varphi}dw+e^{i\varphi}d\overline w\big),\qquad i\,r\,d\varphi=\tfrac12\big(e^{-i\varphi}dw-e^{i\varphi}d\overline w\big).
\end{equation}

\begin{lemma}\label{A:lem:round}
In the charts on $U_N$ and $U_S$ the round metric is $\gac=dx^2+dy^2+(x\,dx+y\,dy)^2/(1-x^2-y^2)$. In particular $\gac$ and $\gac^{-1}$ are
smooth at the poles, and $\gac=dx^2+dy^2$ there.
\end{lemma}
\begin{proof}
$\gac=dx^2+dy^2+dz^2$ and $dz=-(x\,dx+y\,dy)/z$ with $z^2=1-x^2-y^2$.
\end{proof}

The polar frame \eqref{A:eq:polarframe} is a smooth $\gac$-orthonormal frame on $S^2\setminus\{N,S\}$; it does not extend continuously to the
poles (near $N$ it is the Cartesian frame rotated by the angle $\varphi$), and by the hairy ball theorem no frame on $S^2$ does.

\paragraph{Complex components and spin weight.} For a one-form $\alpha$ and a symmetric bilinear form $T$ on $S^2\setminus\{N,S\}$ write
$\alpha_a=\alpha(e_a)$ and $T_{ab}=T(e_a,e_b)$ (the \emph{frame matrix}); $T$ is $\gac$-tracefree if $T_{11}+T_{22}=0$. With the complex
one-form $\zeta:=\theta^1+i\,\theta^2$ and complex-linear products,
\begin{equation}\label{A:eq:zetasquare}
 \zeta\overline\zeta=\gac,\qquad \zeta^2=(\theta^1)^2-(\theta^2)^2+2i\,\theta^1\theta^2 .
\end{equation}

\begin{definition}\label{A:def:complexcomp}
The \emph{complex component} of a real one-form $\alpha$ is $\alpha_\C:=\alpha_1+i\alpha_2$, that of a real $\gac$-tracefree symmetric form $T$ is
$T_\C:=T_{11}+iT_{12}$, and that of a function $f$ is $f_\C:=f$.
\end{definition}

\begin{lemma}\label{A:lem:complexcomp}
$\alpha=\Re\big(\overline{\alpha_\C}\,\zeta\big)$ and $T=\Re\big(\overline{T_\C}\,\zeta^2\big)$; conversely, for $c\in\C$, $\Re(\overline c\,\zeta)$ and
$\Re(\overline c\,\zeta^2)$ are the one-form, resp.\ tracefree symmetric form, with complex component $c$.
\end{lemma}
\begin{proof}
Expand, using \eqref{A:eq:zetasquare} and $T_{22}=-T_{11}$.
\end{proof}

A tracefree $T$ with frame matrix $\rho S(\xi)$ has $T_\C=\rho e^{i\xi}$ (Lemma~\ref{lem:param}(a)). Under a rotation of the frame by an angle
$\beta$ one has $\zeta'=e^{-i\beta}\zeta$, hence $\alpha'_\C=e^{-i\beta}\alpha_\C$ and $T'_\C=e^{-2i\beta}T_\C$: functions, one-forms and tracefree
symmetric forms have \emph{spin weight} $s=0,1,2$.

\paragraph{Metrics and their logarithms.} Let $\gdc$ be a metric on an open set of $S^2$ with $\det(\gac^{-1}\gdc)=1$. The endomorphism
$H:=\gac^{-1}\gdc$ is $\gac$-self-adjoint and positive with determinant one, so by Lemma~\ref{lem:param}(c) there is a unique $\gac$-self-adjoint
tracefree endomorphism field $A$ with $\exp(A)=H$; we write $\gdc=\gac\exp(A)$. Off the poles the frame matrix of $A$ is $\log G$,
$G_{ab}=\gdc(e_a,e_b)$.

\begin{theorem}\label{A:thm:metriclog}
On any open set of $S^2$, including neighbourhoods of the poles, $\gdc$ is smooth if and only if $A$ is smooth, if and only if the
tracefree symmetric form $\gac A:=\gac(\cdot,A\,\cdot)$ is smooth.
\end{theorem}
\begin{proof}
By Lemma~\ref{lem:log}, $A=\log H$ is given by the frame-free formula \eqref{eq:logformula}, which is smooth in $H$, and $H$ is smooth
wherever $\gdc$ is (Lemma~\ref{A:lem:round}). Conversely, $\exp(A)=\sum_{n\ge0}A^n/n!$ is smooth in $A$. Lowering and raising with $\gac$
preserve smoothness.
\end{proof}

\subsection{Smoothness at the poles}\label{A:sec:poles}

\paragraph{Fourier modes in polar coordinates.} We use two classical facts from one-variable calculus: by Taylor's formula with
integral remainder, a smooth function $u$ on an interval around $0$ with $u(x)=O(|x|^k)$ can be written as $u(x)=x^k\tilde u(x)$ with $\tilde u$
smooth; and by Whitney's theorem \cite{Whitney} (see also \cite{Hormander} for Borel's lemma, on which it rests), every smooth even function
$h$ on $(-R,R)$ is of the form $h(x)=g(x^2)$ with $g$ smooth on $(-\infty,R^2)$.

\begin{theorem}[Fourier modes]\label{A:thm:fourier}
Let $k\in\Z$, $R>0$ and $f:(0,R)\to\C$, and put $F(re^{i\varphi}):=f(r)e^{ik\varphi}$ on $0<|w|<R$. Then $F$ extends to a smooth function of $(x,y)$ on
the disc $|w|<R$ if and only if
\[
 f(r)=r^{|k|}g(r^2)\qquad(0<r<R)
\]
for some smooth $g$ on an open interval containing $[0,R^2)$.
\end{theorem}
\begin{proof}
Replacing $F$ by $\overline F$ if necessary, we may assume $k\ge0$. If $f=r^kg(r^2)$, then $F=w^kg(x^2+y^2)$ is smooth. Conversely, let $F$ be
smooth and $u(x):=F(x,0)$; then $u(x)=f(x)$ for $x>0$ and $u(-x)=(-1)^ku(x)$. By Taylor's theorem $F=P+O(r^k)$ with
$P=\sum_{p+q\le k-1}c_{pq}w^p\overline w^q$ and $w^p\overline w^q=r^{p+q}e^{i(p-q)\varphi}$; averaging against $e^{-ik\varphi}$ over the circle of radius $r$ gives
$f(r)=O(r^k)$, since $p-q=k$ forces $p+q\ge k$. Hence $u(x)=x^k\tilde u(x)$ with $\tilde u$ smooth and, by the parity of $u$, even, so
$\tilde u(x)=g(x^2)$ by Whitney's theorem, and $f(r)=r^kg(r^2)$.
\end{proof}

\paragraph{A smooth coframe at each pole.}

\begin{lemma}\label{A:lem:omega}
Let $\nu:=(1-r^2)^{-1/2}$ and $\mu:=\big(\sqrt{1-r^2}\,(1+\sqrt{1-r^2})\big)^{-1}$; these are smooth functions of $r^2=x^2+y^2$ for $r<1$, and
$\nu-1=\mu r^2$.
\begin{enumerate}[label=(\alph*)]
\item On $U_N\setminus\{N\}$, $\zeta=e^{-i\varphi}\omega_N$, where $\omega_N:=\frac12(1+\nu)\,dw+\frac12\mu\,w^2\,d\overline w$ is a smooth complex one-form on
$U_N$ with $\omega_N=dx+i\,dy$ at $N$.
\item On $U_S\setminus\{S\}$, $\zeta=-e^{i\varphi}\omega_S$, where $\omega_S:=\frac12(1+\nu)\,d\overline w+\frac12\mu\,\overline w^2\,dw$ is a smooth complex one-form
on $U_S$ with $\omega_S=dx-i\,dy$ at $S$.
\end{enumerate}
\end{lemma}
\begin{proof}
With $c:=\sqrt{1-r^2}$, $\nu-1=(1-c)/c=(1-c^2)/(c(1+c))=\mu r^2$. (a) On $U_N$, $\cos\vartheta=c$, so $\theta^1=d\vartheta=\nu\,dr$ and $\theta^2=r\,d\varphi$.
By \eqref{A:eq:drdphi},
\[
 \zeta=\nu\,dr+i\,r\,d\varphi=e^{-i\varphi}\Big[\tfrac12(\nu+1)\,dw+\tfrac12(\nu-1)e^{2i\varphi}\,d\overline w\Big],
\]
and $(\nu-1)e^{2i\varphi}=\mu w^2$. (b) On $U_S$, $\cos\vartheta=-c$, so $\theta^1=-\nu\,dr$ and
$\zeta=-(\nu\,dr-i\,r\,d\varphi)=-e^{i\varphi}\big[\frac12(\nu+1)\,d\overline w+\frac12(\nu-1)e^{-2i\varphi}\,dw\big]$, with $(\nu-1)e^{-2i\varphi}=\mu\overline w^2$.
\end{proof}

The factor $e^{-i\varphi}$ in (a) is the rotation of the polar frame relative to the Cartesian one: in the notation of the spin weight,
$\zeta'=e^{-i\beta}\zeta$ with $\beta=\varphi$.

\paragraph{The smoothness criterion.}

\begin{theorem}[Smoothness of spin-weighted fields]\label{A:thm:spins}
Let $s\in\{0,1,2\}$, and let $T$ be a real smooth field on $S^2\setminus\{N,S\}$ which is a function ($s=0$), a one-form ($s=1$), or a
$\gac$-tracefree symmetric bilinear form ($s=2$), with complex component $\tau:=T_\C$ in the polar frame. Then $T$ extends smoothly across $N$
if and only if $\tau e^{is\varphi}$ extends to a smooth function near $N$, and $T$ extends smoothly across $S$ if and only if $\tau e^{-is\varphi}$
extends to a smooth function near $S$.
\end{theorem}
\begin{proof}
For $s=0$ there is nothing to prove; let $s\in\{1,2\}$. Near $N$, by Lemmas~\ref{A:lem:complexcomp} and~\ref{A:lem:omega}(a),
\[
 T=\Re\big(\overline\tau\,\zeta^s\big)=\Upsilon(F),\qquad F:=\tau e^{is\varphi},\quad \Upsilon(F):=\Re\big(\overline F\,\omega_N^s\big)\qquad\text{on }U_N\setminus\{N\}.
\]
Let $\mathcal V_p$ be the two-dimensional real space of one-forms ($s=1$), resp.\ $\gac$-tracefree symmetric forms ($s=2$), at $p\in U_N$. For every
$p\in U_N$, $F\mapsto\Upsilon(F)|_p$ is a real-linear bijection $\C\to\mathcal V_p$: off $N$ because $\Upsilon(F)=\Re\big(\overline{Fe^{-is\varphi}}\zeta^s\big)$ and by
Lemma~\ref{A:lem:complexcomp}; at $N$ because $\omega_N=dx+i\,dy$, $(dx,dy)$ is a $\gac$-orthonormal coframe there (Lemma~\ref{A:lem:round}), and
Lemma~\ref{A:lem:complexcomp} applies with $\zeta$ replaced by $dx+i\,dy$. Hence $E_1:=\Re(\omega_N^s)$ and $E_2:=\Im(\omega_N^s)$ are smooth fields
forming a basis of $\mathcal V_p$ at every $p\in U_N$, and $\Upsilon(F)=(\Re F)E_1+(\Im F)E_2$. The coefficients of a smooth $\mathcal V$-valued field
$T'=c_1E_1+c_2E_2$ on $U_N$ are smooth: they solve $\Gamma c=(\langle T',E_1\rangle,\langle T',E_2\rangle)$, where $\langle\cdot,\cdot\rangle$ is the inner product
induced by $\gac$ and the Gram matrix $\Gamma_{ij}=\langle E_i,E_j\rangle$ is smooth and invertible.

Now, if $F$ extends smoothly to $U_N$, then $\Upsilon(F)$ is a smooth extension of $T$. Conversely, if $T$ extends smoothly to $U_N$, its coefficients
$c_1,c_2$ are smooth, and $\Upsilon(c_1+ic_2)=T=\Upsilon(F)$ off $N$, so $F=c_1+ic_2$ off $N$ by injectivity, and $F$ extends smoothly. At $S$,
Lemma~\ref{A:lem:omega}(b) gives $\zeta^s=(-1)^se^{is\varphi}\omega_S^s$, hence $T=\Re\big(\overline{(-1)^s\tau e^{-is\varphi}}\,\omega_S^s\big)$, and the same argument
applies with $\omega_S$ (at $S$, $\omega_S=dx-i\,dy$ with the orthonormal coframe $(dx,-dy)$) and $F=(-1)^s\tau e^{-is\varphi}$.
\end{proof}

In words: near $N$ the polar frame is the Cartesian frame rotated by $\varphi$, so a component of spin weight $s$ acquires the factor $e^{is\varphi}$
when converted to the Cartesian frame; near $S$ the orientation is reversed and the factor is $e^{-is\varphi}$.

The case used in the main text is the following.

\begin{lemma}[Smoothness at the poles]\label{lem:poles}
Let $a$ be smooth on $[-1,1]$, $\sigma\in\{-1,0,1\}$, and let $T$ be the tracefree symmetric tensor field with frame
matrix $a(z)S(2\sigma\varphi)$. Then $T$ extends smoothly across the north pole $z=1$ if and only if
$a=(1-z)^{|\sigma+1|}b$, and across the south pole $z=-1$ if and only if $a=(1+z)^{|\sigma-1|}b$, with $b$ smooth:
\begin{center}
\begin{tabular}{@{}lccc@{}}
\toprule
 & $\sigma=-1$ & $\sigma=0$ & $\sigma=+1$\\\midrule
north pole $z=1$ & none & $1-z$ & $(1-z)^2$\\
south pole $z=-1$ & $(1+z)^2$ & $1+z$ & none\\\bottomrule
\end{tabular}
\end{center}
The same holds for $\gdc=\gac\exp(T)$, since $\exp$ is an equivariant diffeomorphism from tracefree symmetric
matrices onto unimodular positive ones.
\end{lemma}

\begin{proof}
As usual, smooth on $[-1,1]$ means the restriction of a smooth function on an open interval containing $[-1,1]$. By Lemma~\ref{lem:param}(a), $T$
has the complex component $T_\C=a(z)e^{2i\sigma\varphi}$ and spin weight $2$, so by Theorem~\ref{A:thm:spins} it is smooth at $N$ if and only if
$a(z)e^{ik\varphi}$ is smooth near $N$, where $k=2\sigma+2$, and at $S$ if and only if $a(z)e^{ik'\varphi}$ is smooth near $S$, where $k'=2\sigma-2$. Off the
poles $T$ is smooth. Consider $N$ and put $m:=|k|/2=|\sigma+1|$. Near $N$, $z=\sqrt{1-r^2}$, so by Theorem~\ref{A:thm:fourier} the function is smooth
if and only if $a(\sqrt{1-u})=u^mg(u)$ for small $u=r^2\ge0$, with $g$ smooth near $0$. Since $u=(1-z)(1+z)$ and $1-\sqrt{1-u}=u/(1+\sqrt{1-u})$: if
$a=(1-z)^mb$, then $a(\sqrt{1-u})=u^m(1+\sqrt{1-u})^{-m}b(\sqrt{1-u})$; conversely, if $a(\sqrt{1-u})=u^mg(u)$, then $a(z)=(1-z)^m\,b(z)$ with
$b(z):=(1+z)^mg(1-z^2)$ smooth near $z=1$. At $S$, $z=-\sqrt{1-r^2}$, and the same argument with $1+z$ in place of $1-z$ and $m=|\sigma-1|$ gives
the second condition.

For the metric, $\gac\exp(T)$ stands for $\gac\exp(A)$ with $A:=\gac^{-1}T$, so that $\gac A=T$. It remains to show that, for a neighbourhood $U$ of a pole, $\gdc$ extends smoothly to $U$ if and only if $T$
does. If $T$ extends smoothly to $U$, then so does $A$, and the implication (ii)$\Rightarrow$(i) of
Theorem~\ref{A:thm:metriclog}, which only uses the smoothness of $A$, shows that $\gdc$ extends smoothly to $U$. Conversely, if $\gdc$ extends
smoothly to $U$, then by continuity the extension is positive semidefinite with $\det(\gac^{-1}\gdc)=1$, hence positive definite and unimodular,
and by Theorem~\ref{A:thm:metriclog}, (i)$\Rightarrow$(iii), its logarithm is a smooth tracefree symmetric form on $U$ which coincides with $T$ off the
pole. Hence $\gdc$ is smooth at a pole if and only if $T$ is.
\end{proof}

The phase in Lemma~\ref{lem:poles} is the spin weight at work: at $N$ only $\sigma=-1$ (a constant Cartesian quadrupole, whose polar frame
matrix is $R_\varphi^{-1}\mathrm{diag}(1,-1)R_\varphi=S(-2\varphi)$) needs no vanishing factor, and at $S$ only $\sigma=+1$.

For Section~\ref{ch:linear} we also need the criterion for general Fourier modes with amplitudes that are smooth functions of $\vartheta$
rather than of $z$. A complex function $\tau$ on $S^2\setminus\{N,S\}$ is said to define a smooth field of spin weight $s\in\{0,1,2\}$ if $\tau=T_\C$
for a field $T$ as in Theorem~\ref{A:thm:spins} that extends smoothly to $S^2$, and to define a smooth field of spin weight $-s$ if
$\overline\tau$ defines a smooth field of spin weight $s$.

\begin{corollary}[Pole conditions for Fourier modes]\label{A:cor:modes}
Let $s\in\{-2,-1,0,1,2\}$, $m\in\Z$, and $\tau=a(\vartheta)e^{im\varphi}$ with $a$ smooth on $(0,\pi)$. Then $\tau$ defines a smooth field of spin
weight $s$ on $S^2$ if and only if
\[
 a(\vartheta)=\vartheta^{|m+s|}f_N(\vartheta^2)\ \text{near }\vartheta=0\qquad\text{and}\qquad a(\pi-t)=t^{|m-s|}f_S(t^2)\ \text{near }t=0,
\]
with $f_N,f_S$ smooth on an open interval containing $0$.
\end{corollary}
\begin{proof}
Off the poles the field is smooth, since $\tau$ is smooth in $(\vartheta,\varphi)$. For $s\ge0$, Theorem~\ref{A:thm:spins} says that the field is smooth
at $N$ if and only if $a(\vartheta)e^{i(m+s)\varphi}$ is smooth near $N$, and at $S$ if and only if $a(\vartheta)e^{i(m-s)\varphi}$ is smooth near $S$. For
$s<0$ we apply Theorem~\ref{A:thm:spins} to $\overline\tau=\overline a\,e^{-im\varphi}$ with spin weight $|s|=-s$: the functions are
$\overline a\,e^{-i(m+s)\varphi}$ at $N$ and $\overline a\,e^{-i(m-s)\varphi}$ at $S$, which are smooth if and only if their complex conjugates are. So in all
cases the conditions are the smoothness of $a(\vartheta)e^{ik_N\varphi}$ near $N$ and of $a(\vartheta)e^{ik_S\varphi}$ near $S$, with $k_N=m+s$, $k_S=m-s$.

Consider $N$ and put $k:=|k_N|$. In the chart at $N$, $r=\sin\vartheta$ and $\vartheta=\arcsin r$ for $\vartheta<\frac\pi2$, so the function is $f(r)e^{ik_N\varphi}$
with $f(r)=a(\arcsin r)$, and by Theorem~\ref{A:thm:fourier} it is smooth near $N$ if and only if $f(r)=r^kg(r^2)$ for small $r>0$ with $g$ smooth
near $0$. Now $\sin\vartheta=\vartheta\,c(\vartheta^2)$ and $\arcsin r=r\,d(r^2)$, where $c(u)=\sum_{n\ge0}(-u)^n/(2n+1)!$ and $d$ (the power series of
$\arcsin r/r$ in $r^2$, convergent for $|u|<1$) are smooth and positive near $u=0$. Hence, for small $\vartheta>0$ and $r=\sin\vartheta$: if
$f(r)=r^kg(r^2)$, then $a(\vartheta)=\vartheta^kf_N(\vartheta^2)$ with $f_N(u):=c(u)^k\,g\big(u\,c(u)^2\big)$; conversely, if $a(\vartheta)=\vartheta^kf_N(\vartheta^2)$,
then $f(r)=r^kg(r^2)$ with $g(u):=d(u)^k\,f_N\big(u\,d(u)^2\big)$. In both cases the new function is smooth near $0$. At $S$ the same argument
applies with $t=\pi-\vartheta$: in the chart at $S$, $r=\sin\vartheta=\sin t$ and $t=\arcsin r$ for $t<\frac\pi2$, while the polar angle of $w$ is
again $\varphi$; Theorem~\ref{A:thm:fourier} with $k=|k_S|$ gives the second condition.
\end{proof}

\section{The non-flat second cutoff}\label{app:cutoff}

This appendix proves the cutoff lemma used in Step~4 of Section~\ref{tl:sec:proof}. It is the local modification of \cite[Lemma~5.1]{KU22}
and of the cutoff step in \cite[Section~4]{KU}. Its quantitative non-flat form is proved here; it is not attributed to a stated result in
those papers.

\begin{lemma}\label{tl:lem:cap}
Let $\bg_b$ solve the reduced Einstein equations on a fixed wave-coordinate slab with spacelike time slices, with regularity
$H^s\times H^{s-1}$ for all $s<7/2$. Let $B$ be a smooth compact ball in $\{t=0\}$ such that $K:=\overline{D_b^-(B^\circ)}$ is compact in the
slab and lies in an open vacuum region, and let the background data be smooth in a collar of $\partial B$. Suppose that vacuum Cauchy data on
$B$ and an exterior collar
\begin{enumerate}[label=\textup{(\alph*)}]
\item equal the background data on $B$,
\item match them to orders two (metric) and one (second fundamental form) at $\partial B$, and
\item have the one-sided derivative reserve of Proposition~\ref{tl:prop:input}\textup{(c)}.
\end{enumerate}
Then they have a vacuum development on an open neighbourhood of the whole closed set $K$, which agrees with $\bg_b$ on $K$ and with the
prescribed data near $B$, in the same fractional Cauchy class.
\end{lemma}

The idea is to cut off the difference between the new and the background data in a collar of width $\epsilon$ and to evolve. The cut-off data
satisfy the constraints only on $B_\epsilon:=B\cup\{\text{collar of width }\epsilon\}$, so the solution is vacuum only on the domain of dependence
of $B_\epsilon$. To keep the closed past domain $K$ of $B$ inside it, the new light cones must move by much less than $\epsilon$. A mere modulus
of continuous dependence does not give this; a rate $o(\epsilon)$ does.

\begin{proof}
\emph{Step 1: cutoff rates.} Let $f=g_*-g_b$ and $h=k_*-k_b$, extended by zero into $B$. In terms of the signed distance $x\ge0$ outside $B$, the
vanishing jets give $|\partial^jf|\le Cx^{3-j}$ for $j\le3$ and $|\partial^jh|\le Cx^{2-j}$ for $j\le2$. Let $\theta_\epsilon$ be equal to one on the
exterior collar of width $\epsilon$ and to zero beyond width $2\epsilon$, with $|\partial^i\theta_\epsilon|\le C\epsilon^{-i}$. On the support
($x\le2\epsilon$, volume $O(\epsilon)$) every term of $\partial^j(\theta_\epsilon f)$ is $O(\epsilon^{3-j})$, so
\[
 \|\theta_\epsilon f\|_{H^j}\le C\epsilon^{7/2-j}\quad(0\le j\le3),\qquad \|\theta_\epsilon h\|_{H^j}\le C\epsilon^{5/2-j}\quad(0\le j\le2),
\]
and the same exponents hold at fractional orders by interpolation.

\emph{Step 2: a common existence slab.} We use the perturbed data for the reduced equations, keeping the background initial lapse and shift
and solving the wave-gauge conditions for the remaining initial velocities. For $5/2<r<3$ the perturbation tends to zero in
$H^r\times H^{r-1}$ (with exponent $7/2-r>0$), so Cauchy stability about the fixed background gives solutions $\bg_\epsilon$ on a common slab
containing $K$ \cite{HKM}.

\emph{Step 3: the rate.} One derivative lower, the initial perturbation is smaller by a factor $\epsilon$: it is $O(\epsilon^{9/2-r})$ in
$H^{r-1}\times H^{r-2}$. Subtract the two reduced wave equations, with $\bg_\epsilon^{\alpha\beta}\partial_\alpha\partial_\beta$ as principal part.
The source is a metric difference times $\partial^2\bg_b$ plus terms linear in the difference of first derivatives. Since $H^{r-1}$ is an algebra
and a multiplier on $H^{r-2}$, and the principal coefficients are uniformly Lipschitz, the lower-order wave energy estimate and Gronwall's
inequality give
\begin{equation}\label{tl:eq:rate}
\begin{gathered}
 \sup_t\Big(\|\bg_\epsilon-\bg_b\|_{H^{r-1}}+\|\partial_t(\bg_\epsilon-\bg_b)\|_{H^{r-2}}\Big)\le C\epsilon^{9/2-r},\\
 \text{hence}\qquad \|\bg_\epsilon-\bg_b\|_{C^0}\le C\epsilon^{9/2-r}=o(\epsilon),
\end{gathered}
\end{equation}
using $H^{r-1}\subset C^0$ ($r-1>3/2$) and $9/2-r>1$. For example, $r=11/4$ gives $\epsilon^{7/4}$.

\emph{Step 4: the light cones move by $o(\epsilon)$.} In lapse--shift form, the causal curves parametrised by $t$ are exactly the solutions of
\[
 \dot x=-\beta+N\gamma^{-1/2}a(t),\qquad |a(t)|\le1,
\]
for measurable controls $a$. The coefficients are uniformly Lipschitz in $x$ and, by \eqref{tl:eq:rate}, differ from the background ones by
$o(\epsilon)$ in $C^0$. Comparing, by Gronwall's inequality, the two solutions with the same control and the same initial point, we find that all
future endpoints at $t=0$ of causal curves from $K$ lie within $o(\epsilon)$ of $B$, hence strictly inside $B_\epsilon$, where the new data
satisfy the constraints. Continuity in the starting point gives a margin on the past side, and a uniform speed bound puts a future collar of
the top ball into $D^+(B_\epsilon)$. Hence
\[
 K\subset\operatorname{Int}\big(D^-_{\bg_\epsilon}(B_\epsilon)\cup D^+_{\bg_\epsilon}(B_\epsilon)\big),
\]
and constraint propagation makes this open neighbourhood of $K$ vacuum.

\emph{Step 5: equality on $K$, overlap at the rim, regularity.} The data are unchanged on $B^\circ$, so finite-propagation uniqueness gives
$\bg_\epsilon=\bg_b$ on the background past domain of $B^\circ$, and by continuity on its closure $K$. (In a common gauge, put the background
wave operator on the left of the difference equation; its potential coefficient is $\partial^2\bg_\epsilon\in H^{r-2}_{\loc}\subset L^3_{\loc}$,
so local $H^1$ uniqueness follows from $H^1\subset L^6$.) Near $\partial B$ the prescribed collar data and the cut-off data agree on an open
annulus, which by Cauchy uniqueness gives an actual open overlap with the characteristic collar. Finally the fractional regularity
(Section~\ref{tl:sec:regularity}) propagates. The higher norms of the cut-off data need not be uniform in $\epsilon$; they are never used.
\end{proof}

\subsection{Why the hypotheses hold in the two-stage construction}\label{tl:app:place}

\emph{The margin.} The first, Minkowski-based use of the same argument supplies a fixed open vacuum margin around its closed past ball domain.
Its reduced solution exists on a larger fixed slab, even where the auxiliary data are not constrained.

\emph{The new ball.} Choose $C_B$ close enough to $C_0$ on the fixed early segment. Its bottom ball $B$ is a small perturbation of the original
one, and its closed past domain stays inside the margin. Indeed, the first reduced solution may be taken close to Minkowski on the fixed slab,
so uniform upper and lower bounds on its causal velocities put all these past domains in a common compact set. If they did not stay in the
margin, a limit point would lie outside the old closed domain; it then has a future causal curve to a point strictly outside the old ball;
perturbing that curve to a timelike one, the same escape persists for nearby balls, contradicting membership in their past domains.

\emph{The order.} After fixing this ball we choose the second bridge and its exterior strip, and only then the second cutoff width. The
background is smooth near the new rim (which lies outside the old gluing hypersurface), and has the hyperbolic regularity of
Lemma~\ref{tl:lem:cap} in the centre. The $C^2$ pasting supplies exactly the vanishing jets used in Step~1 of the proof of
Lemma~\ref{tl:lem:cap}.

\emph{No uniformity.} Nothing has to be uniform in $M_f$, in the bridge length, or in either strip width. Equality with the first spacetime
holds on the closed past domain $K$, not on its whole new neighbourhood.

\end{document}